\documentclass[reqno,hidelinks]{amsart}%hidelinfks to remove the red rectangle in the references

\usepackage{a4,amsmath,amssymb,amsthm}
\usepackage[top=1in, bottom=1in, left=1.25in, right=1.25in]{geometry}%set the margins;
\usepackage{xfrac,array}
\usepackage{tikz,slashed,graphicx}
\usetikzlibrary{snakes,arrows}
\usepackage{latexsym}
\usepackage{mathrsfs}
\usepackage[numbers]{natbib}
\usepackage{leftidx}
\usepackage{inputenc}

\usepackage{amsaddr}

\usepackage[T1]{fontenc}

\usepackage{leftidx}
\usepackage{tikz}

\usepackage{hyperref}
\newcounter{mnotecount}[section]

\usepackage{color}

\usepackage{url}

\numberwithin{equation}{section}

\usepackage[mathscr]{eucal}

\usepackage{graphicx}
\graphicspath{pic/}

\newcommand{\half}{\frac{1}{2}}
\newcommand{\Half}{1/2}        %
\newcommand{\veps}{\varepsilon}

\newcommand{\prb}{\partial_{\rb}}
\newcommand{\di}{\mathrm{d}} %{\text{d}}

\newcommand{\Donetwo}{\Omega_{\tb_1,\tb_2}}

\newcommand{\Dzeroinfty}{\Omega_{\tb_0,\infty}}

\newcommand{\sfrak}{\mathfrak{s}}

\newcommand{\psis}{\psi_{s} }
\newcommand{\psips}{\psi_{\sfrak} }
\newcommand{\psins}{\psi_{-\sfrak} }
\newcommand{\psipns}{\psi_{\pm\sfrak} }
\newcommand{\phis}{\phi_{s} }
\newcommand{\phips}{\phi_{\sfrak} }
\newcommand{\phins}{\phi_{-\sfrak} }

\newcommand{\Phis}{\Phi_{s} }
\newcommand{\Phips}{\Phi_{\sfrak} }
\newcommand{\Phins}{\Phi_{-\sfrak} }
\newcommand{\hatPhips}{\Phips^{(1)} }
\newcommand{\Psips}{\Psi_{\sfrak} }
\newcommand{\Psins}{\Psi_{-\sfrak} }
\newcommand{\Psinst}{\tilde{\Psi}_{-\sfrak}}
\newcommand{\Psipns}{\Psi_{\pm\sfrak} }
\newcommand{\Phipns}{\Phi_{\pm\sfrak} }

\providecommand{\PhipsHigh}[1]{\Phips^{(#1)}}
\providecommand{\PhinsHigh}[1]{\Phins^{(#1)}}
\providecommand{\PhipnsHigh}[1]{\Phipns^{(#1)}}
\providecommand{\tildePhipsHigh}[1]{\tilde{\Phi}_{\sfrak}^{(#1)}}
\providecommand{\tildePhinsHigh}[1]{\tilde{\Phi}_{-\sfrak}^{(#1)}}

\providecommand{\tildePhipsHighTI}[1]{\tilde{\Phi}_{\sfrak,TI}^{(#1)}}
\providecommand{\tildePhinsHighTI}[1]{\tilde{\Phi}_{-\sfrak,TI}^{(#1)}}

\newcommand{\iip}{i'}

\newcommand{\PigeonTime}{\timefunc}
\newcommand{\iPigeonReg}{\iip}
\newcommand{\timefunc}{\tb}

\newcommand{\Integers}{\mathbb{Z}}
\newcommand{\Reals}{\mathbb{R}}

\newcommand{\Lxi}{\mathcal{L}_{\xi}}
\newcommand{\Leta}{\mathcal{L}_{\eta}}

\newcommand{\TMEOps}{\widehat{\Box}_{g_M,\sfrak}}
\newcommand{\TMEOpns}{\widehat{\Box}_{g_M,-\sfrak}}

\newcommand{\curlL}[1]{\mathcal{L}_{\sf[#1]}}
\newcommand{\curlLd}[1]{\mathcal{L}^{\dagger}_{\sf[#1]}}
\newcommand{\TAO}{\mathbf{T}}

\newcommand{\curlLs}{\edthR'}
\newcommand{\curlLds}{\edthR}

\newcommand{\pu}{\partial_u}
\newcommand{\pv}{\partial_v}

\newcommand{\DOC}{\mathcal{D}}
\newcommand{\tb}{\tau}
\newcommand{\pb}{\phi}
\newcommand{\rb}{\rho}
\newcommand{\Hyper}{\Sigma}
\newcommand{\Sphere}{S^2}

\newcommand{\Sigmazero}{\Hyper_{\tb_0}}
\newcommand{\Sigmatb}{\Hyper_{\tb}}
\newcommand{\Sigmatwo}{\Hyper_{\tb_2}}
\newcommand{\Sigmaone}{\Hyper_{\tb_1}}

\newcommand{\Horizon}{\mathcal{H}^+}
\newcommand{\Scri}{\mathcal{I}^+}
\newcommand{\Horizononetwo}{\Horizon_{\tb_1,\tb_2}}
\newcommand{\Scrionetwo}{\Scri_{\tb_1,\tb_2}}

\newcommand{\CDeri}{\mathbb{D}}
\newcommand{\CDerit}{\tilde{\mathbb{D}}}

\newcommand{\PDeri}{\mathbb{B}}
\newcommand{\PSDeri}{\widetilde{\mathbb{B}}}
\newcommand{\RDeri}{\mathbb{H}}

\providecommand{\NPCP}[1]{\mathbb{Q}_{\sfrak}^{(#1)}}
\providecommand{\NPCN}[1]{\mathbb{Q}_{-\sfrak}^{(#1)}}
\providecommand{\NPCPTI}[1]{\mathbb{Q}_{{\sfrak}, TI}^{(#1)}}

\newcommand{\VR}{\hat{V}}
\newcommand{\curlVR}{\hat{\mathcal{V}}}
\newcommand{\edthR}{\mathring{\eth}}

\newcommand{\FB}{\mathbf{F}}
\newcommand{\FBT}{\tilde{\mathbf{F}}}

\newcommand{\InitEnerNonvanish}{\mathbf{I}^{\neq 0}_{\delta,\regl}}
\newcommand{\InitEnervanish}{\mathbf{I}^{=0}_{\delta,\regl}}

\providecommand{\abs}[1]{\lvert#1\rvert}
\providecommand{\norm}[1]{\lVert#1\rVert}

\providecommand{\absHighOrder}[3]{\abs{#1}_{#2,#3}}

\providecommand{\absCDeri}[2]{\absHighOrder{#1}{#2}{\CDeri}}

\providecommand{\absRDeri}[2]{\absHighOrder{#1}{#2}{\RDeri}}
\providecommand{\absCDerit}[2]{\absHighOrder{#1}{#2}{\CDerit}}

\makeatletter
  \def\moverlay{\mathpalette\mov@rlay}
  \def\mov@rlay#1#2{\leavevmode\vtop{%
     \baselineskip\z@skip \lineskiplimit-\maxdimen
     \ialign{\hfil$#1##$\hfil\cr#2\crcr}}}
\makeatother

\newcommand{\reg}{k}

\newcommand{\regl}{k'}

\theoremstyle{plain}
\newtheorem{thm}{Theorem}[section]
\newtheorem{cor}[thm]{Corollary}
\newtheorem{lemma}[thm]{Lemma}
\newtheorem{prop}[thm]{Proposition}

\theoremstyle{definition}
\newtheorem{definition}[thm]{Definition}

\newtheorem{remark}[thm]{Remark}

\title{Sharp decay estimates for massless Dirac fields on a Schwarzschild background}

\author[S. Ma and L.Zhang]{Siyuan Ma$^\dagger$ and Lin Zhang$^\star$}
\email{siyuan.ma@sorbonne-universite.fr, linzhang2013@pku.edu.cn}
\address{$^\dagger$Laboratoire Jacques-Louis Lions,
Sorbonne Universit\'{e} Campus Jussieu,
4 place Jussieu 75005 Paris, France.\\
$^\star$College of Mathematics and Statistics, Chongqing University, Chongqing 401331, China.}
\begin{document}
\allowdisplaybreaks

\begin{abstract}
We consider the explicit asymptotic profile of massless Dirac fields on a Schwarzschild background.
First, we prove for the spin $s=\pm \half$ components of the Dirac field a uniform bound of a positive definite energy and an integrated local energy decay estimate from a symmetric hyperbolic wave system. Based on these estimates, we further show that these components have globally pointwise decay $fv^{-3/2-s}\tau^{-5/2+s}$ as both an upper and a lower bound outside the black hole, with function $f$ finite and explicitly expressed in terms of the initial data and the coordinates. This establishes the validity of the conjectured Price's law for massless Dirac fields outside a Schwarzschild black hole.
\end{abstract}

\maketitle
\tableofcontents

%%%%%%%%%%%%%%%
\section{Introduction}
\label{sect:intro}
%%%%%%%%%%%%%%%%%%

In this work, we consider the asymptotics of massless Dirac fields on a Schwarzschild black hole background. Our motivation arises from its relevance to many fundamental problems in classic General Relativity, as this model is closely tied to the black hole stability problem, Strong Cosmic Censorship conjecture, and a complete mathematical understanding of the Hawking radiation, etc.

The metric of a Schwarzschild black hole spacetime \cite{schw1916}, when written in Boyer--Lindquist (B--L) coordinates $(t,r,\theta,\phi)$ \cite{boyer:lindquist:1967}, takes the form of
\begin{align}\label{eq:SchwMetricBoyerLindquistCoord}
g_{M}= & -\mu \di t^2 +\mu^{-1} \di r^2 + r^2 (\di \theta^2 +\sin^2\theta\di \phi^2),
\end{align}
where the function
$\mu=\mu(r,M)=\Delta r^{-2}$ with $\Delta=\Delta(r,M)= r^2 -2Mr$ and $M$ being the mass of the black hole. The larger root $r=2M$ of the function $\Delta$ is the location of the event horizon $\mathcal{H}$, and denote the domain of outer communication (DOC) of a Schwarzschild black hole spacetime as
\begin{equation}\label{def:DOC}
\mathcal{D}=\overline{\{(t,r,\theta,\phi)\in \mathbb{R}\times (2M,\infty)\times \mathbb{S}^2\}}.
\end{equation}
We focus on the future development, hence only the future part of the event horizon, called the future event horizon and denoted as $\Horizon$, is relevant.

The governing equations of massless Dirac fields describe the movement of sourceless neutrino, with no coupling to electrons or muons. These Dirac equations take the form of
\begin{align}
\label{eq:Dirac:spinorform}
\nabla^{AA'}\Phi_{A}=0,
\end{align}
where $\Phi_{A}$ is a two-component spinor.
Choose a Hartle--Hawking null tetrad \cite{HHtetrad72} which is regular at $\Horizon$ and reads in B--L coordinates $(t,r,\theta,\phi)$: \begin{align}\label{eq:HartleHawkingtetrad}
l^\mu &= \half (1 , \mu , 0 , 0), &
n^\mu &= (\mu^{-1}, - 1 , 0 , 0), \notag\\
m^\mu &= \frac{1}{\sqrt{2}r}\left(0,0 , 1, i\csc\theta\right),&
(\bar{m})^{\mu}&= \frac{1}{\sqrt{2}r}\left(0,0 , 1, -i\csc\theta\right).
\end{align}
The vectors $(l^{\mu}, n^\mu, m^\mu, \bar{m}^\mu)$ form a null frame and satisfy
$g_M(m,\bar{m})=g_M(l, n)=-2$ and all other inner products being zero. 
Since the DOC, endowed with a Schwarzschild metric $g_M$, is globally hyperbolic, it is a spin manifold and admits a spin-structure.
We denote by $\mathbb{S}$  the spinor bundle over the DOC, where the spinor space at each point of the  DOC is $\mathbb{C}^2$ with the vector representation of $\mathrm{SL}(2,\mathbb{C})$, and the complex conjugate spinor bundle is denoted by $\overline{\mathbb{S}}$. The vector space $\mathbb{R}^4\otimes\mathbb{C}$ and the spinor space $\mathbb{C}^2\otimes\overline{\mathbb{C}^2}$ have a nature isomorphism correspondence via the soldering form
$g^\mu_{\ AA'}$, e.g. $\nu^{\mu}=g^{\mu}_{\ AA'}\nu^{AA'}$,  and this correspondence can be extended to a correspondence between $\mathrm{T}\mathcal{M}\otimes\mathbb{C}$\footnote{In our situation, $\mathcal{M}$ is the DOC of the Schwarzschild spacetime and $\mathrm{T}\mathcal{M}$ is its tangent bundle.}  and spinor bundle $\mathbb{S}\otimes\overline{\mathbb{S}}$. See \cite[Section 3]{penroserindlerI}.
A local basis $\{o^A, \iota^A\}$ of the spinor bundle $\mathbb{S}$ satisfying $o_A\iota^A=1$ is called a pair of dyad legs. Associated to the Hartle--Hawking null tetrad, there exists a pair of  dyad legs $\{o^A, \iota^A\}$, unique modulo an overall sign change, such that\footnote{In fact, such a pair of dyad legs can be associated to any Newman--Penrose null tetrad.}
\begin{align}
\label{def:oandiota}
l^\mu= o^A \bar{o}^{A'}, \quad n^{\mu}= \iota^A \bar{\iota}^{A'}, \quad m^{\mu}=o^A\bar{\iota}^{A'}, \quad \bar{m}^{\mu}=\iota^A\bar{o}^{A'},
\end{align}
where \textquotedblleft{$=$\textquotedblright} is with respect to the soldering correspondence.
Let $\chi_0$ and $\chi_1$ be the components of $\Phi_A$ along dyad legs $o^A$ and $\iota^A$
\begin{align}
\label{def:compsofDiracspinor}
\chi_0=\Phi_{A}o^{A}, \quad \chi_1=\Phi_A \iota^A.
\end{align}
These two components $\chi_0$ and $\chi_1$ are spin-weight $\half$ and $-\half$ scalars, respectively.\footnote{For more about the Dirac field, see \cite{chandrasekhar1976solution} and \cite[Chapter 10]{chandbook1998}.}
Unless otherwise stated, we shall throughout the paper denote $s$  the spin-weight $\pm \frac{1}{2}$ and $\sfrak$ its absolute value $\frac{1}{2}$.

Define the Teukolsky scalars of Dirac field as
\begin{align}
\label{def:psipns:Schw}
\psis=\left\{
             \begin{array}{ll}
              r\chi_0, & \hbox{$s=\half$;} \\
               2^{-\half}\chi_1, & \hbox{$s=-\half$.}
             \end{array}
           \right.
\end{align}
As shown in Appendix \ref{app:TMEandDiracEq}, the Dirac equations \eqref{eq:Dirac:spinorform} on Schwarzschild simplify to
\begin{subequations}\label{eq:Dirac:TMEscalar}
\begin{align}
\curlLs   \psips={}&
(\Delta^{\Half}\VR)(\Delta^{\Half}\psins),\\
\curlLds \psins={}&
Y\psips,
\end{align}
\end{subequations}
where $Y$ and $\VR$ are
two future-directed ingoing and outgoing null vectors in B-L coordinates
\begin{subequations}
\begin{align}\label{def:VectorFieldYandV}
Y&= \mu^{-1}\partial_t -\partial_r, \ &\ \VR&= \mu^{-1}\partial_t+\partial_r,
\end{align}
and
 $\edthR$ and $\edthR'$ are the spherical edth operators defined, when acting on a spin-weight $s$ scalar $\varphi$, by
\begin{align}
\label{def:edthRandedthR'}
\edthR\varphi={}&\partial_{\theta}\varphi
+i\csc\theta\partial_{\phi}\varphi
-s\cot\theta\varphi, &
\edthR'\varphi={}&\partial_{\theta}\varphi
-i\csc\theta\partial_{\phi}\varphi
+s\cot\theta\varphi.
\end{align}
\end{subequations}

 Define additionally a tortoise coordinate $r^*$ by
\begin{align}
\label{def:r*coordinate}
\di r^*=\mu^{-1}\di r,\qquad r^*(3M)=0.
\end{align}
It is convenient to introduce double null coordinates $(u,v,\theta,\phi)$, where $u=t-r^*$ and $v=t+r^*$. Thus $\pu=\half\mu Y$ and $\pv=\half \mu \VR$. Define additionally a function $h=h(r)$ and a hyperboloidal coordinate system $(\tb,\rb,\theta,\phi)$ as in \cite{andersson2019stability} where $\tb=v-h$. In particular, the function $h$ satisfies $\lim\limits_{r\to r_+}h =r_+$, $\lim\limits_{r\to r_+} \partial_rh=1$, $\partial_r h \geq 0$ for $r\geq r_+$, $h=r^*$ for $r\in [r_{away}, R]$ where $r_{away}$ is away from horizon location $r=2M$ and $R/M$ is a large constant, and $1\lesssim \lim\limits_{r\to\infty}M^{-2}r^2 (\partial_rh - 2\mu^{-1})\vert_{\Sigmatb}<\infty$. See Figure \ref{fig:2}.

%%%%%%%%%%%%%%%%%%%%%%%%%%%%
\begin{figure}[htbp]
\begin{minipage}[t]{0.5\linewidth}
\begin{center}
\begin{tikzpicture}[scale=0.8]
\draw[thin]
(0,0)--(2.45,2.45);
\draw[very thin]
(2.5,2.5) circle (0.05);
\coordinate [label=90:$i_+$] (a) at (2.5,2.5);
\draw[thin]
(0,0)--(2.45,-2.45);
\draw[dashed]
(2.55,2.45)--(4.95,0.05);
\draw[very thin]
(5,0) circle (0.05);
\coordinate [label=360:$i_0$] (a) at (5,0);
\draw[dashed]
(4.95,-0.05)--(2.55,-2.45);
\draw[very thin]
(2.5,-2.5) circle (0.05);
\coordinate [label=270:$i_-$] (a) at (2.5,-2.5);
\draw[thin]
(0.9,0.9) arc (215:320:2.1 and 1.7);
\node at (2.5,-0.15) {\small $\Sigma_{\tau_1}$};
\draw[thin]
(1.5,1.5) arc (212:323:1.25 and 0.9);
\node at (2.5,1.35) {\small $\Sigma_{\tau_2}$};
\draw[very thick]
(0.9,0.9)--(1.5,1.5);
\draw[dashed, very thick]
(3.57,1.43)--(4.28,0.72);
\node at (0.95,1.45) [rotate=45] {\small $\mathcal{H}_{\tau_1,\tau_2}^+$};
\node at (4.15,1.35) [rotate=-45] {\small $\mathcal{I}_{\tau_1,\tau_2}^+$};
\node at (2.5,0.7) {\small $\Omega_{\tau_1,\tau_2}$};
\end{tikzpicture}
\end{center}
\caption{Hyperboloidal foliation and some related definitions.}
\label{fig:2}
\end{minipage}%
\begin{minipage}[t]{0.5\linewidth}
\begin{center}
\begin{tikzpicture}[scale=0.8]
\draw[thin]
(0,0)--(2.45,2.45);
\draw[very thin]
(2.5,2.5) circle (0.05);
\coordinate [label=90:$i_+$] (a) at (2.5,2.5);
\draw[thin]
(0,0)--(2.45,-2.45);
\draw[dashed]
(2.55,2.45)--(4.95,0.05);
\draw[very thin]
(5,0) circle (0.05);
\coordinate [label=360:$i_0$] (a) at (5,0);
\draw[dashed]
(4.95,-0.05)--(2.55,-2.45);
\draw[very thin]
(2.5,-2.5) circle (0.05);
\coordinate [label=270:$i_-$] (a) at (2.5,-2.5);
\draw[thin]
(0.6,0.6) arc (209:335:2.1 and 1.1);
\node at (1.08,1.42) [rotate=45] {\small $\mathcal{H}^+$};
\node at (3.9,1.4) [rotate=-45] {\small $\mathcal{I}^+$};
\node at (2.5,-0.25) {\small $\Sigma_{\tau_0}$};
\end{tikzpicture}
\end{center}
\caption{Initial hypersurface $\Sigma_{\tau_0}$.}
\label{fig:1}
\end{minipage}
\end{figure}
%%%%%%%%%%%%%%%%%%%%%%%%%%%%%%%%%%%%%%

Let $\tb_0\geq 1$, and define for any $\tb_0\leq \tb_1<\tb_2$,
\begin{subequations}
\label{def:domainnotations}
\begin{align}
&\Sigmaone={}\{(\tb, \rb, \theta, \pb)|\tb=\tb_1\}\cap \DOC, \quad \Donetwo={}\bigcup_{\tb\in [\tb_1,\tb_2]}\Sigmatb,\\
&\Scrionetwo={}\lim_{c\to \infty}\{v=c\}\cap \Donetwo, \quad
\Horizononetwo={}\Donetwo\cap\Horizon.
\end{align}
\end{subequations}
We fix $\tb_0$ by requiring  $v\geq M$ on $\Sigmazero$ such that $v\geq c (\tb+\rb)$ in $\Omega_{\tb_0,\infty}$. See Figure \ref{fig:1}. The hypersurface $\Sigmazero$ will be our initial hypersurface on which the initial data are imposed.
The level sets of the time function $\tb$ are strictly spacelike with
\begin{align}
c(M)r^{-2}\leq -g(\nabla \tb,\nabla\tb)\leq C(M) r^{-2}
\end{align}
for two positive universal constants $c(M)$ and $C(M)$,
and they cross the future event horizon regularly, and for large $r$, the level sets of $\tb$ are asymptotic to future null infinity $\Scri$.

Throughout this work, we \underline{always} assume that the initial data on $\Sigmazero$, i.e.  the spin $\pm \half$ components $\psipns$ on $\Sigmazero$, are smooth in a regular coordinate system, for instance, the ingoing Eddington--Finkelstein coordinate system. By standard theory of global well-posedness of linear symmetric hyperbolic systems, the components $\psipns$ are globally smooth up to and including $\Horizon$.

We shall need to decompose the spin $\pm \half$ components into $\ell\geq 2$ part and $\ell=1$ mode, and decompose further $\ell=1$ mode into $(m,\ell=1)$ modes in terms of spin-weighted spherical harmonics, where $m=-\half, \half$, cf. Section \ref{sect:decompIntoModes}.
Let $F^{(2)}(\reg,p,\tb,(\Psipns)^{\ell\geq 2})$ and $F^{(1)}(\reg,p,\tb,(\Psipns)^{\ell=1})$ be defined as in Definition \ref{def:Fenergies:big2} by simply replacing $\Psipns$ therein by $(\Psipns)^{\ell\geq 2}$ and $(\Psipns)^{\ell=1}$, respectively. In the end, define
\begin{subequations}
\begin{align}
\label{def:NPCPell=1:positive:mainthm}
\NPCP{1}(m,\ell=1)={}&\lim\limits_{\rb\to\infty}\rb^2 \VR
(\mu^{-\half}r\psips(m,\ell=1)),\\
\label{def:varphisfrak}
\varphi_{\sfrak}={}&(r-M)^{-1}\psips.
\end{align}
\end{subequations}

\begin{thm}
\label{thm:pricelaw:nonzeroNP} (Price's law in nonvanishing first Newman--Penrose constant case)
Let $j\in \mathbb{N}$. Assume 
\begin{enumerate}
\item $\sum\limits_{m=\pm\half}\abs{\NPCP{1}(m,\ell=1)}\neq 0$;
\item there are constants $\beta\in (0,\half)$ and $D_0\geq 0$ such that for all $0\leq i\leq j$, 
\begin{align}
\label{assump:Pricelaw:nonzero}
\sum_{m=\pm\half}\sup_{\Sigmazero\cap\{\rb\geq 4M\}}\bigg|\rb^i\prb^i\bigg(\rb^2\VR(\mu^{-\half}r\psips(m,\ell=1))
-{\NPCP{1}(m,\ell=1)}\bigg)\bigg|\lesssim \rb^{-\beta}D_0;
\end{align}
\item the following initial bound for a suitably small $\delta\in (0,\half)$ and a suitably large $\regl=\regl(j)$
\begin{align}
\InitEnerNonvanish={}&
\big(F^{(1)}(\regl,3-\delta,\tb_0,(\Psipns)^{\ell=1})\big)^{\half}
+\big(F^{(2)}(\regl,1+\delta,\tb_0,(\Psipns)^{\ell\geq 2})\big)^{\half}
\notag\\
&
+\sum_{m=\pm\half}\abs{\NPCP{1}(m,\ell=1)}+D_0<\infty.
\end{align}
\end{enumerate}
Then there exists an $\epsilon>0$ such that in $\Dzeroinfty$,
\begin{subequations}
\label{eq:Pricelaw:nonzero}
\begin{align}
\bigg|\partial_{\tb}^j \varphi_{\sfrak}-c_{\sfrak,j}v^{-2}\tb^{-1-j}
\sum_{m=\pm\half}\NPCP{1}(m,\ell=1)
Y_{m,\ell=1}^{\sfrak}(\cos\theta)e^{im\pb}\bigg|
\lesssim_{j,\delta} {}& v^{-2}\tb^{-1-j-\epsilon}
\InitEnerNonvanish, \\
\bigg|\partial_{\tb}^j \psins
-c_{-\sfrak,j}v^{-1}\tb^{-2-j}
\sum_{m=\pm\half}\NPCP{1}(m,\ell=1)
Y_{m,\ell=1}^{-\sfrak}(\cos\theta)e^{im\pb}\bigg|
\lesssim_{j,\delta} {}& v^{-1}\tb^{-2-j-\epsilon}\InitEnerNonvanish,
\end{align}
\end{subequations}
where
\begin{subequations}
\label{def:cpnsj}
\begin{align}
c_{\sfrak,j}={}&4 (-1)^j j! \sum_{n=0}^{j}\sum_{i=0}^{n}
\bigg(\frac{\tb}{v}\bigg)^{j-i},\\
c_{-\sfrak,j}={}&4(-1)^j j!
\bigg[
(j+2)\sum_{n=0}^j \bigg(\frac{\tb}{v}\bigg)^{j-n}
-\sum_{n=0}^j \sum_{i=0}^n\bigg(\frac{\tb}{v}\bigg)^{j-i}
+(j+1)\bigg(\bigg(\frac{\tb}{v}\bigg)^{j}
-\bigg(\frac{\tb}{v}\bigg)^{j+2}\bigg)
\bigg].
\end{align}
\end{subequations}
\end{thm}

\begin{remark}
\begin{itemize}
\item
Each of the constants $\{\NPCP{1}(m,\ell=1)\}_{m=\pm\half}$ is the first Newman--Penrose constant for the $(m,\ell=1)$ mode of the spin $\half$ component, and the first assumption justifies that we are considering the nonvanishing first Newman--Penrose constant case. Thus, this result determines the leading asymptotics of the spin $\pm \half$ components from the first Newman--Penrose constants in nonvanishing first Newman--Penrose constant case. If the first Newman--Penrose constants are vanishing, the decay rates can be improved, cf. Theorem \ref{thm:pricelaw:zeroNP}.

\item 
In view of the definition \eqref{def:NPCPell=1:positive:mainthm}, the second assumption characterizes the decay of the scalar $r^2 \VR
(\mu^{-\half}r\psips(m,\ell=1))$ towards its limit $\NPCP{1}(m,\ell=1)$ along $\Sigmazero$. The third assumption contains a natural bound for the constant $D_0$ and the Newman--Penrose constants $\{\NPCP{1}(m,\ell=1)\}_{m=\pm\half}$ and a bound for an initial weighted energy.
\end{itemize}
\end{remark}

\begin{remark}
We remark that the peeling property of massless Dirac fields in a Schwarzschild spacetime is proved and contained in the above theorem. The assumptions can in fact be weaken as can be seen in Section \ref{sect:almostsharpdecayestimates}, but we shall not discuss it further here.
\end{remark}

It is clear from the above theorem that if the first Newman--Penrose constants for all $(m,\ell=1)$ modes vanish, then the scalars $\varphi_{\sfrak}$ and $\psins$ will have faster decay in $\tb$. This is precisely what we will obtain in the theorem below.
To state our main result about the Price's law in the case of vanishing first Newman--Penrose constant, we shall need the following notations and definitions. Define for any spin-weight $\half$ scalar $\varphi$
\begin{align}
\label{def:tildeHs:general}
\tilde{H}_{\sfrak}(\varphi)=(r-M)[r\mu^{\half}(2\mu^{-1}-\partial_r h)\partial_r h\partial_{\tb}\varphi
+2r\mu^{\half}(-\mu^{-1}+\partial_r h)\partial_\rho\varphi
+\partial_r(\Delta^{\half}\partial_r h)\varphi].
\end{align}
We decompose the spin $\pm \half$ components into $\ell\geq 3$ part, $\ell=2$ mode and $\ell=1$ mode, and decompose further $\ell=1$ mode into $(m,\ell=1)$ modes in terms of spin-weighted spherical harmonics, where $m=-\half, \half$, cf. Section \ref{sect:decompIntoModes}.
Let $F^{(3)}(\reg,p,\tb,(\Psipns)^{\ell\geq 3})$, $F^{(2)}(\reg,p,\tb,(\Psipns)^{\ell=2})$ and $F^{(1)}(\reg,p,\tb,(\Psipns)^{\ell=1})$ be defined as in Definition \ref{def:Fenergies:big2} by simply replacing $\Psipns$ therein by $(\Psipns)^{\ell\geq 3}$, $(\Psipns)^{\ell=2}$ and $(\Psipns)^{\ell=1}$, respectively.

\begin{thm}
\label{thm:pricelaw:zeroNP} (Price's law for vanishing first Newman--Penrose constant case)
Let $j\in \mathbb{N}$. Assume
\begin{enumerate}
\item there are constants $\beta\in (0,\half)$, $\tilde{D}_0\geq 0$, and $\{\tilde{D}_1(m,\ell=1)\}_{m=\pm\half}$  such that for all $0\leq i\leq j$,
\begin{align}
\label{assump:Pricelaw:zero}
\sum_{m=\pm\half}\sup_{\Sigmazero\cap\{\rb\geq 4M\}}\bigg|\rb^i\prb^i\bigg(\rb^3\VR(\mu^{-\half}r\psips(m,\ell=1))
-{\tilde{D}_1(m,\ell=1)}\bigg)\bigg|\lesssim \rb^{-\beta}\tilde{D}_0;
\end{align}
\item the following initial bound for a suitably small $\delta\in (0,\half)$ and a suitably large $\regl=\regl(j)$
\begin{align}
\label{thm:Pricelaw:zero:assump2}
\InitEnervanish={}&
\big(F^{(1)}(\regl,5-\delta,\tb_0,(\Psipns)^{\ell=1})\big)^{\half}
+\big(F^{(2)}(\regl(j),3+\delta,\tb_0,(\Psipns)^{\ell=2})\big)^{\half}
 \notag\\
&
+\big(F^{(3)}(\regl(j),1+\delta,\tb_0,(\Psipns)^{\ell\geq 3})\big)^{\half}+\sum_{m=\pm\half}\abs{\tilde{D}_1(m,\ell=1)}+ \tilde{D}_0<\infty.
\end{align}
\end{enumerate}
Then there exists an $\epsilon>0$ such that in $\Dzeroinfty$,
\begin{subequations}
\label{eq:Pricelaw:zero}
\begin{align}
\bigg|\partial_{\tb}^j \varphi_{\sfrak}-c_{\sfrak,j+1}v^{-2}\tb^{-2-j}
\sum_{m=\pm\half}\NPCPTI{1}(m,\ell=1)
Y_{m,\ell=1}^{\sfrak}(\cos\theta)e^{im\pb}\bigg|
\lesssim_{j,\delta} {}& v^{-2}\tb^{-2-j-\epsilon}
\InitEnervanish, \\
\bigg|\partial_{\tb}^j \psins
-c_{-\sfrak,j+1}v^{-1}\tb^{-3-j}
\sum_{m=\pm\half}\NPCPTI{1}(m,\ell=1)
Y_{m,\ell=1}^{-\sfrak}(\cos\theta)e^{im\pb}\bigg|
\lesssim_{j,\delta} {}& v^{-1}\tb^{-3-j-\epsilon}\InitEnervanish,
\end{align}
\end{subequations}
where $c_{\sfrak,j+1}$ and $c_{-\sfrak,j+1}$ are defined as in Definition \ref{def:cpnsj}, and for each $m=-\half,\half$,
\begin{align}
\label{expression:NPCPTI:mainthm}
\NPCPTI{1}(m,\ell=1)=M\int_{2M}^\infty
\tilde{H}_{\sfrak}(\Phips(m,\ell=1))(\tb_0,\rho')\di\rho'
-\frac{2}{3}\tilde{D}_1(m,\ell=1).
\end{align}
\end{thm}

\begin{remark}
\begin{itemize}
\item
The first assumption \eqref{assump:Pricelaw:zero} actually implies that the first Newman--Penrose constant of the $\ell=1$ mode vanishes, and the second assumption is again a boundedness assumption for the constant $\tilde{D}_0$, constants $\{\tilde{D}_1(m,\ell=1)\}_{m=\pm\half}$ which are the limits of $\{r^3\VR(\mu^{-\half}r\psips(m,\ell=1))\}_{m=\pm\half}$ along $\Sigmazero$ and shall be compared to the first Newman--Penrose constants in Theorem \ref{thm:pricelaw:nonzeroNP}, and an initial energy.

 In particular, if the initial data is compactly supported on a spacelike hypersurface terminating at spacelike infinity, these assumptions hold because of the finite speed of propagation and the constants $\tilde{D}_0=0$ and all $\tilde{D}_1(m,\ell=1)=0$, therefore, the above sharp decay estimates \eqref{eq:Pricelaw:zero} are clearly valid.
\item
Furthermore, if the initial data are imposed on a Boyer--Lindquist $t=const$ hypersurface and compactly supported away  from both the bifurcation sphere and spatial infinity, the above decay rates can be improved if and only if all $\NPCPTI{1}(m,\ell=1)=M\int_{2M}^\infty
\tilde{H}_{\sfrak}(\Phips(m,\ell=1))\vert_{\Sigma_{t=const}}\di r=0$ are vanishing. In this case, we have $\partial_r h=\mu^{-1}$ on $t=const$ hypersurface such that $\tau=t$ when away from both the horizon and infinity, and substituting this into the expression \eqref{def:tildeHs:general} gives
$\tilde{H}_{\sfrak}(\varphi)=\mu^{-\frac{3}{2}}(r-M)(r\partial_{\tb}\varphi
+r^{-1}(r-3M)\varphi)$. Hence, the above decay rates can be improved if and only if 
\begin{align}
\quad 0={}&\NPCPTI{1}(m,\ell=1)\notag\\
={}&M\int_{2M}^{\infty} \mu^{-\frac{3}{2}} r(r-M) (\partial_t \psips(m,\ell=1) +r^{-2}(r-3M) \psips(m,\ell=1))\vert_{\Sigma_{t=const}}\di r
\end{align}
holds
for all $(m,\ell=1)$ modes. This is \emph{in contrast to} the case of scalar field $\varphi_{scalar}$ where initially static data ($\partial_t \varphi_{scalar}^{\ell=0}\vert_{\Sigma_{t=const}}=0$) lead to extra time decay in the future development as shown in \cite{angelopoulos2018late,hintz2020sharp}.
\end{itemize}
\end{remark}

\begin{remark} 
 We shall emphasis that the standard spin-weighted spherical harmonics in the case of general spin-weight $s\in \half \mathbb{Z}$ have eigenvalue parameter $\ell'\in \{|s|, |s|+1, \ldots\}$, but we make an overall shift of $\frac{1}{2}$ only for convenience of discussions in this work, where $|s|=\frac{1}{2}$ for the Dirac field, such that $\ell=\ell'+\frac{1}{2}\in \{1,2,\ldots\}$. See more in Section \ref{sect:decompIntoModes}. As has been discussed above, for smooth, compactly supported initial data, we have in a compact region where $2M<r<\infty$ that the spin $\pm \half$ components both have $t^{-4}$ decay as a lower and an upper bound. This bound, equal to $t^{-2\ell'-3}$ for the lowest mode $\ell'=|s|=\frac{1}{2}$, obeys the Price's law \cite{Price1972SchwScalar,Price1972SchwIntegerSpin} which, though, predicts the $t^{-2\ell'-3}$ decay for an $\ell'$ mode of the scalar field on Schwarzschild. We believe also that the Price's law, in particular the $t^{-2\ell'-3}$ decay in a compact region, is valid for an arbitrary $\ell'$ mode, $\ell'\in \{|s|, |s|+1, \ldots\}$, of the spin $\pm \half$ components of the Dirac field on a Schwarzschild background.
\end{remark}

\begin{remark}
The asymptotics \eqref{eq:Pricelaw:zero} hold globally in the future domain of outer communication $\Dzeroinfty$. In some asymptotic regions, for instance, $\rb\leq \tb^{1-\epsilon}$ (hence $|\tb/v-1|\lesssim \tb^{-\epsilon}$) and $\rb\geq \tb^{1+\epsilon}$ (hence $\abs{\tb/v}\lesssim \tb^{-\epsilon}$) for a small positive constant $\epsilon$, one can easily calculate the leading constant part of the functions $c_{\sfrak,j+1}$ and $c_{-\sfrak,j+1}$, the remainder being $O(\tb^{-\epsilon})$, and thus derive a simpler form of the asymptotic profiles. In particular, the asymptotics towards null infinity (which can be put into the region $\rb\geq \tb^{1+\epsilon}$) and in a finite radius region (which  is a subset of the region $\rb\leq \tb^{1-\epsilon}$) are manifest. 
\end{remark}

\begin{remark}
Barack and Ori
\cite{bo99} made a heuristic claim that while the spin $-s$ component, for the spin $s=1,2$, obeys the Price's law, the spin $+s$ component has an extra $v^{-1}$ decay at the event horizon than the conjectured Price's law. We show in Theorem \ref{thm:pricelaw:zeroNP} that such a heuristic claim does not hold in the case of the Dirac field, by deriving the precise asymptotics of the spin $\pm \half$ components up to and including the future event horizon and showing both components obey the Price's law. 
\end{remark}

Additionally, we also obtain a result about almost Price's law for each $\ell=\ell_0\geq 2$ mode of each of the spin $\pm \half$ components. The detailed proof can be found in Section \ref{subsect:almostPricelaw:allmodes}.

\begin{thm}
\label{thm:intro:almostPricelaw}
 Let the spin $\pm \half$ components be supported on $\ell= \ell_0$ mode for an $\ell_0\geq 2$. If the $\ell_0$-th Newman--Penrose constant does not vanish, then we have  in $\Dzeroinfty$ that for any $\delta\in (0,\half)$,
\begin{subequations}
\begin{align}
\abs{\partial_{\tb}^j \varphi_{\sfrak}}\lesssim {} & v^{-2}\tb^{-\ell_0-j+\delta/2}(F^{(\ell_0)}(\regl(j,\ell_0),3-\delta,\tb_0,\Psipns))^{\half}, \\
\abs{\partial_{\tb}^j \psins}\lesssim {} & v^{-1}\tb^{-1-\ell_0-j+\delta/2}(F^{(\ell_0)}(\regl(j,\ell_0),3-\delta,\tb_0,\Psipns))^{\half}.
\end{align}
\end{subequations}
While if the $\ell_0$-th Newman--Penrose constant vanishes,
the $\tb$ power of the above pointwise decay estimates is decreased by $1$ in the region $\Dzeroinfty$:
\begin{subequations}
\begin{align}
\abs{\partial_{\tb}^j \varphi_{\sfrak}}\lesssim {} & v^{-2}\tb^{-1-\ell_0-j+\delta/2}(F^{(\ell_0)}(\regl(j,\ell_0),5-\delta,\tb_0,\Psipns))^{\half}, \\
\abs{\partial_{\tb}^j \psins}\lesssim {} & v^{-1}\tb^{-2-\ell_0-j+\delta/2}(F^{(\ell_0)}(\regl(j,\ell_0),5-\delta,\tb_0,\Psipns))^{\half}.
\end{align}
\end{subequations}
\end{thm}

\begin{remark}
These pointwise decay estimates are almost sharp only in the exterior region $\{r\geq \tb\}$. In the interior region $\{r<\tb\}$, including in particular a finite $r$ region,  it is known from the Price's law that one should have better $\tb$ fall-off at the price of an $r$ growth. See also Remark \ref{rem:elliphyper:ellgeq2:interior}.
\end{remark}

\subsection{Outline of the proof}

We provide an outline of the proof in this subsection. This can roughly be divided into three main steps: to show the energy and Morawetz estimates which contain most of the local information of the field, to prove almost sharp pointwise decay estimates via suitable energy decay estimates, and to derive the precise asymptotics by analyzing the equation of motion. We will give an overview of the main ideas in these three steps in Section \ref{sect:1.1.1}, Sections \ref{sect:1.1.2}--\ref{sect:1.1.3}, and Sections \ref{sect:intro:nonzeroN-P}--\ref{sect:1.1.5}, respectively.

%%%%%%%%%%%%%%%%%
\subsubsection{Basic energy and Morawetz estimates}
\label{sect:1.1.1}
%%%%%%%%%%%%%%%%%

Teukolsky \cite{Teukolsky1973I} found that in a Schwarzschild spacetime, the scalars $\psis$ satisfy the celebrated
\emph{Teukolsky Master Equation} (TME), a separable, decoupled wave equation, which takes the following form in Boyer-Lindquist coordinates:
\begin{align}\label{eq:TME}
&\left(r^2\Box_{g_{M}} + \frac{2is\cos\theta}{\sin^2 \theta}\partial_{\phi} - (s^2\cot^2\theta +s)\right)\psis ={}-2s((r-M)Y-2r\partial_t)\psis,
\end{align}
with $\Box_{g_{M}}$ being the scalar wave operator
\begin{align}
\Box_{g_{M}}={}&-\mu^{-1} \partial^2_t
+ r^{-2}\partial_r \left( \Delta \partial_r \right)
+r^{-2}\left(\frac{1}{\sin^2{\theta}} \partial^2_{\phi}+ \frac{1}{\sin{\theta}} \partial_\theta \left( \sin{\theta} \partial_\theta\right)\right).
\end{align}
In particular, the scalars $\psis$ are regular and non-degenerate at the future event horizon $\mathcal{H}^+$. This is a spin-weighted wave equation in the sense that the operator on the LHS of \eqref{eq:TME} is a spin-weighted wave operator. Such a TME is actually derived in \cite{Teukolsky1973I} for general half integer spin fields on a larger family of spacetimes--the Kerr family of spacetimes \cite{kerr63}, and it serves as a starting model for quite many results in obtaining quantitative estimates for these fields, including the Maxwell field and linearized gravity. See the discussions in Section \ref{sect:relatedworks}.

The Chandrasekhar's transformation \cite{chandrasekhar1975linearstabSchw}, which is a differential transformation utilized to obtain a scalar-wave-like equation (to be more precise, Fackerell--Ipser equation \cite{fackerell:ipser:EM} for Maxwell field and Regge--Wheeler equation \cite{ReggeWheeler1957} for linearized gravity) from the TME of integer spin fields, does not exist anymore for Dirac fields. Instead, one has to couple both first order Dirac equations of the spin $\pm \half$ components into a system, and both second order TME \eqref{eq:TME}  into a wave system in order to prove the energy estimate and the Morawetz estimate. An interesting feature for the Dirac field is that the decoupled TME system can be transformed into a coupled symmetric hyperbolic wave system, and this feature is indispensable in achieving both the energy estimate and the integrated local energy decay estimates (or, Morawetz estimates); Further, making use of this feature allows us to derive  these estimates by employing similar techniques in treating the scalar wave equation. In other words, \textbf{such a feature can be viewed as a substitute of the Chandrasekhar's transformation in the case of the Dirac field.} We believe that such a feature, valid in the Kerr family of spacetimes as well, is a requisite in extending the analysis and generalizing these estimates to more general Kerr backgrounds via the energy method.

These two estimates together--we call as \emph{basic energy and Morawetz estimates} (BEAM estimates)--contain the local information of the field and imply already certain weak decay. More importantly, they serve as precursors in obtaining further stronger decay which we will now discuss.

%%%%%%%%%%%%%%%%%%
\subsubsection{Almost sharp energy decay estimates}
\label{sect:1.1.2}
%%%%%%%%%%%%%%%%%%

The $r^p$ method initiated by Dafermos and Rodnianski in \cite{dafermos2009new} is suited and well-developed in recently years to show some basic energy decay results from the BEAM estimates. An application of these $r^p$ estimates, with $p$ ranging from $0$ to $2$, to a wave system of the spin $\pm \half$ components together with the above BEAM estimates yields $\tb^{-2}$ decay for a basic energy of the Dirac field, from which basic pointwise decay $v^{-1}\tb^{-\half}$ can be derived for scalars $\psipns$. The reason that we can obtain such estimates for $\psips$, instead of $\phips=r^{-1}\psips$ which would cause a loss of $r^{-1}$ decay, is due to the damping effect in the TME \eqref{eq:DTMEP} of $\phips$ near infinity. 

To achieve better energy and pointwise decay estimates, we shall decompose the spin $\pm \half$ components into $\ell$ modes. For a fixed $\ell$ mode $\{\psips^{\ell},\psins^{\ell}\}$, we consider the  following wave systems
\begin{align}\label{wave-system:l:mode}
&\Big\{WS[\ell,j],\ j=1,...,\ell \Big|\ \text{the}\ j\text{-th system}\ WS[\ell,j]\ \text{is the wave equations of}\ \{\Phi_{\sfrak}^{(\ell,j')},\Phi_{-\sfrak}^{(\ell,j')}\}_{ 1\leq j'\leq j}\Big \},
\end{align}
where
\begin{align}
\Phi_{\sfrak}^{(\ell,1)}=&\mu^{-\half}r\psips^{\ell},\qquad\quad
\Phi_{-\sfrak}^{(\ell,1)}=\curlVR(\Delta^{\half}\psins^{\ell}),\notag\\
\Phi_{\sfrak}^{(\ell,i)}=&\curlVR^{i-1}\Phi_{\sfrak}^{(\ell,1)},\qquad
\Phi_{-\sfrak}^{(\ell,i)}=\curlVR^{i-1}\Phi_{-\sfrak}^{(\ell,1)},\quad \forall i\geq 2,
\end{align}
and the differential operator $\curlVR=r^2\hat{V}$ equals precisely $\partial_{r_{BS}}$ in Bondi--Sachs coordinates $(u, r_{BS},\theta,\phi)$ with $r_{BS}=r^{-1}$.
Such a treatment of the wave systems is essential for nonzero-spin fields as these scalars are coupled to each other in their governing equations, and is convenient in achieving further energy decay for the lower-index system in terms of energy of the higher-index system. See, for instance, the works \cite{andersson2019stability,Ma20almost} and the discussions below.
For each  wave equations in \eqref{wave-system:l:mode}, an $r^p$ estimate for $p\in [0,2]$ similar to the above can be proven and yields $\tb^{-2}$ decay for the basic energy of each $j$-th system $WS[\ell,j]$, $j=1,...,\ell$. Meanwhile, for each $j\in\{1,\ldots, \ell-1\}$, the basic energy of the $j$-th system $WS[\ell,j]$ can be shown to have $\tb^{-2}$ decay in terms of the basic energy of the $(j+1)$-th system $WS[\ell,j+1]$, thus one can iteratively show that the basic energy of the first system $WS[\ell,1]$ has $\tb^{-2\ell}$ decay in terms of a $r^2$-weighted energy of the $\ell$-th system $WS[\ell,\ell]$. Moreover, because of the property that there is no $O(1)\Phi_{\pm\sfrak}^{(\ell,\ell)}$ term in the wave equation of $\Phi_{\pm\sfrak}^{(\ell,\ell)}$ (see equation \eqref{eq:Phipshighi:lbig2} for $i=\ell=\ell_0$),  the $r^p$ hierarchy for this wave equation of $\Phi_{\pm\sfrak}^{(\ell,\ell)}$  can be extended to $p\in [0,3)$, but no further. For this reason, the basic energy of the first system of $\{\Phi_{\sfrak}^{(\ell,1)},\Phi_{-\sfrak}^{(\ell,1)}\}$ has $\tb^{-2\ell-1+\delta}$ decay with respect to a $r^{3-\delta}$-weighted energy of the $\ell$-th system $WS[\ell,\ell]$ with $\delta\in (0,\half)$ arbitrary.

In the case that the limit $\lim\limits_{\rb \to \infty}\curlVR\tilde{\Phi}_{\sfrak}^{(\ell,\ell)}\vert_{\Sigmazero}\neq 0$, where $\tilde{\Phi}_{\sfrak}^{(\ell,\ell)}$ is a linear combination of $\Phi_{\sfrak}^{(\ell,j)}, j=1,...,\ell$ as defined in Definition \ref{def:tildePhiplusandminusHigh}, this implies that the $r^3$-weighted initial energy of the $\ell$-th system $WS[\ell,\ell]$ will be infinite, hence the above energy decay of $\{\Phi_{\sfrak}^{(\ell,1)},\Phi_{-\sfrak}^{(\ell,1)}\}$ is in fact sharp.  In particular, this limit is a \textquotedblleft{constant\textquotedblright} independent of $\tb$ at future null infinity, and we call it the $\ell$-th Newman--Penrose constant which denoted as $\NPCP{\ell}$ with respect to the $\ell$-th mode $\psips^{\ell}$ of the spin $\half$ component. The corresponding $\ell$-th N--P constant $\NPCN{\ell}$ with respect to the $\ell$-th mode $\psins^{\ell}$ of the spin $-\half$ component can be similarly defined and equals to a constant times $\NPCP{\ell}$. As a result, the above energy decay result of $\{\Phi_{\sfrak}^{(\ell,1)},\Phi_{-\sfrak}^{(\ell,1)}\}$ is sharp in the case of nonvanishing $\ell$-th N--P constant of the $\ell$-th mode
$\{\psips^{\ell},\psins^{\ell}\}$.

To further enlarge the $p$ range in the $r^p$ hierarchy for the wave equations of $\Phi_{\pm\sfrak}^{(\ell,\ell)}$, one has to remove the $O(1) \PhipnsHigh{\ell,\ell-1}$ term in these equations. It suffices to consider only the spin $\half$ component, since the equation of $\PhinsHigh{\ell,\ell}$ is the same as the one of $\PhipsHigh{\ell,\ell}$. It is surprising that there exists a \emph{unique} linear combination of $\{\PhipsHigh{\ell,i}\}_{i=1,\ldots,\ell}$, denoted as $\tildePhipsHigh{\ell,\ell}$, such that in its governing equation, the first and second order operators remain the same and the troublesome $O(1)\PhipsHigh{\ell,\ell-1}$ term is removed, but at the price of introducing extra $\{O(r^{-1})\PhipsHigh{\ell,j}\}_{j=1,\ldots, \ell}$ terms. These new terms with coefficients decaying as $r^{-1}$ are responsible for achieving an $r^p$ hierarchy for the equation of $\tildePhipsHigh{\ell,\ell}$  for $p$ exactly in the range of $[0,5)$, and no further. As a result, for any $\delta\in (0,\half)$, the basic energy of the first system of $\{\PhinsHigh{\ell,1},\PhipsHigh{\ell,1}\}$ has $\tb^{-2\ell-3+\delta}$ decay with respect to a $r^{5-\delta}$-weighted initial energy of this new $\ell$-th system of $\{\tildePhipsHigh{\ell,\ell},\tildePhinsHigh{\ell,\ell}\}$. One should note that such energy decay estimates hold only in the case of vanishing $\ell$-th N--P constant for the $\ell$-th mode $\{\psips^{\ell},\psins^{\ell}\}$, since requiring the $r^{5-\delta}$-weighted initial energy to be finite excludes the case of nonvanishing $\ell$-th N--P constant.

We shall remark that the extension of the $p$ range beyond $2$ is first due to Angelopoulos--Aretakis--Gajic \cite{angelopoulos2018late} where they prove the $r^p$ estimates for  $p\in [0,5)$ for the spherically symmetric $\ell=0$ mode of the scalar field on a Reissner--N\"{o}rdstrom background. Here, we generalize the ideas therein as well as in \cite{andersson2019stability,Ma20almost} to treat an arbitrary mode of the Dirac field.

%%%%%%%%%%%%%%%%%%%%
\subsubsection{Almost sharp pointwise decay estimates}
\label{sect:1.1.3}
%%%%%%%%%%%%%%%%%%%%

We have proven the sharp basic energy decay results for  $\{\PhinsHigh{1,1},\PhipsHigh{1,1}\}$ in the above discussion,  and for simplicity, we will denote $\{\PhinsHigh{1,1},\PhipsHigh{1,1}\}$ by $\{\PhinsHigh{1},\PhipsHigh{1}\}$. One still needs to derive the decay estimates of a basic energy of the scalars $\{\psips^{\ell=1}=\mu^{\half} r^{-1}\PhipsHigh{1},r\psins^{\ell=1}\}$ in order to achieve almost sharp pointwise decay estimates for the spin $\pm \half$ components $\chi_0$ and $\chi_1$ of the Dirac field. On a $\tau=const$ hypersurface, by rewriting the wave equation of
$\{\PhinsHigh{1},\PhipsHigh{1}\}$ into a form that a $3$-dimensional spatial elliptic operator  acting on $\{\PhinsHigh{1},\PhipsHigh{1}\}$ equals the terms involving $\partial_{\tb}$ derivatives, and making use of the fact that for any $j\in \mathbb{N}$, the basic energy of $\{\partial_\tau^j\PhinsHigh{1},\partial_\tau^j\PhipsHigh{1}\}$ has extra $\tb^{-2j}$ decay than the basic energy of $\{\PhinsHigh{1},\PhipsHigh{1}\}$, this enables us to derive (degenerate) elliptic estimates in terms of the terms involving $\partial_{\tb}$ derivative,  and to conclude that a degenerate, basic energy of $\{\psips^{\ell=1},r\psins^{\ell=1}\}$ has further $\tb^{-2}$  decay. Pointwise decay rates $v^{-\frac{3}{2}-s}\tb^{-\ell-j+s-\half+\frac{\delta}{2}}$ and $v^{-\frac{3}{2}-s}\tb^{-\ell-j-\frac{3}{2}+s+\frac{\delta}{2}}$
for $\{\partial_{\tb}^j((r-M)^{-1}\psi_{\sfrak}^{\ell}), \partial_{\tb}^j\psins^{\ell}\}$ follow easily in the case of nonvanishing $\ell$-th N--P constant and vanishing $\ell$-th N--P constant for the $\ell$-th mode, respectively. In both cases, there is only a $\frac{\delta}{2}$ loss of decay in $\tb$ compared to the sharp asymptotics predicted by Price in \cite{Price1972SchwScalar,Price1972SchwIntegerSpin} and Price--Burko in \cite{price2004late}, where $\delta\in (0,\half)$ is arbitrary. We note also that for $\ell=1$ mode, by iteratively substituting these almost sharp asymptotics into the wave equations \eqref{eq:prbprbPhins:ell=1:impro} and
\eqref{eq:simpleformofeqofpsins},  one can show that $\{\partial_{\tb}^j\prb \varphi_{\sfrak}^{\ell=1}, \partial_{\tb}^j\prb \psins^{\ell=1}\}$ with $\varphi_{\sfrak}^{\ell=1}=(r-M)^{-1}\psips^{\ell=1}$ has faster $\tb^{-1}$ decay compared to $\{\partial_{\tb}^j \varphi_{\sfrak}^{\ell=1}, \partial_{\tb}^j \psins^{\ell=1}\}$.

%%%%%%%%%%%%%%%%%%%
\subsubsection{Asymptotics in the case of nonvanishing first Newman--Penrose constant}
\label{sect:intro:nonzeroN-P}
%%%%%%%%%%%%%%%%%%%

To achieve the precise asymptotics, the first N--P constant of the $\ell=1$ mode is of vital importance in deriving the precise behaviours of $\ell=1$ mode, and the higher modes $\ell\geq 2$ have faster decay from the above discussions.
The first N--P constant is one particular conserved quantity at null infinity and contains all information of the leading asymptotics of $\ell=1$ mode $\{\psips^{\ell=1},\psins^{\ell=1}\}$. Without loss of generality, we consider only a fixed $(m,\ell=1)$ mode of the spin $\half$ component which is the $m$-th spin-weighted spherical harmonic mode of $\psips^{\ell=1}$, as the asymptotics of such a $(m,\ell=1)$ mode of the spin $-\half$ component can be fully determined from the first order Dirac system and the asymptotics of the same mode of the spin $\half$ component. For such a mode, its N--P constant is a constant independent of $\theta, \phi, \tb$.

We shall follow the work \cite{angelopoulos2018late} and derive the precise asymptotics for a fixed $(m,\ell=1)$ mode of the spin $\half$ component. Under a very generic assumption \eqref{assump:Pricelaw:nonzero} which states the quantity $r^2\curlVR \PhipsHigh{1}(m,\ell=1)$ converges to the N--P constant $\NPCP{1}(m,\ell=1)$ in a speed of rate $O(r^{-\beta})$ on the initial hypersurface, one can obtain leading asymptotics of $r^{2} \curlVR \PhipsHigh{1}(u,v)$ in the region where $\{v-u\geq v^{\alpha}\}$, $\alpha\in (\half,1)$, by integrating the wave equation \eqref{eq:PhipsHigh1:l=1:v1} along a $v=const$ hypersurface from the initial hypersurface. One can then integrate along $u=const$ hypersurface and make use of the above leading asymptotics of $\curlVR \PhipsHigh{1}(u,v)$ to obtain precise asymptotics for $\varphi_{\sfrak}^{\ell=1}$ in the region $\{v-u\geq v^{\alpha'}\}$ with $\alpha'\in (\alpha, 1)$ suitably chosen. In the remaining region, it suffices to combine this estimate at the boundary hypersurface $\{v-u= v^{\alpha'}\}$ together with better decay for $\prb \varphi_{\sfrak}^{\ell=1}$ to achieve the leading asymptotics of $\varphi_{\sfrak}^{\ell=1}$. A similar argument can be utilized to derive the asymptotics of $\partial_{\tb}^j\varphi_{\sfrak}^{\ell=1}$.

%%%%%%%%%%%%%%%%%%%
\subsubsection{Asymptotics in the case of vanishing first Newman--Penrose constant}
\label{sect:1.1.5}
%%%%%%%%%%%%%%%%%%

A natural idea would be to reduce this case of vanishing first N--P constant to a case of nonvanishing first N--P constant so that the above results in Section \ref{sect:intro:nonzeroN-P} can be applied. This is exactly the idea behind and realized by the uniqueness and existence of the smooth time integral $g_{\sfrak}$ of $\psips$ which solves the spin $s=\half$ TME and satisfies $\partial_{\tb}g_{\sfrak}=\psips$. The wave equation of $g_{\sfrak}$ then yields equation \eqref{regular:equa:gs:horizon:scaled} on the initial hypersurface $\Sigmazero$, from which one can explicitly calculate the N--P constant of the time integral $g_{\sfrak}$  in terms of the initial data of the spin $\half$ component $\psips$. This part is mostly in the same spirit of the work \cite{angelopoulos2018late}.  The rest of the proof is devoted to showing that the assumption \eqref{assump:Pricelaw:zero} implies an assumption
\eqref{assump:Pricelaw:nonzero} for the time integral and that a $r^{3-\delta}$-weighted energy of the time integral $g_{\sfrak}$ is bounded by a $r^{5-\delta/2}$-weighted energy of $\psipns$. In the end, one applies the results in Theorem \ref{thm:pricelaw:nonzeroNP} to conclude Theorem \ref{thm:pricelaw:zeroNP}.

%%%%%%%%%%%%
\subsection{Related works}
\label{sect:relatedworks}
%%%%%%%%%%%

We now put our results in context and give some background and related results. Teukolsky \cite{Teukolsky1973I} found that the two components of the massless Dirac field in a Kerr spacetime satisfy a separable, decoupled wave equation, known as Teukolsky master equation.  In a seminal work, Chandarasekhar \cite{chandrasekhar1976solution} found that the massive Dirac equations in a Kerr spacetime in Boyer--Lindquist coordinates are also separable. There are extensions \cite{page1976dirac,roken2017massive} to the Kerr--Newman spacetimes and the Eddington--Finkelstein coordinates in Kerr spacetimes. The works of Teukolsky and Chandarasekhar are fundamental since they open the possibility of applying various methods to analyze massless and massive Dirac fields.

There are quite many results on the scattering properties of massless, or massive Dirac field on black hole backgrounds.
The scattering of massless Dirac in Schwarzschild and massive charged Dirac fields in Reissner--Nordstr\"{o}m are obtained by Nicolas \cite{nicolas1995scattering} and Melnyk \cite{melnyk2003scattering} respectively.  Melnyk then used this result to study the Hawking effect for massive charge Dirac fields on Reissner--N{o}rdstr\"{o}m in \cite{melnyk2004hawking}.
These works use trace class perturbation methods and cannot be extended to the Kerr case because of the lack of symmetry in the Kerr geometry.
A complete scattering result for massless Dirac fields outside a subextremal Kerr black-hole is first proven by H\"{a}fner--Nicolas \cite{hafner2004scattering}  using the Mourre theory \cite{mourre1981absence}, and
Batic \cite{batic2007scattering} extended it to massive Dirac fields in a subextremal Kerr spacetime by employing an integral representation for the Dirac propagator. H\"{a}fner--Mokdad--Nicolas obtained the scattering for the massive charged Dirac field inside a Reissner--Nordstr\"{o}m--type black hole in \cite{hafner20scattering}.

The peeling properties of massless Dirac fields in Kerr spacetimes are obtained by Pham \cite{truong2020peeling} following earlier works by Mason--Nicolas \cite{mason2012peeling} and Nicolas--Pham \cite{nicolas2019peeling}, and Smoller--Xie \cite{smoller2012asymptotic} proved $t^{-2\ell}$ decay for each $\ell$ mode of massless Dirac fields in a Schwarzschild spacetime using the Chandrasekhar's separation of variables and a detailed analysis of the associated Green's function. Finster--Kamran--Smoller--Yau \cite{finster2003long,finster2002decay} proved local asymptotical decay $t^{-\frac{5}{6}}$ for the massive Dirac field with bounded angular momentum in a subextremal Kerr--Newman spacetime. Dong--LeFloch--Wyatt \cite{Dong2019HiggsBoson} established a nonlinear stability result for a massive Dirac coupled system in Minkowski.

There is a large amount of works on spin fields in asymptotically flat spacetimes. We list here a few in the literature: \cite{morawetz1968time,klainerman1986null,christodoulou1986global,CK93global,lindblad2010global}
on the wave equations on Minkowski background and nonlinear stability of Minkowski spacetime; \cite{wald1979note,kay1987linear,bluesoffer03mora,blue:soffer:integral,
dafrod09red,dafermos2011bdedness,larsblue15hidden,tataru2011localkerr,
dafermos2016decay,schlue2013decay,moschidis2016r,MMTT,tohaneanu2012strichartz} for energy, Morawetz, Strichartz, and pointwise estimates of scalar field on a Schwarzschild or subextremal Kerr background; \cite{fackerell:ipser:EM,blue08decayMaxSchw,pasqualotto2019spin,sterbenz2015decayMaxSphSym,
larsblue15Maxwellkerr,andersson16decayMaxSchw,Ma2017Maxwell,gudapati2019conserved} for similar estimates for Maxwell field in black hole spacetimes; \cite{dafermos2019linear,hung2017linearstabSchw,Jinhua17LinGraSchw,
johnson2019linear,Giorgi2019linearRNfullcharge,
finster2016linear,Ma17spin2Kerr,dafermos2019boundedness,andersson2019stability,hafner2019linear} on linear stability of Schwarzschild, Reissner--N\"{o}rdstrom and Kerr metrics. There are also results \cite{HintzKds2018,kla2015globalstabwavemapKerr} on nonlinear stability of black hole spacetimes.

Researches toward sharp decay of spin fields in black hole spacetimes are quite active in recent years. The precise upper and lower rates of decay in Schwarzschild are predicted by Price \cite{Price1972SchwScalar,Price1972SchwIntegerSpin} and further completed by Price--Burko in \cite{price2004late}. In these works, they predict that for any fixed $\ell$ mode of spin fields in a Schwarzschild spacetime, if the initial data is compactly supported, this mode should fall off as $\tb^{-2\ell-3}$ at any finite radius,  and this sharp decay is  now called as \textquotedblleft{Price's law\textquotedblright}. Donninger--Schlag--Soffer  proved in \cite{donninger2011proof} $\tb^{-2\ell-2}$ decay for an $\ell$ mode of scalar field and in \cite{donninger2012pointwise} $\tb^{-3}$, $\tb^{-4}$ and $\tb^{-6}$ for scalar field, Maxwell field and gravitational perturbations, respectively. Efforts have also been made in proving Price's law in Kerr or more general spacetimes: under an assumption that a basic energy and Morawetz estimate holds, $\tb^{-3}$ decay for scalar field and $\tb^{-4}$ decay for Maxwell field in a class of non-stationary asymptotically flat spacetimes are proved in a series of works by Tataru \cite{tataru2013local} and Metcalfe--Tataru--Tohaneanu \cite{metcalfe2012price,metcalfe2017pointwise}. For the Maxwell field, decay estimates in the Kerr spacetimes and almost sharp decay estimates in a Schwarzschild spacetime are proven in \cite{Ma20almost}. Recently, there are two approaches succeeding in obtaining $\tb^{-3}$ as both an upper and a lower bound for scalar field: Angelopoulos--Aretakis--Gajic in a series of works \cite{angelopoulos2018vector,angelopoulos2018late,
angelopoulos2019logarithmic} obtained using the vector field method almost sharp decay $\tb^{-3+\epsilon}$, Price's law $\tb^{-3}$ decay, and for the subleading term $\tb^{-3}\log \tb$ decay, respectively outside a Schwarzschild black hole; Hintz \cite{hintz2020sharp} computed the $\tb^{-3}$ leading order term in a subextremal Kerr spacetime and obtained $\tb^{-2\ell-3}$ upper bound for a fixed $\ell$ mode on a Schwarzschild background, and his approach relies on an analysis of the resolvent near zero frequency.

These Price's law decay results, in particular, the lower bound of decay, are crucial in resolving the Strong Cosmic Censorship conjecture, that is,  to prove (in)stability of the Cauchy horizon of black hole spacetimes. We direct the readers to the works \cite{dafermos2017interior,luk2019strong,LukOh2019SCC} and references therein.

%%%%%%%%%%%%%%%%%
\subsection*{Overview of the paper}
%%%%%%%%%%%%%%%%%

We collect in Section \ref{sect:preliminaries} some preliminaries, including more definitions, some general facts and a few useful estimates. Sections \ref{mainsect:energyestimates} and
\ref{sect:mora:Schw} are devoted to proving the uniform boundedness of a nondegenerate energy and an integrated local energy estimate, respectively. In Section \ref{sect:almostsharpdecayestimates}, we utilize the proven energy and Morawetz estimates to achieve almost sharp asymptotics, and in particular, prove Theorem \ref{thm:intro:almostPricelaw}. In the last two sections, we give the proofs of
Theorems \ref{thm:pricelaw:nonzeroNP} and
\ref{thm:pricelaw:zeroNP}, respectively.  In the end, we provide in the appendix  a self-contained derivation of both the Dirac equations and the Teukolsky master equation on a Kerr background and, for convenience, a list of the many different scalars constructed from the spin $\pm \half$ components. 
%%%%%%%%%%%%%%%%%%%%%%%%%%%%%%

%%%%%%%%%%%%%
\section{Preliminaries}
\label{sect:preliminaries}
%%%%%%%%%%%%

%%%%%%%
\subsection{General conventions}
%%%%%%%%%%

Denote $\mathbb{N}$ to be the set of natural numbers $\{0,1,\ldots\}$, $\mathbb{Z}$ the set of integers, $\mathbb{Z}^+$ the set of positive integers, $\mathbb{R}$ the set of real numbers, and $\mathbb{R}^+$ the set of positive real numbers. Denote $\Sphere$ the standard unit round sphere.

The notation $\Re(\cdot)$ is to denote the real part. We use an overline or a bar to denote the complex conjugate.

LHS and RHS are short for left-hand side and right-hand side, respectively.

Throughout this work, $F_1\equiv F_2$ means that the two sides are equal after integration over unit round sphere $\Sphere$, i.e. $\int_{\Sphere}F_1 \di^2\mu=
\int_{\Sphere}F_2 \di^2\mu$.

Denote a large (positive) universal constant by $C$ and a small (positive) universal constant by $c$. These universal constants may change from term to term. We denote it by $C(\mathbf{P})$ (or $c(\mathbf{P})$) if it depends on a set of parameters $\mathbf{P}$. Regularity parameters are generally denoted by $\reg$, and $\regl$ is a universal constant that may change from term to term. Also, $\regl(\mathbf{P})$ means a regularity constant depending on the parameter set $\mathbf{P}$.

Let $F_2$ be a nonnegative function. We denote $F_1\lesssim F_2$ if there exists a universal constant $C$ such that $F_1\leq CF_2$, and similarly for $F_1\gtrsim F_2$. If both $F_1\lesssim F_2$ and $F_1\gtrsim F_2$ hold, we say $F_1\sim F_2$.

Let $\mathbf{P}$ be a set of parameters. We say $F_1\lesssim_{\mathbf{P}} F_2$ if there exists a universal constant $C(\mathbf{P})$ such that $F_1\leq C(\mathbf{P})F_2$. Similarly for $F_1\gtrsim_{\mathbf{P}} F_2$. We say $F_1\sim_{\mathbf{P}} F_2$ if both $F_1\lesssim_{\mathbf{P}} F_2$ and $F_1\gtrsim_{\mathbf{P}} F_2$ hold.

For any $\alpha \in \mathbb{N}$, we say a function $f(r,\theta,\pb)$ is $O(r^{-\alpha})$ if it is a sum of two smooth functions $f_1(\theta,\pb) r^{-\alpha}$ and $f_2(r,\theta,\pb)$ satisfying that for any $j\in \mathbb{N}$,
$\abs{(\partial_r)^j f_2}\leq C(j) r^{-\alpha-1-j}$. In particular, if $f$ is $O(1)$, then $\partial_r f=O(r^{-2})$.

Let  $\chi_1$ be a standard smooth cutoff function which is decreasing, $1$ on $(-\infty,0)$, and $0$ on $(1,\infty)$, and let $\chi=\chi_1((R_0-r)/M)$ with $R_0$ suitably large and to be fixed in the proof.  So $\chi=1$ for $r\geq R_0$ and vanishes identically for $r\leq R_0-M$.

%%%%%%%%%%%%%%%%%
\subsection{Further definitions}
%%%%%%%%%%%%%

In this subsection, we aim to define the energy norms and (spacetime) Morawetz norms. To properly define these norms, we need a few further definitions. 

First, we define a few convenient reference volume forms in calculations and defining the norms. Note that these are not the volume element of DOC or the induced volume form on a $3$-dimensional hypersurface.

\begin{definition}
Define $\di^2\mu=\sin\theta \di \theta \wedge \di \pb$, and define the reference volume forms
\begin{subequations}
\begin{align}
\di^3\mu ={}&\di \rb\wedge \di^2\mu,\\
\di^4\mu ={}&\di \tb\wedge\di^3\mu .
\end{align}
\end{subequations}
Given a $1$-form $\nu$, let $\di^3\mu_{\nu}$ denote a Leray $3$-form such that $\nu\wedge\di^3\mu_{\nu}=\di^4\mu$.
\end{definition}

We then define a few vector fields and operators.

\begin{definition}
Define two Killing vector fields
\begin{align}
\label{def:Killingvectors}
\Lxi={}\partial_{\tb}=\partial_t, \quad \Leta=\partial_{\phi}.
\end{align}
Define the edth operators $\edthR$ and $\edthR'$ as in \eqref{def:edthRandedthR'} and the null vectors $Y$ and $\VR$ as in \eqref{def:VectorFieldYandV}. 
Denote also a regular outgoing vector
\begin{align}
\label{def:regularV}
V=\mu\VR=\partial_t +\mu \partial_r.
\end{align}
Define an operator
\begin{align}
\label{def:curlVR}
\curlVR=r^2\VR.
\end{align}
Define a second order Teukolsky angular operator
\begin{align}
\label{def:TAOsOp}
\TAO_{s}={}&\frac{1}{\sin{\theta}} \partial_{\theta}(\sin \theta \partial_{\theta})
+\frac{\Leta^2}{\sin^2\theta}
+\frac{2is\cos\theta}{\sin^2 \theta}\Leta
-({s^2}{\cot^2 \theta}+\abs{s}).
\end{align}
\end{definition}

\begin{remark}
\begin{itemize}
\item  Remind that when acting the edth operators in \eqref{def:edthRandedthR'} on any spin-weighted scalar $\varphi$,  the obtained scalar has a different spin-weight from the scalar $\varphi$. Specifically, if $\varphi$ is a spin-weight $s$ scalar, then $\edthR\varphi$ and $\edthR'\varphi$ are spin-weight $s+1$ and $s-1$ scalars, respectively.
\item
By the definitions of $\edthR$ and $\edthR'$ in \eqref{def:edthRandedthR'},  one finds when acting on a spin-weight $s$ scalar,
\begin{align}
\label{eq:TAopmsfrak:edthRedthR'}
\TAO_{s}={}&\curlLds  \curlLs - \abs{s}+s=\curlLs \curlLds -\abs{s}-s,
\end{align}
where one can use the commutator \eqref{comm:curlLscurlLds} to derive the second relation.
\end{itemize}
\end{remark}

\begin{definition}
Let $m\in \mathbb{N}$ and $n\in \mathbb{Z}^+$. Let $\mathbb{X}=\{X_1, X_2, \ldots, X_n\}$ be a set of spin-weighted operators, and let a multi-index $\mathbf{a}$ be an ordered set $\mathbf{a}=(a_1,a_2,\ldots,a_m)$ with all $a_i\in \{1,\ldots, n\}$. Define $|\mathbf{a}|=m$ and define $\mathbb{X}^{\mathbf{a}}=X_{a_1}X_{a_2}\cdots X_{a_m}$. Let $\varphi$ be a spin-weighted scalar, and define its pointwise norm of order $k$, $k\in \mathbb{N}$, as
\begin{align}
\absHighOrder{\varphi}{m}{\mathbb{X}}={}\sqrt{\sum_{\abs{\mathbf{a}}\leq m}\abs{\mathbb{X}^{\mathbf{a}}\varphi}^2} .
\end{align}
\end{definition}

Further, we need a few sets of operators.

\begin{definition}
Define a set of operators
\begin{subequations}
\begin{align}
\PDeri={}\{Y,V, r^{-1}\edthR,r^{-1}\edthR'\}
\end{align}
adapted to the Hartle--Hawking tetrad, and its rescaled one
\begin{align}
\PSDeri={}\{rY,rV, \edthR,\edthR'\}.
\end{align}
Define a set of commutators
\begin{align}
\CDeri={}\{Y,rV, \edthR,\edthR'\}.
\end{align}
Define also a set of operators
\begin{align}
\RDeri={}&\{\Lxi, Y, \edthR, \edthR'\}.
\end{align}
\end{subequations}
\end{definition}

Now we are able to define the energy norms and Morawetz norms.
\begin{definition}
\label{def:basicweightednorm}
Let $\varphi$ be a spin-weighted scalar, and let $k\in \mathbb{N}$ and $\gamma\in \mathbb{R}$. Let $\Omega$ be a $4$-dimensional subspace of the DOC, and let $\Sigma$ be a $3$-dimensional space that can be parameterized by $(\rb,\theta,\pb)$. Define
\begin{subequations}
\begin{align}
\norm{\varphi}_{W_{\gamma}^{\reg}(\Omega)}^2
={}&\int_{\Omega} r^{\gamma}\absCDeri{\varphi}{\reg}^2\di^4\mu,\\
\norm{\varphi}_{W_{\gamma}^{\reg}(\Sigma)}^2
={}&\int_{\Hyper} r^{\gamma}\absCDeri{\varphi}{\reg}^2\di^3\mu.
\end{align}
\end{subequations}
And define
\begin{align}
\norm{\varphi}_{\tilde{W}_{\gamma}^{\reg}(\Omega)}^2
={}\int_{\Omega} r^{\gamma}\absRDeri{\varphi}{\reg}^2\di^4\mu,\quad
\norm{\varphi}_{\tilde{W}_{\gamma}^{\reg}(\Sigma)}^2
={}\int_{\Hyper} r^{\gamma}\absRDeri{\varphi}{\reg}^2\di^3\mu,
\end{align}
\end{definition}

In the end, we define a few subsets of a hypersurface or a spacetime region. These can be viewed as some choices of the regions $\Omega$ and $\Sigma$ in Definition \ref{def:basicweightednorm} and are of particular importance in deriving the decay estimates.

\begin{definition}
Let $\tb_2>\tb_1\geq \tb_0$  and let $r_2>r_1\geq 2M$. Define
\begin{subequations}
\label{def:domainnotations:subdomains}
\begin{align}
\Sigmaone^{\geq r_1}={}&\Sigmaone\cap \{r\geq r_1\}, & \Donetwo^{\geq r_1}={}&\Donetwo\cap\{r\geq r_1\},\\
\Sigmaone^{r_1,r_2}={}&\Sigmaone\cap \{r_1\leq r\leq r_2\}, &\Donetwo^{r_1,r_2}={}&\Donetwo\cap\{r_1\leq r\leq r_2\},\\
\Sigmaone^{\leq r_1}={}&\Sigmaone\cap \{2M\leq r\leq r_1\}, & \Donetwo^{\leq r_1}={}&\Donetwo\cap\{2M\leq r\leq r_1\}.
\end{align}
\end{subequations}
\end{definition}

%%%%%%%%%%%%%%%
\subsection{General facts}
%%%%%%%%%%%%%%%

Since we are treating complex scalars and are using the edth operators, the following integration by parts over sphere is necessary. Although it is standard, we state it for completeness.

\begin{lemma}
\label{lem:IBPonSphere}
Let $s\in \half \mathbb{Z}$. For two spin-weighted scalars $f$ and $h$ with spin-weight $s+1$ and $s$ respectively, we have
\begin{align}
\label{eq:IBPonsphere:1}
\int_{\Sphere}\Re(\bar{f}\edthR h)\di^2\mu={}&
-\int_{\Sphere}\Re(\overline{\edthR'f}h)\di^2 \mu,
\end{align}
\end{lemma}

\begin{proof}
We calculate the LHS of \eqref{eq:IBPonsphere:1}:
\begin{align*}
\int_{\Sphere}\Re(\bar{f}\edthR h)\di^2\mu ={}&
\int_{\Sphere}\Re(\bar{f} (\partial_{\theta}h +i\csc\theta\partial_{\phi}h-s\cot\theta h)\sin\theta\di\theta\di\phi\notag\\
={}&\int_{\Sphere}\Big(\partial_{\theta}\big(\Re(\bar{f} h \sin\theta)\big)
+\partial_{\phi}\big(\Re(i \bar{f}h)\big)\Big)\di\theta\di\phi\notag\\
&
+\int_{\Sphere}\Re\Big(
-\overline{\partial_{\theta}f}h+\overline{i\csc\theta\partial_{\phi}f}h-(s+1)\cot\theta\bar{f}h\Big)
\Big)\sin\theta\di \theta\di \phi.
\end{align*}
The second last line vanishes, and the last line equals the RHS of \eqref{eq:IBPonsphere:1} by the definition \eqref{def:edthRandedthR'}  of the operator $\edthR'$.
\end{proof}

The following commutators are useful.
\begin{lemma}
We have the following commutators
\begin{subequations}
\label{eq:commutators}
\begin{align}
[\mu Y,\mu\VR]={}&
0,\\
\label{comm:DeltahalfVRY}
[\Delta^{\Half}\VR,\Delta^{\Half}Y]
={}&(r-3M)(Y+\VR)
\end{align}
and 
\begin{align}
\label{comm:edthRedthR'LxiLeta}
[\edthR, \Lxi]={}&[\edthR', \Lxi]=[\edthR, \Leta]=[\edthR', \Leta]=0,\\
\label{comm:edthRedthR'YVR}
[\edthR, \VR]={}&[\edthR', \VR]=[\edthR, Y]=[\edthR', Y]=0,
\end{align}
and when acting on a spin-weight $s$ scalar $\varphi$, 
\begin{align}
\label{comm:curlLscurlLds}
[\curlLs, \curlLds]\varphi={}&
2s\varphi.
\end{align}
\end{subequations}
\end{lemma}

\begin{proof}
From \eqref{def:VectorFieldYandV} and \eqref{def:r*coordinate}, we have $\mu Y=\partial_t - \partial_{r^*}$ and $\mu \VR=\partial_t +\partial_{r^*}$, hence $[\mu Y, \mu \VR]=0$. Meanwhile, $\Delta^{\Half}\VR= \mu^{-\Half} r \partial_t + \Delta^{\Half}\partial_r$ and $\Delta^{\Half}Y= \mu^{-\Half} r \partial_t - \Delta^{\Half}\partial_r$, hence
\begin{align*}
[\Delta^{\Half}\VR,\Delta^{\Half}Y]
={}&[\mu^{-\Half} r \partial_t + \Delta^{\Half}\partial_r, \mu^{-\Half} r \partial_t - \Delta^{\Half}\partial_r]\\
={}&\Delta^{\Half}\partial_r(\mu^{-\Half} r)\partial_t +\Delta^{\Half}\partial_r(\mu^{-\Half} r)\partial_t - \Delta^{\Half}\partial_r(\Delta^{\Half})\partial_r +\Delta^{\Half}\partial_r(\Delta^{\Half})\partial_r \\
={}&2\Delta^{\Half}\partial_r(\mu^{-\Half} r)\partial_t \\
={}&2\mu^{-1}(r-3M)\partial_t,
\end{align*}
which proves \eqref{comm:DeltahalfVRY}. In addition, the commutators \eqref{comm:edthRedthR'LxiLeta} are manifest, and the commutators \eqref{comm:edthRedthR'YVR} holds by simply observing that the coefficients in the expression \eqref{def:edthRandedthR'} of $\edthR$ and $\edthR'$ are $(r, t)$-independent and the coefficients in the expression \eqref{def:VectorFieldYandV} of $Y$ and $\VR$ are $(\theta,\phi)$-independent. The last commutator \eqref{comm:curlLscurlLds} is a.well-known fact, nevertheless, we show it in details. Recall that $\edthR'\varphi$ has spin-weight $s-1$ and $\edthR\varphi$ has spin-weight $s+1$ since $\edthR$ increases the spin-weight by $1$ while $\edthR'$ decreases the spin-weight by $1$. Consequently, using \eqref{def:edthRandedthR'}, we compute
\begin{align*}
[\curlLs, \curlLds]\varphi={}&(\partial_{\theta}-i\csc\theta\partial_{\phi} +(s+1)\cot\theta)(\partial_{\theta}+i\csc\theta\partial_{\phi} -s\cot\theta)\varphi\notag\\
&
-(\partial_{\theta}+i\csc\theta\partial_{\phi} -(s-1)\cot\theta)(\partial_{\theta}-i\csc\theta\partial_{\phi} +s\cot\theta)\varphi\notag\\
={}&\bigg[\partial_{\theta}(i\csc\theta)\partial_{\phi} -s\cot\theta \partial_{\theta}
+\frac{is\cos\theta}{\sin^2\theta}\partial_{\phi} +(s+1)\cot\theta\partial_{\theta}
+\frac{i(s+1)\cos\theta}{\sin^2\theta}\partial_{\phi}\bigg]\varphi
\notag\\
&-
\bigg[
-\partial_{\theta}(i\csc\theta)\partial_{\phi} +s\cot\theta \partial_{\theta}
+\frac{is\cos\theta}{\sin^2\theta}\partial_{\phi}
-(s-1)\cot\theta\partial_{\theta}
+\frac{i(s-1)\cos\theta}{\sin^2\theta}\partial_{\phi}
\bigg]\varphi
\notag\\
&+(-2\partial_{\theta}(s\cot\theta)-s(s+1)\cot^2\theta+s(s-1)\cot^2\theta)\varphi,
\end{align*}
where we have dropped the second order derivative terms in the second step since they cancel exactly by a simple observation. In the above equation, the sum of the third last and the second last lines equals zero, and the last line equals $2s\varphi$, hence proving \eqref{comm:curlLscurlLds}.
\end{proof}

\begin{lemma}
One can express the two principal null vectors in the hyperboloidal foliation as
\begin{align}
\label{eq:YV:hyper}
Y=-\prb +\partial_r h \partial_{\tb},\quad \VR=\prb +(2\mu^{-1}-\partial_r h)\partial_{\tb}.
\end{align}
\end{lemma}

The following lemma is to expand out a spin-weighted wave operator on Schwarzschild.
\begin{lemma}
For a spin-weight $s$ scalar $\psi$, $s=\pm \half$,
\begin{align}
\l
\hspace{4ex}&\hspace{-4ex}
\left( -\mu^{-1} r^2\partial^2_t
+ \partial_r \left( \Delta \partial_r \right)
+\frac{1}{\sin^2{\theta}} \partial^2_{\phi}+ \frac{1}{\sin{\theta}} \partial_\theta \left( \sin{\theta} \partial_\theta\right) +\frac{ 2is\cos\theta}{\sin^2 \theta}\Leta
- (s^2\cot^2\theta +\sfrak)\right)\psi\notag\\
={}&r^{-1}(-r^2YV
+\TAO_{s}
-2M r^{-1})(r\psi).
\end{align}
\end{lemma}

\begin{proof}
In view of the definition of $\TAO_s$ in \eqref{def:TAOsOp}, it suffices to show
\begin{align}
\label{eq:r2YVexpansion}
(-\mu^{-1} r^2\partial^2_t
+ \partial_r \left( \Delta \partial_r \right))\psi ={}&
r^{-1}(-r^2YV
-2M r^{-1})(r\psi).
\end{align}
The RHS equals
\begin{align*}
\hspace{4ex}&\hspace{-4ex}
-r(\mu^{-1}\partial_t -\partial_r)(\partial_t +\mu\partial_r)(r\psi)-2Mr^{-1}\psi\notag\\
={}&-\mu^{-1}r^2\partial_t^2\psi
+r\partial_r(\mu \partial_r (r\psi))
-r\partial_t \partial_r (r\psi)
+r\partial_r \partial_t (r\psi)
-2Mr^{-1}\psi\notag\\
={}&-\mu^{-1}r^2\partial_t^2\psi+r\partial_r (r^{-1}\Delta\partial_r \psi+\mu\psi)-2Mr^{-1}\psi\notag\\
={}&-\mu^{-1}r^2\partial_t^2\psi+\partial_r (\Delta\partial_r\psi)+(-r^{-1}\Delta +\mu r )\partial_r \psi +(r\partial_r \mu -2Mr^{-1})\psi,
\end{align*}
and the coefficients of the last two terms both vanish, hence the proof is completed.
\end{proof}

\subsection{Decomposition into modes for spin-weight $s$ scalars}
\label{sect:decompIntoModes}

A theory of decomposing any spin-weighted scalar into spin-weighted spherical harmonics is standard,  and the eigenvalue parameter $\ell$ takes its value in the set $\{\abs{s}, \abs{s}+1,\ldots\}$.  See \cite[Section 4]{penroserindlerI}. In this work, for any spin-weight $s$ scalar $\varphi$, $s-\half \in \mathbb{Z}$, we make an overall shift $\half$ for the eigenvalue parameter $\ell$ such that the new parameter is $\tilde{\ell}=\ell+\half$, thus the eigenvalue parameter $\tilde{\ell}$ takes values in $\{\abs{s}+\half, \abs{s}+\frac{3}{2}, \ldots \}$ and, in particular, is a positive integer. This overall shift of the eigenvalue parameter is solely for convenience of discussions and stating the estimates  in the context of this work, and without confusion, we denote $\tilde{\ell}$ by $\ell$. Meanwhile, in the rest of this subsection, we consider only spin-weighted scalar with spin-weight $s$ satisfying $s-\half\in \mathbb{Z}$.

For any spin-weight $s$ scalar $\varphi$, $s- \half\in \mathbb{Z}$, we can decompose it into modes $\varphi=\sum\limits_{\ell_0=\abs{s}+1/2}^{\infty}\varphi^{\ell=\ell_0}$, with $\ell\in \mathbb{N}$ and each mode
$\varphi^{\ell=\ell_0}
=\sum\limits_{m}\varphi_{m,\ell_0}
(\tb,\rb)Y_{m,\ell_0}^{s}(\cos\theta)e^{im\pb}$, with $m$ taking all values  satisfying $\ell_0-\abs{s}-\abs{m}\in \mathbb{N}$.
Here, $\left\{Y_{m,\ell}^{s}(\cos\theta)e^{im\pb}\right\}_{m,\ell}$ are the eigenfunctions, called as \textquotedblleft{\emph{spin-weighted spherical harmonics},\textquotedblright} of a self-adjoint operator
$\edthR\edthR'$, form a complete orthonormal basis on $L^2(\sin\theta \di\theta\di \pb)$ and have eigenvalues $-\Lambda_{\ell}=-(\ell-\half+s)(\ell-s+\half)$ defined by
\begin{equation}
\label{eq:eigenvalueSWSHO}
\edthR\edthR'(Y_{m,\ell}^{s}(\cos\theta)e^{im\pb})=
-\Lambda_{\ell}
Y_{m,\ell}^{s}(\cos\theta)e^{im\pb}.
\end{equation}
In particular,
\begin{subequations}
\label{eq:edthRandedthR':eigenvalues}
\begin{align}
\edthR (Y_{m,\ell}^{s}(\cos\theta)e^{im\pb})
={}&-\sqrt{\Big(\ell+s+\half\Big)\Big(\ell-s-\half\Big)}
Y_{m,\ell}^{s+1}(\cos\theta)e^{im\pb},\\
\edthR' (Y_{m,\ell}^{s}(\cos\theta)e^{im\pb})
={}&\sqrt{\Big(\ell+s-\half\Big)\Big(\ell-s+\half\Big)}
Y_{m,\ell}^{s-1}(\cos\theta)e^{im\pb}
\end{align}
\end{subequations}
and
\begin{subequations}
\begin{align}
\label{eq:l=l0mode:eigenvalue}
\edthR\edthR'\varphi^{\ell=\ell_0}
={}\Big(\Big(s-\half\Big)^2-\ell_0^2 \Big)\varphi^{\ell=\ell_0}, \quad
\edthR'\edthR\varphi^{\ell=\ell_0}
={}\Big(\Big(s+\half\Big)^2-\ell_0^2 \Big)\varphi^{\ell=\ell_0}.
\end{align}
The above are standard facts, see \cite{E82E} and \cite[Section 4]{penroserindlerI}.
Meanwhile, in view of \eqref{eq:TAopmsfrak:edthRedthR'},
\begin{align}
\label{eq:l=l0mode:eigenvalueofTAOpmsfrak}
\TAO_{s}\varphi^{\ell=\ell_0}
={}\Big(\Big(s-\half\Big)^2-\ell_0^2 -\abs{s}+s\Big)
\varphi^{\ell=\ell_0}=\Big(\Big(\abs{s}-\half\Big)^2-\ell_0^2 \Big)
\varphi^{\ell=\ell_0}.
\end{align}
\end{subequations}

The following proposition collects a few estimates for the edth operators $\edthR$ and $\edthR$' and the Teukolsky angular operator $\TAO_s$.

\begin{prop}
\label{prop:basicesti}
\begin{itemize}
\item
Let $\varphi$ be a spin-weight $s$ scalar and supported on $\ell\geq \ell_0$ modes, then
\begin{align}
\label{eq:ellip:highermodes}
\hspace{4ex}&\hspace{-4ex}
\int_{\mathbb{S}^2}\left(\abs{\edthR'\varphi}^2
-\Big(\ell_0^2- \Big(s-\half\Big)^2 \Big)\abs{\varphi}^2\right) \di^2\mu
=\int_{\mathbb{S}^2}\left(\abs{\edthR\varphi}^2
-\Big(\ell_0^2- \Big(s+\half\Big)^2 \Big)\abs{\varphi}^2\right) \di^2\mu \geq 0.
\end{align}
Let $\varphi$ be an arbitrary spin-weight $s$ scalar, then
\begin{align}
\label{eq:ellipestis}
\int_{\mathbb{S}^2}\left(\abs{\edthR'\varphi}^2
-(s+\abs{s})\abs{\varphi}^2\right) \di^2\mu =\int_{\mathbb{S}^2}\left(\abs{\edthR\varphi}^2
-(\abs{s}-s)\abs{\varphi}^2\right) \di^2\mu \geq 0.
\end{align}
\item Let $\varphi$ be a spin-weight $s$ scalar and supported on $\ell\geq \ell_0$ modes, then
\begin{align}
\label{eq:ellip:highermodes:TAO}
\int_{\mathbb{S}^2}\Re(-\TAO_s\varphi \bar{\varphi})
\di^2\mu
\geq{}&\int_{\mathbb{S}^2}\Big(\ell_0^2 - \Big(\abs{s}-\half\Big)^2\Big)\abs{\varphi}^2 \di^2\mu.
\end{align}
In particular, let $\varphi$ be an arbitrary spin-weight $s$ scalar, then
\begin{align}
\label{eq:ellipestis:TAO}
\int_{\mathbb{S}^2}\Re(-\TAO_s\varphi \bar{\varphi})
\di^2\mu
\geq{}&2\abs{s}\int_{\mathbb{S}^2}\abs{\varphi}^2 \di^2\mu.
\end{align}

\end{itemize}
\end{prop}

\begin{proof}
The proof for the estimates of the edth operators is similar to the one in \cite[Lemma 4.25]{andersson2019stability}.
By decomposing the spin-weight $s$ scalar $\varphi$, which is supported on $\ell\geq\ell_0$ modes, into the spin-weighted spherical harmonics as above and using equation \eqref{eq:edthRandedthR':eigenvalues}, equation \eqref{eq:ellip:highermodes}  is manifest. For an arbitrary spin-weight $s$ scalar, one can view it as a spin-weight $s$ scalar supported on $\ell\geq \ell_0=\abs{s}+\half$ modes. Equation \eqref{eq:ellipestis}  then follows from \eqref{eq:ellip:highermodes} by simply taking $\ell_0=\abs{s}+\half$.

We decompose $\varphi$ into $\ell$ modes and, using \eqref{eq:l=l0mode:eigenvalueofTAOpmsfrak}, we easily deduce \eqref{eq:ellip:highermodes:TAO}. The other equation follows from \eqref{eq:ellip:highermodes:TAO} by taking $\ell_0$ to be the lowest mode, that is, $\ell_0=\abs{s}+\half$.
\end{proof}

\subsection{Hardy and Sobolev estimates}

The following simple Hardy's inequality will be useful.
\begin{lemma}
Let $\varphi$ be a spin-weight $s$ scalar. Then for any $r'>r_+$,
\begin{align}
\label{eq:Hardy:trivial}
\int_{r_+}^{r'}\abs{\varphi}^2\di r
\lesssim{}&\int_{r_+}^{r'}\mu^2r^2\abs{\partial_r\varphi}^2\di r
+(r'-r_+)\abs{\varphi(r')}^2.
\end{align}
In particular, if $\lim\limits_{r\to \infty} r\abs{\varphi}^2 =0$, then
\begin{align}
\label{eq:Hardy:trivial:1}
\int_{r_+}^{\infty}\abs{\varphi}^2\di r
\lesssim{}&\int_{r_+}^{\infty}\mu^2r^2\abs{\partial_r\varphi}^2\di r.
\end{align}
\end{lemma}

\begin{proof}
It follows easily by integrating the following equation
\begin{align}
\partial_r((r-r_+)\abs{\varphi}^2)=\abs{\varphi}^2
+2(r-r_+)\Re(\bar{\varphi}\partial_r\varphi)
\end{align}
from $r_+$ to $r'$ and applying the Cauchy-Schwarz inequality to the last product term.
\end{proof}

We will also use the following standard Hardy's inequality, cf. \cite[Lemma 4.30]{andersson2019stability}.
\begin{lemma}[One-dimensional Hardy estimates]
\label{lem:HardyIneq}
Let $\alpha \in \mathbb{R}\setminus \{0\}$  and $h: [r_0,r_1] \rightarrow \mathbb{R}$ be a $C^1$ function.
\begin{enumerate}
\item \label{point:lem:HardyIneqLHS} If $r_0^{\alpha}\vert h(r_0)\vert^2 \leq D_0$ and $\alpha<0$, then
\begin{subequations}
\label{eq:HardyIneqLHSRHS}
\begin{align}\label{eq:HardyIneqLHS}
-2\alpha^{-1}r_1^{\alpha}\vert h(r_1)\vert^2+\int_{r_0}^{r_1}r^{\alpha -1} \vert h(r)\vert ^2 \di r \leq \frac{4}{\alpha^2}\int_{r_0}^{r_1}r^{\alpha +1} \vert \partial_r h(r)\vert ^2 \di r-2\alpha^{-1}D_0.
\end{align}
\item \label{point:lem:HardyIneqRHS} If $r_1^{\alpha}\vert h(r_1)\vert^2 \leq D_0$ and $\alpha>0$, then
\begin{align}\label{eq:HardyIneqRHS}
2\alpha^{-1}r_0^{\alpha}\vert h(r_0)\vert^2+\int_{r_0}^{r_1}r^{\alpha -1} \vert h(r)\vert ^2 \di r \leq \frac{4}{\alpha^2}\int_{r_0}^{r_1}r^{\alpha +1} \vert \partial_r h(r)\vert ^2 \di r +2\alpha^{-1}D_0.
\end{align}
\end{subequations}
\end{enumerate}
\end{lemma}

Recall the following Sobolev-type estimates from \cite[Lemmas 4.32 and 4.33]{andersson2019stability}.
\begin{lemma}
\label{lem:Sobolev}
Let $\varphi$ be a spin-weight $s$ scalar. Then
\begin{align}
\label{eq:Sobolev:1}
\sup_{\Sigmatb}\abs{\varphi}^2\lesssim_{s}{} \norm{\varphi}_{W_{-1}^3(\Sigmatb)}^2.
\end{align}
If $\alpha\in (0,1]$, then
\begin{align}
\label{eq:Sobolev:2}
\sup_{\Sigmatb}\abs{\varphi}^2\lesssim_{s,\alpha} {} (\norm{\varphi}_{W_{-2}^3(\Sigmatb)}^2
+\norm{rV\varphi}_{W_{-1-\alpha}^2(\Sigmatb)}^2)^{\half}
(\norm{\varphi}_{W_{-2}^3(\Sigmatb)}^2
+\norm{rV\varphi}_{W_{-1+\alpha}^2(\Sigmatb)}^2)^{\half}.
\end{align}
If $\lim\limits_{{\tb\to\infty}}\abs{r^{-1}\varphi}=0$ pointwise in $(\rb,\theta,\pb)$, then
\begin{align}
\label{eq:Sobolev:3}
\abs{r^{-1}\varphi}^2\lesssim_{s} {}\norm{\varphi}_{W_{-3}^3(\DOC_{\tb,\infty})}
\norm{\Lxi\varphi}_{W_{-3}^3(\DOC_{\tb,\infty})}.
\end{align}
For any $r'>2M$ away from horizon, if $\lim\limits_{{\tb\to\infty}}\abs{r^{-1}\varphi}=0$ pointwise in $(\rb,\theta,\pb)$, then
\begin{align}
\label{eq:Sobolev:4}
\abs{(r')^{-1}\varphi(r')}^2\lesssim_{s,r'} {}\norm{\varphi}_{W_{-3}^3(\DOC_{\tb,\infty}^{\geq ({r'+2M})/{2}})}
\norm{\Lxi\varphi}_{W_{-3}^3(\DOC_{\tb,\infty}^{\geq ({r'+2M})/{2}})}.
\end{align}
\end{lemma}

%%%%%%%%%%%%%%
\subsection{$r^p$ estimate for a general spin-weighted wave equation on Schwarzschild}
%%%%%%%%%%%%

We state here an $r^p$ estimate for a general spin-weighted wave equation, which is crucial in obtaining energy decay estimates as shown originally in \cite{dafermos2009new} for the scalar field and later in \cite{andersson2019stability} for the linearized gravity. The following statement and its proof are adapted from \cite{andersson2019stability,Ma20almost}. 

\begin{prop}
\label{prop:wave:rp}
Let $\reg\in \mathbb{N}$, $\abs{s}\in \half \mathbb{N}$, $\abs{s}\leq 2$\footnote{This proposition actually applies to a more general case where the spin-weight $s$ is an arbitrary half integer.}, and $p\in [0,2]$. Let $\delta\in (0,1/2)$ be arbitrary. Let $\varphi$ and $\vartheta=\vartheta(\varphi)$ be spin-weight $s$ scalars satisfying
\begin{align}
\label{eq:wave:rp}
-r^2YV\varphi
+\edthR\edthR'\varphi -b_V V\varphi -b_0\varphi=\vartheta.
\end{align}
Let the maximal eigenvalue of $\edthR\edthR'$ be $-\Lambda_s\leq 0$, i.e. $\abs{\edthR'\varphi}^2\geq \Lambda_s \abs{\varphi}^2$.
Let $b_{V}$, $b_{\phi}$ and $b_0$ be smooth real functions of $r$ such that
\begin{enumerate}
\item $\exists b_{V,-1}\in \mathbb{R}^+\cup\{0\}$ such that $b_V=b_{V,-1} r +O(1)$, and
\item $\exists b_{0,0}\in \mathbb{R}$ such that $b_0=b_{0,0}+O(r^{-1})$ and $b_{0,0}+\Lambda_s\geq 0$.
\end{enumerate}
Then there is a constant $\hat{R}_0=\hat{R}_0(p,b_0,b_V)$ such that for all $R_0\geq \hat{R}_0$ and $\tb_2>\tb_1\geq \tb_0$,
\begin{enumerate}
\item for $p\in (0, 2)$,
    \begin{align}\label{eq:rp:less2:2}
\hspace{2ex}&\hspace{-2ex}
\norm{rV\varphi}^2_{W_{p-2}^\reg(\Sigmatwo^{\geq R_0})}
+\norm{\varphi}^2_{W_{-2}^{\reg+1}(\Sigmatwo^{\geq R_0})}
+\norm{\varphi}^2_{W_{p-3}^{\reg+1}(\Donetwo^{\geq R_0})}
+\norm{Y\varphi}^2_{W_{-1-\delta}^{\reg}(\Donetwo^{\geq R_0})}
\notag\\
&\lesssim_{[R_0-M,R_0]} {}\norm{rV\varphi}^2_{W_{p-2}^\reg(\Sigmaone^{\geq R_0})}
+\norm{\varphi}^2_{W_{-2}^{\reg+1}(\Sigmaone^{\geq R_0})}
+\norm{\vartheta}^2_{W_{p-3}^{\reg}(\Donetwo^{\geq R_0})}
;
\end{align}
\item\label{pt:2:prop:wave:rp}for $p=2$,
\begin{align}\label{eq:rp:p=2:2}
\hspace{2ex}&\hspace{-2ex}
\norm{rV\varphi}^2_{W_{0}^\reg(\Sigmatwo^{\geq R_0})}
+\norm{\varphi}^2_{W_{-2}^{\reg+1}(\Sigmatwo^{\geq R_0})}
+\norm{\varphi}^2_{W_{-1-\delta}^{\reg+1}(\Donetwo^{\geq R_0})}
+\norm{rV\varphi}^2_{W_{-1}^{\reg}(\Donetwo^{\geq R_0})}
\notag\\
&\lesssim_{[R_0-M,R_0]} {}\norm{rV\varphi}^2_{W_{0}^\reg(\Sigmaone^{\geq R_0})}
+\norm{\varphi}^2_{W_{-2}^{\reg+1}(\Sigmaone^{\geq R_0})}
+\norm{\vartheta}^2_{W_{-1}^{\reg}(\Donetwo^{\geq R_0})},
\end{align}
and the term $\norm{\vartheta}^2_{W_{-1}^{\reg}(\Donetwo^{\geq R_0})}$ can be replaced by
\begin{align}
\label{eq:rp:p=2:2:errorterm}
\sum\limits_{\abs{\mathbf{a}}\leq \reg}\bigg|\int_{\Donetwo^{R_0}}\Re\Big(
V\overline{\mathbb{D}^{\mathbf{a}}\varphi}
\mathbb{D}^{\mathbf{a}}\vartheta\Big) \di^4 \mu\bigg|
+\norm{\vartheta}^2_{W_{-1-\delta}^{\reg}(\Donetwo^{R_0})};
\end{align}
\item for $p=0$ and $b_{V,-1}>0$,
\begin{align}\label{eq:rp:p=0}
\norm{\varphi}^2_{W_{-2}^{\reg+1}(\Sigmatwo^{R_0})}
+\norm{\varphi}^2_{W_{-3}^{\reg+1}(\Donetwo^{R_0})}
&\lesssim_{[R_0-M,R_0]} {}\norm{\varphi}^2_{W_{-2}^{\reg+1}(\Sigmaone^{R_0})}
+\norm{\vartheta}^2_{W_{-3}^{\reg}(\Donetwo^{R_0})};
\end{align}
\item for $p=0$ and $b_{V,-1}=0$,
\begin{align}\label{eq:rp:p=0:v2}
\norm{\varphi}^2_{W_{-2}^{\reg+1}(\Sigmatwo^{R_0})}
+\norm{\varphi}^2_{W_{-3}^{\reg+1}(\Donetwo^{R_0})}
&\lesssim_{[R_0-M,R_0]} {}\norm{\varphi}^2_{W_{-2}^{\reg+1}(\Sigmaone^{R_0})}
+\norm{rV\varphi}^2_{W_{-3}^{\reg}(\Donetwo^{R_0})}
+\norm{\vartheta}^2_{W_{-3}^{\reg}(\Donetwo^{R_0})},
\end{align}
\end{enumerate}
where integral terms $\norm{\varphi}^2_{W_{0}^{\reg+1}(\Sigmatwo^{R_0-M,R_0})}
+\norm{\varphi}^2_{W_{0}^{\reg+1}(\Sigmaone^{R_0-M,R_0})} +\norm{\varphi}^2_{W_{0}^{\reg+1}(\Donetwo^{R_0-M,R_0})}
+\norm{\vartheta}^2_{W_{0}^{\reg}(\Donetwo^{R_0-M, R_0})}$ supported on $[R_0-M, R_0]$ are implicit in the symbol $\lesssim_{[R_0-M, R_0]}$.

\end{prop}

\begin{proof}
The estimate \eqref{eq:rp:less2:2} with $p\in(0,2)$ has been proven in \cite[Lemmas 5.5 and 5.6]{andersson2019stability}. Note that there is an overall sign between equation and the one in \cite[Equation (5.29)]{andersson2019stability}, and this is responsible for the sign change in some terms in \eqref{eq:wave:rp}. This estimate is achieved by the following steps. We multiply equation \eqref{eq:wave:rp} by $-2\chi^2 r^{p-2}V\bar{\varphi}$ and take the real part, and based on the following calculations:
\begin{align}
\Re\big((-r^2 YV\varphi)(-2\chi^2 r^{p-2}V\bar{\varphi})\big)={}&Y(\chi^2 r^p \abs{V\varphi}^2)+\partial_r (\chi^2 r^p)\abs{V\varphi}^2,\\
\Re\big((\edthR\edthR'\varphi-b_0\varphi)(2\chi^2 r^{p-2}V\bar{\varphi})\big)={}&\Re\big(\edthR(-2\chi^2r^{p-2}\edthR'\varphi V\bar{\varphi})\big)
+V\big(\chi^2 r^{p-2}(\abs{\edthR'\varphi}^2+b_0\abs{\varphi}^2)\big)\notag\\
&
-\big(\mu \partial_r (\chi^2r^{p-2})\abs{\edthR'\varphi}^2 +\mu\partial_r (b_0\chi^2 r^{p-2})\abs{\varphi}^2\big),\\
\Re\big((-b_V V\varphi)(-2\chi^2 r^{p-2}V\bar{\varphi})\big)={}&2\chi^2 b_Vr^{p-2}\abs{V\varphi}^2,
\end{align}
we arrive at
\begin{align}
\label{eq:rplemma:0to2:general:mul}
&\Re(\edthR(-2\chi^2r^{p-2}\edthR'\varphi  V\overline{\varphi}))
+V(\chi^2 r^{p-2}(\abs{ \edthR'\varphi}^2-\Lambda_s\varphi^2
+(b_{0}+\Lambda_s)\abs{\varphi}^2))
+Y(\chi^2r^p\abs{V\varphi}^2)
\notag\\
&
+(\partial_r(\chi^2r^p)
+2\chi^2 b_{V}r^{p-2})\abs{V\varphi}^2
-\mu\partial_r(\chi ^2r^{p-2})
(\abs{\edthR'\varphi}^2-\Lambda_s\abs{\varphi}^2
+(b_{0}+\Lambda_s)\abs{\varphi}^2)
-\mu\partial_r b_0 \chi^2 r^{p-2} \abs{\varphi}^2\notag\\
&
={}-2\chi^2 r^{p-2}\Re(V\overline{\varphi}\vartheta).
\end{align}
As shown in  \cite[Lemmas 5.5 and 5.6]{andersson2019stability}, the estimate \eqref{eq:rp:less2:2}  then follows from a standard Morawetz estimate near infinity together with an integration over $\Donetwo$ and we omit the details.

For the remaining estimates, it suffices to show the case $\reg=0$, since the proof for the general $\reg\geq 1$ case is the same as the one in \cite[Lemmas 5.5 and 5.6]{andersson2019stability} or \cite[Proposition 2.9]{Ma20almost} and hence omitted.

To show point \eqref{pt:2:prop:wave:rp} for $p=2$, we take $p=2$ in equation \eqref{eq:rplemma:0to2:general:mul} and substitute  in $b_0=b_{0,0}+ (b_0-b_{0,0})$ and $b_V=rb_{V,-1}+ (b_{V}- rb_{V,-1})$, then we obtain
\begin{align}
&\Re(\edthR(-2\chi^2\edthR'\varphi  V\overline{\varphi}))
+V(\chi^2 (\abs{ \edthR'\varphi}^2-\Lambda_s\varphi^2
+(b_{0,0}+\Lambda_s)\abs{\varphi}^2))
+Y(\chi^2r^2\abs{V\varphi}^2)
\notag\\
&-\mu\partial_r(\chi ^2)
(\abs{\edthR'\varphi}^2-\Lambda_s\abs{\varphi}^2
+(b_{0,0}+\Lambda_s)\abs{\varphi}^2)
+(\partial_r(\chi^2r^2)
+2\chi^2 (rb_{V,-1}+ (b_{V}- rb_{V,-1})))\abs{V\varphi}^2
\notag\\
&
+V\big(\chi^2 (b_0-b_{0,0})\abs{\varphi}^2\big)
-\big(\mu \partial_r (\chi^2) (b_0-b_{0,0})
+\mu\partial_r (b_0-b_{0,0}) \chi^2 \big)\abs{\varphi}^2
\notag\\
&
={}-2\chi^2 \Re(V\overline{\varphi}\vartheta).
\end{align}
By integrating over $\Donetwo$ with a reference volume element $\di^4\mu$, the integral of the first term vanishes, the integral of the second term of the first line gives positive contribution of energy at $\Sigmatwo^{R_0}$ in terms of energy at $\Sigmaone^{R_0-M}$ by assymption, and the integral of the second line is positive definite and bounded below by $c\norm{rV\varphi}^2_{W_{-1}^{0}(\Donetwo^{\geq R_0})}$ by assumption. Further, since $b_0-b_{0,0}=O(r^{-1})$ and $b_V- rb_{V,-1}=O(1)$, the integral of the third line is bounded below by $-C\big(\norm{\varphi}^2_{W_{-3}^{0}(\Sigmatwo^{\geq R_0})}+\norm{\varphi}^2_{W_{-3}^{0}(\Sigmaone^{\geq R_0})}
+\norm{\varphi}^2_{W_{-2}^{0}(\Donetwo^{\geq R_0})}\big)$ modulo terms supported on $[R_0-M,R_0]$. Therefore, we add this estimate to the estimate with $p=2-\delta$ and infer
\begin{align}
&
\norm{rV\varphi}^2_{W_{0}^0(\Sigmatwo^{\geq R_0})}
+\norm{\varphi}^2_{W_{-2}^{1}(\Sigmatwo^{\geq R_0})}
+\norm{\varphi}^2_{W_{-1-\delta}^{1}(\Donetwo^{\geq R_0})}
+\norm{rV\varphi}^2_{W_{-1}^{0}(\Donetwo^{\geq R_0})}
\notag\\
&\lesssim_{[R_0-M,R_0]} {}\norm{rV\varphi}^2_{W_{0}^0(\Sigmaone^{\geq R_0})}
+\norm{\varphi}^2_{W_{-2}^{1}(\Sigmaone^{\geq R_0})}
+\norm{\vartheta}^2_{W_{-1-\delta}^{0}(\Donetwo^{R_0})}
+\bigg|\int_{\Donetwo^{R_0}}\Re(
V\bar{\varphi}
\vartheta) \di^4 \mu\bigg|.
\end{align}
One can also apply the
Cauchy--Schwarz inequality to the last term. and this proves the statement in point \eqref{pt:2:prop:wave:rp}.

The estimate \eqref{eq:rp:p=0} can be proved in a similar way. We multiply the wave equation \eqref{eq:wave:rp} by $-2\chi^2 r^{-2}V\bar{\varphi}$ and take the real part, arriving at
\begin{align}
&\Re(\edthR(-2\chi^2r^{-2}\edthR'\varphi  V\overline{\varphi}))
+V(r^{-2}\chi^2(\abs{ \edthR'\varphi}^2-\Lambda_s\varphi^2
+(b_{0,0}+\Lambda_s)\abs{\varphi}^2))
+Y(\chi^2\abs{V\varphi}^2)
\notag\\
&
+(\partial_r(\chi^2)
+2\chi^2 r^{-1}b_{V,-1})\abs{V\varphi}^2
-2\mu\partial_r(|\chi |^2r^{-2})
(\abs{\edthR'\varphi}^2-\Lambda_s\abs{\varphi}^2
+(b_{0,0}+\Lambda_s)\abs{\varphi}^2)\notag\\
&
+2\chi^2 r^{-2} \Re(V\bar{\varphi}[(b_V-rb_{V,-1})V\varphi+(b_0-b_{0,0})\varphi])\notag\\
&\quad
={}-2\chi^2 r^{-2}\Re(V\overline{\varphi}\vartheta).
\end{align}
By integrating over $\Donetwo$ with a reference volume element $\di^4\mu$, the integral of the first term vanishes, the integral of the second term gives positive contribution of energy at $\Sigmatwo^{R_0}$ in terms of energy at $\Sigmaone^{R_0-M}$, the integral of the second line dominates over
\begin{align}
\label{eq:rp:p=0:term1}
\int_{\Donetwo^{R_0}} r^{-3}(\abs{rV\varphi}^2+(\abs{\edthR'\varphi}^2-\Lambda_s\abs{\varphi}^2
+(b_{0,0}+\Lambda_s)\abs{\varphi}^2))\di^4\mu,
\end{align}
and the absolute value of the integrals of the last two lines are bounded from above using the Cauchy--Schwarz inequality by
\begin{align}
\label{eq:rp:p=0:term2}
\int_{\Donetwo^{R_0-M}} (\veps r^{-3}\abs{rV\varphi}^2
+\veps^{-1}r^{-5}\abs{\varphi}^2)\di^4\mu
+\veps^{-1}\int_{\Donetwo^{R_0-M}} r^{-3}\abs{\vartheta}^2\di^4\mu.
\end{align}
By using the  Hardy's inequality \eqref{eq:HardyIneqLHS}, the integral term \eqref{eq:rp:p=0:term1} is further bounded below by
\begin{align}
c_s\int_{\Donetwo^{R_0}} r^{-3}(\abs{rV\varphi}^2+\abs{\varphi}^2)\di^4\mu-C\int_{\Donetwo^{R_0-M,R_0}} \abs{\varphi}^2\di^4\mu.
\end{align}
and the first integral of this expression dominates over the first integral of \eqref{eq:rp:p=0:term2} by first taking $\veps$ small and then choosing $\hat{R}_0$ sufficiently large. Thus, this proves the $\reg=0$ case of inequality \eqref{eq:rp:p=0}.

In the last case that $b_{V,-1}=0$, we can subtract $-rV\varphi$ on both sides of \eqref{eq:wave:rp} such that the obtained equation satisfies the estimate \eqref{eq:rp:p=0}. The source term of this new equation becomes $\vartheta -rV\varphi$, hence by taking into account of this replacement,
the estimate \eqref{eq:rp:p=0:v2} follows manifestly from \eqref{eq:rp:p=0}.
\end{proof}

%%%%%%%%%%%%%%55
\subsection{Decay estimates}
%%%%%%%%%%%%%%

The following two lemmas are quite useful in deriving energy decay estimates.

The first one proves that a hierarchy of energy and Morawetz estimate implies a decay rate for the energy terms in the hierarchy. The current statement of this lemma is essentially the same as \cite[Lemma 5.2]{andersson2019stability}, and it can be proved in the exactly same way. In applications, $\iPigeonReg$ represents a level of regularity, $\alpha$ represents a weight, and $\PigeonTime$ represents a time coordinate. The weights take values in an interval, whereas the levels of regularity are discrete.

\begin{lemma}[A hierarchy of estimates implies decay rates]
\label{lem:hierarchyImpliesDecay}
Let $D\geq 0$. Let $\alpha_1,\alpha_2\in\Reals$ and $i\in\Integers^+$ be such that $\alpha_1\leq\alpha_2-1$, and $\alpha_2-\alpha_1\leq i$. Let $F:\{-1,\ldots,i\}\times[\alpha_1-1,\alpha_2]\times[\timefunc_0,\infty)\rightarrow[0,\infty)$ be such that $F(\iPigeonReg,\alpha,\PigeonTime)$ is Lebesgue measurable in $\PigeonTime$ for each $\alpha$ and $\iPigeonReg$. Let $\gamma\geq 0$.

If
\begin{subequations}
\begin{enumerate}
\item{} [monotonicity] \label{assump:HierarchyToDecay(1)}for all $\iPigeonReg,\iPigeonReg_1,\iPigeonReg_2\in\{-1,\ldots,i\}$ with $\iPigeonReg_1\leq \iPigeonReg_2$, all $\beta, \beta_1,\beta_2\in[\alpha_1,\alpha_2]$ with $\beta_1\leq\beta_2$, and all $\PigeonTime\geq \timefunc_0$,
\begin{align}
F(\iPigeonReg_1,\beta,\PigeonTime)\lesssim{}& F(\iPigeonReg_2,\beta,\PigeonTime) ,
\label{eq:Rev:HierarchyToDecayReal:MonotonicityHypothesis:j}\\
F(\iPigeonReg,\beta_1,\PigeonTime)\lesssim{}& F(\iPigeonReg,\beta_2,\PigeonTime) ,
\label{eq:Rev:HierarchyToDecayReal:MonotonicityHypothesis:beta}
\end{align}
\item{} [interpolation] \label{assump:HierarchyToDecay(2)}for all $\iPigeonReg\in\{-1,\ldots,i\}$, all $\alpha,\beta_1,\beta_2\in[\alpha_1,\alpha_2]$ such that $\beta_1\leq\alpha\leq\beta_2$, and all $\PigeonTime\geq \timefunc_0$,
\begin{align}
F(\iPigeonReg,\alpha,\PigeonTime)
\lesssim{}&
F(\iPigeonReg,\beta_1,\PigeonTime)^{\frac{\beta_2-\alpha}{\beta_2-\beta_1}}
F(\iPigeonReg,\beta_2,\PigeonTime)^{\frac{\alpha-\beta_1}{\beta_2-\beta_1}} ,
\label{eq:Rev:HierarchyToDecayReal:InterpolationHypothesis}
\end{align}
\item{} [energy and Morawetz estimate] for all $\iPigeonReg\in\{0,\ldots,i\}$, $\alpha\in[\alpha_1,\alpha_2]$, and $\PigeonTime_2\geq \PigeonTime_1\geq \timefunc_0$,
\begin{align}
F(\iPigeonReg,\alpha,\PigeonTime_2)
+\int_{\PigeonTime_1}^{\PigeonTime_2} F(\iPigeonReg-1,\alpha-1,t) \di \PigeonTime
\lesssim F(\iPigeonReg,\alpha,\PigeonTime_1) +D \PigeonTime_1^{\alpha-\alpha_2-\gamma} ,
\label{eq:Rev:HierarchyToDecayReal:EvolutionHypothesis}
\end{align}
and
\item{} [initial decay rate] if $\gamma>0$, then for any  $\PigeonTime\geq \timefunc_0$,
\begin{align}
F(i,\alpha_2,\PigeonTime)
\lesssim \PigeonTime^{-\gamma} \left(F(i,\alpha_2,\timefunc_0) +D\right) ,
\label{eq:Rev:HierarchyToDecayReal:InitialDecay}
\end{align}
\end{enumerate}
\end{subequations}
then, for all $\iPigeonReg\in\{0,\ldots,i\}$, all $\alpha\in[\max\{\alpha_1,\alpha_2-\iPigeonReg\},\alpha_2]$, and all $\tb\geq 2\tb_0$,
\begin{align}
F(i-\iPigeonReg,\alpha,\PigeonTime) \lesssim{}& \PigeonTime^{\alpha-\alpha_2-\gamma}  (F(i,\alpha_2,\timefunc/2) +D),
\end{align}
and for all $\PigeonTime\geq \timefunc_0$,
\begin{align}
F(i-\iPigeonReg,\alpha,\PigeonTime) \lesssim{}& \PigeonTime^{\alpha-\alpha_2-\gamma}  (F(i,\alpha_2,\timefunc_0) +D),
\label{eq:Rev:HierarchyToDecayReal}
\end{align}
where the implicit constant in $\lesssim$ can depend on $\alpha_2$ and $\alpha_1$.
\end{lemma}

The second one is one type of Gr\"{o}nwall inequality cited from \cite[Lemma 7.4]{angelopoulos2018late}.

\begin{lemma}
\label{lem:AAGlemma}
Let $f: [\tb_0,\infty)\rightarrow \mathbb{R}^+$ be a continuous, positive function. Assume there exist positive constants $E_0, C_0, b$ and $p$ such that for all $\tb_0\leq \tb_1< \tb_2$,
\begin{subequations}
\begin{align}
f(\tb_2) +b\int_{\tb_1}^{\tb_2}f(\tb)\di \tb
\leq {} f(\tb_1) +
E_0(\tb_2-\tb_1)\tb_1^{-p},\\
f(\tb_2) +b\int_{\tb_1}^{\tb_2}f(\tb)\di \tb
\leq {} f(\tb_1) + C_0(\tb_2 -\tb_1)f(\tb_0).
\end{align}
\end{subequations}
Then for all $\tb\geq \tb_0$,
\begin{subequations}
\begin{align}
f(\tb)\lesssim_{C_0,b} f(\tb_0)
\end{align}
and
\begin{align}
f(\tb)\lesssim_{E_0, C_0, b, p}\tb^{-p}(f(\tb_0)+E_0).
\end{align}
\end{subequations}
\end{lemma}

%%%%%%%%%%%%%%%%%%%%
\section{Energy estimates}
\label{mainsect:energyestimates}
%%%%%%%%%%%%%%%%%%

In this section, we prove the uniform boundedness for a nondegenerate, positive definite energy. In contrast to the scalar field case where the spin-weight is zero, it is usually nontrivial to derive a uniform energy bound for the TME \eqref{eq:TME} under non-zero spin-weight case.
For Dirac field, it differs from the integer spin-weight case (including the scalar field case) in the absence of the Chandrasekhar's transformation \cite{chandrasekhar1975linearstabSchw} which is a differential operator and transforms the TME \eqref{eq:TME} into a scalar-wave-like equation such that a conserved energy can be easily constructed and energy and Morawetz estimates can be derived following a rather similar proof as the scalar field case. See \cite{dafermos2019linear,pasqualotto2019spin,Ma2017Maxwell,Ma17spin2Kerr} the application of the Chandasekhar's transformation to the spin $\pm 1$ and $\pm 2$ cases on Schwarzschild and Kerr backgrounds. An interesting feature about the Dirac equation is that one can use the first order Dirac equations to rewrite the TME for both spin $\pm \half$ components and obtain \emph{a second order symmetric hyperbolic wave system} for these two components. This feature enables us to derive a conservation law, the density of which contains both components, from this symmetric hyperbolic wave system, and it is also an essential property used in proving the Morawetz estimate in the next section.

\subsection{Rewrite the TME and Dirac equations}

%%%%%%%%%%%%%%%

The purpose of this subsection is to rewrite the first order Dirac equations and the TME into more concise forms. 

We first introduce a few scalars and wave operators. Apart from the scalars $\psis$, we shall need as well the following ones which are defined by performing $r$-rescalings on them.
\begin{definition}
Define
\begin{subequations}
\begin{align}
\label{def:phis}
\phis={}&\left\{
     \begin{array}{ll}
       \psips/r^{2\sfrak}=r^{-1}\psips, & \hbox{$s=\sfrak$;} \\
       \Delta^{\sfrak}\psins/r^{2\sfrak}=\mu^{\half}\psins, & \hbox{$s=-\sfrak$,}
     \end{array}
   \right.\\
\label{def:varphis}
\Phis={}&r\phis
={}\left\{
     \begin{array}{ll}
       \psips, & \hbox{$s=\sfrak$;} \\
       \Delta^{\half}\psins, & \hbox{$s=-\sfrak$.}
     \end{array}
   \right.
\end{align}
\end{subequations}
\end{definition}

\begin{remark}
In particular, the scalars $\phins$ and $\Phins$ are degenerate at $\Horizon$, and, $\Delta^{-\sfrak}\phins$ and $\Delta^{-\sfrak}\Phins$ are nondegenerate at $\Horizon$.
\end{remark}

\begin{definition}
Define Teukolsky wave operators 
\begin{subequations}
\label{eq:Teukolskywaveoperators}
\begin{align}
\label{eq:TMEope:posi}
\TMEOps={}&r^2\Box_{g_{M}} + \tfrac{2i\sfrak\cos\theta}{\sin^2 \theta}\Leta - (\sfrak^2\cot^2\theta +\sfrak)
-(r-3M)(Y-r^{-1})
+2Mr^{-1}\notag\\
={}&-\mu^{-1}r^2 \partial^2_t
+ \partial_r \left( \Delta \partial_r \right)
+\TAO_{\sfrak}-(r-3M)(Y-r^{-1})
+2Mr^{-1},\\
\label{eq:TMEope:nega}
\TMEOpns={}&r^2\Box_{g_{M} } -\tfrac{2i\sfrak\cos\theta}{\sin^2 \theta}\Leta - (\sfrak^2\cot^2\theta +\sfrak)
+(r-3M)(\VR+r^{-1})
+2Mr^{-1}\notag\\
={}&-\mu^{-1}r^2 \partial^2_t
+ \partial_r \left( \Delta \partial_r \right)
+\TAO_{-\sfrak}+(r-3M)(\VR+r^{-1})
+2Mr^{-1}.
\end{align}
\end{subequations}
\end{definition}

We can then rewrite the Dirac equations and the TME as follows. 

\begin{prop}\label{prop:DiracandTME}
The Dirac equations for $\Phis$ are
\begin{subequations}\label{eq:Dirac:sym}
\begin{align}
\label{eq:Dirac:sym:posi}
\curlLs   \Phips ={}&
(\Delta^{\Half}\VR)\Phins ,\\
\label{eq:Dirac:sym:nega}
\curlLds  \Phins ={}&
(\Delta^{\Half}Y)\Phips ,
\end{align}
\end{subequations}
and the wave equations for $\Phis$ are
\begin{subequations}
\label{eq:TME:varphis}
\begin{align}
\label{eq:TME:varphis:posi}
\curlLds  \curlLs   \Phips
-\Delta^{\Half}\VR(\Delta^{\Half}Y\Phips )={}&0,\\
\label{eq:TME:varphis:nega}
\curlLs   \curlLds  \Phins
-\Delta^{\Half}Y(\Delta^{\Half}\VR\Phins )={}&0,
\end{align}
\end{subequations}
or, equivalently, 
\begin{subequations}
\label{eq:TME:varphis:exf}
\begin{align}
\label{eq:TME:varphis:posi:exf}
-r^2 YV\Phips+\curlLds  \curlLs   \Phips
={}&(r-3M)Y\Phips,\\
\label{eq:TME:varphis:nega:exf}
-r^2 YV\Phins+\curlLs   \curlLds  \Phins
={}&-(r-3M)\VR\Phins,
\end{align}
\end{subequations}
Meanwhile, $\phis$ satisfies the Teukolsky master equation (TME) 
\begin{subequations}
\label{eq:Dirac:TME:phis}
\begin{align}
\TMEOps\phips={}0,
\label{eq:DTMEP}\\
\TMEOpns\phins={}&0.
\label{eq:DTMEN}
\end{align}
\end{subequations}
\end{prop}

\begin{proof}
Equations \eqref{eq:Dirac:sym} are direct from equations \eqref{eq:Dirac:TMEscalar}.

We apply $\curlLds  $ to \eqref{eq:Dirac:sym:posi} and $\Delta^{\Half}V$ to \eqref{eq:Dirac:sym:nega} and then take the difference. Owing to the commuatator  $[\edthR, \VR]=0$ from equation \eqref{comm:edthRedthR'YVR}, we thus eliminate $\Phins$ and arrive at \eqref{eq:TME:varphis:posi}. The other equation \eqref{eq:TME:varphis:nega} can be derived in a similar fashion.

The TME system \eqref{eq:TME:varphis} can be written as
\begin{subequations}
\begin{align}
\edthR\edthR'\Phips -r^2 YV \Phips+\Delta^{\half}\partial_r (\mu^{-1}\Delta^{\half}) V\Phips -(r-3M)(Y+\VR)\Phips={}&0,\\
\edthR'\edthR\Phins-r^2 YV \Phins+\Delta^{\half}\partial_r (\mu^{-1}\Delta^{\half}) V\Phins={}&0.
\end{align}
\end{subequations}
Note that $\Delta^{\half}\partial_r (\mu^{-1}\Delta^{\half})V\Phins=\mu^{-1}(r-3M)V\Phins=(r-3M)\VR\Phins$ and $\Delta^{\half}\partial_r (\mu^{-1}\Delta^{\half}) V\Phips -(r-3M)(Y+\VR)\Phips=-(r-3M)Y\Phips$, hence equations \eqref{eq:TME:varphis:exf} follow.

Equation \eqref{eq:TME:varphis:nega} can be written as
\begin{align*}
\TAO_{-\sfrak}\Phins-r^2 YV \Phins+\Delta^{\half}\partial_r (\mu^{-1}\Delta^{\half}) V\Phins={}0
\end{align*}
We use \eqref{eq:r2YVexpansion} and definition \eqref{def:varphis} and find $\TAO_{-\sfrak}\Phins-r^2 YV \Phins=r\TAO_{-\sfrak}\phins+r[-\mu^{-1} r^2\partial^2_t
+ \partial_r \left( \Delta \partial_r \right)+2Mr^{-1}]\phins$ and $\Delta^{\half}\partial_r (\mu^{-1}\Delta^{\half}) V\Phins=r(r-3M)(\VR+r^{-1})\phins$, hence, in view of  \eqref{eq:TMEope:nega}, this proves equation \eqref{eq:DTMEN}. Similarly, we use the commutator \eqref{comm:DeltahalfVRY} to find that the LHS of \eqref{eq:TME:varphis:posi} equals
\begin{align*}
\hspace{4ex}&\hspace{-4ex}
\TAO_{\sfrak}\Phips- \Delta^{\Half}Y(\Delta^{\Half}\VR\Phips)
-[\Delta^{\half}\VR, \Delta^{\half}Y]\Phips\notag\\
={}&r[\TAO_{\sfrak} -\mu^{-1} r^2\partial^2_t
+ \partial_r \left( \Delta \partial_r \right)+2Mr^{-1}]\phips
+r(r-3M)(\VR+r^{-1})\phips
-(r-3M)(Y+\VR)(r\phips)\notag\\
={}&r[\TAO_{\sfrak} -\mu^{-1} r^2\partial^2_t
+ \partial_r \left( \Delta \partial_r \right)+2Mr^{-1}-(r-3M) (Y-r^{-1})]\phips,
\end{align*}
and this proves \eqref{eq:DTMEP}.
\end{proof}

\begin{remark}
\label{rem:symhypform:TMEsystem}
The rescaling in defining $\phi_s$ is chosen such that the RHS contains only $Y$ or $\VR$ derivative. In particular, by rewriting $(Y-r^{-1})\phips$ and $(\VR+r^{-1})\phins$  as $\curlLds   \phins$ and $\curlLs   \phips$ respectively in the wave system \eqref{eq:Dirac:TME:phis},  we obtain a second order symmetric hyperbolic system. See Section \ref{sect:symhypwavesys}.
\end{remark}

\begin{remark}
The coefficients of the first order $\partial_{\tb}$ operator in the Teukolsky wave operators in \eqref{eq:Teukolskywaveoperators} change the sign at $r=3M$ and have opposite signs at any fixed radius for spin $\half$ and $-\half$ components. As a result, damping or antidamping occurs in different radius regions for different spin components.
\end{remark}

%%%%%%%%%%%%%%%%%%%
\subsection{Symmetric hyperbolic wave system}
\label{sect:symhypwavesys}
%%%%%%%%%%%%%%%%%%%

As is discussed in Remark \ref{rem:symhypform:TMEsystem}, one can
rewrite $(Y-r^{-1})\phips$ and $(\VR+r^{-1})\phins$  as $\curlLds   \phins$ and $\curlLs   \phips$ respectively in the TME system to achieve a second order symmetric hyperbolic wave system. We show this explicitly here, and the obtained symmetric hyperbolic wave system is  crucial in deriving both the energy estimate and the integrated local energy decay (or Morawetz) estimate in the later sections.

\begin{prop}
The spin $\pm \half$ components $\Phis$ satisfy the following symmetric hyperbolic wave system
\begin{subequations}
\label{eq:TME:varphis:mode}
\begin{align}
\label{eq:TME:varphis:mode:posi}
\Delta^{-1}r^4(\partial_t +\mu \partial_r)(\partial_t -\mu \partial_r)\Phips -\TAO_{\sfrak} \Phips={}&-\Delta^{-\half}(r-3M)\curlLds\Phins,\\
\label{eq:TME:varphis:mode:nega}
\Delta^{-1}r^4(\partial_t +\mu \partial_r)(\partial_t -\mu \partial_r)\Phins
-\TAO_{-\sfrak}  \Phins ={}&\Delta^{-\half}(r-3M)\curlLs\Phips.
\end{align}
\end{subequations}
Also, the spin $\pm \half$ components $\phis$ satisfy  the following symmetric hyperbolic wave system
\begin{subequations}
\label{eq:Dirac:TME:phis:SH}
\begin{align}
-\mu^{-1}r^2 \partial^2_t\phips
+ \partial_r \left( \Delta \partial_r \phips\right)
+\edthR\edthR'\phips
+2Mr^{-1}\phips={}&\Delta^{-\half}(r-3M)\edthR\phins,
\label{eq:DTMEP:SH}\\
-\mu^{-1}r^2 \partial^2_t\phins
+ \partial_r \left( \Delta \partial_r \phins\right)
+\edthR'\edthR\phins
+2Mr^{-1}\phins={}&-\Delta^{-\half}(r-3M)\edthR'\phips.
\label{eq:DTMEN:SH}
\end{align}
\end{subequations}

\end{prop}

\begin{proof}
Recall the system \eqref{eq:TME:varphis:exf} and the expressions $Y=\mu^{-1}(\partial_t -\mu \partial_r)$ and $V=\partial_t+\mu \partial_r$. We use the first order Dirac equations \eqref{eq:Dirac:sym} to rewrite the RHS of each subequation in the system \eqref{eq:TME:varphis:exf}  as the angular derivative of the other spin component and this proves \eqref{eq:TME:varphis:mode}.

We expand the Teukolsky wave operators in the TME system \eqref{eq:Dirac:TME:phis} using equation \eqref{eq:Teukolskywaveoperators}, and, substituting in the relations $(Y-r^{-1})\phips=\Delta^{-\half}\edthR\phins$ and $(\VR+r^{-1})\phins=\Delta^{-\half}\edthR'\phips$ which follow directly from the first order Dirac equations \eqref{eq:Dirac:sym}, the system \eqref{eq:Dirac:TME:phis:SH} then follows.
\end{proof}

%%%%%%%%%%%%%%%%%%%%%%%%%

\subsection{A conservation law for the Dirac system}
%%%%%%%%%%%%
We prove a conservation law for the Dirac system \eqref{eq:Dirac:sym}. This is usually interpreted as a conserved charge. 

\begin{definition}
Let $\mathbf{a}=(a_1,a_2,a_3,a_4)$ be a multi-index with $a_1\in \{0,1\}$, $a_3$ and $a_4$ being any nonnegative integers, and $a_2$ being any nonnegative even integer. Define for a spin-weight $-\half$ component $\varphi$ that
\begin{subequations}
\begin{align}
\varphi^{(\mathbf{a})}={}&
(\curlLds  )^{a_1}
(\curlLs   \curlLds  )^{a_2/2}(\Lxi)^{a_3} (\Leta)^{a_4}\varphi ,
\end{align}
and for a spin-weight $\half$ component $\varphi$ that
\begin{align}
\varphi^{(\mathbf{a})}={}&(\curlLs   )^{a_1}
(\curlLds  \curlLs   )^{a_2/2}(\Lxi)^{a_3} (\Leta)^{a_4}\varphi .
\end{align}
\end{subequations}
\end{definition}

\begin{prop}
\label{prop:conservationlaw:Highorder}
It holds true that for any $\tb_0\leq \tb_1<\tb_2$, 
\begin{align}
\label{eq:current:highorder}
\hspace{4ex}&\hspace{-4ex}
\int_{\Sigmatwo}
\Big[\partial_r h |\Phips ^{(\mathbf{a})}|^2
+(2\mu^{-1}-\partial_rh)|\Phins ^{(\mathbf{a})}|^2
\Big]\di^3 \mu
+\int_{\Horizononetwo}|
\Phips ^{(\mathbf{a})}|^2 \di v\di^2 \mu
+\int_{\Scrionetwo}|\Phins ^{(\mathbf{a})}|^2\di u\di^2 \mu\notag\\
={}&\int_{\Sigmaone}
\Big[\partial_r h |\Phips ^{(\mathbf{a})}|^2
+(2\mu^{-1}-\partial_rh)|\Phins ^{(\mathbf{a})}|^2
\Big]\di^3 \mu.
\end{align}
In particular, from the choice of the function $h$, the integrand in the integrals over $\Sigmaone$ and $\Sigmatwo$ is positive definite and is equivalent to $|\Phips ^{(\mathbf{a})}|^2 + \mu^{-1}r^{-2}|\Phins ^{(\mathbf{a})}|^2$.
\end{prop}

\begin{remark}
\label{rem:conservationlaw:degeneracy}
For $\abs{\mathbf{a}}\geq 1$, one can make use of the Dirac equations \eqref{eq:Dirac:sym} to obtain estimates for one principle null derivative of each component and thus find the first integral term on the LHS of \eqref{eq:current:highorder} bounds over $c(r')\big(\norm{\psips}_{\tilde{W}_{0}^{\reg+1}(\Sigmatwo^{\geq r'})}^2
+\norm{\psins}_{\tilde{W}_{0}^{\reg+1}(\Sigmatwo^{\geq r'})}^2\big)$ for any $r'>2M$. However, this energy does not have full control over all derivatives up to  event horizon (for instance, an integral of $\abs{Y^{\abs{\mathbf{a}}} \psins}^2$ over $\Sigmatwo$ can not be dominated by such an energy) or null infinity, and it is from this respect that such an energy is degenerate.
\end{remark}

\begin{proof}
By multiplying \eqref{eq:Dirac:sym:posi} by $2\Delta^{-\Half}\overline{\Phins }$ and \eqref{eq:Dirac:sym:nega} by $2\Delta^{-\Half}\overline{\Phips }$, taking the real part, and adding the obtained two identities together, one obtains
\begin{align}
\int_{\Sphere}\VR(|\Phins |^2) + Y(|\Phips |^2)\di^2 \mu
={}&\int_{\Sphere}2\Delta^{-\Half}\Re\big(\curlLs   \Phips
\overline{\Phins } +\curlLds   \Phins
\overline{\Phips }\big)\di^2 \mu=0.
\end{align}
As shown in Proposition \ref{prop:DiracandTME}, the wave equations \eqref{eq:TME:varphis} can be rewritten using the Dirac equations \eqref{eq:Dirac:sym} as
\begin{subequations}\label{eq:Dirac:sym:firstorder}
\begin{align}
\label{eq:Dirac:sym:firstorder:posi}
\curlLds  (\curlLs   \Phips )={}&
(\Delta^{\Half}\VR)(\curlLds  \Phins ),\\
\label{eq:Dirac:sym:firstorder:nega}
\curlLs   (\curlLds  \Phins )={}&
(\Delta^{\Half}Y)(\curlLs   \Phips ).
\end{align}
\end{subequations}
Similarly as above, one can obtain an equality
\begin{align}
\int_{\Sphere}\big[\VR(|\curlLds  \Phins |^2) + Y(|\curlLs   \Phips |^2)\big]\di^2 \mu
={}&\int_{\Sphere}2\Delta^{-\Half}\Re\big(\curlLds
\curlLs   \Phips
\overline{\curlLds  \Phins } +\curlLs   \curlLds  \Phins
\overline{\curlLs   \Phips }\big)\di^2 \mu=0.
\end{align}
Additionally, the Killing vectors $\Lxi$ and $\Leta$ and Killing tensor $\TAO_s$ commute with the Dirac equations \eqref{eq:Dirac:sym},
therefore, we have the following equality
\begin{align}
\label{eq:current:firstorder}
\int_{\Sphere}\big[\VR(|\Phins ^{(\mathbf{a})}|^2) + Y(|
\Phips ^{(\mathbf{a})}|^2)\big]\di^2 \mu
={}&0.
\end{align}
By integrating over $\Donetwo$ and making use of \eqref{eq:YV:hyper},
this implies the desired conservation law.
\end{proof}

%%%%%%%%%%%%%%%%%%%%%%
\subsection{A conservation law for the symmetric hyperbolic wave system}
\label{sect:ConserLaw:Wavesystem}
%%%%%%%%%%%%%%%%

It is difficult to find a conserved energy for one single spin component from each of the wave system \eqref{eq:TME:varphis} because of the presence of the first order null derivative term. Nevertheless, we show below that we can construct a conserved energy, the density of which involves \underline{both} spin $\pm \half$ components, from an alternative form of the TME system \eqref{eq:TME:varphis} -- the symmetric hyperbolic system \eqref{eq:Dirac:TME:phis:SH}. This energy (albeit indefinite) conservation allows us to bound an energy flux at $\Scrionetwo$ and a flux on $\Horizononetwo$ in terms of energies on both $\Sigmatwo$ and $\Sigmaone$. Together with the conserved energy in Proposition \ref{prop:conservationlaw:Highorder} which is constructed from the first order Dirac system \eqref{eq:Dirac:sym} and a red-shift estimate near horizon, we will achieve the uniform boundedness of a nondegenerate positive definite energy in Section \ref{sect:uniformbd:nondegenergy}.

\begin{definition}
\label{def:TAOshalfop}
Denote $\TAO_{s}^{\half}=\edthR'$ when acting on a spin-weight $s=\half$ scalar and $\TAO_{s}^{\half}=\edthR$ for $s=-\half$ if acting on a spin-weight $-\half$ scalar.
\end{definition}

\begin{remark}
\label{rem:TAOshalfandTAOsellip}
By this definition, equation \eqref{eq:TAopmsfrak:edthRedthR'} and Lemma \ref{lem:IBPonSphere}, we have for any spin-weight $s=\pm \half$ scalar $\varphi$ that
\begin{align}
\int_{\Sphere} -\Re(\TAO_{s}\varphi \bar{\varphi})\di^2\mu=\int_{\Sphere}|\TAO_{s}^{\half}\varphi|^2\di^2\mu.
\end{align}
\end{remark}

\begin{prop}
\label{prop:conservationlaw:symhypwavesys}
For any $\tb_0\leq \tb_1<\tb_2$, there is a conservation law
\begin{align}
\label{eq:EnerEsti:(4)}
\hspace{4ex}&\hspace{-4ex}
\int_{\Horizononetwo}\left[
\mu e_t(\phi_{\pm\sfrak}^{(\mathbf{a})})
-e_r(\phi_{\pm\sfrak}^{(\mathbf{a})})\right]\di v\di^2\mu
+\int_{\Scrionetwo}\left[
e_t(\phi_{\pm\sfrak}^{(\mathbf{a})})
+\mu^{-1}e_r(\phi_{\pm\sfrak}^{(\mathbf{a})})\right]\di u\di^2\mu\notag\\
\hspace{4ex}&\hspace{-4ex}
+\int_{\Sigmatwo}\left[e_t(\phi_{\pm\sfrak}^{(\mathbf{a})})
+(\mu^{-1}-\partial_r h)e_r(\phi_{\pm\sfrak}^{(\mathbf{a})})\right]\di^3\mu
=\int_{\Sigmaone}\left[e_t(\phi_{\pm\sfrak}^{(\mathbf{a})})
+(\mu^{-1}-\partial_r h)e_r(\phi_{\pm\sfrak}^{(\mathbf{a})})\right]\di^3\mu,
\end{align}
where
\begin{subequations}
\begin{align}
\label{eq:etHigh:conservationlaw:wavesystem}
e_t(\phi_{\pm\sfrak}^{(\mathbf{a})})={}&
r^2 \mu^{-1}
|\partial_t\phi_{\pm\sfrak}^{(\mathbf{a})}|^2
+r^2 \mu \abs{\partial_r \phi_{\pm\sfrak}^{(\mathbf{a})}}^2
+\abs{\edthR'\phi_{\sfrak}^{(\mathbf{a})}}^2
+\abs{\edthR\phi_{-\sfrak}^{(\mathbf{a})}}^2-2Mr^{-1}
|\phi_{\pm\sfrak}^{(\mathbf{a})}|^2\notag\\
&+{2(r-3M)}{\Delta^{-1/2}}
\Re\big(\TAO_{s}^{\half}\phins^{(\mathbf{a})}\overline{   \phips^{(\mathbf{a})}}\big),
\\
e_r(\phi_{\pm\sfrak}^{(\mathbf{a})})={}&-2\Delta\Re\big(
\partial_{t}\overline{\phi_{\pm\sfrak}^{(\mathbf{a})}}
\partial_{r}{\phi_{\pm\sfrak}^{(\mathbf{a})}}\big).
\end{align}
\end{subequations}
\end{prop}

\begin{remark}
This conservation law alone does not provide a bound on a positive definite energy, as the indefinite term in last line of \eqref{eq:etHigh:conservationlaw:wavesystem} can not be bounded by the first line due to the blowup factor $\Delta^{-\half}$ near horizon. \end{remark}

\begin{proof}
We prove only for $\abs{\mathbf{a}}=0$, and the general $\abs{\mathbf{a}}\geq 0$ follows in the same manner as proving Proposition \ref{prop:conservationlaw:Highorder}.
 Multiplying equation \eqref{eq:DTMEP:SH} by $-2\partial_t\overline{\phips}$  and equation \eqref{eq:DTMEN:SH} by $-2\partial_t\overline{\phins}$, taking the real part, and summing together, we obtain
\begin{align}
\label{eq:EnerEsti:t}
\partial_t e^{(1)}_t(\phi_{\pm\sfrak})+\partial_r e^{(1)}_r(\phi_{\pm\sfrak})
\equiv{}&
{2(r-3M)}\Delta^{-\half}\Re\left(
-\curlLds  \phins \partial_t\overline{\phips }
+\curlLs   \phips \partial_t\overline{\phins }\right),
\end{align}
where
\begin{subequations}
\begin{align}
e^{(1)}_t(\phi_{\pm\sfrak})={}&
r^2 \mu^{-1}
|\partial_t\phi_{\pm\sfrak}|^2
+r^2 \mu \abs{\partial_r \phi_{\pm\sfrak}}^2
+\abs{\edthR'\phi_{\sfrak}}^2
+\abs{\edthR\phi_{-\sfrak}}^2-2Mr^{-1}
|\phi_{\pm\sfrak}|^2,\\
e^{(1)}_r(\phi_{\pm\sfrak})={}&
-2\Delta\Re\left(\partial_{t}\overline{\phi_{\pm\sfrak}}
\partial_{r}{\phi_{\pm\sfrak}}\right).
\end{align}
\end{subequations}
We then use  the Dirac equations \eqref{eq:Dirac:sym}  and Lemma \ref{lem:IBPonSphere} and find that the terms on the RHS of equation \eqref{eq:EnerEsti:t}, after integration over sphere, are total derivatives:
\begin{align*}
 \int_{\mathbb{S}^2}{2(r-3M)}\Delta^{-\half}\Re\left(
-\curlLds  \Phins \partial_t\overline{\Phips }
+\curlLs   \Phips \partial_t\overline{\Phins }\right)
={}&-\int_{\mathbb{S}^2}{2(r-3M)}\Delta^{-\half}
\partial_t\left(\Re\big(
\edthR\Phins \overline{   \Phips }\big)\right).
\end{align*}
It thus follows from substituting this equality  into equality \eqref{eq:EnerEsti:t} and integrating over $\Donetwo$ that
\begin{align}
\label{eq:EnerEsti:(4):trcoord}
\int_{\Donetwo}\left(\partial_t e_t(\phi_{\pm\sfrak})+\partial_r e_r(\phi_{\pm\sfrak})\right)={}&0.
\end{align}
The conservation law \eqref{eq:EnerEsti:(4)} for the case of $\abs{\mathbf{a}}=0$ is then manifest.
\end{proof}

%%%%%%%%%%%%%%
\subsection{Uniform boundedness of a nondegenerate positive definite energy}
\label{sect:uniformbd:nondegenergy}
%%%%%%%%%%%%%%%%%%%
As illustrated in Remark \ref{rem:conservationlaw:degeneracy}, the energy in \eqref{eq:current:highorder} shows degeneracy at $\Horizon$, and the following red-shift estimates will be utilized to remove this degeneracy.

\begin{prop}
\textbf{(Red-shift estimates near horizon).}
There exist two constants $2M<r_0<r_1<2.1M$ such that for any $\tb_2>\tb_1\geq \tb_0$ and any $\reg\in \mathbb{N}$,
\begin{align}\label{eq:redshift:psipns}
\hspace{4ex}&\hspace{-4ex}
\norm{\psipns}^2_{W_{0}^{\reg+1}(\Sigmatwo^{\leq r_0})}
+\sum_{\abs{\mathbf{a}}\leq \reg+1}\int_{\Horizononetwo}\abs{\psipns^{(\mathbf{a})}}^2 \di v\di^2\mu
+\norm{\psipns}^2_{W_{0}^{\reg+1}(\Donetwo^{\leq r_0})}
\notag\\
\lesssim {}&\norm{\psipns}^2_{W_{0}^{\reg+1}(\Sigmaone^{r_1})}
+\norm{\psipns}^2_{W_{0}^{\reg+1}(\Donetwo^{r_0,r_1})}.
\end{align}

\end{prop}

\begin{proof}
We first consider spin $-\half$ component.
Let $\Psinst=r^2\psins$. Equation \eqref{eq:TME} satisfied by $\psins$ can be reformulated in terms of $\Psinst$ as
\begin{align}\label{eq:TME:nega:rewrite}
&-r^2 YV\Psinst+ \edthR'\edthR\Psinst
-(1+2Mr^{-1}-2\mu)\Psinst-[(r-M)-2\mu r] Y\Psinst ={}0.
\end{align}
Multiplying this equation by $-2r^{-2}fY\overline{\Psinst}$, taking the real part, and integrating over $\Donetwo^{\leq r_1}$, this yields
\begin{align}
&\int_{\Donetwo^{\leq r_1}}\Big[\VR(\mu f\abs{Y\Psinst}^2)
+Y(f r^{-2}(\abs{\edthR \Psinst}^2+ f (1+2Mr^{-1}-2\mu)\abs{\Psinst}^2))\notag\\
&\qquad
+  (-\mu \partial_r f+2fr^{-1}(1-2\mu ))\abs{Y\Psinst}^2
+\partial_r (f r^{-2})\abs{\edthR \Psinst}^2
+\partial_r (f (1+2Mr^{-1}-2\mu))\abs{\Psinst}^2\Big]\di^4\mu\notag\\
&\lesssim {}\norm{\Psinst}^2_{W_{0}^{1}(\Donetwo^{r_0,r_1})}.
\end{align}
We choose $f=\chi_0^2 (1+A\mu)$, with $\chi_0=\chi_0(r)$ being a smooth cutoff function which equals to $1$ for $r\leq r_0$ and vanishes identically in $[r_1,\infty)$, and $A$ large enough such that the coefficients of both the $\abs{\edthR\Psinst}^2$ term and the $\abs{\Psinst}^2$ term are bigger than a positive universal constant $c$ for $r\leq r_0$ with $r_0$ sufficiently close to $2M$. It is manifest that the coefficient of $\abs{Y\Psinst}^2$ term in the second line is also positive in $[2M, r_0]$ for $r_0$ close to $2M$. Therefore, there exist two positive universal constants $c$ and $C$ such that
\begin{align}\label{eq:redshift:psins:t2}
&\int_{\Donetwo^{\leq r_1}}\Big[\VR(\mu f\abs{Y\Psinst}^2)
+Y(f r^{-2}(\abs{\edthR \Psinst}^2+ f (1+2Mr^{-1}-2\mu)\abs{\Psinst}^2))\notag\\
&\qquad\quad
+  c(\abs{Y\Psinst}^2
+\abs{\edthR \Psinst}^2
+\abs{\Psinst}^2)\Big]\di^4\mu\notag\\
&\leq {}C\norm{\Psinst}^2_{W_{0}^{1}(\Donetwo^{r_0,r_1})}.
\end{align}

Consider then the spin $\half$ component. Recall the form of equation \eqref{eq:TME:varphis:posi:exf}
\begin{align}
\curlLds  \curlLs   \Phips
-r^2 Y V\Phips -(r-3M)Y\Phips={}&0
\end{align}
By multiplying this equation by $-2r^{-2}fY\overline{\Phips}$, taking the real part, and integrating over $\Donetwo$ with reference volume element $\di^4\mu$, this yields
\begin{align}
0\equiv{}&\int_{\Donetwo}\Big[\VR(\mu f\abs{Y\Phips}^2)
+Y(f r^{-2}\abs{\edthR' \Phips}^2)
\notag\\
&\qquad \qquad
+(-\mu \partial_r f + 2fM r^{-2}+2r^{-2}(r-3M)f)\abs{Y\Phips}^2
+\partial_r (f r^{-2})\abs{\edthR' \Phips}^2\Big]\di^4\mu.
\end{align}
Similarly we choose $f=\chi_0^2 (1+A\mu)$ with $A$ large enough such that the coefficient $\partial_r (fr^{-2})$ of $\abs{\edthR' \Phips}^2$ term is bigger than a positive universal constant $c$ in $r\leq r_0$ for $r_0$ sufficiently close to $2M$. In view of the estimate \eqref{eq:redshift:psins:t2} which bounds over spacetime integral of $\edthR\Psinst$ and the Dirac equations \eqref{eq:Dirac:TMEscalar} which says $Y\Phips=\edthR\psins$, one can bound the spacetime integral of $\abs{Y\Phips}^2$ over $\Donetwo^{\leq r_0}$ by the RHS of \eqref{eq:redshift:psins:t2}. As a result, there exist two positive universal constants $c$ and $C$ such that
\begin{align}\label{eq:redshift:Phips:t2}
&\int_{\Donetwo^{\leq r_1}}\Big[\VR(\mu f\abs{Y\Phips}^2)
+Y(f r^{-2}(\abs{\edthR\Phips}^2)
+  c(\abs{Y\Phips}^2
+\abs{\edthR \Phips}^2
+\abs{\Phips}^2)\Big]\di^4\mu\notag\\
&\leq {}C\Big(\norm{\Psinst}^2_{W_{0}^{1}(\Donetwo^{r_0,r_1})}
+\norm{\Phips}^2_{W_{0}^{1}(\Donetwo^{r_0,r_1})}\Big).
\end{align}
One can multiply equation \eqref{eq:TME} satisfied by the spin $\half$ component $\psips$ by $-2\chi_0^2\partial_t \overline{\psips}$, take the real part and integrate over $\Donetwo$, and this allows us in particular to bound $\int_{\Donetwo^{\leq r_0}} \abs{\partial_t \psips}^2 \di^4\mu$ by $C\big(\int_{\Donetwo^{\leq r_1}} \abs{Y \psips}^2 \di^4\mu + \norm{\Phips}^2_{W_{0}^{1}(\Sigmaone^{r_1})}
+\norm{\Phips}^2_{W_{0}^{1}(\Donetwo^{r_0,r_1})}\big)$. Together with the estimate \eqref{eq:redshift:Phips:t2}, this implies that there exist two constants $2M<r_0<r_1$ such that for any $\tb_2>\tb_1\geq \tb_0$,
\begin{align}\label{eq:redshift:psips}
\norm{\Phips}^2_{W_{0}^{1}(\Sigmatwo^{\leq r_0})}
+\norm{\Phips}^2_{W_{0}^{1}(\Donetwo^{\leq r_0})}
\lesssim {}\norm{\Phips}^2_{W_{0}^{1}(\Sigmaone^{r_1})}
+\norm{\Phips}^2_{W_{0}^{1}(\Donetwo^{r_0,r_1})}
+\norm{\Psinst}^2_{W_{0}^{1}(\Donetwo^{r_0,r_1})}.
\end{align}
In view of the above estimate and the Dirac equations \eqref{eq:Dirac:TMEscalar},  the RHS of \eqref{eq:redshift:psips} also dominates over some integrals of $\Psinst$:
\begin{align}\label{eq:redshift:psins:t1}
\hspace{4ex}&\hspace{-4ex}
\norm{V\Psinst}^2_{W_{0}^{0}(\Sigmatwo^{\leq r_0})}
+\norm{\edthR\Psinst}^2_{W_{0}^{0}(\Sigmatwo^{\leq r_0})}
+\norm{\edthR\Psinst}^2_{W_{0}^{0}(\Donetwo^{\leq r_0})}
+\norm{V\Psinst}^2_{W_{0}^{0}(\Donetwo^{\leq r_0})}\notag\\
\lesssim {}&\norm{\Phips}^2_{W_{0}^{1}(\Sigmaone^{r_1})}
+\norm{\Phips}^2_{W_{0}^{1}(\Donetwo^{r_0,r_1})}
+\norm{\Psinst}^2_{W_{0}^{1}(\Donetwo^{r_0,r_1})}.
\end{align}
The estimates \eqref{eq:redshift:Phips:t2}, \eqref{eq:redshift:psips} and \eqref{eq:redshift:psins:t1} together them yields the $\reg=0$ case of the estimate \eqref{eq:redshift:psipns}.
The general $\reg\geq 0$ case follow in a standard way by commuting with $\Lxi$ and $Y$ and making use of elliptic estimates (since $\Lxi$ and $Y$ span a timelike direction everywhere in $\Dzeroinfty$).
\end{proof}

By utilizing the above red-shift estimates near horizon for the spin $\pm \half$ components of Dirac field and the conservation laws in Propositions \ref{prop:conservationlaw:Highorder} and \ref{prop:conservationlaw:symhypwavesys}, we deduce the following uniform bound of a nondegenerate, positive definite energy.

\begin{thm}
\label{thm:uniformenergybd}\textbf{(Uniform energy boundedness).}
It holds true on a Schwarzschild background that for any $\tb_0\leq \tb_1<\tb_2$ and any $\reg\in \mathbb{N}$,
\begin{align}
\label{eq:uniformenergybdness}
\hspace{4ex}&\hspace{-4ex}
\norm{\psips}_{\tilde{W}_{0}^{\reg+1}(\Sigmatwo)}^2
+\norm{\psins}_{\tilde{W}_{0}^{\reg+1}(\Sigmatwo)}^2
+\sum\limits_{\abs{\mathbf{a}}\leq \reg+1}\bigg(\int_{\Horizononetwo}|
\Phips ^{(\mathbf{a})}|^2\di v\di^2\mu
+\int_{\Scrionetwo}|\Phipns ^{(\mathbf{a})}|^2\di u\di^2\mu\bigg)\notag\\
\lesssim{}&\norm{\psips}_{\tilde{W}_{0}^{\reg+1}(\Sigmaone)}^2
+\norm{\psins}_{\tilde{W}_{0}^{\reg+1}(\Sigmaone)}^2.
\end{align}
\end{thm}

\begin{proof}
We add the estimate \eqref{eq:current:highorder} for all $\abs{\mathbf{a}}\leq \reg+1$ to the estimate \eqref{eq:redshift:psipns} to obtain
\begin{align}
\label{eq:uniformenerbd:t1}
\norm{\psipns}_{\tilde{W}_{0}^{\reg+1}(\Sigmatwo)}^2
+
\norm{\psipns}_{\tilde{W}_{0}^{\reg+1}(\Donetwo^{\leq r_0})}^2
\lesssim{}&\norm{\psipns}_{\tilde{W}_{0}^{\reg+1}(\Sigmaone)}^2
+
\norm{\psipns}_{\tilde{W}_{0}^{\reg+1}(\Donetwo^{r_0,r_1})}^2.
\end{align}
Here we have used a simple fact that
\begin{align}
\sum\limits_{\abs{\mathbf{a}}\leq \reg+1}\int_{\Sigmatb}
(\partial_r h |\Phips ^{(\mathbf{a})}|^2
+(2\mu^{-1}-\partial_rh)|\Phins ^{(\mathbf{a})}|^2
)+\norm{\psipns}^2_{W_{0}^{\reg+1}(\Sigmatb^{\leq r_0})}
\sim{}& \norm{\psipns}_{\tilde{W}_{0}^{\reg+1}(\Sigmatb)}^2
.
\end{align}
Denote
\begin{align}
\label{def:fkandtildefk}
f_{\reg, \Sigmatb} =\norm{\psipns}_{\tilde{W}_{0}^{\reg}(\Sigmatb)}^2,
\quad \tilde{f}_{\reg,\Sigmatb}=\sum\limits_{\abs{\mathbf{a}}\leq \reg}\int_{\Sigmatb}
(\partial_r h |\Phips ^{(\mathbf{a})}|^2
+(2\mu^{-1}-\partial_rh)|\Phins ^{(\mathbf{a})}|^2
).
\end{align}
We can add $\int_{\tb_1}^{\tb_2}\tilde{f}_{\reg+1,\Sigmatb} \di \tb$ to both sides of \eqref{eq:uniformenerbd:t1} such that the last two spacetime integrals of \eqref{eq:uniformenerbd:t1} are absorbed by LHS, leading to
\begin{align}
\label{eq:uniformenerbd:t2}
f_{\reg+1, \Sigmatwo}
+\int_{\tb_1}^{\tb_2}f_{\reg+1, \Sigmatb} \di \tb
\lesssim{}&f_{\reg+1, \Sigmaone}
+\int_{\tb_1}^{\tb_2}\tilde{f}_{\reg+1,\Sigmatb} \di \tb.
\end{align}
From Proposition \ref{prop:conservationlaw:Highorder}, one has $\tilde{f}_{\reg+1,\Sigmatb}\leq \tilde{f}_{\reg+1,\Sigmaone}$ for any $\tb\geq \tb_1$, which implies the last term of \eqref{eq:uniformenerbd:t2} is bounded by $(\tb_2-\tb_1)\tilde{f}_{\reg+1,\Sigmaone}$, and is further bounded by $(\tb_2-\tb_1){f}_{\reg+1,\Sigmaone}$. An application of Lemma \ref{lem:AAGlemma} then yields
$f_{\reg+1, \Sigmatwo}
\lesssim f_{\reg+1, \Sigmaone}$. Together with the estimate \eqref{eq:current:highorder}, we have
\begin{align}
\label{eq:uniformenergybdness:v1}
f_{\reg+1, \Sigmatwo}
+\sum_{\abs{\mathbf{a}}\leq \reg+1}
\bigg(\int_{\Horizononetwo}|
\Phips ^{(\mathbf{a})}|^2 \di v\di^2 \mu
+\int_{\Scrionetwo}|\Phins ^{(\mathbf{a})}|^2\di u\di^2 \mu\bigg)
\lesssim{}&f_{\reg+1, \Sigmaone}.
\end{align}
In the end, we add to this inequality the
estimate \eqref{eq:EnerEsti:(4)}, and this allows us in addition to bound $\int_{\Scrionetwo}|\Phips ^{(\mathbf{a})}|^2\di u\di^2 \mu$ by the RHS of \eqref{eq:uniformenergybdness:v1}. Thus, we achieve the estimate \eqref{eq:uniformenergybdness}.
\end{proof}

%%%%%%%%%%%%%%%%%
\section{Integrated local energy decay estimates}
\label{sect:mora:Schw}
%%%%%%%%%%%%%%%%%%

In this section, we aim to prove the integrated local energy decay, or Morawetz, estimates for the Dirac field. Our main statement is the following high order
basic energy and Morawetz (BEAM) estimates, which are a combination of the energy estimate proven in the previous section and the Morawetz estimate. A new and fundamental feature in treating the Dirac field which differentiates itself from the scalar field case is the indispensable coupling of the spin $\pm \half$ components in both the energy estimate, as is shown in Theorem \ref{thm:uniformenergybd}, and the Morawetz estimate which will be shown in the following theorem. We will come back to this point in Remark \ref{remLessentialcouplinginBEAM}.

\begin{thm}
\label{thm:BEAM:Dirac:sect4}\textbf{(High order BEAM estimates).}
Consider the Dirac field on a Schwarzschild spacetime. For any $\reg\in \mathbb{N}$, any $2M<R'<\infty$ and any $\tb_0\leq \tb_1<\tb_2$,
\begin{align}\label{eq:BEAM:v1}
\hspace{8ex}&\hspace{-8ex}
\norm{\psipns}_{\tilde{W}_{0}^{\reg+1}(\Sigmatwo)}^2
+\sum_{\abs{\mathbf{a}}\leq \reg+1}\bigg(\int_{\Horizononetwo}\abs{\psipns^{(\mathbf{a})}}^2 \di v\di^2\mu
+\int_{\Scrionetwo}|\Phipns ^{(\mathbf{a})}|^2\di u\di^2\mu\bigg)
+
\norm{\psipns}^2_{W_{0}^{\reg}(\Donetwo^{\leq R'})}
\notag\\
\lesssim_{\reg, R'}{}& \norm{\psipns}_{\tilde{W}_{0}^{\reg+1}(\Sigmaone)}^2.
\end{align}
\end{thm}

\begin{proof}
Recall the symmetric hyperbolic wave system \eqref{eq:TME:varphis:mode}. We put both equations in \eqref{eq:TME:varphis:mode} into the following form 
\begin{align}
-\TAO_{s}   \varphi
+\Delta^{-1}r^4(\partial_t +\mu \partial_r)(\partial_t -\mu \partial_r)\varphi +G={}&0,
\end{align}
Here, $\varphi=\Phips$  and $G=G_+=\Delta^{-\half}(r-3M)\curlLds\Phins$ for $s=\half$,  and $\varphi=\Phins$ and $G=G_-=-\Delta^{-\half}(r-3M)\curlLs\Phips$ for $s=-\half$. Define $V_{P}=-2M r^{-1}$. By multiplying these two subequations by
\begin{equation}
r^{-2}X\varphi=r^{-2}(f\partial_r \varphi +q\varphi),
\end{equation}
one obtains
\begin{align}\label{eq:MoraEstidegSchw}
\hspace{4ex}&\hspace{-4ex}
\partial_t\left(\Re\left(\mu^{-1}X(\varphi)
\partial_t\bar{\varphi}\right)\right)
+\tfrac{1}{2}\partial_r\left(r^{-2}f\left[|\TAO_{s}^{\half}\varphi|^2
-\mu^{-1}|\partial_t\varphi|^2-\Delta |\partial_r(r^{-1}\varphi)|^2+V_{P}\abs{\varphi}^2\right]\right)\notag\\
\hspace{4ex}&\hspace{-4ex}
+\tfrac{1}{2}\partial_r\left(r^{-2}
\left[\Delta\partial_r(q+r^{-1}f)|\varphi|^2-2\Delta (q+r^{-1}f)\Re(\bar{\varphi}r\partial_r(r^{-1}\varphi))
-r^{-1}B^r|\varphi|^2\right]\right)\notag\\
\hspace{4ex}&\hspace{-4ex}
+r^{-2}B(\varphi)+r^{-2}\Re(X\overline{\varphi}{G})\equiv{}0.
\end{align}
Here, the bulk term
\begin{align}
B(\varphi)=B^t|\partial_t\varphi|^2+B^r|\partial_r\varphi|^2
+B^{a}|\TAO_{s}^{\half}\varphi|^2+B^0|\varphi|^2,
\end{align}
with
\begin{align}
\label{eq:Bterms:general}
B^t={}&\tfrac{1}{2}\partial_r
(r^4\Delta^{-1}{f})
-(q+r^{-1}f)r^4\Delta^{-1},\notag\\
B^r={}&\tfrac{1}{2}\partial_r(\Delta{f})
-2{f}(r-M)+\Delta (q+r^{-1}f), \notag\\
B^{a}={}&-\tfrac{1}{2}\partial_r{f}+(q+r^{-1}f),\notag\\
B^0={}&\partial_r((q+r^{-1}f)(r-M))
-\tfrac{1}{2}\partial_{r}^2(\Delta (q+r^{-1}f))+r^2(\partial_r (r^{-3}B^r(r))+ r^{-4}B^r(r))\notag\\
&+V_{P}(q+r^{-1}f)-\tfrac{1}{2} \partial_r (V_{P}f)\notag\\
={}&-\tfrac{1}{2}\partial_{r}(\Delta \partial_r(q+r^{-1}f))
+r^2(\partial_r (r^{-3}B^r)+ r^{-4}B^r)+V_{P}(q+r^{-1}f)-\tfrac{1}{2} \partial_r (V_{P}f).
\end{align}

We follow \cite{larsblue15hidden,Ma17spin2Kerr} by taking
\begin{align}\label{eq:multiplier:Mora:Schw}
f=\frac{2(r-2M)(r-3M)}{r^2},\quad q=\mu \partial_r \left(\half \mu^{-1}f\right)=\frac{3M\Delta}{r^4}
\end{align}
and calculate using the expressions \eqref{eq:Bterms:general} that
\begin{align}
B^t=0,\quad
B^r=6M \mu^2, \quad
B^{a}=2r^{-3}(r-3M)^2,\quad
B^0=-3M r^{-4}(3r^2 -20Mr +30M^2).
\end{align}
We sum over the terms from the source terms $G_{\pm}$:
\begin{align}
\label{eq:MoraSchw:sourceterm:1}
\hspace{4ex}&\hspace{-4ex}
r^{-2}\Re(X\overline{\Phips}{G_{+}})
+r^{-2}\Re(X\overline{\Phins}{G_{-}})\notag\\
={}&r^{-2} {(r-3M)}{\Delta^{-\half}}\Re\left((f\partial_r \overline{\Phips} +q\overline{\Phips})\curlLds\Phins
-(f\partial_r \overline{\Phins} +q\overline{\Phins})\curlLs\Phips\right)\notag\\
\equiv{}&\partial_r \left(\frac{r-3M}{2r^2\Delta^{\half}}f
\left[\overline{\Phips}\curlLds\Phins- \overline{\Phins}\curlLs\Phips\right]\right)\notag\\
&+\left\{-\partial_r \left(\frac{r-3M}{2r^2\Delta^{\half}}f\right)
+\frac{r-3M}{r^2\Delta^{\half}}q\right\}\Re\left(\overline{\Phips}\curlLds\Phins- \overline{\Phins}\curlLs\Phips\right).
\end{align}
Using the first-order Dirac equations \eqref{eq:Dirac:sym}, we have 
\begin{align*}
\Re\Big(\overline{\Phips}\curlLds\Phins- \overline{\Phins}\curlLs\Phips\Big)={}&
\Delta^{\half}\Re\Big(\big(\overline{\Phips}(\mu^{-1}\partial_t -\partial_r )\Phips- \overline{\Phins}(\mu^{-1}\partial_t +\partial_r)\Phins\big)\Big)\notag\\
={}&\frac{\Delta^{\half}}{2\mu}\partial_t (\abs{\Phips}^2-\abs{\Phins}^2)
-\half\Delta^{\half}\partial_r (\abs{\Phips}^2+\abs{\Phins}^2),
\end{align*}
and substituting this into the last term of \eqref{eq:MoraSchw:sourceterm:1}, we infer
\begin{align}
\label{eq:MoraSchw:sourceterm:2}
\hspace{4ex}&\hspace{-4ex}
\left\{-\partial_r \left(\frac{r-3M}{2r^2\Delta^{\half}}f\right)
+\frac{r-3M}{r^2\Delta^{\half}}q\right\}\Re\left(\overline{\Phips}\curlLds\Phins- \overline{\Phins}\curlLs\Phips\right)\notag\\
={}&\partial_t \left(-\frac{\Delta^{\half}}{2\mu}\left[\partial_r
\left(\frac{f(r-3M)}{2 r^2 \Delta^{\half}}\right)
-\frac{r-3M}{r^2\Delta^{\half}}q\right]
(\abs{\Phips}^2-\abs{\Phins}^2)\right)\notag\\
&
+\partial_r \left(\frac{\Delta^{\half}}{2}\left[\partial_r
\left(\frac{f(r-3M)}{2 r^2 \Delta^{\half}}\right)
-\frac{r-3M}{r^2\Delta^{\half}}q\right]
(\abs{\Phips}^2+\abs{\Phins}^2)\right)\notag\\
&-\partial_r \left(\frac{\Delta^{\half}}{2}\left[\partial_r
\left(\frac{f(r-3M)}{2 r^2 \Delta^{\half}}\right)
-\frac{r-3M}{r^2\Delta^{\half}}q\right]\right)
(\abs{\Phips}^2+\abs{\Phins}^2).
\end{align}
Note that by making use of the second order symmetric hyperbolic wave system, we have transformed the  error term $r^{-2}\Re(X\overline{\Phips}{G_{+}})
+r^{-2}\Re(X\overline{\Phins}{G_{-}})$ coming from the source terms $G_{\pm}$ into total derivative terms plus the last line of \eqref{eq:MoraSchw:sourceterm:2}.
Combining the estimates \eqref{eq:MoraEstidegSchw}, \eqref{eq:MoraSchw:sourceterm:1} and \eqref{eq:MoraSchw:sourceterm:2} together, we arrive at an estimate of the form 
\begin{align}
\label{eq:MoraSchw:enerMoraeq}
\partial_t \mathbf{F}_t
+\partial_r \mathbf{F}_r
+ \mathbf{B}\equiv 0,
\end{align}
where $\mathbf{F}_t$ and $\mathbf{F}_r$ are given by
\begin{subequations}
\begin{align}
\mathbf{F}_t={}&\sum_{\varphi=\Phips,\Phins}\Re\left(\mu^{-1}X(\varphi)
\partial_t\bar{\varphi}\right)
-\frac{\Delta^{\half}}{2\mu}\left[\partial_r
\left(\frac{f(r-3M)}{2 r^2 \Delta^{\half}}\right)
-\frac{r-3M}{r^2\Delta^{\half}}q\right]
(\abs{\Phips}^2-\abs{\Phins}^2),\\
\mathbf{F}_r={}&\sum_{\varphi=\Phips,\Phins}\half r^{-2}f\left[|\TAO_{s}^{\half}\varphi|^2
-\mu^{-1}|\partial_t\varphi|^2-\Delta |\partial_r(r^{-1}\varphi)|^2+V_{P}\abs{\varphi}^2\right] \notag\\
&+\sum_{\varphi=\Phips,\Phins}\half r^{-2}
\left[\Delta\partial_r(q+r^{-1}f)|\varphi|^2-2\Delta (q+r^{-1}f)\Re(\bar{\varphi}r\partial_r(r^{-1}\varphi))
-r^{-1}B^r|\varphi|^2\right]\notag\\
&+\frac{r-3M}{2r^2\Delta^{\half}}f
\left[\overline{\Phips}\curlLds\Phins- \overline{\Phins}\curlLs\Phips\right] \notag\\
&+\frac{\Delta^{\half}}{2}\left[\partial_r
\left(\frac{f(r-3M)}{2 r^2 \Delta^{\half}}\right)
-\frac{r-3M}{r^2\Delta^{\half}}q\right]
(\abs{\Phips}^2+\abs{\Phins}^2),
\end{align}
\end{subequations}
and the bulk term $\mathbf{B}$ equals
\begin{align}
\mathbf{B}={}&r^{-2}B^r(|\partial_r\Phips|^2
+|\partial_r\Phins|^2)
+r^{-2}B^{a}(|\TAO_{\sfrak}^{\half}\Phips|^2
+|\TAO_{-\sfrak}^{\half}\Phins|^2)\notag\\
&
+\left[r^{-2}B^0-\partial_r \left(\frac{\Delta^{\half}}{2}\left[\partial_r
\left(\frac{f(r-3M)}{2 r^2 \Delta^{\half}}\right)
-\frac{r-3M}{r^2\Delta^{\half}}q\right]\right)\right]
(|\Phips|^2
+|\Phins|^2).
\end{align}

By integrating over $\Donetwo$, one obtains
\begin{align}
\int_{\Donetwo}\mathbf{B}\di^4\mu
={}&-\int_{\Horizononetwo}\left[
\mu \mathbf{F}_t
-\mathbf{F}_r\right]\di v\di^2\mu
-\int_{\Scrionetwo}\left[
\mathbf{F}_t
+\mu^{-1}\mathbf{F}_r\right]\di u\di^2\mu\notag\\
&
-\int_{\Sigmatwo}\left[\mathbf{F}_t
+(\mu^{-1}-\partial_r h)\mathbf{F}_r\right]\di^3\mu
+\int_{\Sigmaone}\left[\mathbf{F}_t
+(\mu^{-1}-\partial_r h)\mathbf{F}_r\right]\di^3\mu.
\end{align}
From Remark \ref{rem:TAOshalfandTAOsellip} and equation \eqref{eq:ellipestis:TAO}, we have $\int_{\Sphere}\abs{\TAO_{s}^{\half}\Phis}^2\geq \int_{\Sphere}\abs{\Phis}^2$,  therefore, it suffices to check the following relation outside the black hole
\begin{align}
r^{-2}(B^{a} +B^0)
-\partial_r \left(\frac{\Delta^{\half}}{2}\left[\partial_r
\left(\frac{f(r-3M)}{2 r^2 \Delta^{\half}}\right)
-\frac{r-3M}{r^2\Delta^{\half}}q\right]\right)>0
\end{align}
such that the bulk term is nonnegative.
In fact, a simple direct calculation gives the LHS is equal to
\begin{align}
\frac{3}{2} M r^{-6} (6 r^2 - 32 Mr +45 M^2)>0.
\end{align}
On the other hand, by using the fact that $\psins$ (or equivalently $\mu^{-\half}\Phins$) is regular and nondegenerate at $\Horizon$, the integrals of flux terms are bounded by the LHS of \eqref{eq:uniformenergybdness} with $\reg=0$, hence by $Cf_{1, \Sigmaone}$ from Theorem \ref{thm:uniformenergybd}, where
$f_{\reg, \Sigmatb}$ for any $\reg\in \mathbb{N}$ is defined as in \eqref{def:fkandtildefk}.
In total, we arrive at
\begin{align}
\label{eq:MoraSchw:finalform:1}
\hspace{4ex}&\hspace{-4ex}
\int_{\Donetwo}\mu^2 r^{-2}(|\partial_r\Phips|^2
+|\partial_r\Phins|^2)
+r^{-6}(r-3M)^2(|\edthR'\Phips|^2
+|\edthR\Phins|^2) +r^{-4}(|\Phips|^2
+|\Phins|^2)\di^4\mu\notag\\
\lesssim{}&f_{1, \Sigmaone}.
\end{align}

Instead, if we choose $f=0$, $q=-M\Delta(r-3M)^2r^{-6}$, one finds that $B^t=M(r-3M)^2 r^{-2}$,
\begin{align}
\abs{B^r}\lesssim M\mu^2,\quad  \abs{B^{a}}\lesssim M\mu (r-3M)^2r^{-4}, \quad \abs{B^0}\lesssim{}Mr^{-2},
\end{align}
and
\begin{align}
\left|-\partial_r \left(\frac{\Delta^{\half}}{2}\left[\partial_r
\left(\frac{f(r-3M)}{2 r^2 \Delta^{\half}}\right)
-\frac{r-3M}{r^2\Delta^{\half}}q\right]\right)
(\abs{\Phips}^2+\abs{\Phins}^2)\right|\lesssim{}Mr^{-4}(\abs{\Phips}^2+\abs{\Phins}^2).
\end{align}
This gives an upper bound for the integral $\int_{\Donetwo} (r-3M)^2 r^{-4}(|\partial_t\Phips|^2
+|\partial_t\Phins|^2)\di^4 \mu$ in terms of the LHS of inequality \eqref{eq:MoraSchw:finalform:1} plus the corresponding boundary flux terms which are bounded by $Cf_{1, \Sigmaone}$. Together with inequality \eqref{eq:MoraSchw:finalform:1}, we eventually conclude
\begin{align}
\label{eq:MoraSchw:finalform:2}
\hspace{4ex}&\hspace{-4ex}
\int_{\Donetwo}
\Big[(r-3M)^2 r^{-4}(|\partial_t\Phips|^2
+|\partial_t\Phins|^2)
+\mu^2 r^{-2}(|\partial_r\Phips|^2
+|\partial_r\Phins|^2)\notag\\
\hspace{4ex}&\hspace{-4ex}
\qquad \quad+r^{-6}(r-3M)^2(|\edthR'\Phips|^2
+|\edthR\Phins|^2) +r^{-4}(|\Phips|^2
+|\Phins|^2)\Big]\di^4\mu\notag\\
\lesssim{}& f_{1, \Sigmaone}.
\end{align}
By commuting with $\Lxi$, $\edthR$ and $\edthR'$, we can obtain a higher order regularity version of Morawetz estimate: for any $\reg\in \mathbb{N}$, any $2M<r'<R'<\infty$ and any $\tb_0\leq \tb_1<\tb_2$,
\begin{align}
\label{eq:MoraSchw:finalform:highorder:2}
\norm{\psips}^2_{W_{0}^{\reg}(\Donetwo^{r',R'})}
+\norm{\psins}^2_{W_{0}^{\reg}(\Donetwo^{r',R'})}
\lesssim_{\reg, r',R'}{}& f_{\reg+1, \Sigmaone}.
\end{align}
We combine this estimate with the uniform energy boundedness estimate \eqref{eq:uniformenergybdness} and the red-shift estimate \eqref{eq:redshift:psipns}
to conclude the BEAM estimate \eqref{eq:BEAM:v1}.
\end{proof}

\begin{remark}
\label{remLessentialcouplinginBEAM}
If we instead start with the TME for one component, say, the spin $+\half$ component without of loss of generality, we have from \eqref{eq:TME:varphis:posi:exf} that
\begin{align}
-\TAO_{\sfrak} \Phips
+\Delta^{-1}r^4(\partial_t +\mu \partial_r)(\partial_t -\mu \partial_r)\Phips
+G_{+}={}&0,
\end{align}
with $G_+=(r-3M)(\mu^{-1}\partial_t -\partial_r)\Phips$. By running the above proof, the error term coming from the source $G_+$ equals
$
r^{-2}\Re(X\overline{\Phips}{G_{+}})=\mu^{-1} r^{-2}(r-3M) \Re(X\overline{\Phips}{(\partial_t \Phips -\mu \partial_r \Phips)})$. With the choice of the multiplier $X\Phips$ as in the above proof, it is impossible to control this error term by the other bulk term $r^{-2}B(\Phips)$ neither near horizon (because of the extra $\mu^{-1}$ factor) nor in a large $r$ region (in view of the $r$ weight near infinity). Therefore, a Morawetz estimate for the Dirac field can only be achieved from the wave system of both spin $\pm \half$ components, not one single TME of one component. Meanwhile,
by making use of the symmetric hyperbolic wave system, the integral of the error terms from the source terms $G_{\pm}$ is transformed into pure error flux terms plus one single spacetime error integral, all of which have dependence only on $\Phipns$ and first order angular derivatives of $\Phipns$. These error flux terms are bounded by the nondegenerate, positive definite energy constructed in the previous section and the spacetime error integral is absorbed by the integral of $r^{-2}(B(\Phips)+B(\Phins))$ by a simple elliptic estimate over sphere for non-zero spin-weighted scalars.
In summary, by making use of the coupled symmetric hyperbolic wave system,  the analysis has been reduced to some extent to a scalar-wave-like equation analysis. This can be compared to the Chandrasekhar's transformation \cite{chandrasekhar1975linearstabSchw} which transforms the TME \eqref{eq:TME} for integer spin fields into a scalar-wave-like equation.
\end{remark}

%%%%%%%%%%%%%%%%%
\section{Almost sharp decay estimates}
\label{sect:almostsharpdecayestimates}
%%%%%%%%%%%%%%%%%

In this section, we prove the almost sharp energy and pointwise decay estimates based on the BEAM estimates in Theorem \ref{thm:BEAM:Dirac:sect4}. 

%%%%%%%%%%
\subsection{Decay of basic energy}
%%%%%%%%%%%%%%%%%

We first upgrade the BEAM estimates \eqref{eq:BEAM:v1}  to some basic energy decay estimates.

It is convenient to use the wave equation \eqref{eq:TME:varphis:posi:exf} for $\Phips$ and equation \eqref{eq:TME:varphis:mode:nega} for $\Phins$, which read
\begin{subequations}
\label{eq:TME:rp}
\begin{align}
\label{eq:TME:rp:posi}
\curlLds  \curlLs   \Phips
-r^2YV\Phips
 ={}&(r-3M)Y  \Phips,\\
\label{eq:TME:rp:nega}
\curlLs   \curlLds  \Phins
-r^2YV\Phins ={}&-\Delta^{-\half}(r-3M)\curlLs   \Phips,
\end{align}
\end{subequations}
respectively.
By defining
\begin{align}
\label{def:hatPhips}
\hatPhips=\Delta^{-\half}r^2\Phips
\end{align}
and using the commutator \eqref{comm:curlLscurlLds},
 the equations of the scalars $\hatPhips$ and $\Phins$ are\begin{subequations}
\label{eq:TME:rp1}
\begin{align}
\label{eq:TME:rp1:posi}
&-r^2YV\hatPhips
+\curlLds  \curlLs   \hatPhips
-(r-3M)\VR \hatPhips
+(1-6Mr^{-2})\hatPhips ={}0,\\
\label{eq:TME:rp1:nega}
&-r^2YV\Phins+ \curlLds   \curlLs  \Phins- \Phins
 ={}-(r-3M)r^{-2}\curlLs   \hatPhips.
\end{align}
\end{subequations}
The way of defining $\hatPhips$ as in \eqref{def:hatPhips} is such that the only first order derivative  appearing in its equation \eqref{eq:TME:rp1:posi} is $\VR$ derivative, a property which is quite helpful in proving the $r^p$ estimates, and the reason of using equation \eqref{eq:TME:varphis:mode:nega} for $\Phins$ instead of the decoupled TME \eqref{eq:TME:varphis:nega:exf} is that one can couple this equation with \eqref{eq:TME:rp1:posi} to achieve the $r^p$ estimates for $p\in[0,2]$, while proving an $r^p$ estimate for $p\in [0,2)$ is impossible for the TME \eqref{eq:TME:varphis:nega:exf} since one would obtain a negative spacetime integral of $\abs{V\Phins}^2$. All of these will be manifest in the proof of Lemma \ref{lem:BED:0to2:v1}.

The idea of proving basic energy decay estimates is to put each of the coupled equations \eqref{eq:TME:rp1} into the form of \eqref{eq:wave:rp} and apply the $r^p$ estimates in Proposition \ref{prop:wave:rp}.

We define the scalar $\Psipns$ which are actually the radiation fields of the components $\psipns$ and a $p$-weighted energy. 

\begin{definition}
\label{def:Fenergies:0to2}
Let
\begin{align}
\label{def:PsipsPsins}
\Psips=r\psips, \qquad \Psins=r\psins.
 \end{align}
Let $\varphi$ be a spin-weight $\pm \half$ scalar. Define for $0\leq p\leq 2$ that
\begin{align}
F(\reg,p,\tb,\varphi)={}&\norm{rV\varphi}^2_{W_{p-2}^{\reg-1}(\Sigmatb)}
+\norm{\varphi}^2_{W_{-2}^{\reg}(\Sigmatb)},
\end{align}
for $-1<p<0$ that $F(\reg,p,\tb,\varphi)=0$, and for $p=-1$ that
$F(\reg,p,\tb,\varphi)=\norm{\varphi}^2_{W_{-3}^{\reg}(\Sigmatb)}$.
Define moreover
\begin{align}
F(\reg,p,\tb,\Psipns)={}F(\reg,p,\tb,\Psips)+F(\reg,p,\tb,\Psins).
\end{align}
\end{definition}

The $r^p$ estimates, as well as some basic energy and pointwise decay estimates, are as follow.

\begin{lemma}
\label{lem:BED:0to2:v1}
Given the BEAM estimates \eqref{eq:BEAM:v1} on a Schwarzschild spacetime, then for any $j\in \mathbb{N}$, there exists a $\reg'(j,\reg)$ such that for any $p\in [0,2]$, it holds for any $\tb\geq \tb_0$ that
\begin{align}
\label{eq:BED2:Lxi:v1}
F(\reg,p,\tb,\Lxi^j\Psipns)
+\int_{\tb}^\infty F(\reg,p-1,\tb',\Lxi^j\Psipns) \di\tb'
\lesssim{}&\tb^{p-2-2j}
F(\reg+\regl(j,\reg),2,\tb_0,\Psipns)
\end{align}
and
\begin{align}
\label{eq:ptwdecay:basic:v1}
\absCDeri{\Lxi^j\psips}{\reg}+
\absCDeri{\Lxi^j\psins}{\reg}\lesssim_{j,\reg}  {} v^{-1}\tb^{-\half-j} (F(\reg+\regl(j,\reg),2,\tb_0,\Psipns))^{\half} .
\end{align}
\end{lemma}

\begin{proof}
Each subequation of the system \eqref{eq:TME:rp1} can be put into the form of \eqref{eq:wave:rp}, and the assumptions in Proposition \ref{prop:wave:rp} are all satisfied with $b_{0,0}(\Phins)+\Lambda_{-\sfrak}=1+0=1$, $\vartheta(\Phins)=(r-3M)r^{-2}\curlLs   \hatPhips$, $b_{0,0}(\hatPhips)+\Lambda_{\sfrak}=-1+1=0$ and $\vartheta(\hatPhips)=0$. Therefore, we can apply the estimate \eqref{eq:rp:less2:2} with $p\in (0,2)$ to both subequations of the system \eqref{eq:TME:rp1}, the estimate \eqref{eq:rp:p=2:2} with $p=2$ to both subequations of the system \eqref{eq:TME:rp1}, respectively. This gives that there exists a constant $\hat{R}_0=\hat{R}_0(p)$ such that for all $R_0\geq \hat{R}_0$ and $\tb_2>\tb_1\geq \tb_0$,
\begin{enumerate}
\item for $p\in (0, 2)$,
\begin{subequations}\label{eq:rp:less2}
    \begin{align}\label{eq:rp:less2:posi}
\hspace{4ex}&\hspace{-4ex}
\norm{rV\hatPhips}^2_{W_{p-2}^\reg(\Sigmatwo^{R_0})}
+\norm{\hatPhips}^2_{W_{-2}^{\reg+1}(\Sigmatwo^{R_0})}
+\norm{\hatPhips}^2_{W_{p-3}^{\reg+1}(\Donetwo^{R_0})}
+\norm{Y\hatPhips}^2_{W_{-1-\delta}^{\reg}(\Donetwo^{R_0})}
\notag\\
&\lesssim_{[R_0-M, R_0]} {}\norm{rV\hatPhips}^2_{W_{p-2}^\reg(\Sigmaone^{R_0})}
+\norm{\hatPhips}^2_{W_{-2}^{\reg+1}(\Sigmaone^{R_0})},\\
\label{eq:rp:less2:nega}
\hspace{4ex}&\hspace{-4ex}
\norm{rV\Phins}^2_{W_{p-2}^\reg(\Sigmatwo^{R_0})}
+\norm{\Phins}^2_{W_{-2}^{\reg+1}(\Sigmatwo^{R_0})}
+\norm{\Phins}^2_{W_{p-3}^{\reg+1}(\Donetwo^{R_0})}
+\norm{Y\Phins}^2_{W_{-1-\delta}^{\reg}(\Donetwo^{R_0})}
\notag\\
&\lesssim_{[R_0-M, R_0]} {}\norm{rV\Phins}^2_{W_{p-2}^\reg(\Sigmaone^{R_0})}
+\norm{\Phins}^2_{W_{-2}^{\reg+1}(\Sigmaone^{R_0})}
+\norm{\curlLs   \hatPhips}^2_{W_{p-5}^{\reg}(\Donetwo^{R_0})};
\end{align}
\end{subequations}
\item\label{pt:2:prop:wave:rp} for $p=2$,
\begin{subequations}\label{eq:rp:p=2}
\begin{align}\label{eq:rp:p=2:posi}
\hspace{4ex}&\hspace{-4ex}
\norm{rV\hatPhips}^2_{W_{0}^\reg(\Sigmatwo^{R_0})}
+\norm{\hatPhips}^2_{W_{-2}^{\reg+1}(\Sigmatwo^{R_0})}
+\norm{\hatPhips}^2_{W_{-1-\delta}^{\reg+1}(\Donetwo^{R_0})}
+\norm{rV\hatPhips}^2_{W_{-1}^{\reg}(\Donetwo^{R_0})}
\notag\\
&\lesssim_{[R_0-M, R_0]} {}\norm{rV\hatPhips}^2_{W_{0}^\reg(\Sigmaone^{R_0})}
+\norm{\hatPhips}^2_{W_{-2}^{\reg+1}(\Sigmaone^{R_0})},\\
\label{eq:rp:p=2:nega}
\hspace{4ex}&\hspace{-4ex}
\norm{rV\Phins}^2_{W_{0}^\reg(\Sigmatwo^{R_0})}
+\norm{\Phins}^2_{W_{-2}^{\reg+1}(\Sigmatwo^{R_0})}
+\norm{\Phins}^2_{W_{-1-\delta}^{\reg+1}(\Donetwo^{R_0})}
+\norm{rV\Phins}^2_{W_{-1}^{\reg}(\Donetwo^{R_0})}
\notag\\
&\lesssim_{[R_0-M, R_0]} {}\norm{rV\Phins}^2_{W_{0}^\reg(\Sigmaone^{R_0})}
+\norm{\Phins}^2_{W_{-2}^{\reg+1}(\Sigmaone^{R_0})}
+\norm{\curlLs   \hatPhips}^2_{W_{-3}^{\reg}(\Donetwo^{R_0})}.
\end{align}
\end{subequations}
\end{enumerate}
Adding these estimates together, and plugging in the BEAM estimates \eqref{eq:BEAM:v1} to absorb the terms which are implicit in $\lesssim_{[R_0-M, R_0]}$ and supported on $[R_0-M,R_0]$, one can thus obtain for $p\in (0,2)$ that
\begin{subequations}
\begin{align}
F(\reg,p,\tb_2,\Psipns)
+\int_{\tb_1}^{\tb_2}F(\reg-\reg',p-1,\tb,\Psipns)\di\tb
\lesssim{}&F(\reg,p,\tb_1,\Psipns),
\end{align}
and for $p=2$,
\begin{align}
F(\reg,2,\tb_2,\Psipns)
+\int_{\tb_1}^{\tb_2}F(\reg-\reg',1,\tb,\Psipns)\di\tb
\lesssim{}&F(\reg,2,\tb_1,\Psipns).
\end{align}
We remark that $\regl$ here is a general parameter of derivative loss, although it can be chosen explicitly to be $1$.
For $p=0$, we can apply the estimate \eqref{eq:rp:p=0} to the subequation \eqref{eq:TME:rp1:posi} and the estimate \eqref{eq:rp:p=0:v2} to the subequation \eqref{eq:TME:rp1:nega} respectively. We note from the Dirac equations \eqref{eq:Dirac:TMEscalar} that $r^{-1}\Delta\edthR'\hatPhips=rV\Phins$, hence the last two terms in \eqref{eq:TME:rp1:nega} for $\varphi=\Phins$ are bounded by the RHS of the estimate \eqref{eq:rp:p=0:v2} for $\varphi=\hatPhips$. Therefore,
\begin{align}
F(\reg,0,\tb_2,\Psipns)
+\int_{\tb_1}^{\tb_2}F(\reg-1,-1,\tb,\Psipns)\di\tb
\lesssim{}&F(\reg,0,\tb_1,\Psipns).
\end{align}
\end{subequations}
In total, one has for any $p\in[0,2]$ and $\tb_2>\tb_1\geq \tb_0$ that
\begin{align}
F(\reg,p,\tb_2,\Psipns)
+\int_{\tb_1}^{\tb_2}F(\reg-\reg',p-1,\tb,\Psipns)\di\tb
\lesssim{}&F(\reg,p,\tb_1,\Psipns).
\end{align}
An application of Lemma \ref{lem:hierarchyImpliesDecay}  then implies for any $p\in[0,2]$ and $\tb\geq 2\tb_0$ that
\begin{align}
\label{eq:rp:bothspin}
F(\reg,p,\tb,\Psipns)
\lesssim{}&\tb^{-2+p}F(\reg+\regl,2,\tb/2,\Psipns)
\lesssim{}\tb^{-2+p}F(\reg+\regl,2,\tb_0,\Psipns).
\end{align}

To show better decay for $\Lxi$ derivative, one just needs to note that  away from horizon to rewrite $r^2V\Lxi \Phins$ as a weighted sum of $(rV)^2 \Phins$, $\edthR'\edthR\Phins$,
$\Lxi^2\Phins$,  $\Lxi\Phins$, $r^{-1}\Phins$ and $r^{-1}\curlLs   \hatPhips$ all with $O(1)$ coefficients using the wave equation \eqref{eq:TME:rp1:nega} and $Y=\mu^{-1}(2\Lxi-V)$. Similarly, away from horizon, one can express $r^2V\Lxi \hatPhips$ as a weighted sum of $(rV)^2 \hatPhips$, $\edthR\edthR'\hatPhips$,
$\Lxi^2\hatPhips$, $rV\hatPhips$, $\Lxi\hatPhips$ and $r^{-1}\hatPhips$ all with $O(1)$ coefficients using the wave equation \eqref{eq:TME:rp1:posi} and $Y=\mu^{-1}(2\Lxi-V)$. As a result,
\begin{subequations}
\begin{align}
F(\reg,2,\tb,\Lxi\Psips)={}&
\norm{rV\Lxi\Psips}^2_{W_{0}^{\reg-1}(\Sigmatb)}
+\norm{\Lxi\Psips}^2_{W_{-2}^{\reg}(\Sigmatb)}\notag\\
\lesssim{}& \norm{rV\Psips}^2_{W_{-2}^{\reg+\regl-1}(\Sigmatb)}
+\norm{\Psips}^2_{W_{-2}^{\reg+\regl}(\Sigmatb)}\notag\\
\lesssim{}& F(\reg+\regl,0,\tb,\Psips),\\
F(\reg,2,\tb,\Lxi\Psins)={}&
\norm{rV\Lxi\Psins}^2_{W_{0}^{\reg-1}(\Sigmatb)}
+\norm{\Lxi\Psins}^2_{W_{-2}^{\reg}(\Sigmatb)}\notag\\
\lesssim{}& \norm{rV\Psins}^2_{W_{-2}^{\reg+\regl-1}(\Sigmatb)}
+\norm{\Psins}^2_{W_{-2}^{\reg+\regl}(\Sigmatb)}
+\norm{\Psips}^2_{W_{-4}^{\reg+\regl}(\Sigmatb)}\notag\\
\lesssim{}& F(\reg+\regl,0,\tb,\Psipns),
\end{align}
\end{subequations}
which together give $F(\reg,2,\tb,\Lxi\Psipns)\lesssim F(\reg+\regl,0,\tb,\Psipns)$. Substituting this back into \eqref{eq:rp:bothspin}, we then have for any $p\in [0,2]$ and $\tb\geq 4\tb_0$ that
\begin{align}
F(\reg,p,\tb,\Lxi\Psipns)
\lesssim{}&\tb^{-2+p}F(\reg+\regl,2,\tb/2,\Lxi\Psipns)
\lesssim{}\tb^{-4+p}F(\reg+\regl,2,\tb/4,\Psipns)\notag\\
\lesssim{}&\tb^{-4+p}F(\reg+\regl,2,\tb_0,\Psipns).
\end{align}
Repeating the above discussions then proves the general $j\in \mathbb{N}$ cases of the estimate \eqref{eq:BED2:Lxi:v1}.

Turn in the end to the pointwise estimates \eqref{eq:ptwdecay:basic:v1}. Applying the inequality \eqref{eq:Sobolev:2} with  $p=1-\alpha$ and $p=1+\alpha$ of the energy decay estimate \eqref{eq:BED2:Lxi:v1} gives the  decay estimate \eqref{eq:ptwdecay:basic:v1} but with decay rate $r^{-1}\tb^{-\half-j}$. In addition, one can make use of the inequality \eqref{eq:Sobolev:3},  the energy decay estimate  \eqref{eq:BED2:Lxi:v1} with $p=0$, and the fact that $\int_{\tb}^{\infty}F(\reg-\reg',-1,\tb,\Psipns)\di\tb
\lesssim F(\reg,0,\tb,\Psipns)$ to achieve the decay estimate \eqref{eq:ptwdecay:basic:v1} with decay rate $\tb^{-\frac{3}{2}-j}$. These two estimates together $v\sim_R \tau$ as $r\leq R$ and $v\sim_R r$ as $r\geq R$  prove \eqref{eq:ptwdecay:basic:v1}.
\end{proof}

It is convenient in the latter discussions to utilize instead the following slightly different basic and pointwise energy decay estimate, in particular in deriving the properties of Newman--Penrose constants in Section \ref{sect:N-Pconst}.

\begin{lemma}
\label{lem:BED:0to2}
 Let $\PhinsHigh{1}=r^2\VR \Phins$ be defined as in Definition \ref{def:tildePhiplusandminusHigh}, and let
$F^{(1)}(\reg,p,\tb,\Psipns)$, for any $p\in [-1,2]$, be defined as in Definition \ref{def:Fenergies:big2}.
Given the BEAM estimates \eqref{eq:BEAM:v1} on a Schwarzschild spacetime, then for any $j\in \mathbb{N}$, there exists a $\reg'(j,\reg)$ such that for any $p\in [0,2]$, it holds for any $\tb\geq \tb_0$ that
\begin{align}
\label{eq:BED2:Lxi}
F^{(1)}(\reg,p,\tb,\Lxi^j\Psipns)
+\int_{\tb}^\infty F^{(1)}(\reg,p-1,\tb',\Lxi^j\Psipns) \di\tb'
\lesssim{}&\tb^{p-2-2j}
F^{(1)}(\reg+\regl(j,\reg),2,\tb_0,\Psipns)
\end{align}
and
\begin{align}
\label{eq:ptwdecay:basic}
\absCDeri{\Lxi^j\psips}{\reg}+
\absCDeri{\Lxi^j\psins}{\reg}
+\absCDeri{\Lxi^j(rV(r\psins))}{\reg}
\lesssim_{j,\reg}  {} v^{-1}\tb^{-\half-j} (F^{(1)}(\reg+\regl(j,\reg),2,\tb_0,\Psipns))^{\half}.
\end{align}
\end{lemma}

\begin{remark}
The main difference of these estimates from the ones in Lemma \ref{lem:BED:0to2:v1} is that the energy and pointwise decay estimates for $\PhinsHigh{1}=r^2\VR \Phins$ or $r^2V(r\psins)$ are improved.
\end{remark}

\begin{proof}
We note from Proposition \ref{prop:eqstildePhipnsHigh} that $\PhinsHigh{1}$ satisfies the same equation as $\PhipsHigh{1}$, therefore, the $r^p$ estimates of $\PhipsHigh{1}$ in the proof of Lemma \ref{lem:BED:0to2:v1} hold for $\PhinsHigh{1}$ as well. The same way of arguing therein applies and yields the energy decay \eqref{eq:BED2:Lxi} and a pointwise decay estimate
\begin{align}
\label{eq:ptwdecay:basic:PhinsHigh1}
\absCDeri{\Lxi^j(\mu^{\half}r^{-1}\PhinsHigh{1})}{\reg}
\lesssim_{j,\reg}  {} v^{-1}\tb^{-\half-j} (F^{(1)}(\reg+\regl(j,\reg),2,\tb_0,\Psipns))^{\half}.
\end{align}
The pointwise estimates for both $\psips$ and $\psins$ in \eqref{eq:ptwdecay:basic} are immediate from \eqref{eq:ptwdecay:basic:v1}, and together with the above estimate \eqref{eq:ptwdecay:basic:PhinsHigh1}, these prove the estimate of $rV(r\psins)$ in \eqref{eq:ptwdecay:basic}.
\end{proof}

%%%%%%%%%%%%%%%%%%%%5
\subsection{Newman--Penrose constants}
\label{sect:N-Pconst}
%%%%%%%%%%%%%%%%%%%%%

In this subsection, we introduce the (conserved) Newman--Penrose constants for any mode of the spin $\pm \half$ components.

Recall from \eqref{def:curlVR} that $\curlVR=r^2\VR$. We define a few scalars constructed from the spin $\pm \half$ components and then derive their governing equations in Proposition \ref{prop:eqstildePhipnsHigh}. 

\begin{definition}
\label{def:tildePhiplusandminusHigh}
For any $i\in \mathbb{N}^+$, let $f_{i,1}=i^2$, $f_{i,2}=-2i-1$, $g_i=6 \sum\limits_{j=0}^i f_{j,1}= i(i-1)(2i-1)$, $x_{i+1,i}=\frac{g_{i+1}}{f_{i+1,1} -f_{i,1}}=i(i+1)$, and $x_{i+1,j}=-\frac{g_{i+1}x_{i,j}}{f_{i+1,1}-f_{j,1}}$ for $1\leq j\leq i-1$. Define 
\begin{align}
\PhipsHigh{i}={}&\curlVR^{i-1}\hatPhips, &\PhinsHigh{i}={}&\curlVR^i\Phins,
\end{align}
and
\begin{subequations}
\begin{align}
\tildePhipsHigh{1}={}&\PhipsHigh{1},& \tildePhipsHigh{i+1}={}&\PhipsHigh{i+1}+\sum_{j=1}^i x_{i+1,j}M^{i+1-j}\tildePhipsHigh{j},\\
\tildePhinsHigh{1}={}&\PhinsHigh{1},& \tildePhinsHigh{i+1}={}&\PhinsHigh{i+1}+\sum_{j=1}^i x_{i+1,j}M^{i+1-j}\tildePhinsHigh{j}.
\end{align}
\end{subequations}
\end{definition}

\begin{prop}
\label{prop:eqstildePhipnsHigh}
Let $i\in \mathbb{N}^+$.
\begin{enumerate}
  \item The equation of $\hatPhips$ is
  \begin{subequations}
\begin{align}
\label{eq:PhipsHigh1}
-2\pu\curlVR \hatPhips +(\edthR\edthR'+1)\hatPhips-{3(r-3M)r^{-2}}\curlVR\hatPhips -{6M}{r^{-1}}\hatPhips={}&0,
\end{align}
 the equation of $\PhipsHigh{i}$ is
\begin{align}
\label{eq:Phipshighi}
&-2\pu \curlVR  \PhipsHigh{i} +(\edthR\edthR'+f_{i,1})\PhipsHigh{i}
+{f_{i,2}(r-3M)r^{-2}}\curlVR\PhipsHigh{i}
-6f_{i,1}Mr^{-1}\PhipsHigh{i}
+g_i M\PhipsHigh{i-1}={}0,
\end{align}
and the equation of $\tildePhipsHigh{i}$ is
\begin{align}
\label{eq:tildePhipshighi}
&-2\pu \curlVR  \tildePhipsHigh{i} +(\edthR\edthR'+f_{i,1})\tildePhipsHigh{i}
+{f_{i,2}(r-3M)r^{-2}}\curlVR\tildePhipsHigh{i}
-6f_{i,1}Mr^{-1}\tildePhipsHigh{i}
+\sum_{j=1}^{i}h_{i,j} \PhipsHigh{j}={}0,
\end{align}
\end{subequations}
with $h_{i,j}=O(r^{-1})$ for all $j\in \{1,2,\ldots, i\}$.
  \item The equation of $\PhinsHigh{1}$ is
  \begin{subequations}
\begin{align}
\label{eq:PhinsHigh1}
-2\pu \curlVR\PhinsHigh{1}+(\edthR'\edthR+1)\PhinsHigh{1}
-3(r-3M)r^{-2}
{\curlVR\PhinsHigh{1}}-6Mr^{-1}\PhinsHigh{1}
={}0,
\end{align}
and the equation of $\PhinsHigh{i}$  is
\begin{align}
\label{eq:PhinsHighi}
&-2\pu \curlVR  \PhinsHigh{i} +(\edthR'\edthR+f_{i,1})\PhinsHigh{i}
+{f_{i,2}(r-3M)r^{-2}}\curlVR\PhinsHigh{i}
-6f_{i,1}Mr^{-1}\PhinsHigh{i}
+g_i M\PhinsHigh{i-1}={}0,
\end{align}
and the equation of $\tildePhinsHigh{i}$ is
\begin{align}
\label{eq:tildePhinsHighi}
&-2\pu \curlVR  \tildePhinsHigh{i} +(\edthR\edthR'+f_{i,1})\tildePhinsHigh{i}
+{f_{i,2}(r-3M)r^{-2}}\curlVR\tildePhinsHigh{i}
-6f_{i,1}Mr^{-1}\tildePhinsHigh{i}
+\sum_{j=1}^{i}h_{i,j} \PhinsHigh{j}={}0,
\end{align}
\end{subequations}
with $h_{i,j}$ being the same as the ones in \eqref{eq:tildePhipshighi} and satisfying $h_{i,j}=O(r^{-1})$ for all $j\in \{1,2,\ldots, i\}$.
\end{enumerate}
\end{prop}

\begin{remark}
The definition of the above scalars $\PhipsHigh{i}$ and $\PhinsHigh{i}$ and the derivation of their equations are generalized from the work \cite{Ma20almost} of the first author for the Maxwell field. It is shown therein that these scalars are fundamental in both defining the Newman--Penrose constants for an arbitrary mode and extending the hierarchy of $r^p$ estimates to $p\in [4,5)$. See Section \ref{sect:improveddecbasicenerg} for the extension of the $r^p$ hierarchy.
\end{remark}

\begin{proof}
The wave equation \eqref{eq:PhipsHigh1} is manifestly equation \eqref{eq:TME:rp1:posi}. Equations of $\PhipsHigh{i}$ $(i\in \mathbb{N}^+)$ follow from induction together with a commutator
\begin{align}
\label{commutato:curlVRandwave}
[\curlVR, -r^2YV]\varphi={}&-\curlVR\Big(\tfrac{2(r-3M)}{r^{2}} \curlVR\varphi\Big)
=-\tfrac{2(r-3M)}{r^2}\curlVR^2\varphi
+(2-12Mr^{-1})\curlVR\varphi.
\end{align}
We prove equation \eqref{eq:tildePhipshighi} by induction. Assume it holds for $\tildePhipsHigh{i'}$ for all $1\leq i'\leq i$, we prove the equation for $\tildePhipsHigh{i+1}$. We add $x_{i+1,j} M^{i+1-j}$ multiple of equation \eqref{eq:tildePhipshighi} of $\tildePhipsHigh{j}$ for all $j=1,2,\ldots, i$ to equation \eqref{eq:Phipshighi} of $\PhipsHigh{i+1}$, rearrange the terms on the LHS, and arrive at an equation
\begin{align}
\label{eq:tildePhipsHigh:choices}
&-2\pu \curlVR  \tildePhipsHigh{i+1} +(\edthR\edthR'+f_{i+1,1})\tildePhipsHigh{i+1}
+{f_{i+1,2}(r-3M)r^{-2}}\curlVR\tildePhipsHigh{i+1}
-6f_{i+1,1}Mr^{-1}\tildePhipsHigh{i+1}\notag\\
&-\sum_{j=1}^i x_{i+1,j} M^{i+1-j}(f_{i+1,1}- f_{j,1})\tildePhipsHigh{j}
+g_{i+1}M\PhipsHigh{i}\notag\\
&+6Mr^{-1}\sum_{j=1}^i x_{i+1,j} M^{i+1-j}(f_{i+1,1}- f_{j,1})\tildePhipsHigh{j}\notag\\
&
-(r-3M) r^{-2}\sum_{j=1}^i (f_{i+1,2}-f_{j,2})x_{i+1,j}M^{i+1-j}  \curlVR\tildePhipsHigh{j}
+\sum_{j=1}^i x_{i+1,j} M^{i+1-j}\sum_{j'=0}^{j}h_{j,j'} \PhipsHigh{j'}={}0.
\end{align}
By substituting $\PhipsHigh{i}=\tildePhipsHigh{i} -\sum\limits_{j=1}^{i-1}x_{i,j}M^{i-j}\tildePhipsHigh{j}$ into the last term of the second line, one finds the second line equals $\sum\limits_{j=1}^id_{i+1,j}M^{i+1-j}\tildePhipsHigh{j}$, with
\begin{subequations}
\begin{align}
d_{i+1,i}={}&-x_{i+1,i}(f_{i+1,1}-f_{i,1}) +g_{i+1},\\
d_{i+1,j}={}&-x_{i+1,j}(f_{i+1,1}-f_{j,1}) -g_{i+1}x_{i,j}, \quad \text{for} \quad 1\leq j\leq i-1.
\end{align}
\end{subequations}
Note that the values of $\{x_{i+1,j}\}|_{j=0,\ldots, i}$ in Definition \ref{def:tildePhiplusandminusHigh} are exactly the ones such that all $\{d_{i+1,j}\}_{j=1,2,\ldots,i}$ vanish. So far, the second line of equation \eqref{eq:tildePhipsHigh:choices} vanishes, and, by using Definition \ref{def:tildePhiplusandminusHigh} to write $\curlVR\tildePhipsHigh{j}$ as a weighted sum of $\{\tildePhipsHigh{j'}\}|_{j'=1,\ldots, j+1}$ with all coefficients being $O(1)$, the last two lines of the LHS of \eqref{eq:tildePhipsHigh:choices} are manifestly in the form of $\sum\limits_{j=1}^{i+1}h_{i+1,j} \PhipsHigh{j}$ with $h_{i+1,j}=O(r^{-1})$ for all $j\in \{1,2,\ldots, i+1\}$, hence proving equation \eqref{eq:tildePhipshighi} for $\tildePhipsHigh{i+1}$.

The wave equation \eqref{eq:TME:varphis:nega}
of $\Phins$ is
\begin{align}
\curlLs   \curlLds  \Phins
-r^2YV\Phins={}&-(r-3M)r^{-2}\curlVR\Phins.
\end{align}
We utilize the commutator \eqref{commutato:curlVRandwave}
and thus obtain an equation for $\PhinsHigh{1}$:
\begin{align}
\label{eq:PhinsHigh1:1}
\curlLs   \curlLds  \PhinsHigh{1}
-r^2YV\PhinsHigh{1}
-(r-3M)r^{-2}\curlVR\PhinsHigh{1}
+(1-6Mr^{-1})\PhinsHigh{1}={}0.
\end{align}
This is exactly equation \eqref{eq:PhinsHigh1}.
We simply note that equation \eqref{eq:PhinsHigh1} is of the same form as equation \eqref{eq:PhipsHigh1}, hence the above discussions for the spin $\half$ component apply and yield the equations \eqref{eq:PhinsHighi} and \eqref{eq:tildePhinsHighi}.
\end{proof}

We are now ready to define from the scalars $\PhipsHigh{i}$ and $\PhinsHigh{i}$ a crucial notation: the Newman--Penrose (N--P) constants.

\begin{definition}
\label{def:NPCs}
Let $i\in \mathbb{N}^+$. Assume the spin $\pm \half$ components are supported on $\ell=i$ mode. Define the $i$-th N--P constants of these spin $\half$ and $-\half$ components to be $\NPCP{i}(\theta,\pb)=\lim\limits_{\rb\to\infty}
\curlVR\tildePhipsHigh{i}$ and $\NPCN{i}(\theta,\pb)=\lim\limits_{\rb\to\infty}\curlVR
\tildePhinsHigh{i}$, respectively.
\end{definition}

\begin{remark}
As will be shown in Proposition \ref{prop:NPCsindepentonu} below, these N--P constants are independent of $\tb$ under very general conditions, hence they are only dependent on $\theta$ and $\pb$. In particular, for an $(\ell,m)$ mode, the above N--P constants are in fact constants, independent of the coordinates $(\tb,\rb, \theta,\phi)$.
\end{remark}

Further, the N--P constants for the spin $\half$ and $-\half$ components are related to each other. 

\begin{lemma}
\label{lem:relationoftwoNPconsts}
On Schwarzschild, it holds true that
$\NPCN{i}=\edthR'\NPCP{i}$ for $i\in \mathbb{N}^+$. In particular, if $\NPCN{i}$ vanishes, then $\NPCP{i}$ vanishes, and vice versa.
\end{lemma}

\begin{proof}
Equation \eqref{eq:Dirac:TMEscalar} is $\edthR'\Phips=\Delta^{\half}\VR \Phins$, or, $\edthR'\hatPhips=\curlVR \Phins=\PhinsHigh{1}$. Thus, by definition, $\NPCN{i}=\edthR'\NPCP{i}$ for any $i \in \mathbb{N}^+$.  The other statement follows from the fact that $\edthR'$ has trivial kernel when acting on spin-weight $\half$ scalar.
\end{proof}

The following three propositions are to utilize the equations derived in Proposition \ref{prop:eqstildePhipnsHigh} to conduct some basic asymptotic analysis. This is in the spirit of \cite[Section 3]{angelopoulos2018vector}.

\begin{prop}
\label{prop:nullinfBeha:PhiplusiandPhinsHighi} Let $i\in \mathbb{N}^+$ and $\reg\in \mathbb{N}$.  Let $\regl=\regl(i)>0$ be suitably large. Assume $\sum\limits_{j=1}^iF^{(i)}(\reg+\regl,0,\tb_0,\Psipns)<\infty$
as defined in Definition \ref{def:Fenergies:big2}.
\begin{enumerate}
  \item[$(i)$] If
     $\lim\limits_{r\to\infty}\sum\limits_{j=1}^i\absCDeri{
     \PhipsHigh{j}}{\reg}\vert_{\Sigmazero}
      <\infty$, then for any $\tb\geq \tb_0$, $\lim\limits_{r\to\infty}\sum\limits_{j=1}^i
      \absCDeri{\PhipsHigh{j}}{\reg}\vert_{\Sigmatb}<\infty$. The same statement holds if one replaces all $\PhipsHigh{j}$ by $\tildePhipsHigh{j}$;
  \item[$(ii)$]  If $\lim\limits_{r\to\infty}\Big(\sum\limits_{j=1}^i\absCDeri{
      \edthR\edthR'\PhipsHigh{j}}{\reg}\vert_{\Sigmazero}
      +r^{-\alpha}\absCDeri{\PhipsHigh{i+1}}{\reg}\vert_{\Sigmazero}\Big)
      <\infty$ for some $\alpha\in [0,2]$, then for any $\tb\geq \tb_0$, $\lim\limits_{r\to\infty}
      \Big(\sum\limits_{j=1}^i
      \absCDeri{\edthR\edthR'\PhipsHigh{j}}{\reg}\vert_{\Sigmatb}
      +r^{-\alpha}\absCDeri{\PhipsHigh{i+1}}{\reg}\vert_{\Sigmatb}\Big)<\infty$.
      The same statement holds if one replaces all $\PhipsHigh{j}$ by $\tildePhipsHigh{j}$;
  \item[$(iii)$] If $\lim\limits_{r\to\infty}\sum\limits_{j=1}^{i}\absCDeri{
     \PhinsHigh{j}}{\reg}\vert_{\Sigmazero}
      <\infty$, then for any $\tb\geq \tb_0$, $\lim\limits_{r\to\infty}\sum\limits_{j=1}^{i}
      \absCDeri{\PhinsHigh{j}}{\reg}\vert_{\Sigmatb}<\infty$.
      The same statement holds if one replaces all $\PhinsHigh{j}$ by $\tildePhinsHigh{j}$;
  \item[$(iv)$] If $\lim\limits_{r\to\infty}\Big(\sum\limits_{j=1}^{i}
      \absCDeri{
      \edthR'\edthR\PhinsHigh{j}}{\reg}\vert_{\Sigmazero}
      +r^{-\alpha}\absCDeri{
      \PhinsHigh{i+1}}{\reg}\vert_{\Sigmazero}\Big)
      <\infty$  for some $\alpha\in [0,2]$, then for any $\tb\geq \tb_0$, $\lim\limits_{r\to\infty}\Big(\sum\limits_{j=1}^{i}
      \absCDeri{\edthR'\edthR\PhinsHigh{j}}{\reg}\vert_{\Sigmatb}
      +r^{-\alpha}\absCDeri{\PhinsHigh{i+1}}{\reg}\vert_{\Sigmatb}\Big)<\infty$.
      The same statement holds if one replaces all $\PhinsHigh{j}$ by $\tildePhinsHigh{j}$.
\end{enumerate}
\end{prop}

\begin{proof}
The assumption $\sum\limits_{j=1}^iF^{(i)}(\reg+\regl,0,\tb_0,\Psipns)<\infty$ in particular yields that for any $\tb\geq \tb_0$ and any $1\leq j\leq i$,
\begin{align}
\label{eq:NPconstantProp:energy}
\norm{\Psipns}^2_{W_{-2}^{\reg+\regl}(\Sigmatb)}
+\sum_{j=1}^i\norm{\PhipnsHigh{j}}^2_{W_{-2}^{\reg+\regl}(\Sigmatb^{\geq 4M})}
<\infty
\end{align}
and
\begin{align}
\label{eq:NPconstantProp:ptw}
\sup_{\Sigmatb}\int_{\Sphere} r^{-1}\absCDeri{\Psipns}{\reg+\regl}^2\di^2\mu
+\sup_{\Sigmatb\cap\{\rb\geq 4M\}}\sum_{j=1}^i\int_{\Sphere} r^{-1}\absCDeri{\PhipnsHigh{j}}{\reg+\regl}^2\di^2\mu <\infty.
\end{align}
Note that the first estimate \eqref{eq:NPconstantProp:energy} is contained in the proof of Proposition \ref{eq:NPconstantProp:energy} and the second estimate \eqref{eq:NPconstantProp:ptw} follows from the Sobolev-type estimate \eqref{eq:Sobolev:1} together with the estimate \eqref{eq:NPconstantProp:energy}.
The rest of the proof is similar to the one of \cite[Propositions 3.4 and 3.5]{angelopoulos2018vector} and we omit it.
\end{proof}

\begin{prop}
\label{prop:NPCsindepentonu}
Let $\ell\in \mathbb{N}$ and let $\regl=\regl(\ell,i)>0$ be suitably large. Assume $F^{(\ell)}(\reg+\regl,0,\tb_0,\Psipns)<\infty$ as defined in Definition \ref{def:Fenergies:big2}.
\begin{enumerate}
  \item Let the spin $\pm \half$ components of Dirac field be supported on $\ell=\ell_0=1$ mode.
      \begin{itemize}
      \item Assume $\lim\limits_{r\to\infty}\big(\abs{
      \PhipsHigh{1}}
      +\abs{
      \curlVR\PhipsHigh{1}}\big)\vert_{\Sigmazero}
      <\infty$, then the first N--P constant $\NPCP{1}$ is finite and independent of $\tb$;
      \item Assume $\lim\limits_{r\to\infty}\big(\abs{
      \PhinsHigh{1}}
      +\abs{
      \curlVR\PhinsHigh{1}}\big)\vert_{\Sigmazero}
      <\infty$, then the first N--P constant $\NPCN{1}$ is finite and independent of $\tb$.
      \end{itemize}
  \item Let  the  spin $\pm \half$ components of Dirac field be supported on $\ell=\ell_0$ $(\ell_0\geq 2)$ mode.
      \begin{itemize}
      \item Assume $\lim\limits_{r\to\infty}
          \sum\limits_{j=1}^{\ell_0}
          (\abs{
      \tildePhipsHigh{j}}+\abs{
      \curlVR\tildePhipsHigh{j}})\vert_{\Sigmazero}
      <\infty$, then the $\ell_0$-th N--P constant $\NPCP{\ell_0}$ is finite and independent of $\tb$;
       \item Assume $\lim\limits_{r\to\infty}\sum\limits_{j=1}^{\ell_0}
          (\abs{
      \tildePhinsHigh{j}}+\abs{
      \curlVR\tildePhinsHigh{j}})\vert_{\Sigmazero}
      <\infty$, then the $\ell_0$-th N--P constant $\NPCN{\ell_0}$ is finite and independent of $\tb$.
      \end{itemize}
\end{enumerate}
\end{prop}

\begin{proof}
If the field is supported on $\ell=1$ mode, then from Proposition \ref{prop:eqstildePhipnsHigh}, $\Psi=\PhipsHigh{1}$ or $\Psi=\PhinsHigh{1}$ solves
\begin{align}
-2\pu\curlVR \Psi -{3(r-3M)r^{-2}}\curlVR\Psi -{6M}{r^{-1}}\Psi={}&0.
\end{align}
The results in Proposition \ref{prop:nullinfBeha:PhiplusiandPhinsHighi} implies $\lim\limits_{r\to \infty}\big(\abs{\Psi}+\abs{\curlVR\Psi}\big)\vert_{\Sigmatb}<\infty$
for any $\tb\geq \tb_0$, which thus implies
$\lim\limits_{r\to \infty}\pu (\curlVR \Psi)\vert_{\Sigmatb}=0$ for any $\tb\geq \tb_0$.  The conclusion follows  from the bounded convergence theorem.

Instead, if the field is supported on $\ell=\ell_0$ mode for some $\ell_0\geq 2$, equations
 for $\tildePhipsHigh{\ell_0}$ and for $\tildePhinsHigh{\ell_0}$ become
\begin{align}
\label{eq:268}
&-2\pu \curlVR \tildePhipsHigh{\ell_0}
-{(2\ell_0+1)(r-3M)r^{-2}}\curlVR\tildePhipsHigh{\ell_0}
+\sum_{j=1}^{\ell_0}O(r^{-1})\tildePhipsHigh{j}={}0,\\
&-2\pu \curlVR \tildePhinsHigh{\ell_0}
-{(2\ell_0+1)(r-3M)r^{-2}}\curlVR\tildePhinsHigh{\ell_0}
+\sum_{j=1}^{\ell_0}O(r^{-1})\tildePhinsHigh{j}={}0.
\end{align}
One also obtains $\lim\limits_{r\to \infty}\pu (\curlVR \tildePhipsHigh{\ell_0})\vert_{\Sigmatb}=\lim\limits_{r\to \infty}\pu (\curlVR\tildePhinsHigh{\ell_0})\vert_{\Sigmatb}=0$ from Proposition \ref{prop:nullinfBeha:PhiplusiandPhinsHighi}, and by the same way of arguing, the statement follows.
\end{proof}

\begin{prop}
\label{prop:vanishingNPC:betterasymnearscri}
Let the  spin $\pm \half$ components of Dirac field be supported on an $\ell= \ell_0$ mode with $\ell_0\geq 1$. Let $\alpha\in [0,1]$ be arbitrary and let $\reg\in \mathbb{N}$. Assume the $\ell_0$-th N--P constant $\NPCN{\ell_0}$  vanishes.
\begin{itemize}
\item There exists a $\regl=\regl(\ell_0)$ such that if $F^{(\ell_0)}(\reg+\regl,0,\tb_0,\Psipns)
     +\lim\limits_{r\to\infty}\absCDeri{
      r^{1-\alpha}\curlVR\tildePhinsHigh{\ell_0}}{\reg}\vert_{\Sigmazero}
    +\lim\limits_{r\to\infty}\sum\limits_{j=1}^{\ell_0}
    \absCDeri{
      \tildePhinsHigh{j}}{\reg}\vert_{\Sigmazero}
      <\infty$,
      then there is a  constant $C_{\ell_0}(\tb,\theta,\pb)<\infty$ such that for any $\tb\geq \tb_0$, $\lim\limits_{r\to\infty}
\absCDeri{r^{1-\alpha}\curlVR\tildePhinsHigh{\ell_0}}{\reg}
\vert_{\Sigmatb}<C_{\ell_0}(\tb,\theta,\pb)$.
In particular, if $\alpha>0$, then $\lim\limits_{r\to\infty}\absCDeri{r^{1-\alpha}
\curlVR\tildePhinsHigh{\ell_0}}{\reg}
\vert_{\Sigmatb}$ is independent of $\tb$;
\item There exists a $\regl=\regl(\ell_0)$ such that if $F^{(\ell_0)}(\reg+\regl,0,\tb_0,\Psipns)
    +\lim\limits_{r\to\infty}r^{1-\alpha}
          \absCDeri{
      \curlVR\tildePhipsHigh{\ell_0}}{\reg}\vert_{\Sigmazero}
    +\lim\limits_{r\to\infty}\sum\limits_{j=1}^{\ell_0}
          \absCDeri{
     \tildePhipsHigh{j}}{\reg}\vert_{\Sigmazero}
      <\infty$, then there is a constant $C_{\ell_0}(\tb,\theta,\pb)<\infty$ such that for any $\tb\geq \tb_0$, $\lim\limits_{r\to\infty}
\absCDeri{r^{1-\alpha}
\curlVR\tildePhipsHigh{\ell_0}}{\reg}
\vert_{\Sigmatb}<C_{\ell_0}(\tb,\theta,\pb)$. In particular, if $\alpha>0$, then $\lim\limits_{r\to\infty}\absCDeri{r^{1-\alpha}
\curlVR\tildePhipsHigh{\ell_0}}{\reg}
\vert_{\Sigmatb}$ is independent of $\tb$.
\end{itemize}
\end{prop}

\begin{proof}
We show it only for the spin $\half$ component, the proof of the spin $-\half$ component being the same. Consider first the $\ell=\ell_0$ mode $\Psips^{\ell=\ell_0}$. The scalar $\tildePhipsHigh{\ell_0}$  satisfies equation \eqref{eq:268}, and hence performing a rescaling gives
\begin{align}
-\pu (r^{1-\alpha}\curlVR \tildePhipsHigh{\ell_0})
={}O(r^{-\alpha})\curlVR\tildePhipsHigh{\ell_0}
+\sum_{j=1}^{\ell_0}O(r^{-\alpha})\tildePhipsHigh{j}.
\end{align}
By Proposition \ref{prop:nullinfBeha:PhiplusiandPhinsHighi} and the assumption of vanishing $\ell_0$-th N--P constant, in the case that $\alpha>0$, this yields $\lim\limits_{r\to \infty}\absCDeri{r^{-\alpha}
\curlVR\tildePhipsHigh{\ell_0}}{\reg}
\vert_{\Sigmatb}=0$, and one obtains
$\lim\limits_{r\to \infty}\pu (r^{1-\alpha}\curlVR \tildePhipsHigh{\ell_0})\vert_{\Sigmatb})=0$ for any $\tb\geq \tb_0$. The conclusion for $\alpha>0$ follows  from the bounded convergence theorem. For $\alpha=0$, the RHS is bounded by a $\tb$-dependent constant, hence $\lim\limits_{r\to\infty}
\absCDeri{r
\curlVR\tildePhipsHigh{\ell_0}}{\reg}
\vert_{\Sigmatb}<C(\tb)$.
\end{proof}

%%%%%%%%%%%%%%%%%%%%%%%%
\subsection{Improved decay of basic energy}
\label{sect:improveddecbasicenerg}
%%%%%%%%%%%%%%%%%%

Following Definition \ref{def:Fenergies:0to2}, we further define the following $p$-weighted energies for the purpose of obtaining $r^p$ estimates for a larger range of $p$, hence also improved energy and pointwise decay estimates.

\begin{definition}
\label{def:Fenergies:big2}
Let $i\in \mathbb{N}^+$. Define
\begin{subequations}
\begin{align}
F^{(i)}(\reg,-1,\tb,\Psips)=
F(\reg,-1,\tb,\Psips)
+\norm{\PhipsHigh{i}}^2_{W_{-3}^{\reg-i+1}(\Sigmatb^{\geq 4M})},\\
F^{(i)}(\reg,-1,\tb,\Psins)=
F(\reg,-1,\tb,\Psins)
+\norm{\PhinsHigh{i}}^2_{W_{-3}^{\reg-i+1}(\Sigmatb^{\geq 4M})},
\end{align}
\end{subequations}
for any $-1< p<0$ that
\begin{align}
F^{(i)}(\reg,p,\tb,\Psips)=F^{(i)}(\reg,p,\tb,\Psins)={}0,
\end{align}
for any $0\leq p\leq 2$ that
\begin{subequations}
\begin{align}
F^{(i)}(\reg,p,\tb,\Psips)={}&F(\reg,0,\tb,\Psips)
+\norm{rV\PhipsHigh{i}}^2_{W_{p-2}^{\reg-i}(\Sigmatb^{\geq 4M})}
+\norm{\PhipsHigh{i}}^2_{W_{-2}^{\reg-i+1}(\Sigmatb^{\geq 4M})},\\
F^{(i)}(\reg,p,\tb,\Psins)={}&F(\reg,0,\tb,\Psins)
+\norm{rV\PhinsHigh{i}}^2_{W_{p-2}^{\reg-i}(\Sigmatb^{\geq 4M})}
+\norm{\PhinsHigh{i}}^2_{W_{-2}^{\reg-i+1}(\Sigmatb^{\geq 4M})},
\end{align}
\end{subequations}
and for any $2< p<5$ that
\begin{subequations}
\begin{align}
F^{(i)}(\reg,p,\tb,\Psips)={}&F(\reg,0,\tb,\Psips)
+\norm{rV\tildePhipsHigh{i}}^2_{W_{p-2}^{\reg-i}(\Sigmatb^{\geq 4M})}
+\norm{\tildePhipsHigh{i}}^2_{W_{-2}^{\reg-i+1}(\Sigmatb^{\geq 4M})},\\
F^{(i)}(\reg,p,\tb,\Psins)={}&F(\reg,0,\tb,\Psins)
+\norm{rV\tildePhinsHigh{i}}^2_{W_{p-2}^{\reg-i}(\Sigmatb^{\geq 4M})}
+\norm{\tildePhinsHigh{i}}^2_{W_{-2}^{\reg-i+1}(\Sigmatb^{\geq 4M})}.
\end{align}
\end{subequations}
Define in the end for any $p\in [-1,5)$ that
\begin{align}
F^{(i)}(\reg,p,\tb,\Psipns)={}
F^{(i)}(\reg,p,\tb,\Psips)+F^{(i)}(\reg,p,\tb,\Psins).
\end{align}
\end{definition}

The main statement in this subsection is as follows.

\begin{prop}
\label{prop:BED:Psiplus:l=1}
Given $\ell\in \mathbb{N}^+$. Let $j\in \mathbb{N}$ and let $\reg\in \mathbb{N}$. Let $\Psips$ and $\Psins$ be supported on an $\ell$ mode.
Then,
\begin{enumerate}
\item if the $\ell$-th N-P constant $\NPCP{\ell}$ does not vanish,
there is a constant $\regl(j,\ell)$ such that for any small $\delta>0$, any $p\in (1,3-\delta]$, any $p'\in [0,\min\{p,2\}]$ and any $\tb\geq\tb_0$,
\begin{align}
\label{eq:BEDC:Phiplus:l=1:0to3}
\hspace{8ex}&\hspace{-8ex}
F^{(1)}(\reg,p',\tb,\Lxi^j\Psipns)
+\int_{\tb}^{\infty}F^{(1)}(\reg,p',\tb',\Lxi^j\Psipns)\di\tb'\notag\\
\lesssim_{\delta,j,\reg,\ell} {}&\tb^{-2(\ell-1)-2j+p'-p}F^{(\ell)}(\reg+\regl(j,\ell),p,\tb_0,\Psipns)
\end{align}
and for any $p\in (1,3-\delta]$,
\begin{align}
\label{ptw:BEDC:Phiplus:l=1:0to3}
\hspace{8ex}&\hspace{-8ex}
\absCDeri{\Lxi^j\psips}{\reg}+
\absCDeri{\Lxi^j\psins}{\reg} +\absCDeri{\Lxi^j(\mu^{\half}r^{-1}\PhinsHigh{1})}{\reg}\notag\\
\lesssim_{\delta,j,\reg,\ell}{} & v^{-1}\tb^{-\ell-j+\frac{3-p}{2}} (F(\reg+\regl(j,\reg),p,\tb_0,\Psipns))^{\half};
\end{align}
\item
if the $\ell$-th N-P constant $\NPCP{\ell}$ vanishes, there is a constant $\regl(j,\ell)$ such that for any small $\delta>0$, any $p\in (1, 5-\delta]$, any $p'\in [0,\min\{p,2\}]$ and any $\tb\geq\tb_0$,
\begin{align}
\label{eq:BEDC:Phiplus:l=1:0to5}
\hspace{8ex}&\hspace{-8ex}
F^{(1)}(\reg,p',\tb,\Lxi^j\Psipns)
+\int_{\tb}^{\infty}F^{(1)}(\reg,p',\tb',\Lxi^j\Psipns)\di\tb'\notag\\ \lesssim_{\delta,j,\reg,\ell} {}&\tb^{-2(\ell-1)-2j+p'-p}
F^{(\ell)}(\reg+\regl(j),p,\tb_0,\Psipns)
\end{align}
and for any $p\in (1,5-\delta]$,
\begin{align}
\label{ptw:BEDC:Phiplus:l=1:0to5}
\hspace{8ex}&\hspace{-8ex}
\absCDeri{\Lxi^j\psips}{\reg}+
\absCDeri{\Lxi^j\psins}{\reg}
+\absCDeri{\Lxi^j(\mu^{\half}r^{-1}\PhinsHigh{1})}{\reg}\notag\\
\lesssim_{\delta,j,\reg,\ell} {} & v^{-1}\tb^{-\ell-j+\frac{3-p}{2}} (F(\reg+\regl(j,\reg),p,\tb_0,\Psipns))^{\half} .
\end{align}
\end{enumerate}
\end{prop}

In the following discussions, we prove the above proposition for $\ell=1$ case and $\ell\geq 2$ case in Sections \ref{sect:almostsharpenergy:l=1} and \ref{sect:almostsharpenergy:lbigger1}, respectively. The $r^p$ estimates for the spin $\pm \half$ components are proven by putting the derived equations in Proposition \ref{prop:eqstildePhipnsHigh} into the form \eqref{eq:wave:rp} and applying the general $r^p$ estimates in Proposition \ref{prop:wave:rp}. Moreover, we consider only the border case that $p=3-\delta$ or $p=5-\delta$ as the other cases where $p\in (1,3-\delta)$ or $p\in (1,5-\delta)$ are much easier and can be proven in an exactly same way.

\begin{remark}
If the $\ell$-th N--P constant $\NPCP{\ell}$ does not vanish, the energy $F^{(\ell)}(\reg,3,\tb_0,\Psipns)$ for any $\reg\geq 1$ is infinite. Thus, in this respect, one can only extend the upper bound of the $p$ range up to $3-$ in this case, and the energy decay estimate \eqref{eq:BEDC:Phiplus:l=1:0to3} with $p=3-\delta$ is sharp. It is also from this respect that we separate out the vanishing $\ell$-th N--P constant case and the nonvanishing $\ell$-th N--P constant case.
\end{remark}

\begin{remark}
The viability of extending the $r^p$ hierarchy from $[0,2]$ to $[0,5)$ is first realized in \cite{angelopoulos2018vector} in treating the spherically symmetric $\ell=0$ mode of the scalar field. We here generalize it to any mode of the Dirac field.
\end{remark}

%%%%%%%%%%%%
\subsubsection{$\ell=1$ mode}
\label{sect:almostsharpenergy:l=1}
%%%%%%%%%%%%

We shall remind the readers that the discussions for this $\ell=1$ mode in this subsection can be put into the context of Section \ref{sect:almostsharpenergy:lbigger1} in which an $\ell=\ell_0\geq 2$ mode is considered, as Section \ref{sect:almostsharpenergy:lbigger1} in fact provides a general approach of proving energy decay estimates for any $\ell=\ell_0\geq 1$ mode. The reason that we separate the discussions for the $\ell=1$ mode from the ones for the $\ell=\ell_0\geq 2$ mode is twofold: one is that the decay of the $\ell=1$ mode, which is much slower than any $\ell=\ell_0\geq 2$ mode as shown in Proposition \ref{prop:BED:Psiplus:l=1},  dominates the decay of the entire Dirac field and hence it is worth estimating this mode alone; the other one is that the proof is simpler than an $\ell=\ell_0\geq 2$ mode but at the same time encodes already the basic idea in treating a general higher mode.

We start with equation \eqref{eq:PhipsHigh1} of $\hatPhips$. For the $\ell=1$ mode, we use \eqref{eq:l=l0mode:eigenvalue} and find that it simplifies to
\begin{align}
\label{eq:PhipsHigh1:l=1}
-r^2YV\hatPhips -{(r-3M)}\VR\hatPhips -{6M}{r^{-1}}\hatPhips={}&0.
\end{align}
We multiply this equation by $-2r^{p-2}\chi^2 V\overline{\hatPhips}$ and take the real part, then the LHS becomes 
\begin{align*}
\hspace{2.5ex}&\hspace{-2.5ex}
r^{p}\chi^2 Y(\abs{V\hatPhips}^2)+2\mu^{-1}(r-3M)r^{p-2}\chi^2 \abs{V\hatPhips}^2
+12Mr^{p-3}\chi^2 \Re(V\overline{\hatPhips} \hatPhips)\notag\\
={}&Y(r^{p}\chi^2 \abs{V\hatPhips}^2)
+\big(\partial_r (r^{p}\chi^2 )+2\Delta^{-1}(r-3M)r^{p}\chi^2 \big)\abs{V\hatPhips}^2
+12Mr^{p-3}\chi^2 \Re(V\overline{\hatPhips} \hatPhips)\notag\\
={}&Y(r^{p}\chi^2 \abs{V\hatPhips}^2)+
\big((p+r\partial_r  )\chi^2 r^{p-1}+2r^{p}\Delta^{-1} (r-3M)\chi^2\big) \abs{V\hatPhips}^2
+12Mr^{p-3}\chi^2 \Re(V\overline{\hatPhips} \hatPhips).
\end{align*}
We then integrate over $\Donetwo$ with a measure $\di^4 \mu$, arriving at
\begin{align}
\label{eq:multiplierintegral:generalp:l=1:Schw}
\int_{\Donetwo}\Big(&Y(r^{p}\chi^2 \abs{V\hatPhips}^2)+12Mr^{p-3} \chi^2\Re(V\overline{\hatPhips} \hatPhips)\notag\\
&+\big((p+r\partial_r  )\chi^2 r^{p-1}+2r^{p}\Delta^{-1} (r-3M)\chi^2\big) \abs{V\hatPhips}^2\Big)\di^4\mu
={}0.
\end{align}

For $0\leq p< 3$, the second line on the LHS is bounded by a bulk integral $\int_{\Donetwo^{R_0-M}}\chi^2 r^{p-1}\abs{V\hatPhips}^2\di^4\mu$ from below, and one applies an integration by parts to the second term on the LHS to obtain both positive fluxes at $\Sigmatwo$ and a positive spacetime integral.  By adding this to the BEAM estimate, this gives for any $p\in [0,3)$ and $\reg\geq 1$
\begin{align}
\label{eq:rphierachyintermsofF:Psiminus}
F^{(1)}(\reg,p,\tb_2,\Psips)
+\int_{\tb_1}^{\tb_2}F^{(1)}(\reg,p-1,\tb,\Psips)\di\tb
\lesssim_{p,\reg}{} F^{(1)}(\reg+\regl,p,\tb_1,\Psips),
\end{align}
where the $\reg\geq 2$ cases follow in the same way as in Proposition \ref{prop:wave:rp}.  Since the equation of $\PhinsHigh{1}$ is the same as equation \eqref{eq:PhipsHigh1:l=1}, one can obtain \eqref{eq:rphierachyintermsofF:Psiminus} as well for the spin $-\half$ component. Thus, for any $p\in [0,3)$ and $\reg\geq 1$,
\begin{align}
\label{eq:rphierachyintermsofF:Psipns:0to3}
F^{(1)}(\reg,p,\tb_2,\Psipns)
+\int_{\tb_1}^{\tb_2}F^{(1)}(\reg,p-1,\tb,\Psipns)\di\tb
\lesssim_{p,\reg}{} F^{(1)}(\reg+\regl,p,\tb_1,\Psipns),
\end{align}
and this gives an extended $r^p$ hierarchy for $p\in [0,3)$, which then implies the estimate \eqref{eq:BEDC:Phiplus:l=1:0to3} with $\ell=1$ and $j=0$ by using Lemma \ref{lem:hierarchyImpliesDecay}.
To show the general $j\in \mathbb{N}$ case, one can follow the discussions after equation \eqref{eq:rp:bothspin} by using the wave equation of $\PhipsHigh{1}$ to rewrite $r^2 V\Lxi\PhipsHigh{1}$. Similarly, we have
\begin{align}
F^{(1)}(\reg,2,\tb,\Lxi\Psipns)\lesssim F^{(1)}(\reg+\regl,0,\tb,\Psipns).
\end{align}
One can then obtain
\begin{align}
F^{(1)}(\reg,p,\tb,\Lxi\Psipns)&\lesssim {} \tb^{-2+p}F^{(1)}(\reg+\regl,2,\tb/2,\Lxi\Psipns)
\lesssim \tb^{-2+p}F^{(1)}(\reg+\regl,0,\tb/2,\Psipns)\notag\\
&\lesssim_{\delta,\reg} {}\tb^{-5+\delta+p}F^{(1)}(\reg+\regl,3-\delta,\tb_0,\Psipns).
\end{align}
This proves $j=1$ case, and the above procedures can be applied to prove the general $j\in \mathbb{N}$ case of the estimate \eqref{eq:BEDC:Phiplus:l=1:0to3}.

Consider then the case that the first N-P constant $\NPCP{1}$ vanishes. We can extend the upper bound of the $p$ range beyond $3$ via the essentially same idea. The second term in the first line of equation \eqref{eq:multiplierintegral:generalp:l=1:Schw} can be bounded using the Cauchy--Schwarz inequality by
\begin{align}
\label{eq:rp:3-4:posi:CS}
\hspace{4ex}&\hspace{-4ex}
\bigg|\int_{\Donetwo}12Mr^{p-3} \chi^2\Re(V\overline\hatPhips \hatPhips)\di^4\mu\bigg|\notag\\
\lesssim{}&\veps\int_{\Donetwo^{R_0-M}}r^{p-1} \chi^2 \abs{V\hatPhips}^2\di^4\mu
+\veps^{-1}\int_{\Donetwo^{R_0-M}}r^{p-5} \chi^2 \abs{\hatPhips}^2\di^4\mu,
\end{align}
and this last term is  bounded via the Hardy's inequality \eqref{eq:HardyIneqLHSRHS} by $\veps^{-1}\big(\int_{\Donetwo^{R_0-M}}r^{p-3} \abs{\prb  \hatPhips}^2\di^4\mu+\int_{\Donetwo^{R_0-M, R_0}}r^{p-5} \abs{ \hatPhips}^2\di^4\mu\big)$ since $\lim\limits_{r\to \infty} r^{p-4}\abs{\hatPhips}^2=0$. Combined with the BEAM estimates, these terms can be easily absorbed by choosing $\veps$ small and $R_0$ sufficiently large, and this proves the estimate \eqref{eq:rphierachyintermsofF:Psiminus} for $p\in [3,4)$. We have similar estimates for the spin $-\half$ component since it satisfies the same equation as the spin $\half$ component.  In summary, we have thus obtained an $r^p$ hierarchy for $p\in [0,4)$, i.e., the estimate \eqref{eq:rphierachyintermsofF:Psipns:0to3} holds for $p\in [0,4)$. The above discussions applied here then yield that
there is a constant $\regl(j)$ such that for any small $\delta>0$, any $p\in [0,4-\delta]$ and any $\tb\geq2\tb_0$,
\begin{align}
\label{eq:BEDC:Phiplus:l=1:0to4}
F^{(1)}(\reg,p,\tb,\Lxi^j\Psipns)
\lesssim_{\delta,j,\reg} {}&\tb^{-4+\delta-2j+p}
F^{(1)}(\reg+\regl(j),4-\delta,\tb/2,\Psipns).
\end{align}

Further, we can extend the hierarchy to $p\in [0,5)$.
For $4\leq p\leq 5-\delta$ where $\delta>0$ is small and arbitrary, the above way of applying the Cauchy--Schwarz inequality does not work any more since the last term in \eqref{eq:rp:3-4:posi:CS} can not be bounded anymore. Instead, we estimate the second term on the LHS  of \eqref{eq:multiplierintegral:generalp:l=1:Schw} by a different way of applying the Cauchy--Schwarz inequality 
\begin{align}
\hspace{4ex}&\hspace{-4ex}
\bigg|\int_{\Donetwo}
12Mr^{p-3} \chi^2_R\Re(V\overline\hatPhips \hatPhips )\di^4\mu\bigg|\notag\\
\lesssim{}&\veps\int_{\Donetwo}r^{p} \tb^{-1-\delta}\chi^2_R \abs{V\hatPhips }^2 \di^4\mu +\veps^{-1}\int_{\Donetwo}r^{p-6}\tb^{1+\delta} \chi^2_R \abs{\hatPhips }^2 \di^4\mu.
\end{align}
The first term on the RHS can be absorbed by choosing $\veps$ small, and the second term is bounded using the estimate \eqref{eq:BEDC:Phiplus:l=1:0to4} by
$\int_{\tb_1}^{\tb_2}\tb^{1+\delta}F^{(1)}(1,p-4,\tb,\Psips)\di\tb
\lesssim_{\delta}  \tb_1^{-6+2\delta+p}F^{(1)}(\regl,4-\delta,\tb_0,\Psips)$. Therefore, one obtains for any $p\in [4,5-\delta]$ and $\tb_2>\tb_1\geq \tb_0$,
\begin{align}
\hspace{8ex}&\hspace{-8ex}
F^{(1)}(\reg,p,\tb_2,\Psipns)
+\int_{\tb_1}^{\tb_2}F^{(1)}(\reg,p-1,\tb,\Psipns)\di\tb
\notag\\
\lesssim_{\delta,\reg} {} &F^{(1)}(\reg+\reg',p,\tb_1,\Psipns)
+\tb_1^{-6+2\delta+p}F^{(1)}(\reg+\regl,4-\delta,\tb_1,\Psipns).
\end{align}
Thus, for any $p\in [4,5-\delta)$, Lemma \ref{lem:hierarchyImpliesDecay} implies
\begin{align}
F^{(1)}(\reg,p,\tb,\Psipns)\lesssim{}
\tb^{-5+\delta+p}F^{(1)}(\reg+\regl,5-\delta,\tb/2,\Psipns).
\end{align}
The estimate \eqref{eq:BEDC:Phiplus:l=1:0to5} for $\ell=1$ then follows from this estimate combined with the estimate \eqref{eq:BEDC:Phiplus:l=1:0to4}. \qed

%%%%%%%%%%%%%
\subsubsection{$\ell=\ell_0\geq 2$ mode}
\label{sect:almostsharpenergy:lbigger1}
%%%%%%%%%%%%

The wave equation \eqref{eq:Phipshighi} now takes the form of
\begin{align}
\label{eq:Phipshighi:lbig2}
&-r^2 Y V  \PhipsHigh{i} -(\ell_0^2-i^2)\PhipsHigh{i}
-(2i-1)(r-3M)r^{-2}\curlVR\PhipsHigh{i}
-6i^2Mr^{-1}\PhipsHigh{i}
+g_i M\PhipsHigh{i-1}={}0.
\end{align}
For any $1\leq i\leq \ell_0-1$, this equation can be put into the form of equation \eqref{eq:wave:rp}, and the assumptions in Proposition \ref{prop:wave:rp} are all satisfied with $b_{0,0}(\PhipsHigh{i})+\ell_0^2=\ell_0^2-i^2>0$, $\vartheta(\PhipsHigh{i})=-g_{i}M \PhipsHigh{i-1}$; for $i=\ell_0$, this can also be put into the form of equation \eqref{eq:wave:rp},  and the assumptions in Proposition \ref{prop:wave:rp} are satisfied with $b_{0,0}(\PhipsHigh{\ell_0})+\ell_0^2=0$ and $\vartheta(\PhipsHigh{\ell_0})=-g_{\ell_0}M \PhipsHigh{\ell_0-1}$.
The estimates in Proposition \ref{prop:wave:rp} then applies: for any $p\in [0,2)$, the error terms arising from $\{\vartheta(\PhipsHigh{i})\}_{i=2,\ldots,\ell_0}$ are bounded by  the corresponding estimate of $\PhipsHigh{i-1}$.

For any $1\leq i\leq \ell_0$ and $p\in [0,2]$, let
\begin{align}
\label{def:Ffts:PhiplusHigh:0to2}
\tilde{F}^{(i)}(\reg,p,\tb,\Psipns)={}&{F}(\reg,p,\tb,\Psipns)
+\sum_{j=1}^i\Big(\norm{rV\PhipnsHigh{j}}^2_{W_{p-2}^{\reg-1-j}(\Sigmatb^{4M})}
+\norm{\PhipnsHigh{j}}^2_{W_{-2}^{\reg-j}(\Sigmatb^{4M})}\Big);
\end{align}
for any $1\leq i\leq \ell_0$ and $p\in (-1,0)$, let $\tilde{F}^{(i)}(\reg,p,\tb,\Psipns)=0$; and for any $1\leq i\leq \ell_0$ and $p=-1$, let $\tilde{F}^{(i)}(\reg,-1,\tb,\Psipns)
=\norm{\PhipnsHigh{1}}^2_{W_{-3}^{\reg-1}(\Sigmatb)}
+\sum\limits_{m=1}^i\norm{\PhipnsHigh{m}}^2_{W_{-3}^{\reg-m}(\Sigmatb^{4M})}$. Then
it holds for any $1\leq i\leq \ell_0$, $p\in [0,2)$ and $\tb_2>\tb_1\geq \tb_0$,
\begin{align}
\label{eq:EnerDecay:Impro:Psiplus:l0:0to2:v2:Integralform}
\tilde{F}^{(i)}(\reg,p,\tb_2,\Psips)
+\int_{\tb_1}^{\tb_2}\tilde{F}^{(i)}(\reg,p-1,\tb,\Psips)\di\tb
\lesssim_{p,\reg} {}&\tilde{F}^{(i)}(\reg+\regl,p,\tb_1,\Psips).
\end{align}
This yields by using Lemma \ref{lem:hierarchyImpliesDecay}  that for  any $1\leq i\leq \ell_0$, $\delta\in (0,\half)$ and $p\in [0,2-\delta]$,
\begin{align}
\tilde{F}^{(i)}(\reg,p,\tb,\Psips)\lesssim{}&
\tb^{-2+\delta+p}\tilde{F}^{(i)}(\reg,2-\delta,\tb/2,\Psips).
\end{align}
Together with the fact that the relation $\tilde{F}^{(i+1)}(\reg,0,\tb,\Psips)\lesssim \tilde{F}^{(i)}(\reg+\regl,2,\tb,\Psips)\lesssim \tilde{F}^{(i+1)}(\reg,0,\tb,\Psips)$ holds true for any $i\in \mathbb{N}$ since one can always rewrite $rV \PhipnsHigh{j}=\mu r^{-1}\PhipnsHigh{j+1}$ by Definition \ref{def:tildePhiplusandminusHigh},  we conclude that for any $p\in [0,2)$ and any $1\leq i\leq \ell_0$, there exists a constant $\regl(j,\ell_0-i)$ such that
\begin{align}
\label{eq:EnerDecay:Impro:Psiplus:l0:0to2:v2}
\tilde{F}^{(i)}(\reg,p,\tb,\Lxi^j\Psips)\lesssim_{\delta,j,\ell_0,\reg}{}&
\tb^{-2(\ell_0-i)-2j-2+p+C\delta}
\tilde{F}^{(\ell_0)}(\reg+\regl(j,\ell_0-i),2-\delta,\tb_0,\Psips).
\end{align}

To apply the estimate \eqref{eq:rp:p=2:2} to equation \eqref{eq:Phipshighi:lbig2}, we replace the error term by \eqref{eq:rp:p=2:2:errorterm} and find the error term \eqref{eq:rp:p=2:2:errorterm} arising from the last term on the LHS of \eqref{eq:Phipshighi:lbig2} is bounded using a Cauchy--Schwarz inequality by
\begin{align}
\hspace{4ex}&\hspace{-4ex}
\veps\int_{\tb_1}^{\tb_2}\frac{1}{\tb^{1+\delta}}
\Big(\norm{rV\PhipsHigh{i}}^2_{W_{0}^{\reg}(\Sigmatb^{R_0})}
+\norm{\PhipsHigh{i}}^2_{W_{-2}^{\reg}(\Sigmatb^{R_0})}\Big)\di\tb\notag\\
&
+\frac{C}{\veps}\int_{\tb_1}^{\tb_2}
\tb^{1+\delta}\norm{\PhipsHigh{i-1}}^2_{W_{-2}^{\reg}(\Sigmatb^{R_0})}
\di\tb
+C\norm{\PhipsHigh{i-1}}^2_{W_{-1-\delta}^{\reg}(\Donetwo^{R_0})}.
\end{align}
The first line is absorbed by choosing $\veps$ small and the second line is bounded from the estimates \eqref{eq:EnerDecay:Impro:Psiplus:l0:0to2:v2} by $C\tb_1^{-2(\ell_0-i)+C\delta}\tilde{F}^{(\ell_0)}
(\reg+\regl,2-\delta,\tb_0,\Psipns)$.  The treatment for the spin $-\half$ component is the same.
One can apply again the above argument and eventually obtains
for any $p\in [0,2]$ and any $1\leq i\leq \ell_0$,
\begin{subequations}
\begin{align}
\label{eq:EnerDecay:Impro:Psiplus:l0:0to2:anyi}
F^{(i)}(\reg,p,\tb,\Lxi^j\Psipns)&\lesssim_{j,\ell_0,i,\reg} {}\tb^{-2(\ell_0-i)-2j-2+p}
F^{(\ell_0)}(\reg+\regl(j,\ell_0-i),2,\tb_0,\Psipns),\\
\label{eq:EnerDecay:Impro:Psiplus:l0:0to2}
F^{(1)}(\reg,p,\tb,\Lxi^j\Psipns)&\lesssim_{j,\ell_0,\reg} {}\tb^{-2(\ell_0-1)-2j-2+p}
F^{(\ell_0)}(\reg+\regl(j,\ell_0),2,\tb/2,\Psipns)\notag\\
&\lesssim_{j,\ell_0,\reg} {}\tb^{-2(\ell_0-1)-2j-2+p}
F^{(\ell_0)}(\reg+\regl(j,\ell_0),2,\tb_0,\Psipns).
\end{align}
\end{subequations}
Here, we have utilized
\begin{align}
F^{(i)}(\reg,p,\tb,\Psipns)\lesssim \tilde{F}^{(i)}(\reg+\regl,p,\tb,\Psipns)\lesssim F^{(i)}(\reg+\regl,p,\tb,\Psipns),
\end{align}
which holds true by the Hardy's inequality \eqref{eq:HardyIneqLHSRHS} and rewriting $rV \PhipnsHigh{j}=\mu r^{-1}\PhipnsHigh{j+1}$ by Definition \ref{def:tildePhiplusandminusHigh}.

We then turn to equation \eqref{eq:tildePhipshighi} of $\tildePhipsHigh{\ell_0}$, which is
\begin{align}
\label{eq:tildePhipshighi:l0}
&-r^2 Y V  \tildePhipsHigh{\ell_0}
-(2\ell_0-1)(r-3M)r^{-2}\curlVR\tildePhipsHigh{\ell_0}
-6\ell_0^2Mr^{-1}\tildePhipsHigh{\ell_0}
+\sum_{j=1}^{\ell_0}h_{\ell_0j} \PhipsHigh{j}={}0,
\end{align}
and prove the $r^p$ estimate for this equation for $p\in (2,4)$. We only need to bound $W_{p-3}^{\reg-1}(\Donetwo^{R_0-M})$ norm square of the last term on the LHS. In view that all $h_{\ell_0j}$ are $O(r^{-1})$ functions, one can use a Hardy's inequality and find that this is in turn bounded by $\sum\limits_{j=1}^{\ell_0}
\norm{rV\PhipsHigh{j}}^2_{W_{p-5}^{\reg-1}(\Donetwo^{R_0-M})}
+\norm{\PhipsHigh{j}}^2_{W_{p-7}^{\reg}(\Donetwo^{R_0-M})}$.
Thus, we can take $R_0$ large enough such that these terms are absorbed by the LHS of the $r^p$ estimate, leading to
\begin{align}
\label{eq:EnerDecay:Impro:Psiplus:l0:2to4:Integralform}
\tilde{F}^{(\ell_0)}(\reg,p,\tb_2,\Psipns)
+\int_{\tb_1}^{\tb_2}\tilde{F}^{(\ell_0)}(\reg,p-1,\tb,\Psipns)\di\tb
\lesssim_{p,\reg} {}&\tilde{F}^{(\ell_0)}(\reg+\regl,p,\tb_1,\Psipns)
\end{align}
for any $p\in (2,4)$. With an application of Lemma \ref{lem:hierarchyImpliesDecay}, this yields that for $p\in [2,3-\delta)$,
\begin{align}
\label{eq:Integralformdecay:ell0:3-}
\tilde{F}^{(\ell_0)}(\reg,p,\tb_2,\Psipns)\lesssim{}&
\tb^{-3+\delta+p}\tilde{F}^{(\ell_0)}(\reg+\regl,3-\delta,\tb/2,\Psipns)
\notag\\
\lesssim {}& \tb^{-3+\delta+p}\tilde{F}^{(\ell_0)}(\reg+\regl,3-\delta,\tb_0,\Psipns),
\end{align}
and for $p\in [2,4-\delta]$,
\begin{align}
\label{eq:Integralformdecay:ell0:4-}
\tilde{F}^{(\ell_0)}(\reg,p,\tb_2,\Psipns)\lesssim{}&
\tb^{-4+\delta+p}\tilde{F}^{(\ell_0)}(\reg+\regl,4-\delta,\tb/2,\Psipns)
\notag\\
\lesssim {}& \tb^{-4+\delta+p}\tilde{F}^{(\ell_0)}(\reg+\regl,4-\delta,\tb_0,\Psipns),
\end{align}
The estimate \eqref{eq:Integralformdecay:ell0:3-} together with \eqref{eq:EnerDecay:Impro:Psiplus:l0:0to2} proves the estimate \eqref{eq:BEDC:Phiplus:l=1:0to3}.

In the end, we consider the $r^p$ estimates for $p\in [4,5)$.
The error term from the last term on the LHS of \eqref{eq:tildePhipshighi:l0} is bounded via the Cauchy--Schwarz inequality by
\begin{align}
\veps\int_{\tb_1}^{\tb_2}\frac{1}{\tb^{1+\delta}}
\norm{rV\tildePhipsHigh{\ell_0}}^2_{W_{p-2}^{\reg-1}(\Sigmatb^{R_0-M})}
\di\tb
+\frac{C}{\veps}\sum_{j=1}^{\ell_0}\int_{\tb_1}^{\tb_2}
\tb^{1+\delta}\norm{\PhipsHigh{j}}^2_{W_{p-6}^{\reg-1}(\Sigmatb^{R_0-M})}
\di\tb.
\end{align}
Again, the first part is absorbed after taking $\veps$ small enough and the second term is bounded by
$C\tb_1^{-6+2\delta+p}F^{(\ell_0)}
(\reg+\regl,4-\delta,\tb_1/2,\Psipns)$.  Thus,
\begin{align}
F^{(\ell_0)}(\reg+\regl(j,\ell_0),2,\tb,\Psipns)\lesssim{}&
\tb^{-3+\delta}F^{(\ell_0)}(\reg+\regl(j,\ell_0),5-\delta,\tb/2,\Psipns).
\end{align}
Finally, combining this with the estimate \eqref{eq:EnerDecay:Impro:Psiplus:l0:0to2} proves the estimate \eqref{eq:BEDC:Phiplus:l=1:0to5}.

The pointwise decay estimates \eqref{ptw:BEDC:Phiplus:l=1:0to3} and \eqref{ptw:BEDC:Phiplus:l=1:0to5} can be analogously obtained as proving the estimate \eqref{eq:ptwdecay:basic} in
Lemma \ref{lem:BED:0to2}. \qed

%%%%%%%%%%%%%%%
\subsection{Further energy decay and almost Price's law for $\ell\geq 2$ modes}
%%%%%%%%%%%%%%%

In this and the next subsections, we prove for the $\ell\geq 2$ modes and the $\ell=1$ mode respectively suitable elliptic-type estimates on $\Sigmatb$ to derive improved energy and pointwise decay than the ones in Proposition \ref{prop:BED:Psiplus:l=1}.  This idea can be traced back to the work \cite{angelopoulos2018vector} where the authors obtained the almost sharp decay for the scalar field by employing a degenerate type of elliptic estimate for the spherically symmetric, $\ell=0$ lowest mode.
We will explain the underlying reason why we consider the $\ell\geq 2$ modes and $\ell=1$ mode separately in the next subsection, cf. Remark \ref{rem:ell=1andellhigh:diff:ellip}. 

To state the elliptic-type estimates, we first define $\CDerit=\{\Lxi, \Delta^{\half}\prb, \edthR,\edthR'\}$.

\begin{prop}
\label{prop:degellip:Phips:ellbig2}
Assume the spin $\pm \half$ components are supported on $\ell\geq 2$ modes. Then,
\begin{align}
\label{eq:degellip:Phipns:ellbig2}
\hspace{4ex}&\hspace{-4ex}
\int_{\Sigmatb}r^{-3}\Big[\mu^{-\half}\big(
\abs{\edthR\edthR'\Phips}^2
+\abs{\Delta^{\half}\prb(\Delta^{\half}\prb\Phips )}^2
+\abs{\Delta^{\half}\prb\edthR'\Phips}^2)\notag\\
&\qquad \qquad
+\big(\abs{\edthR'\edthR\Psins}^2
+\abs{\Delta^{\half}\prb(\Delta^{\half}\prb\Psins )}^2
+\abs{\Delta^{\half}\prb\edthR\Psins}^2\big)
\Big]\di^3\mu\notag\\
\lesssim{}&\int_{\Sigmatb}\Big[r^{-3}\mu^{-\half}
(\abs{\Lxi^2\Phips}^2
+\abs{r^2\Lxi\prb\Phips}^2
+\abs{r\Lxi\Phips}^2)\notag\\
&\qquad +r^{-3}
(\abs{\Lxi^2\Psins}^2
+\abs{r^2\Lxi\prb\Psins}^2
+\abs{\Lxi\Psins}^2)\Big]\di^3\mu,
\end{align}
and for any $\reg\in \mathbb{N}$,
\begin{align}
\label{eq:degellip:Phipns:ellbig2:Highorder}
\hspace{6ex}&\hspace{-6ex}
\int_{\Sigmatb} r^{-3}\Big(
\absCDerit{\Psins}{\reg}^2
+\mu^{-\half}\absCDerit{\Phips}{\reg}^2
\Big)\di^3\mu\notag\\
\lesssim_{\reg}{} &
\norm{\Lxi\Psips}^2_{W_{-3}^{\reg+2}(\Sigmatb)}
+\norm{\Lxi\PhinsHigh{1}}^2_{W_{-3}^{\reg}(\Sigmatb^{\geq 4M})}
+\norm{\Lxi\Psins}^2_{W_{-3}^{\reg+1}(\Sigmatb)}.
\end{align}
\end{prop}

\begin{proof}
Let $H=2\mu^{-1}-\partial_r h(r)$, then
 one can express  $Y$ and $\VR$ as
\begin{align}
\label{def:vectorVRintermsofprb}
Y={}-\prb+(2\mu^{-1}-H)\Lxi,\quad
\VR={}\partial_{\rb}+H \Lxi.
\end{align}
By the choice of the hyperboloidal coordinates, there exist positive constants $c_0$ and $c_1$ such that
\begin{align}
\label{eq:propertyofHfunction}
\lim_{r\to \infty}r^2 H=c_0, \quad \text{and} \quad \abs{ H-2\mu^{-1}-c_1}\lesssim \mu  \quad
\text{as } r\to r_+.
\end{align}
The wave equation \eqref{eq:TME:varphis:posi} can thus be rewritten as
\begin{align}
\label{eq:ellip:varphis:posi}
\edthR\edthR'   \Phips
+\Delta^{\half}\prb(\Delta^{\half}\prb\Phips )={}&H_{\sfrak}(\Phips),
\end{align}
where
\begin{align}
H_{\sfrak}={}&
\Delta(2\mu^{-1}-H)H\Lxi^2
+2\Delta (\mu^{-1} -H) \Lxi\prb
+\Delta^{\half}\partial_r (\Delta^{\half}(2\mu^{-1}-H))
\Lxi.
\end{align}

Multiplying equation \eqref{eq:ellip:varphis:posi} by $-f_2 \overline{\Phips}$ and taking the real part gives
\begin{align}
\label{eq:ellip:f2trial}
\hspace{4ex}&\hspace{-4ex}
\prb(\Re(-f_2\overline{\Phips}\Delta\prb\Phips ))
+f_2\abs{\edthR' \Phips}^2+f_2 \Delta \abs{\prb\Phips}^2
+\partial_r (f_2 \Delta^{\half})
\Re(\overline{\Phips}\Delta^{\half}\prb\Phips)\notag\\
\equiv{}&-f_2\Re(H_{\sfrak}(\Phips)\overline{\Phips}).
\end{align}
We take $f_2=\mu^{-\half}r^{-3}$ and the above equation \eqref{eq:ellip:f2trial} becomes
\begin{align}
\label{eq:ellip:f2trial:1:1}
\hspace{4ex}&\hspace{-4ex}
\prb(-r^{-2}\Delta^{\half}\Re(\overline{\Phips}\prb\Phips) )
+r^{-2}(\Delta^{-\half}\abs{\edthR' \Phips}^2+\Delta^{\half} \abs{\prb\Phips}^2)
-2r^{-3}\Delta^{\half}\Re(\overline{\Phips}\prb\Phips)\notag\\
\equiv{}&-r^{-2}\Delta^{-\half}
\Re(H_{\sfrak}(\Phips)\overline{\Phips}).
\end{align}

If the spin $\pm \half$ components are supported on $\ell\geq 2$ modes, then we have from equation \eqref{eq:ellip:highermodes}
 that
\begin{align}
\label{eq:ellipsphe:ellgeq2:deg}
\int_{S^2}\abs{\edthR' \Phips}^2\di^2\mu\geq 4\int_{S^2}\abs{ \Phips}^2\di^2\mu,
\end{align}
and the last term in the first line of equality \eqref{eq:ellip:f2trial:1:1} is dominated by the middle term in the first line by Cauchy--Schwarz inequality. As a result, by integrating over $\Sigmatb$, this yields
\begin{align}
\int_{\Sigmatb}r^{-3}\mu^{-\half}(\abs{\edthR' \Phips}^2+\mu \abs{r\prb\Phips}^2)\di^3\mu
\lesssim{}&\int_{\Sigmatb}\big|r^{-3}\mu^{-\half}
\Re(H_{\sfrak}(\Phips)\overline{\Phips})\big|\di^3\mu,
\end{align}
and hence,
\begin{align}\label{eq:degellip:Phips:ellbig2:1}
\int_{\Sigmatb}r^{-3}\mu^{-\half}(\abs{\edthR' \Phips}^2+\mu \abs{r\prb\Phips}^2)\di^3\mu
\lesssim{}&\int_{\Sigmatb}r^{-3}\mu^{-\half}\abs{H_{\sfrak}(\Phips)}^2\di^3\mu.
\end{align}

We take a square of both sides of \eqref{eq:ellip:varphis:posi}, multiply by $\mu^{-\half}r^{-3}$, integrate over $\Sigmatb$ and arrive at
\begin{align}
\hspace{4ex}&\hspace{-4ex}
\int_{\Sigmatb}r^{-3}\mu^{-\half}\left(\abs{\edthR\edthR'\Phips}^2
+\abs{\Delta^{\half}\prb(\Delta^{\half}\prb\Phips )}^2
+2\Re(\Delta^{\half}\prb(\Delta^{\half}\prb\overline{\Phips} )\edthR\edthR'\Phips)\right)\di^3\mu\notag\\
={}&\int_{\Sigmatb}r^{-3}\mu^{-\half}\abs{H_{\sfrak}(\Phips)}^2\di^3\mu.
\end{align}
For the third term on the LHS, it equals after applying integration by parts \begin{align}
\label{eq:balabala:ellbig2}
\hspace{4ex}&\hspace{-4ex}
\int_{\Sigmatb}r^{-3}\mu^{-\half} 2\Re(\Delta^{\half}\prb(\Delta^{\half}\prb\overline{\Phips} )\edthR\edthR'\Phips)\di^3\mu\notag\\
={}&
\int_{\Sigmatb}\left[\prb\left(-2r^{-2}\Delta^{\half}
\Re\Big(\prb\edthR'\Phips\overline{\edthR'\Phips}\Big)\right)
+2r^{-3}\mu^{\half}\abs{r\prb\edthR'\Phips}^2
-4r^{-3}\mu^{\half} \Re\Big(r\prb\edthR'\Phips\overline{\edthR'\Phips}\Big)\right]\di^3\mu.
\end{align}
The integral of the total derivative $\prb$ part vanishes, and we combine the above two equalities together. Since $\edthR'\Phips$ has spin-weight $-\half$, we have from equation \eqref{eq:ellip:highermodes} that $\int_{\Sphere}\abs{\edthR\edthR'\Phips}^2\di^2\mu\geq \int_{\Sphere}4\abs{\edthR'\Phips}^2\di^2\mu$, hence the last term on the RHS of \eqref{eq:balabala:ellbig2} can be dominated by the other terms, and we obtain
\begin{align}
\label{eq:degellip:Phips:ellbig2:2}
\int_{\Sigmatb}r^{-3}\mu^{-\half}\left(\abs{\edthR\edthR'\Phips}^2
+\abs{\Delta^{\half}\prb(\Delta^{\half}\prb\Phips )}^2
+\mu\abs{r\prb\edthR'\Phips}^2
\right)\di^3\mu
\lesssim{}&\int_{\Sigmatb}r^{-3}\mu^{-\half}\abs{H_{\sfrak}(\Phips)}^2\di^3\mu.
\end{align}
Combining inequalities \eqref{eq:degellip:Phips:ellbig2:1} and \eqref{eq:degellip:Phips:ellbig2:2} together and taking into account of the following estimate
\begin{align}
\int_{\Sigmatb}r^{-3}\mu^{-\half}\abs{H_{\sfrak}(\Phips)}^2\di^3\mu
\lesssim{}&\int_{\Sigmatb}r^{-3}\mu^{-\half}(\abs{\Lxi^2\Phips}^2
+\abs{r^2\Lxi\prb\Phips}^2
+\abs{r\Lxi\Phips}^2)\di^3\mu,
\end{align}
we conclude an estimate
\begin{align}
\label{eq:degellip:Phips:ellbig2}
\hspace{4ex}&\hspace{-4ex}
\int_{\Sigmatb}r^{-3}\mu^{-\half}\Big(\abs{\edthR' \Phips}^2+\mu \abs{r\prb\Phips}^2
+\abs{\edthR\edthR'\Phips}^2
+\abs{\Delta^{\half}\prb(\Delta^{\half}\prb\Phips )}^2
+\mu\abs{r\prb\edthR'\Phips}^2
\Big)\di^3\mu\notag\\
\lesssim{}&\int_{\Sigmatb}r^{-3}\mu^{-\half}(\abs{\Lxi^2\Phips}^2
+\abs{r^2\Lxi\prb\Phips}^2
+\abs{r\Lxi\Phips}^2)\di^3\mu.
\end{align}

For the spin $-\half$ component, equation \eqref{eq:TME:varphis:nega} can be written as
\begin{align}
\label{eq:ellip:varphis:nega}
\edthR'\edthR   \Phins
+\Delta^{\half}\prb(\Delta^{\half}\prb\Phins )={}&H_{-\sfrak}(\Phins),
\end{align}
where
\begin{align}
H_{-\sfrak}={}&
\Delta(2\mu^{-1}-H)H\Lxi^2
+2\Delta (\mu^{-1} -H) \Lxi\prb
-\Delta^{\half}\partial_r (\Delta^{\half}H)
\Lxi.
\end{align}
In particular,  in terms of the regular scalar $\Psins$, one finds
\begin{align}
\label{expression:Hns}
\mu^{-\half}H_{-\sfrak}(\Phins)={}&
\Delta (2\mu^{-1}-H)H \Lxi^2\Psins
+2r^2 (1-\mu H) \Lxi \prb\Psins\notag\\
&
+[M(2\mu^{-1}-H) -r \mu H
-\partial_r (\Delta H)]\Lxi\Psins.
\end{align}
Equation \eqref{eq:ellip:varphis:nega} is exactly in the same form as equation \eqref{eq:ellip:varphis:posi}, hence the same form of  \eqref{eq:ellip:f2trial} holds.
Then,  by taking $f_2=-r^{-3}\mu^{-1}$ and writing down all $\Phi_{-\sfrak}$ terms in terms of $\Psins$  using $\Phi_{-\sfrak}=\mu^{\half}\Psi_{-\sfrak}$,  we obtain
\begin{align}
\label{eq:ellip:f2trial:1:1:nega}
\hspace{4ex}&\hspace{-4ex}
r^{-3}(\abs{\edthR \Psins}^2+r \abs{\prb(\mu^{\half}\Psins)}^2)
+\partial_r(r^{-2}\mu^{-\half})
\mu r \Re(\overline{\Psins}\prb(\mu^{\half}\Psins))\notag\\
&+\prb(-r^{-1}\mu^{\half}\Re(\overline{\Psins}\prb(\mu^{\half}\Psins)) )\equiv-r^{-3}\mu^{-\half}
\Re(H_{-\sfrak}(\Phins)\overline{\Psins}).
\end{align}
Expanding out the LHS of \eqref{eq:ellip:f2trial:1:1:nega}, one finds the first line equals
\begin{align}
\hspace{4ex}&\hspace{-4ex}
r^{-3}(\abs{\edthR \Psins}^2-2\mu \abs{\Psins}^2)
+r^{-1} \abs{\mu^{\half}\prb\Psins}^2
-\prb(r^{-2}\mu \abs{\Psins}^2)
-r^{-1}\mu^{\half} \partial_r (\mu^{\half})\Re(\prb\Psins\overline{\Psins})\notag\\
={}&\frac{1}{r^3}\Big(\abs{\edthR \Psins}^2-\Big(\frac{3M}{2r}+2\mu\Big) \abs{\Psins}^2\Big)
+\frac{1}{r} \abs{\mu^{\half}\prb\Psins}^2
-\prb\Big(\half r^{-1}\mu^{\half}\partial_r(\mu^{\half})\abs{\Psins}^2
+r^{-2}\mu \abs{\Psins}^2\Big),
\end{align}
and the second line on the LHS is
\begin{align}
\prb(-r^{-1}\mu^{\half}\partial_r(\mu^{\half})\abs{\Psins}^2)
+\prb(-r^{-1}\mu \Re(\overline{\Psins}\prb\Psins))
\end{align}
Therefore, equation \eqref{eq:ellip:f2trial:1:1:nega} becomes
\begin{align}
&{}-\partial_\rho\Big(\frac{3M}{2r^3}|\Psins|^2+r^{-2}\mu|\Psins|^2\Big)
+r^{-3}\Big(|\edthR \Psins|^2
-2\mu|\Psins|^2-\frac{3M}{2r}|\Psins|^2
+|\Delta^{\half}\partial_\rho\Psins|^2\Big)\notag\\
&\equiv-r^{-3}\mu^{-\half}
\Re(H_{-\sfrak}(\Phins)\overline{\Psins}).
\end{align}
By integrating over $\Sigmatb$, the total derivative part equals $\frac{3}{2} Mr^{-3}\abs{\Psins}^2\vert_{r=2M}$, and this yields
\begin{align}\label{eq:degellip:Phins:ellbig2:1}
\int_{\Sigmatb}[r^{-3}(\abs{\edthR' \Psins}^2+ \abs{\Delta^{\half}\prb\Psins}^2)
]\di^3\mu
\lesssim{}&\int_{\Sigmatb}r^{-3}\abs{\mu^{-\half}H_{-\sfrak}(\Phins)}^2\di^3\mu.
\end{align}

In addition, we can take a square of both sides of \eqref{eq:ellip:varphis:nega}, multiply by $r^{-3}\mu^{-1}$, integrate over $\Sigmatb$, and arrive at
\begin{align}\label{integal:equ:Phins}
\hspace{4ex}&\hspace{-4ex}
\int_{\Sigmatb}r^{-3}\mu^{-1}\left(\abs{\edthR'\edthR\Phins}^2
+\abs{\Delta^{\half}\prb(\Delta^{\half}\prb\Phins )}^2
+2\Re(\Delta^{\half}\prb(\Delta^{\half}\prb\overline{\Phins} )\edthR'\edthR\Phins)\right)\di^3\mu\notag\\
={}&\int_{\Sigmatb}r^{-3}\abs{\mu^{-\half}H_{-\sfrak}(\Phins)}^2\di^3\mu.
\end{align}
For the third term on the LHS, it equals after applying integration by parts \begin{align}
\label{eq:balabala:ellbig2:new}
\hspace{4ex}&\hspace{-4ex}
\int_{\Sigmatb}r^{-3}\mu^{-1} 2\Re(\Delta^{\half}\prb(\Delta^{\half}\prb\overline{\Phins} )\edthR'\edthR\Phins)\di^3\mu\notag\\
={}&\int_{\Sigmatb}2r^{-2} \Re(\prb(\mu r\prb\overline{\Psins} +{M}{r}^{-1}\overline{\Psins})\edthR'\edthR\Psins)\di^3\mu\notag\\
={}&\int_{\Sigmatb}\Big[-\prb(2r^{-1}\mu\Re(\prb\edthR\Psins\overline{\edthR\Psins})
+Mr^{-3}|\edthR\Psins|^2)\notag\\
&\qquad+
2\mu r^{-1}|\prb\edthR\Psins|^2
-Mr^{-4}|\edthR\Psins|^2
-\frac{4\mu}{r^2}\Re(\prb\edthR\Psins\overline{\edthR\Psins})\Big]\di^3\mu.
\end{align}
The integral of the total derivative $\prb$ part in above equation is equals to $Mr^{-3}|\edthR\Psins|^2|_{r=2M}$, and we combine the above two equalities together. Note that $\int_{\Sphere}\abs{\edthR'\edthR\Phins}^2\di^2\mu\geq \int_{\Sphere}4\abs{\edthR\Phins}^2\di^2\mu$, hence the last term on the RHS of \eqref{eq:balabala:ellbig2:new} can be dominated by the other terms, and we obtain a similar estimate as \eqref{eq:degellip:Phips:ellbig2:2}:
\begin{align}
\label{eq:degellip:Phins:ellbig2:2:v1}
\int_{\Sigmatb}\frac{1}{r^3}\left(\abs{\edthR'\edthR\Psins}^2
+\abs{r\prb(\Delta^{\half}\prb\Phins )}^2
+\mu\abs{r\prb\edthR\Psins}^2
\right)\di^3\mu
\lesssim{}&\int_{\Sigmatb}\frac{1}{r^3}\abs{\mu^{-\half}H_{-\sfrak}(\Phins)}^2\di^3\mu.
\end{align}
For the second term on the LHS of \eqref{eq:degellip:Phins:ellbig2:2:v1}, one can expand it out, apply integration by parts for the product term,  and obtain
\begin{align}
\hspace{2ex}&\hspace{-2ex}
\int_{\Sigmatb}\frac{1}{r^3}\abs{r\prb(\Delta^{\half}\prb\Phins )}^2\di^3\mu\notag\\
={}&\int_{\Sigmatb}\Big[
\frac{1}{r^{3}}|\Delta^{\half}\prb(\Delta^{\half}\prb\Psins)|^2+\frac{8M}{r^4}|\Delta^{\half}
\prb\Psins|^2+
\frac{4M^2}{r^3}|\prb\Psins|^2\notag\\
&\qquad -\prb(M^2r^{-4}|\Psins|^2)-(3M^2r^{-5}-\prb(3M\mu r^{-3}))|\Psins|^2\Big]\di^3\mu.
\end{align}
The above two estimates together with \eqref{eq:degellip:Phins:ellbig2:1} yield
\begin{align}
\label{eq:degellip:Phins:ellbig2:2}
\hspace{4ex}&\hspace{-4ex}\int_{\Sigmatb}\frac{1}{r^3}\left(\abs{\edthR'\edthR\Psins}^2
+\abs{\Delta^{\half}\prb(\Delta^{\half}\prb\Psins )}^2
+\mu\abs{r\prb\edthR\Psins}^2
+\abs{\prb\Psins}^2
\right)\di^3\mu\notag\\
\lesssim{}&\int_{\Sigmatb}\frac{1}{r^3}\abs{\mu^{-\half}H_{-\sfrak}(\Phins)}^2\di^3\mu.
\end{align}
From \eqref{expression:Hns}, we have
\begin{align}
\int_{\Sigmatb}r^{-3}\abs{\mu^{-\half}H_{-\sfrak}(\Phins)}^2\di^3\mu
\lesssim{}&\int_{\Sigmatb}r^{-3}(\abs{\Lxi^2\Psins}^2
+\abs{r^2\Lxi\prb\Psins}^2
+\abs{\Lxi\Psins}^2)\di^3\mu,
\end{align}
thus
it holds that
\begin{align}
\label{eq:degellip:Phins:ellbig2}
\hspace{4ex}&\hspace{-4ex}
\int_{\Sigmatb}r^{-3}\Big(
\abs{\edthR'\edthR\Psins}^2
+\abs{\Delta^{\half}\prb(\Delta^{\half}\prb\Psins )}^2
+\mu\abs{r\prb\edthR\Psins}^2
+\abs{r\prb\Psins}^2
\Big)\di^3\mu\notag\\
\lesssim{}&\int_{\Sigmatb}r^{-3}(\abs{\Lxi^2\Psins}^2
+\abs{r^2\Lxi\prb\Psins}^2
+\abs{\Lxi\Psins}^2)\di^3\mu.
\end{align}
This estimate and the inequality \eqref{eq:degellip:Phips:ellbig2} together prove the estimate \eqref{eq:degellip:Phipns:ellbig2}. Moreover, by using the expression of $\PhinsHigh{1}$ in Definition \ref{def:tildePhiplusandminusHigh}, the RHS of \eqref{eq:degellip:Phins:ellbig2} is further bounded by $\norm{\Lxi\PhinsHigh{1}}^2_{W_{-3}^{0}(\Sigmatb^{\geq 4M})}
+\norm{\Lxi\Psins}^2_{W_{-3}^{1}(\Sigmatb)}$.

By commuting with $\Lxi$, $\edthR$, $\edthR'$ and $\Delta^{\half}\prb$, and by the above process for the spin $-\half$ component, one can obtain
\begin{align}
\label{eq:degellip:Phipns:ellbig2:Highorder:v1}
\hspace{6ex}&\hspace{-6ex}
\int_{\Sigmatb} r^{-3}\Big[\big(
\absCDerit{\Delta^{\half}\prb (\Delta^{\half}\prb\Psins)}{\reg}^2
+\absCDerit{\edthR'\edthR\Psins}{\reg}^2
+\mu\absCDerit{r\prb\edthR\Psins}{\reg}^2\big)\notag\\
\hspace{4ex}&\hspace{-4ex}\qquad \qquad
+\mu^{-\half}\big(\absCDerit{\Delta^{\half}\prb (\Delta^{\half}\prb \Phips)}{\reg}^2
+\absCDerit{\edthR\edthR'\Phips}{\reg}^2
+\mu\absCDerit{r\prb\edthR'\Phips}{\reg}^2\big)
\Big]\di^3\mu\notag\\
\lesssim_{\reg}{} &
\norm{\Lxi\Psips}^2_{W_{-3}^{\reg+2}(\Sigmatb)}
+\norm{\Lxi\PhinsHigh{1}}^2_{W_{-3}^{\reg}(\Sigmatb^{\geq 4M})}
+\norm{\Lxi\Psins}^2_{W_{-3}^{\reg+1}(\Sigmatb)}.
\end{align}
Note that on the RHS of \eqref{eq:degellip:Phipns:ellbig2:Highorder:v1}, one more regularity is needed for the $\Phips$ term compared to the $\Phins$ term since a Hardy's inequality is utilized to control the $\mu^{-\half}$ factor on the RHS of \eqref{eq:degellip:Phipns:ellbig2}. The estimate \eqref{eq:degellip:Phipns:ellbig2:Highorder} is manifest from \eqref{eq:degellip:Phipns:ellbig2:Highorder:v1}.
\end{proof}

\begin{remark}
\label{rem:elliphyper:ellgeq2:interior}
In this proof, we crucially rely on the elliptic estimate \eqref{eq:ellipsphe:ellgeq2:deg} on the sphere to dominate the last term by the other terms in the first line of equation \eqref{eq:ellip:f2trial}. For a higher $\ell$ mode, the estimate \eqref{eq:ellipsphe:ellgeq2:deg} can be improved to $\int_{S^2}\abs{\edthR' \Phips}^2\di^2\mu\geq \ell^2\int_{S^2}\abs{ \Phips}^2\di^2\mu$, and this yields that function $f_2$ can be chosen to decay even faster as $r\to +\infty$. With this observation, one can in fact derive faster decay in $\tau$ for a higher $\ell$ mode by allowing for an $r$ growth in the interior region $\{r\leq \tau\}$ and achieve almost sharp decay in the interior region. We will however not discuss this issue in depth in this work.
\end{remark}

Let us consider now the case that the $\ell_0$-th N--P constant of the $\ell_0$ mode does not vanish, with $\ell_0$ being the lowest mode of the spin $\pm \half$  components which does not vanish. By integrating the inequality \eqref{eq:degellip:Phipns:ellbig2:Highorder} over $[\tb,\infty)$, the RHS is bounded by $F^{(1)}(\reg+\regl,0,\tb,\Lxi\Psipns)$, which is in turn bounded by $\tb^{-3-2\ell_0+\delta}F^{(\ell_0)}(\reg+\regl(\ell_0),3-\delta,\tb_0,\Psipns)$ from Proposition \ref{prop:BED:Psiplus:l=1}. Hence, by making use of the inequality \eqref{eq:Sobolev:3}, we obtain for any $r\geq 2M$,
\begin{align}
\absCDerit{r^{-1}\Phips}{\reg}
+\absCDerit{r^{-1}\Psins}{\reg}
\lesssim_{\reg,\ell_0,\delta}{}&\tb^{-2-\ell_0+\delta/2}(F^{(\ell_0)}(\reg+\regl(\ell_0),3-\delta,\tb_0,\Psipns))^{\half}.
\end{align}
For the other case that the $\ell_0$-th N--P constant of the $\ell_0$ mode vanishes, one can argue in a similar way.
In the end, we combine these pointwise estimates with the ones in Proposition \ref{prop:BED:Psiplus:l=1} and conclude the following pointwise decay estimates.
\begin{prop}
\label{prop:almostpricelaw:ellgeq2}
Let the lowest mode of the spin $\pm \half$ components be the $\ell_0$ mode with $\ell_0\geq 2$.
\begin{itemize}
  \item If the $\ell_0$-th N--P constant of the $\ell_0$ mode does not vanish, then  for any $\reg\in \mathbb{N}$ and $j\in\mathbb{N}$, there exists $\regl=\regl(j,\ell_0)>0$ such that for any $\tb\geq\tb_0$ and any $1< p\leq 3-\delta$,
      \begin{subequations}
\begin{align}
\absCDerit{ r^{-1}\Lxi^j\Phips}{\reg}
\lesssim_{\reg, j,p,\ell_0}{}&v^{-2}\tb^{-(1+p)/2-(\ell_0-2)-j}
(F^{(\ell_0)}(\reg+\regl,p,\tb_0,\Psipns))^{\half},\\
\absCDerit{r^{-1}\Lxi^j\Psins}{\reg}
\lesssim_{\reg, j,p,\ell_0}{}&v^{-1}\tb^{-(1+p)/2-(\ell_0-1)-j}
(F^{(\ell_0)}(\reg+\regl,p,\tb_0,\Psipns))^{\half};
\end{align}
\end{subequations}
  \item If the $\ell_0$-th N--P constant of the $\ell_0$ mode vanishes, then  for any $\reg\in \mathbb{N}$ and $j\in\mathbb{N}$, there exists $\regl=\regl(j,\reg)>0$ such that for any $\tb\geq\tb_0$ and any $1< p\leq 5-\delta$,
 \begin{subequations}
\begin{align}
\absCDerit{ r^{-1}\Lxi^j\Phips}{\reg}
\lesssim_{\reg, j,p}{}&v^{-2}\tb^{-(1+p)/2-(\ell_0-2)-j}
(F^{(\ell_0)}(\reg+\regl,p,\tb_0,\Psipns))^{\half},\\
\absCDerit{r^{-1}\Lxi^j\Psins}{\reg}
\lesssim_{\reg, j,p}{}&v^{-1}\tb^{-(1+p)/2-(\ell_0-1)-j}
(F^{(\ell_0)}(\reg+\regl,p,\tb_0,\Psipns))^{\half}.
\end{align}
\end{subequations}
\end{itemize}
\end{prop}

%%%%%%%%%%%%%%%
\subsection{Further energy decay and almost Price's law for $\ell=1$ mode}
\label{sect:almostPricelaw:ell=1}
%%%%%%%%%%%%%%%%%%

In a similar way, we prove an elliptic-type energy estimate for the $\ell=1$ mode. 

\begin{prop}
\label{prop:degellip:Phips:ell=2:new}
Assume the spin $\pm \half$ components are supported on $\ell=1$ mode. Then
for any $\reg\in\mathbb{N}$,
\begin{align}
\label{eq:degellip:Phipns:ell=1:Highorder}
\hspace{6ex}&\hspace{-6ex}
\int_{\Sigmatb} r^{-2}\Big(
\absCDerit{\Psins}{\reg+1}^2
+\mu^{-\half}\absCDerit{\Phips}{\reg+1}^2
\Big)\di^3\mu\notag\\
\lesssim_{\reg}{} &
\norm{\Lxi\Psips}^2_{W_{-2}^{\reg+2}(\Sigmatb)}
+\norm{\Lxi\PhinsHigh{1}}^2_{W_{-2}^{\reg}(\Sigmatb^{\geq 4M})}
+\norm{\Lxi\Psins}^2_{W_{-2}^{\reg+1}(\Sigmatb)}.
\end{align}
\end{prop}

\begin{proof}
If the spin $\pm \half$ components are supported on $\ell=1$ mode, then equality \eqref{eq:ellip:f2trial} becomes
\begin{align}
\label{eq:ellip:f2trial:ell=1}
\hspace{4ex}&\hspace{-4ex}
\prb(\Re(-f_2\overline{\Phips}\Delta\prb\Phips ))
+f_2\abs{ \Phips}^2+f_2 \Delta \abs{\prb\Phips}^2
+\partial_r (f_2 \Delta^{\half})
\Re(\overline{\Phips}\Delta^{\half}\prb\Phips)\notag\\
\equiv{}&-f_2\Re(H_{\sfrak}(\Phips)\overline{\Phips}).
\end{align}
By take $f_2=\mu^{-\half} r^{-2}$, the above equation \eqref{eq:ellip:f2trial:ell=1} becomes
\begin{align}
\label{eq:ellip:f2trial:1}
\hspace{4ex}&\hspace{-4ex}
\prb(-r^{-1}\Delta^{\half}\Re(\overline{\Phips}\prb\Phips) )
+r^{-1}(\Delta^{-\half}\abs{ \Phips}^2+\Delta^{\half} \abs{\prb\Phips}^2)
-r^{-2}\Delta^{\half}\Re(\overline{\Phips}\prb\Phips)\notag\\
\equiv{}&-r^{-1}\Delta^{-\half}
\Re(H_{\sfrak}(\Phips)\overline{\Phips}).
\end{align}
After integrating over $\Sigmatb$, the first term on the LHS vanishes, and on the LHS, the middle term dominates over the third term by Cauchy--Schwarz inequality, thus we arrive at
\begin{align}
\label{eq:ellip:l=1:ps:1}
\int_{\Sigmatb}r^{-2}(\mu^{-\half}\abs{ \Phips}^2+\mu^{\half} \abs{r\prb\Phips}^2)\di^3 \mu
\lesssim{}&\int_{\Sigmatb} \abs{r^{-2}\mu^{-\half}
\Re(H_{\sfrak}(\Phips)\overline{\Phips})}\di^3\mu.
\end{align}
By using the Cauchy--Schwarz inequality for the RHS of this estimate and in view of equation \eqref{eq:ellip:varphis:posi},  one achieves
\begin{align}
\label{eq:degellip:Phips:ell=1}
\hspace{4ex}&\hspace{-4ex}
\int_{\Sigmatb}r^{-2}\mu^{-\half}\left(\abs{\Phips}^2
+\abs{\Delta^{\half}\prb(\Delta^{\half}\prb\Phips )}^2
+\mu\abs{r\prb\Phips}^2\right)\di^3\mu\notag\\
\lesssim{}&\int_{\Sigmatb}r^{-2}\mu^{-\half}\abs{H_{\sfrak}(\Phips)}^2\di^3\mu\notag\\
\lesssim{}&\int_{\Sigmatb}r^{-2}\mu^{-\half}(\abs{\Lxi^2\Phips}^2
+\abs{r^2\Lxi\prb\Phips}^2
+\abs{r\Lxi\Phips}^2)\di^3\mu.
\end{align}
Similarly as in the proof of Proposition \ref{prop:degellip:Phips:ellbig2}, we have for the spin $-\half$ component that
\begin{align}
\label{eq:degellip:Phins:ell=1}
\hspace{4ex}&\hspace{-4ex}
\int_{\Sigmatb}r^{-2}\Big(
\abs{\Psins}^2
+\abs{\Delta^{\half}\prb(\Delta^{\half}\prb\Psins )}^2
+\mu\abs{r\prb\Psins}^2
\Big)\di^3\mu\notag\\
\lesssim{}&\int_{\Sigmatb}r^{-2}\mu^{-1}\abs{H_{-\sfrak}(\Phins)}^2\di^3\mu\notag\\
\lesssim{}&\int_{\Sigmatb}r^{-2}(\abs{\Lxi^2\Psins}^2
+\abs{r^2\Lxi\prb\Psins}^2
+\abs{r^{-1}\Lxi\Psins}^2)\di^3\mu,
\end{align}
where in the last step we have used the expression \eqref{expression:Hns}.
The essentially same proof of Proposition \ref{prop:degellip:Phips:ellbig2} but with $f_2=-\mu^{-1}r^{-2}$ then yields
\begin{align}
\label{eq:degellip:Phipns:ell=1}
\hspace{4ex}&\hspace{-4ex}
\int_{\Sigmatb}r^{-2}\Big[\mu^{-\half}\big(\abs{ \Phips}^2
+\mu \abs{r\prb\Phips}^2
+\abs{\Delta^{\half}\prb(\Delta^{\half}\prb\Phips )}^2)
\notag\\
&\qquad \qquad
+\big(\abs{\Psins}^2
+\abs{\Delta^{\half}\prb(\Delta^{\half}\prb\Psins )}^2
+\mu\abs{r\prb\Psins}^2
\big)\Big]\di^3\mu\notag\\
\lesssim{}&\int_{\Sigmatb}r^{-2}\Big[\mu^{-\half}(\abs{\Lxi^2\Phips}^2
+\abs{r^2\Lxi\prb\Phips}^2
+\abs{r\Lxi\Phips}^2)\notag\\
&\qquad\qquad
+(\abs{\Lxi^2\Psins}^2
+\abs{r^2\Lxi\prb\Psins}^2
+\abs{r^{-1}\Lxi\Psins}^2)\Big]\di^3\mu,
\end{align}
and thus the estimate \eqref{eq:degellip:Phipns:ell=1:Highorder} for any $\reg\in \mathbb{N}$.
\end{proof}

\begin{remark}
\label{rem:ell=1andellhigh:diff:ellip}
If we choose $f_2=\mu^{-\half}r^{-3}$ in equation \eqref{eq:ellip:f2trial:ell=1} as in the $\ell\geq 2$ case, the last three terms in the first line equal
$$
\mu^{-\half}r^{-3}\abs{ \Phips}^2+\mu^{\half}r^{-3}\abs{r\prb\Phips}^2
-2r^{-3}
\Re(\overline{\Phips}\Delta^{\half}\prb\Phips)=\mu^{-\half}r^{-3}(1-\mu)
\abs{ \Phips}^2+\mu^{\half}r\abs{\prb (r^{-1}\Phips)}^2.
$$
After integrating over $\Sigmatb$, we obtain an analogous estimate to the estimate \eqref{eq:ellip:l=1:ps:1}
\begin{align}
\int_{\Sigmatb}\Big(\mu^{-\half}r^{-3}(1-\mu)
\abs{ \Phips}^2+\mu^{\half}r\abs{\prb (r^{-1}\Phips)}^2\Big)\di^3 \mu
\lesssim{}&\int_{\Sigmatb} \abs{r^{-3}\mu^{-\half}
\Re(H_{\sfrak}(\Phips)\overline{\Phips})}\di^3\mu.
\end{align}
By applying the Cauchy--Schwarz to the RHS, this does not provide a desired elliptic estimate because of the degenerate factor $1-\mu$ near infinity, thus the same choice of function $f_2$ as in the $\ell\geq 2$ case does not work anymore in the $\ell=1$ case, and this is the main reason that we treat the $\ell=1$ case and the $\ell\geq 2$ case separately.
\end{remark}

Let us consider the case that the first N--P constant of the $\ell=1$ mode does not vanish. The RHS of \eqref{eq:degellip:Phipns:ell=1:Highorder} is bounded by $F^{(1)}(\reg+\regl,0,\tb,\Lxi\Psipns)$, which is in turn bounded by $\tb^{-5+\delta}F^{(1)}(\reg+\regl,3-\delta,\tb_0,\Psipns)$ from Proposition \ref{prop:BED:Psiplus:l=1}. Hence, by making use of the inequality \eqref{eq:Sobolev:1},  we obtain for any $r\geq r_+$,
\begin{align}
\label{almostPrice:posi:weak:NPnotvanish}
r^{\half}(\absCDerit{r^{-1}\Phips}{\reg}
+\absCDerit{ r^{-1}\Psins}{\reg})
\lesssim_{\delta,\reg}{}&\tb^{-5/2+\delta/2}(F^{(1)}(\reg+\regl,3-\delta,\tb_0,\Psipns))^{\half}.
\end{align}
For the other case that the first N--P constant of the $\ell=1$ mode vanishes, we can similarly obtain
\begin{align}
\absCDerit{ r^{-1}\Lxi^j\Phips}{\reg}
+\absCDerit{r^{-1}\Lxi^j\Psins}{\reg}
\lesssim_{\reg, j,\delta}{}&r^{-\half}\tb^{-7/2-j+\delta/2}(F^{(1)}(\reg+\regl,5-\delta,\tb_0,\Psipns))^{\half}.
\end{align}

We shall now further improve these pointwise estimates. We consider first the case that the first N--P constant does not vanish. Let us focus on the interior region where $\{\rb\leq\tb\}$. In this case, the following pointwise decay estimates hold
\begin{align}
\absCDerit{ r^{-1}\Lxi^j\Phips}{\reg}\lesssim {} v^{-2}r^{-\half}\tb^{-\half-j+\delta/2}
(F^{(1)}(\reg+\regl,3-\delta,\tb_0,\Psipns))^{\half}.
\end{align}
The wave equation \eqref{eq:ellip:varphis:posi} simplifies to
\begin{align}
-  \Phips
+\Delta^{\half}\prb(\Delta^{\half}\prb\Phips )={}&H_{\sfrak}(\Phips).
\end{align}
For $\varphi_{\sfrak}=(r-M)^{-1} \Phips$, the above equation reduces to
\begin{align}
\label{eq:prbprbPhins:ell=1:impro}
(r-M)^{-1}\Delta^{\half}\prb((r-M)^2\Delta^{\half}\prb\varphi_{\sfrak})
=H_{\sfrak}(\Phips).
\end{align}
Since
\begin{align}
\label{eq:sourceterm:weak:ell1:improve}
\abs{H_{\sfrak}(\Phips)}
+\abs{\rho\prb(H_{\sfrak}(\Phips))}\lesssim_{\delta} r^2 v^{-2}r^{-\half}\tb^{-\frac{3}{2}+\delta/2}
(F^{(1)}(\regl,3-\delta,\tb_0,\Psipns))^{\half},
\end{align}
one can integrate the above equation from horizon to obtain
\begin{align}
\abs{(r-M)^2\Delta^{\half}\prb\varphi_{\sfrak}}
\lesssim_{\delta} {}&\Delta^{\half} r^2 v^{-2}r^{-\half}\tb^{-\frac{3}{2}+\delta/2}
(F^{(1)}(\regl,3-\delta,\tb_0,\Psipns))^{\half},
\end{align}
that is,
\begin{align}
\abs{\prb\varphi_{\sfrak}}\lesssim_{\delta} {}&
v^{-2}r^{-\half}\tb^{-\frac{3}{2}+\delta/2}
(F^{(1)}(\regl,3-\delta,\tb_0,\Psipns))^{\half},
\end{align}
Integrating now from $\rb=\tb$ then gives that in the interior region
\begin{align}
\abs{\Lxi^j \varphi_{\sfrak}}\lesssim_{\delta}  {} v^{-2}\tb^{-1-j+\delta/2}
(F^{(1)}(\regl,3-\delta,\tb_0,\Psipns))^{\half}.
\end{align}
One substitutes this back into \eqref{eq:prbprbPhins:ell=1:impro} and finds $\abs{H_{\sfrak}(\Phips)}\lesssim_{\delta}  r^2 v^{-2}\tb^{-2+\delta/2}
(F^{(1)}(\regl,3-\delta,\tb_0,\Psipns))^{\half}$, thus applying the above discussions again gives an improved estimate for the radial derivative
\begin{align}
\abs{(r-M)^2\Delta^{\half}\prb\varphi_{\sfrak}}
\lesssim_{\delta} {}&\Delta^{\half} r^2 v^{-2}\tb^{-2+\delta/2}
(F^{(1)}(\regl,3-\delta,\tb_0,\Psipns))^{\half}.
\end{align}
This is equivalent to $\abs{\prb\varphi_{\sfrak}}\lesssim_{\delta}  v^{-2}\tb^{-2+\delta/2}
(F^{(1)}(\regl,3-\delta,\tb_0,\Psipns))^{\half}$, and applying $\Lxi$ gives extra $\tb^{-1}$ decay, i.e.
\begin{align}
\abs{\Lxi^j\prb\varphi_{\sfrak}}\lesssim_{\delta} {}&v^{-2}\tb^{-2-j+\delta/2}
(F^{(1)}(\regl,3-\delta,\tb_0,\Psipns))^{\half}.
\end{align}
Applying $(\Delta^{\half}\prb)^i$ to equation \eqref{eq:ellip:varphis:posi} and since $\Delta^{\half}\prb$ commutes with the LHS of \eqref{eq:ellip:varphis:posi}, one obtains
\begin{align}
\label{eq:highorder:prbprbPhins:ell=1:impro}
\edthR\edthR'   ((\Delta^{\half}\prb)^i\Phips)
+\Delta^{\half}\prb(\Delta^{\half}
\prb((\Delta^{\half}\prb)^i\Phips ))={}&(\Delta^{\half}\prb)^i (H_{\sfrak}(\Phips)).
\end{align}
Thus, one arrives at the equation \eqref{eq:prbprbPhins:ell=1:impro} but with $(r-M)^{-1}(\Delta^{\half}\prb)^i \Phips$ and $(\Delta^{\half}\prb )^i(H_{\sfrak}(\Phips))$ in place of $\varphi_{\sfrak}$ and $H_{\sfrak}(\Phips)$ respectively. In particular, one has a similar estimate as \eqref{eq:sourceterm:weak:ell1:improve} for the RHS of \eqref{eq:highorder:prbprbPhins:ell=1:impro}. The above discussions for $i=0$ go through here for general $i\in \mathbb{N}$, and we obtain
\begin{align}
\abs{\Lxi^j ((r-M)^{-1}(\Delta^{\half}\prb)^i\Phips)}\lesssim_{\delta,j}  {} v^{-2}\tb^{-1-j+\delta/2}
(F^{(1)}(\regl+i,3-\delta,\tb_0,\Psipns))^{\half}.
\end{align}
As a result, one has
\begin{align}
\absCDerit{ \Lxi^j\varphi_{\sfrak}}{\reg}
\lesssim_{\reg, j,\delta}{}&v^{-2}\tb^{-1-j+\delta/2}
(F^{(1)}(\reg+\regl,3-\delta,\tb_0,\Psipns))^{\half}.
\end{align}

We can similarly treat the case that the first N--P constant vanishes and obtain
\begin{subequations}
\label{eq:almostPrice:psips:1}
\begin{align}
\abs{\Lxi^j (r^{-1}\Phips)}\lesssim_{\delta,j}  {} & v^{-2}\tb^{-2-j+\delta/2}(F^{(1)}(\regl,5-\delta,\tb_0,\Psipns))^{\half}, \\ \abs{\Lxi^j\prb\varphi_{\sfrak}}\lesssim_{\delta,j}{}&v^{-3}\tb^{-2-j+\delta/2}
(F^{(1)}(\regl,5-\delta,\tb_0,\Psipns))^{\half}.
\end{align}
\end{subequations}

Turn to the spin $-\half$ component. Consider the case that the first N--P constant does not vanish. Similarly, we consider only the interior region $\{\rb\leq \tb\}$. Equation \eqref{eq:ellip:varphis:nega} then simplifies to
\begin{align}
\label{eq:simpleformofeqofpsins}
\prb(\mu^{\frac{3}{2}}r^3\prb\psins)=H_{-\sfrak}(\Phins).
\end{align}
From the estimates \eqref{ptw:BEDC:Phiplus:l=1:0to3} and \eqref{almostPrice:posi:weak:NPnotvanish},
\begin{align}
\abs{H_{-\sfrak}(\Phins)}\lesssim_{\delta} {}& \mu^{\half}(\abs{\Lxi^2\Psins}
+\abs{\mu^{-\half}\Lxi\PhinsHigh{1}}
+\abs{\Lxi\Psins})
\notag\\
\lesssim_{\delta} {}& \mu^{\half}(r^{\half}v^{-1}\tb^{-\frac{5}{2}+\frac{\delta}{2}}
+rv^{-1}\tb^{-2+\frac{\delta}{2}})
(F^{(1)}(\regl,3-\delta,\tb_0,\Psipns))^{\half}.
\end{align}
Thus, integrating equation \eqref{eq:simpleformofeqofpsins} from horizon $\rb=2M$ gives
\begin{align}
\abs{r\prb\psins}\lesssim_{\delta}{}&
(r^{-\half}v^{-1}
\tb^{-\frac{5}{2}+\frac{\delta}{2}}
+v^{-1}\tb^{-2+\frac{\delta}{2}})
(F^{(1)}(\regl,3-\delta,\tb_0,\Psipns))^{\half}\notag\\
\lesssim_{\delta} {}&v^{-1}\tb^{-2+\frac{\delta}{2}}
(F^{(1)}(\regl,3-\delta,\tb_0,\Psipns))^{\half}.
\end{align}
We substitute this back to estimate $\abs{H_{-\sfrak}(\Phins)}$:
\begin{align}
\abs{H_{-\sfrak}(\Phins)}\lesssim_{\delta} {}& \mu^{\half}(\abs{\Lxi^2\Psins}
+r^2\abs{\Lxi(r\prb\psins)}
+\abs{r\Lxi\Psins})\notag\\
\lesssim_{\delta} {}& \mu^{\half}(r^{\frac{3}{2}}v^{-1}\tb^{-\frac{5}{2}+\frac{\delta}{2}}
+r^2v^{-1}\tb^{-3+\frac{\delta}{2}})
(F^{(1)}(\regl,3-\delta,\tb_0,\Psipns))^{\half}.
\end{align}
Integrating equation \eqref{eq:simpleformofeqofpsins} again from horizon $\rb=2M$ gives
\begin{align}
\abs{\prb\psins}\lesssim_{\delta}{}&
(r^{-\half}v^{-1}
\tb^{-\frac{5}{2}+\frac{\delta}{2}}
+v^{-1}\tb^{-3+\frac{\delta}{2}})
(F^{(1)}(\regl,3-\delta,\tb_0,\Psipns))^{\half}\notag\\
\lesssim_{\delta} {}& r^{-\half}v^{-1}
\tb^{-\frac{5}{2}+\frac{\delta}{2}}
(F^{(1)}(\regl,3-\delta,\tb_0,\Psipns))^{\half}.
\end{align}
Integrating along $\Sigmatb$ from the hypersurface $\{\rb=\tb\}$ thus gives
\begin{align}
\abs{\Lxi^j\psins}\lesssim_{\delta} v^{-1}\tb^{-2-j+\frac{\delta}{2}}
(F^{(1)}(\regl,3-\delta,\tb_0,\Psipns))^{\half}.
\end{align}
We plug these two estimates back to estimate $H_{-\sfrak}(\Phins)$:
\begin{align}
\abs{H_{-\sfrak}(\Phins)}\lesssim_{\delta} {}& \mu^{\half}(\abs{\Lxi^2\Psins}
+r^2\abs{\Lxi(r\prb\psins)}
+\abs{r\Lxi\Psins})\notag\\
\lesssim_{\delta} {}& \mu^{\half}r^2v^{-1}\tb^{-3+\frac{\delta}{2}}
(F^{(1)}(\regl,3-\delta,\tb_0,\Psipns))^{\half}.
\end{align}
Integrating equation \eqref{eq:simpleformofeqofpsins} from horizon $\rb=2M$ gives
\begin{align}
\abs{\prb\psins}\lesssim_{\delta}{}v^{-1}\tb^{-3+\frac{\delta}{2}}
(F^{(1)}(\regl,3-\delta,\tb_0,\Psipns))^{\half}.
\end{align}
In the same fashion as for the spin $\half$ component, one can obtain decay estimates for higher order pointwise norm:
\begin{align}
\absCDerit{\Lxi^j\psins}{\reg}\lesssim_{\delta,j,\reg}{}
v^{-1}\tb^{-2-j+\frac{\delta}{2}}
(F^{(1)}(\reg+\regl,3-\delta,\tb_0,\Psipns))^{\half}.
\end{align}

In the case that the first N--P constant vanishes, a similar treatment gives that
\begin{subequations}
\begin{align}
\absCDerit{\Lxi^j\psins}{\reg}\lesssim_{\delta,j,\reg}{}&v^{-1}\tb^{-3+\frac{\delta}{2}}
(F^{(1)}(\reg+\regl,5-\delta,\tb_0,\Psipns))^{\half},\\
\abs{\prb\Lxi^j\psins}\lesssim_{\delta}{}&v^{-1}\tb^{-4-j+\frac{\delta}{2}}
(F^{(1)}(\regl,5-\delta,\tb_0,\Psipns))^{\half}.
\end{align}
\end{subequations}

%%%%%%%%%%%%%%%%%%
\subsection{Almost Price's law decay}
\label{subsect:almostPricelaw:allmodes}
%%%%%%%%%%%%%%%

We collect the main statement about almost Price's law decay in the theorem below.
\begin{thm}
\label{thm:almostPrice}
Consider a Dirac field on a Schwarzschild black hole spacetime.
\begin{enumerate}
  \item Let  the spin $\pm \half$ components be supported on $\ell\geq \ell_0$ modes for an $\ell_0\geq 2$. If the $\ell_0$-th Newman--Penrose constant of the $\ell_0$ mode does not vanish, we have
      \begin{subequations}
\begin{align}
\absCDerit{\Lxi^j \varphi_{\sfrak}}{\reg}\lesssim_{\delta,j,\reg,\ell_0} {} & v^{-2}\tb^{-\ell_0-j+\delta/2}
\Big[(F^{(\ell_0)}(\reg+\regl,3-\delta,\tb_0,(\Psipns)^{\ell=\ell_0}))^{\half}\notag\\
&\qquad\qquad\qquad\quad
+(F^{(\ell_0+1)}(\reg+\regl,1-\delta,\tb_0,(\Psipns)^{\ell\geq\ell_0+1}))^{\half}\Big],\\
\absCDerit{\Lxi^j \psins}{\reg}\lesssim_{\delta,j,\reg,\ell_0} {} & v^{-1}\tb^{-1-\ell_0-j+\delta/2}\Big[(F^{(\ell_0)}(\reg+\regl,3-\delta,\tb_0,(\Psipns)^{\ell=\ell_0}))^{\half}\notag\\
&\qquad\qquad\qquad\qquad
+(F^{(\ell_0+1)}(\reg+\regl,1-\delta,\tb_0,(\Psipns)^{\ell\geq\ell_0+1}))^{\half}\Big].
\end{align}
\end{subequations}
And if the $\ell_0$-th Newman--Penrose constant of the $\ell_0$ mode vanishes,
the $\tb$ power of the above pointwise decay estimates is decreased by $1$, and the terms in the square brackets are replaced by
$(F^{(\ell_0)}(\reg+\regl,5-\delta,\tb_0,(\Psipns)^{\ell=\ell_0}))^{\half}
+(F^{(\ell_0+1)}(\reg+\regl,3-\delta,\tb_0,(\Psipns)^{\ell\geq\ell_0+1}))^{\half}$.
  \item Let the spin $\pm \half$ components be supported on $\ell=1$ mode. Then, if the first N--P constant does not vanish, we have for the spin $\half$ component that
\begin{subequations}
\label{dec:almostPrice:posi}
\begin{align}
\absCDerit{\Lxi^j\varphi_{\sfrak}}{\reg}\lesssim_{\delta,j,\reg}
{}&v^{-2}\tb^{-1+\frac{\delta}{2}}
(F^{(1)}(\reg+\regl,3-\delta,\tb_0,\Psipns))^{\half},\\
\abs{\prb\Lxi^j\varphi_{\sfrak}}\lesssim_{\delta,j}
{}&v^{-2}\tb^{-2-j+\frac{\delta}{2}}
(F^{(1)}(\regl,3-\delta,\tb_0,\Psipns))^{\half}
\end{align}
\end{subequations}
and the spin $-\half$ component that
\begin{subequations}
\label{dec:almostPrice:nega}
\begin{align}
\absCDerit{\Lxi^j\psins}{\reg}\lesssim_{\delta,j,\reg}{}&v^{-1}\tb^{-2+\frac{\delta}{2}}
(F^{(1)}(\reg+\regl,3-\delta,\tb_0,\Psipns))^{\half},\\
\abs{\prb\Lxi^j\psins}\lesssim_{\delta,j}{}&v^{-1}\tb^{-3-j+\frac{\delta}{2}}
(F^{(1)}(\regl,3-\delta,\tb_0,\Psipns))^{\half}.
\end{align}
\end{subequations}
Moreover, if the first Newman--Penrose constant vanishes,
the $\tb$  power of the above pointwise decay estimates is decreased by $1$ and the argument $3-\delta$ is replaced by $5-\delta$.
\end{enumerate}
\end{thm}

\begin{proof}
In the first case that the spin $\pm \half$ components are supported on $\ell\geq \ell_0$ modes for an $\ell_0\geq 2$, we utilize the estimates in Proposition \ref{prop:almostpricelaw:ellgeq2} for $\ell=\ell_0$ mode and $\ell\geq \ell_0+1$ modes respectively and add them together to achieve the desired estimates.
The estimates of $\ell=1$ mode are from Section \ref{sect:almostPricelaw:ell=1}.
\end{proof}
\begin{remark}
In the case that the components are supported on $\ell=1$ mode, the above decay estimates for both $\varphi_{\sfrak}$ and $\psins$ and for the radial tangential derivative of both $\varphi_{\sfrak}$ and $\psins$ are almost sharp.
\end{remark}

%%%%%%%%%%%%%%%%%%%%%
\section{Price's law decay for nonvanishing first Newman--Penrose constant}
\label{sect:Pricelaw:nonzeroNPconst}
%%%%%%%%%%%%%%%%%%%

The aim of this section is to derive the precise asymptotic behaviours of the spin $\pm \half$ components on a Schwarzschild spacetime  in the case that the first Newman--Penrose constant of $\ell=1$ mode is nonzero.

We shall focus  only on the $\ell=1$ mode of the spin $\pm \half$ components, i.e. $(\psips)^{\ell=1}$ and $(\psins)^{\ell=1}$, since the higher modes, $\ell\geq 2$, have faster decay as shown in Theorem \ref{thm:almostPrice}. As discussed in Section \ref{sect:decompIntoModes}, this mode for each component can further be expanded in terms of the spin-weighted spherical harmonics:
\begin{subequations}
\label{eq:azimuthalmodes:decomp}
\begin{align}
(\psips)^{\ell=1}(\tb,\rb,\theta,\pb)
={}&\sum\limits_{m=\pm\half}(\psips)_{m,\ell=1}
(\tb,\rb)Y_{m,1}^{\sfrak}(\cos\theta)e^{im\pb},\\
(\psins)^{\ell=1}(\tb,\rb,\theta,\pb)
={}&\sum\limits_{m=\pm\half}(\psins)_{m,\ell=1}
(\tb,\rb)Y_{m,1}^{\sfrak}(\cos\theta)e^{im\pb}.
\end{align}
\end{subequations}
Each $(m,\ell=1)$ mode can be treated in the same way, thus we shall simply  drop the subscript $m,\ell=1$ and allow them  to share the same notation with the spin $\pm \half$ components. For each separate $(m,\ell=1)$ mode of either of the spin $\pm \half$ components,  its corresponding N--P constant $\NPCP{1}$ (and $\NPCN{1}$ which is equal to $\NPCP{1}$ by Lemma \ref{lem:relationoftwoNPconsts}) as defined in Definition \ref{def:NPCs} is a constant independent of $\tb$, $\rb$, $\theta$, and $\pb$.

For any $\delta>0$, we denote
\begin{align}
\FB^{\reg}_{\delta}={}&(F^{(1)}(\reg,3-\delta,\tb_0,\Psipns))^{\half},
\end{align}
where the RHS is defined as in Definition \ref{def:Fenergies:big2}. The regularity parameter $\reg$, which depends only on $j$, may always be suppressed, and we simply write $\FB_{\delta}$ for $\FB^{\reg}_{\delta}$.
For simplicity, we shall also denote
\begin{align}
\FBT=\FB_{\delta}+\FB_{\delta'}+\abs{\NPCP{1}}+ D_0,
\end{align}
where $\delta$ and $\delta'$ are to be fixed in the proof, and $D_0$ is a constant appearing in the assumptions below.
%%%%%%%%%%%%%%%%%
\subsection{Spin $\half$ component}
\label{sect:Pricelaw:ell=1:posi}
%%%%%%%%%%%%%%%%%%%%%%%%%%

Consider a $(m, \ell=1)$ mode of the spin $ \half$ component. In this case, the spin $\half$ component satisfies equation \eqref{eq:PhipsHigh1:l=1}, which can also be written as
\begin{align}
\label{eq:wave:PhipsHigh1:simple}
-r^2Y(\mu^{\half} r^{-1}V\hatPhips) -{6M}{r^{-1}}(\mu^{\half} r^{-1}\hatPhips) ={}&0,
\end{align}
or equivalently, 
\begin{align}
\label{eq:PhipsHigh1:l=1:v1}
\pu(\mu^{\half} r^{-1}V\hatPhips) ={}-{3M}\mu^{\frac{3}{2}} r^{-4}\hatPhips.
\end{align}

We will frequently use also the double-null coordinates $(u,v,\theta,\pb)$, and the DOC will be divided into different regions as in Figure \ref{fig:4}. The following lemma lists all relations and estimates among $u$, $v$, $r$, and $\tb$ that will be utilized in these different regions.

%%%%%%%%%%%%%%%%%%%%%%%%%%%%%%%%%%
\begin{figure}[htbp]
\begin{center}
\begin{tikzpicture}[scale=0.8]
\draw[thin]
(0,0)--(2.45,2.45);
\draw[thin]
(0,0)--(-0.25,-0.25);
\draw[very thin]
(2.5,2.5) circle (0.05);
\coordinate [label=90:$i_+$] (a) at (2.5,2.5);
\draw[dashed]
(2.55,2.45)--(5.5,-0.5);
\node at (0.68,1.02) [rotate=45] {\small $\mathcal{H}^+$};
\node at (5.3,0) [rotate=-45] {\small $\mathcal{I}^+$};
\draw[thin]
(2.5,2.45) arc (20:-70:0.3 and 3);
\node at  (2.2,-0.5) [rotate=85] {\small $r=R$};
\draw[thin]
(2.52,2.45) arc (50:2:2.2 and 5);
\draw[thin]
(2.53,2.46) arc (50:18:6 and 8);
\node at (3.6,-1.2)  {\small $\gamma_{\alpha}$};
\node at (4.45,-0.4) [rotate=-67] {\scriptsize $\{v-u=\half v\}$};
\end{tikzpicture}
\end{center}
\caption{For $v$ large enough, there are some useful curves in spacetime, where $\gamma_\alpha=\{v-u=v^\alpha\}$ for an $\alpha\in (0,1)$.}
\label{fig:4}
\end{figure}
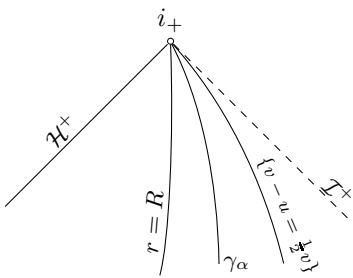
%%%%%%%%%%%%%%%%%%%%%%%%%%%%%%%%%%%%%%%%%%%%%%

\begin{lemma}
For any $\alpha\in (\half, 1)$, let $\gamma_{\alpha}=\{v-u=v^\alpha\}$. For any $u$ and $v$, let $u_{\gamma_\alpha}(v)$ and $v_{\gamma_{\alpha}}(u)$ be such that $(u_{\gamma_\alpha}(v),v), (u, v_{\gamma_{\alpha}}(u))\in \gamma_{\alpha}$.
In the region $v-u\geq v^{\alpha}$,
\begin{subequations}
\begin{align}
\label{eq:rela:a}
&r\gtrsim{} v^{\alpha}+u^{\alpha},\\
\label{eq:rela:b}
&\abs{u-v_{\gamma_{\alpha}}(u)}\lesssim{}u^{\alpha},\\
\label{eq:rela:c}
&\abs{2r-(v-u)}\leq C\log (r-2M);
\end{align}
in the region $\{v-u\geq v^\alpha\}\cap\{v-u\geq\frac{v}{2}\}$,
\begin{align}
\label{eq:rela:d}
&v+u\lesssim r\lesssim v;
\end{align}
in the region $\{v-u\geq v^{\alpha}\}\cap\{v-u\leq\frac{v}{2}\}$,
\begin{align}
\label{eq:rela:e}
&u \sim v, \quad r\gtrsim v^{\alpha};
\end{align}
in the region $\{r\geq R\}\cap\{v-u\leq v^{\alpha}\}$,
\begin{align}
\label{eq:rela:f}
&r\lesssim\min\{v^{\alpha}, u^{\alpha}\};
\end{align}
in the region $\{2M\leq r\leq R\}$, there is a constant $C_R$ depending on $R$ such that
\begin{align}
\label{eq:rela:g}
\abs{v|_{\Sigma_\tau}(R)-v|_{\Sigma_\tau}(r)}+\abs{v-\tb}\leq C_R.
\end{align}
\end{subequations}
On $\Sigmazero$, for $r$ large,
\begin{align}\label{asymp-factor}
\abs{r^{-1}v- 2-4Mr^{-1}\log(r-2M)}\lesssim r^{-1}.
\end{align}
\end{lemma}

Let us emphasis that most of the rest of the discussions in this subsection is in the spirit of \cite{angelopoulos2018late} where the authors derived the asymptotics for the $\ell=0$ mode of the scalar field. 

%%%%%%%%%%%
\subsubsection{Asymptotics of $\varphi_{\sfrak}$}
\label{sect:asymofvarphi+}
%%%%%%%%%%

We first obtain the asymptotics for $\VR\hatPhips$ in the region $\{v-u\geq v^{\alpha}\}$.

\begin{lemma}
 Assume on $\Sigma_{\tau_0}$ that there is a constant $\beta\in (0,\half)$ and a constant $D_0$ such that
\begin{align}
\label{assump:initialdata:zeroorder:NVNP}
\bigg|v^2\VR\hatPhips(\tau_0,v)-4{\NPCP{1}}\bigg|\lesssim D_0v^{-\beta},
\end{align}
then for any $v-u\geq v^\alpha$ with $\half<\alpha<1$,
\begin{align}
\label{eq:VPhinsHigh1:awayregion}
|\mu^{\frac{3}{2}} r^{-1}v^3\VR\hatPhips(u,v)-8\NPCP{1}|
\lesssim (v^{-\beta}+v^{-\eta})\FBT,
\end{align}
where $\eta=2\alpha-1-\delta$.
\end{lemma}

\begin{proof}
In double null coordinates $(u,v)$, we integrate \eqref{eq:PhipsHigh1:l=1:v1} along $v=const$ from $(u_{\Sigma_{\tau_0}}(v),v)\in \Sigma_{\tau_0}$ as in Figure \ref{fig:5}
%%%%%%%%%%%%%%%%%%%%%%%%%%%%%%%%%%%%%%%%%
\begin{figure}[htbp]
\begin{center}
\begin{tikzpicture}[scale=0.8]
\draw[thin]
(0,0)--(2.45,2.45);
\draw[thin]
(0,0)--(-0.25,-0.25);
\draw[very thin]
(2.5,2.5) circle (0.05);
\coordinate [label=90:$i_+$] (a) at (2.5,2.5);
\draw[dashed]
(2.55,2.45)--(5.5,-0.5);
\node at (0.68,1.02) [rotate=45] {\small $\mathcal{H}^+$};
\node at (5.5,-0.2) [rotate=-45] {\small $\mathcal{I}^+$};
\draw[thin]
(2.5,2.45) arc (20:-70:0.3 and 3);
\node at  (2.3,0.2) [rotate=86] {\small $r=R$};
\draw[thin]
(2.52,2.45) arc (50:2:2.2 and 5);
\node at (3.5,-1.4)  {\small $\gamma_{\alpha}$};
\draw[thin]
(0,0) arc (206:337:2.7 and 1.5);
\node at (1.4,-1) {\small $\Sigma_{\tau_0}$};
\draw[thick]
(4.3,-0.4)--(3.5,0.4);
\draw[fill] (4.3,-0.425) circle (1.5pt);
\draw[fill] (3.5,0.4) circle (1.5pt);
\node at (4.6,-0.8) {\scriptsize $(u_{\Sigma_{\tau_0}}(v),v)$};
\node at (3.7,0.6) {\scriptsize $(u,v)$};
\end{tikzpicture}
\end{center}
\caption{For any point $(u,v)$ in $\{r\geq R\}\cap\{v-u\geq v^\alpha\}$, i.e. $u_{\Sigma_{\tau_0}}(v)\leq u\leq u_{\gamma_\alpha}(v)$, integrate along $v=const$ from $(u_{\Sigma_{\tau_0}}(v),v)\in \Sigma_{\tau_0}$.}
\label{fig:5}
\end{figure}
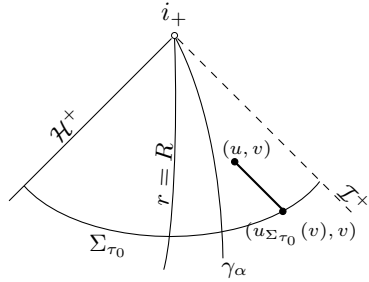
%%%%%%%%%%%%%%%%%%%%%%%%%%%%%%%%%%%%%%%%%%%%%%
and obtain
\begin{align}
\mu^{\frac{3}{2}} r^{-1}v^3\VR\hatPhips(u,v)-\mu^\frac{3}{2} r^{-1}v^3\VR\hatPhips(u_{\Sigma_{\tau_0}}(v),v)=-3Mv^3\int_{u_{\Sigma_{\tau_0}}(v)}^u
\mu^{\frac{3}{2}}r^{-4}\hatPhips(u',v)\di u'.
\end{align}
Here, we have assumed $u_{\Sigma_{\tau_0}}(v)\geq 1$ for all $v$ larger than a fixed constant without loss of generality.
We can now estimate the RHS of the above equality by Theorem \ref{thm:almostPrice} and \eqref{eq:rela:a}:
\begin{align}\begin{split}
\bigg|v^3\int_{u_{\Sigma_{\tau_0}}(v)}^u
\mu^{\frac{3}{2}}r^{-4}\hatPhips(u',v)\di u'\bigg|
&\lesssim v\int_{u_{\Sigma_{\tau_0}}(v)}^u
\mu^{\frac{3}{2}}\Delta^{-\half}r^{-1}\tau^{-1+\frac{\delta}{2}}(u',v)\di u' \FB_{\delta} \\
&=v^{-\eta}\int_{u_{\Sigma_{\tau_0}}(v)}^u
\mu r^{-2}v^{1+\eta}\tau^{-1+\frac{\delta}{2}}(u',v)\di u' \FB_{\delta} \\
&\lesssim v^{-\eta}\int_{u_{\Sigma_{\tau_0}}(v)}^u
\mu r^{-2}v^{1+\eta}(u')^{-1+\frac{\delta}{2}}(u',v)\di u' \FB_{\delta} \\
&\lesssim v^{-\eta}\int_{u_{\Sigma_{\tau_0}}(v)}^u
\mu r^{-2+\frac{1+\eta}{\alpha}}(u')^{-1+\frac{\delta}{2}}(u',v)du' \FB_{\delta} \\
&\lesssim v^{-\eta}\int_{u_{\Sigma_{\tau_0}}(v)}^u
(u')^{-2\alpha+\eta+\frac{\delta}{2}}(u',v)\di u' \FB_{\delta} \\
&\lesssim v^{-\eta} \FB_{\delta} ,
\end{split}\end{align}
where we have used $\eta-2\alpha+\frac{\delta}{2}=-1-\frac{\delta}{2}<-1$.
This yields
\begin{align}
\label{eq:VPhinsHigh1:awayregion:1}
\abs{\mu^{\frac{3}{2}} r^{-1}v^3\VR\hatPhips(u,v)-\mu^\frac{3}{2} r^{-1}v^3\VR\hatPhips(u_{\Sigma_{\tau_0}}(v),v)}\lesssim{}v^{-\eta} \FB_{\delta} .
\end{align}
On the other hand, we have from the assumption \eqref{assump:initialdata:zeroorder:NVNP} that on $\Sigma_{\tau_0}$, for $r$ large,
\begin{align*}
|\mu^\frac{3}{2} r^{-1}v^3\VR\hatPhips(u_{\Sigma_{\tau_0}}(v),v)-8\NPCP{1}|
\lesssim \mu^\frac{3}{2} r^{-1}vv^{-\beta}D_0  +|\mu^{\frac{3}{2}}r^{-1}v -2|\abs{\NPCP{1}}.
\end{align*}
Meanwhile, on $\Sigma_{\tau_0}$, we utilize  $r=r_{\Sigma_{\tau_0}}(v)\lesssim v$, $v^{\alpha}\lesssim r(u,v)\leq r_{\Sigma_{\tau_0}}(v)$ and the estimate \eqref{asymp-factor} to derive
\begin{align*}
|\mu^{\frac{3}{2}}r^{-1}v -2|\lesssim |\mu^{\frac{3}{2}}(r^{-1}v -2)|+\abs{1-\mu^{\frac{3}{2}}}\lesssim r^{-1}+r^{-1}\log (r-2M)\lesssim v^{-\alpha}\log v.
\end{align*}
This thus enables us to conclude 
\begin{align}
\label{eq:VPhinsHigh1:awayregion:2}
|\mu^\frac{3}{2} r^{-1}v^3\VR\hatPhips(u_{\Sigma_{\tau_0}}(v),v)-8\NPCP{1}|
\lesssim v^{-\beta}D_0  +v^{-\alpha}\log v \abs{\NPCP{1}}.
\end{align}
The estimate \eqref{eq:VPhinsHigh1:awayregion} then follows from the estimates \eqref{eq:VPhinsHigh1:awayregion:1} and \eqref{eq:VPhinsHigh1:awayregion:2}.
\end{proof}

Now we estimate $\hatPhips$ in $v-u\geq v^\alpha$. The estimate \eqref{eq:VPhinsHigh1:awayregion} yields
\begin{align}
\label{eq:VPhinsHigh1:awayregion:v1}
|V\hatPhips(u,v)-8rv^{-3}\NPCP{1}|\lesssim
rv^{-3}(v^{-\beta}+v^{-\eta})\FBT+v^{-3}\abs{\NPCP{1}}.
\end{align}
 We integrate along $u=const$ as in Figure \ref{fig:6}
 %%%%%%%%%%%%%%%%%%%%%%%%%%%%
\begin{figure}[htbp]
\begin{center}
\begin{tikzpicture}[scale=0.8]
\draw[thin]
(0,0)--(2.45,2.45);
\draw[thin]
(0,0)--(-0.25,-0.25);
\draw[very thin]
(2.5,2.5) circle (0.05);
\coordinate [label=90:$i_+$] (a) at (2.5,2.5);
\draw[dashed]
(2.55,2.45)--(5.5,-0.5);
\node at (0.68,1.02) [rotate=45] {\small $\mathcal{H}^+$};
\node at (5.5,-0.2) [rotate=-45] {\small $\mathcal{I}^+$};
\draw[thin]
(2.5,2.45) arc (20:-70:0.3 and 3);
\node at  (2.3,0.2) [rotate=86] {\small $r=R$};
\draw[thin]
(2.52,2.45) arc (50:2:2.2 and 5);
\node at (3.3,-1.4)  {\small $\gamma_{\alpha}$};
\draw[thick]
(3.275,-0.625)--(3.95,0.1);
\draw[fill] (3.275,-0.625) circle (1.5pt);
\draw[fill] (3.95,0.1) circle (1.5pt);
\node at (4.2,-0.75) {\scriptsize $(u,v_{\gamma_\alpha}(u))$};
\node at (3.9,0.3) {\scriptsize $(u,v)$};
\end{tikzpicture}
\end{center}
\caption{For any point $(u,v)$ in $\{r\geq R\}\cap\{v-u\geq v^\alpha\}$, i.e.  $v\geq v_{\gamma_\alpha}(u)$, integrate along $u=const$ from $(u, v_{\gamma_\alpha}(u))\in \gamma_\alpha$.}
\label{fig:6}
\end{figure}
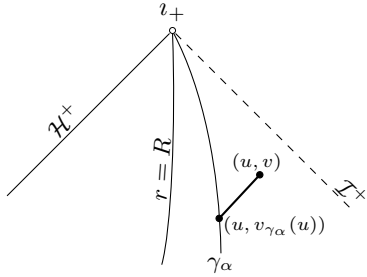
%%%%%%%%%%%%%%%%%%%%%%%%
 to obtain
\begin{align}\label{eq-v-hatPhips}
(r^{-2}\hatPhips)(u,v)=(r(u,v))^ {-2}\hatPhips(u,v_{\gamma_\alpha}(u))
+\half (r(u,v))^ {-2}\int_{v_{\gamma_\alpha}(u)}^{v}
 V\hatPhips (u,v')\di v'.
 \end{align}
We utilize  \eqref{eq:rela:b}, \eqref{eq:rela:c} and \eqref{eq:VPhinsHigh1:awayregion:v1} to estimate the last term of \eqref{eq-v-hatPhips}:
\begin{align}
\hspace{2ex}&\hspace{-2ex}
\frac{1}{2}\bigg|\int_{v_{\gamma_\alpha}(u)}^{v}
 \bigg(V\hatPhips (u,v')
 -8rv^{-3}\NPCP{1}\bigg)\di v'\bigg|\notag\\
 &\lesssim \int_{v_{\gamma_\alpha}(u)}^{v}
 \big(rv^{-3}(v^{-\beta}+v^{-\eta})\FBT +v^{-3}\abs{\NPCP{1}}\big)\di v'\notag\\
 &\lesssim \big((v_{\gamma_\alpha}(u))^{-1-\beta}-v^{-1-\beta}
 +(v_{\gamma_\alpha}(u))^{-1-\eta}-v^{-1-\eta} \big)\FBT +\big((v_{\gamma_\alpha}(u))^{-2}-v^{-2}\big)\abs{\NPCP{1}}\notag\\
 &\lesssim \big(u^{-1-\beta}
 +u^{-1-\eta}\big)\FBT   +u^{-2}\abs{\NPCP{1}}
\end{align}
and
\begin{align}
\hspace{2ex}&\hspace{-2ex}
\half \int_{v_{\gamma_\alpha}(u)}^{v}
8rv^{-3}\NPCP{1}\di v'\notag\\
&=\int_{v_{\gamma_\alpha}(u)}^{v}
[(2v^{-2}-2u v^{-3})+2v^{-3}(2r-(v-u))]\NPCP{1}\di v'\notag\\
&=\NPCP{1}[(uv^{-2}-2v^{-1})
-(uv^{-2}-2v^{-1})(u,v_{\gamma_\alpha}(u))] +\NPCP{1}\int_{v_{\gamma_\alpha}(u)}^{v}2v^{-3}(2r-(v-u))\di v'\notag\\
&=\NPCP{1}\bigg[u^{-1}v^{-2}(v-u)^2
+\frac{(v_{\gamma_\alpha}(u))^2-u^2}{u(v_{\gamma_\alpha}(u))^2}
+\frac{2(u-v_{\gamma_\alpha}(u))}{uv_{\gamma_\alpha}(u)}\bigg]\notag\\
&\qquad +\NPCP{1}\int_{v_{\gamma_\alpha}(u)}^{v}2v^{-3}(2r-(v-u))\di v'.
\end{align}
Note that we can use \eqref{eq:rela:b} and \eqref{eq:rela:c} to achieve
\begin{subequations}
\begin{align}
\abs{u^{-1}v^{-2}(v-u)^2-4r^2u^{-1}v^{-2}}\lesssim {}& u^{-1}v^{-2}r\log r,\\
\bigg|\frac{(v_{\gamma_\alpha}(u))^2-u^2}{u(v_{\gamma_\alpha}(u))^2}
+\frac{2(u-v_{\gamma_\alpha}(u))}{uv_{\gamma_\alpha}(u)}\bigg|
\lesssim{}& u^{-2+\alpha},\\
\bigg|\int_{v_{\gamma_\alpha}(u)}^{v}2v^{-3}(2r-(v-u))\bigg|
\lesssim{}& u^{-2}(\log u)^2.
\end{align}
\end{subequations}
As a result,
\begin{align}\begin{split}
\hspace{4ex}&\hspace{-4ex}
\bigg|\half r^{-2}\int_{v_{\gamma_\alpha}(u)}^{v}
V\hatPhips (u,v')\di v'-4u^{-1}v^{-2}\NPCP{1}\bigg|\\
\lesssim{}&
r^{-2}\big[\big(u^{-1-\beta}
 +u^{-1-\eta}\big)\FBT
 +\big(u^{-2}
 +u^{-1}v^{-2}r\log r +u^{-2+\alpha}
 +u^{-2}(\log u)^2\big)\abs{\NPCP{1}}\big]
\\
\lesssim{}& r^{-2}\big[\big(u^{-1-\beta}
+u^{-2\alpha+\delta}\big)\FBT +u^{-2+\alpha}\abs{\NPCP{1}}\big].
 \end{split}
\end{align}
For the first term on the RHS of \eqref{eq-v-hatPhips}, one uses \eqref{dec:almostPrice:posi} and \eqref{eq:rela:a} to obtain
\begin{align}\begin{split}
|(r(u,v))^ {-2}\hatPhips(u,v_{\gamma_\alpha}(u))|&\lesssim \mu^{-\half}(r(u,v))^ {-2} (r(u,v_{\gamma_\alpha}(u)))^ {2}\cdot (v^{-2}\tau^{-1+\frac{\delta}{2}})
(u,v_{\gamma_\alpha}(u))\FB_{\delta}\\
&\lesssim (r(u,v))^ {-2}u^{-3+2\alpha+\frac{\delta}{2}} \FB_{\delta}.
\end{split}\end{align}
Hence, we conclude the following estimates for $\varphi_{\sfrak}$ in the region $\{v-u\geq v^{\alpha}\}$.

\begin{lemma}
For $v-u\geq v^\alpha$ with $\half<\alpha<1$, we have
\begin{align}\label{asymp-hatPhips-near-nullinfinity}
\big|\varphi_{\sfrak}-4u^{-1}v^{-2}\NPCP{1}\big|
\lesssim {}&r^{-2}\big(u^{-1-\beta}
+u^{-2\alpha+\delta}
+u^{-3+2\alpha+\frac{\delta}{2}}+u^{-2+\alpha}\big)\FBT  .
\end{align}
\end{lemma}

\begin{proof}
The above discussions imply that the estimate \eqref{asymp-hatPhips-near-nullinfinity} holds with $\varphi_{\sfrak}$ replaced by $r^{-2}\PhipsHigh{1}$. It remains to estimate the difference between these two scalars.
By definition, $\varphi_{\sfrak}=r(r-M)^{-1}\mu^{\half}r^{-2}\PhipsHigh{1}$, and $\abs{r(r-M)^{-1}\mu^{\half}-1}\lesssim M^2r^{-2}$. Thus the estimate \eqref{asymp-hatPhips-near-nullinfinity} follows in view of Theorem \ref{thm:almostPrice}.
\end{proof}

We are now ready to consider the asymptotics for $\varphi_{\sfrak}$ in the entire $\Dzeroinfty$.
For $\{v-u\geq v^\alpha\}\cap\{v-u\geq\frac{v}{2}\}$, the RHS of \eqref{asymp-hatPhips-near-nullinfinity} is bounded using \eqref{eq:rela:d} by
\begin{align}
\label{eq:pricelaw:errorterm:veryfarregion}
Cv^{-2}\big(u^{-1-\beta}
+u^{-2\alpha+\delta}
+u^{-3+2\alpha+\frac{\delta}{2}}+u^{-2+\alpha}\big)\FBT  .
\end{align}
For $\{v-u\geq v^{\alpha'}\}\cap\{v-u\leq\frac{v}{2}\}$, where $\alpha'\in (\alpha,1)$, one can utilize \eqref{eq:rela:e} to bound the RHS of \eqref{asymp-hatPhips-near-nullinfinity} by
\begin{align}
\label{eq:pricelaw:errorterm:farregion}
Cv^{-2\alpha'}\big(v^{-1-\beta}+v^{-2\alpha+\delta} +v^{-3+2\alpha+\frac{\delta}{2}}+v^{-2+\alpha}\big)\FBT  .
\end{align}
By taking $0<\delta\leq \min\{0.4, 2\beta\}$, $\alpha=1-\frac{3\delta}{8}$ and $\alpha'=1-\frac{\delta}{16}$, the expression \eqref{eq:pricelaw:errorterm:veryfarregion} is bounded by $Cv^{-2}u^{-1-\frac{\delta}{8}}\FBT $, and the expression \eqref{eq:pricelaw:errorterm:farregion} is bounded by $Cv^{-2+\frac{\delta}{8}}(v^{-1-\beta}
+v^{-2+\frac{7\delta}{4}}+v^{-1-\frac{\delta}{4}}
+v^{-1-\frac{3\delta}{8}})\FBT
\leq Cv^{-3-\frac{\delta}{8}}\FBT$. Additionally, in the region $\{v-u\geq v^{\alpha'}\}$, one has $\abs{u-\tb}\lesssim 1$, Thus, these discussions together with the estimate \eqref{asymp-hatPhips-near-nullinfinity} yield that in $\{v-u\geq v^{\alpha'}\}$,
\begin{align}
\abs{\varphi_{\sfrak}-4\tb^{-1}v^{-2}\NPCP{1}}
\lesssim {}& v^{-2}\tb^{-1-\frac{\delta}{8}}\FBT.
\end{align}

For $\{r\geq R\}\cap\{v-u\leq v^{\alpha'}\}$, we integrate along $\Sigma_\tau$ from a point $(\tb,r_{\gamma_{\alpha'}}(\tb))\in \gamma_{\alpha'}=\{v-u=v^{\alpha'}\}$, (See Figure \ref{fig:7}.)
%%%%%%%%%%%%%%%%%%%%%%%%%%
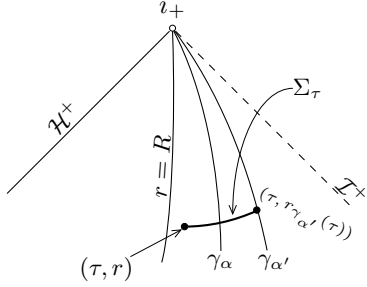
\begin{figure}[htbp]
\begin{center}
\begin{tikzpicture}[scale=0.8]
\draw[thin]
(0,0)--(2.45,2.45);
\draw[thin]
(0,0)--(-0.25,-0.25);
\draw[very thin]
(2.5,2.5) circle (0.05);
\coordinate [label=90:$i_+$] (a) at (2.5,2.5);
\draw[dashed]
(2.55,2.45)--(5.5,-0.5);
\node at (0.68,1.02) [rotate=45] {\small $\mathcal{H}^+$};
\node at (5.5,-0.2) [rotate=-45] {\small $\mathcal{I}^+$};
\draw[thin]
(2.5,2.45) arc (20:-70:0.3 and 3);
\node at  (2.3,0.2) [rotate=86] {\small $r=R$};
\draw[thin]
(2.52,2.45) arc (50:2:2.2 and 5);
\node at (3.3,-1.4)  {\small $\gamma_{\alpha}$};
\draw[thin]
(2.53,2.46) arc (50:18:5 and 8);
\node at (4.2,-1.4) {\small $\gamma_{\alpha'}$};
\draw[thick]
(2.7,-0.8) arc (270:284:5 and 9);
\draw[fill] (2.7,-0.8) circle (1.5pt);
\draw[fill] (3.9,-0.525) circle (1.5pt);
\node at (4.7,-0.625) [rotate=-25] {\tiny $(\tau,r_{\gamma_{\alpha'}}(\tau))$};
\draw[very thin, ->, >=angle 45]
(4.5,1.5) arc (90:180:1 and 2.1);
\node at (4.75,1.5) {\small $\Sigma_\tau$};
\draw[very thin, ->, >=angle 45]
(1.8,-1.3)--(2.65,-0.85);
\node at (1.4,-1.5) {\small $(\tau,r)$};
\end{tikzpicture}
\end{center}
\caption{For any point $(\tau,r)$ in $\{r\geq R\}\cap\{v-u\leq v^{\alpha'}\}$ with $\alpha'\in (\alpha,1)$ suitably chosen, integrate along $\Sigma_\tau$ from $(\tau,r_{\gamma_{\alpha'}}(\tb))\in \gamma_{\alpha'}$. }
\label{fig:7}
\end{figure}
%%%%%%%%%%%%%%%%%%%%%%%%%%%%%%%%
thus,
\begin{align}
\varphi_{\sfrak}(\tau,r)={}\varphi_{\sfrak}(\tau,r_1)
-\int_r^{r_1}\partial_\rho\varphi_{\sfrak}(\tau,\rho)\di \rho.
\end{align}
Note that $\tb= u|_{(\tau,r_{\gamma_{\alpha'}}(\tb))}\sim v$, $v(r_{\gamma_{\alpha'}}(\tb))\sim v\sim u$ and $\abs{v(r_{\gamma_{\alpha'}}(\tb))-v}\lesssim v^{\alpha'}=v^{1-\frac{\delta}{16}}$ on $\Sigmatb$.  Thus, using the estimate \eqref{dec:almostPrice:nega} but with $\delta$ replaced by a $\delta'\in (0,\delta)$ to be fixed gives
\begin{align}\begin{split}
\int_r^{r_{\gamma_{\alpha'}}(\tb)}
|\partial_\rho\varphi_{\sfrak}|(\tau,\rho)\di \rho \lesssim {}&\tau^{-1+\frac{\delta'}{2}}\int_r^{r_{\gamma_{\alpha'}}(\tb)} v^{-3}|_{\Sigma_\tau}\di \rho\\
\lesssim {}&v^{-1+\frac{\delta'}{2}}(v^{-2}-(v(r_{\gamma_{\alpha'}}(\tb)))^{-2}) \FBT \\
\lesssim {}&v^{-1+\frac{\delta'}{2}}\frac{(v(r_{\gamma_{\alpha'}}(\tb))-v)
(v(r_{\gamma_{\alpha'}}(\tb))+v)}{v^2(v(r_{\gamma_{\alpha'}}(\tb)))^2} \FBT \\
\lesssim {}&v^{-3+\frac{\delta'}{2}-\frac{\delta}{16}} \FBT .
\end{split}\end{align}
By taking $\delta'=\frac{\delta}{16}$, the above is bounded by $Cv^{-3-\frac{\delta}{32}} \FBT$.
Hence,
\begin{align}
\label{eq:Pricelaw:posi:bigr}
\hspace{4ex}&\hspace{-4ex}
|\varphi_{\sfrak}(\tau,r)
-4\tb^{-1}v^{-2}\NPCP{1}|\notag\\
\leq {}&\bigg|\varphi_{\sfrak}(\tau,r)
-\varphi_{\sfrak}(\tau,r_{\gamma_{\alpha'}}(\tb))\bigg|
+\bigg|\varphi_{\sfrak}(\tau,r_{\gamma_{\alpha'}}(\tb))
-\frac{4\NPCP{1}}{\tb(v(r_{\gamma_{\alpha'}}(\tb)))^2}\bigg|
+\bigg|\frac{4\NPCP{1}}{\tb(v(r_{\gamma_{\alpha'}}(\tb)))^2}
-\frac{4\NPCP{1}}{\tb v^2}\bigg|\notag\\
\lesssim{}&v^{-3-\frac{\delta}{32}} \FB_{\delta}
+v^{-3-\frac{\delta}{8}}\FBT
+v^{-3-\frac{\delta}{16}}\abs{\NPCP{1}}\notag\\
\lesssim{}&v^{-3-\frac{\delta}{32}}\FBT.
\end{align}

In the end, we consider $r\leq R$ region. Integrating along $\Sigmatb$ from the point $(\tb, R)$, and utilizing the estimate \eqref{eq:Pricelaw:posi:bigr} at $r=R$,  the estimate \eqref{dec:almostPrice:posi}  for $\prb \varphi_{\sfrak}$, and the estimate \eqref{eq:rela:g}, we obtain
\begin{align}\begin{split}
\hspace{2ex}&\hspace{-2ex}
\abs{\varphi_{\sfrak}(\tau,r)
-4\tau^{-1}v^{-2}\NPCP{1}}\\
&\leq{}|\varphi_{\sfrak}(\tau,R)
-4\tau^{-1}(v|_{\Sigma_\tau}(R))^{-2}\NPCP{1}|\\
&\quad
+|(4\tau^{-1}(v|_{\Sigma_\tau}(R))^{-2}-4\tau^{-1}v^{-2})\NPCP{1}|
+\bigg|\int_r^{R}\partial_\rho\varphi_{\sfrak}(\tau,\rho)\di \rho\bigg|\\
&\lesssim_{R}{}v^{-2}\tau^{-1-\frac{\delta}{32}}\FBT
+v^{-4}\abs{\NPCP{1}}
+v^{-3}\tau^{-1+\frac{\delta}{2}}\FB_{\delta} \lesssim_R v^{-2}\tau^{-1-\frac{\delta}{32}}\FBT.
\end{split}
\end{align}

In summary, we achieve the following estimate.
\begin{prop}
Assume on $\Sigma_{\tau_0}$ that there are constants $\beta\in(0,\half)$ and $D_0$ such that for $r\geq R$,
\begin{align}
\bigg|\rb^2\VR\hatPhips(\tau_0,v)-\NPCP{1}\bigg|\lesssim v^{-\beta}D_0.
\end{align}
Then for any $0<\delta\leq \min\{0.4, 2\beta\}$,  we have in $\Dzeroinfty$ that
\begin{align}
\abs{\varphi_{\sfrak}(\tau,r)
-4\tau^{-1}v^{-2}\NPCP{1}}
\lesssim{}v^{-2}\tau^{-1-\frac{\delta}{32}}\FBT.
\end{align}
\end{prop}

%%%%%%%%%%%
\subsubsection{Asymptotics of $\Lxi^j\varphi_{\sfrak}$}
\label{sect:asymofvarphi+:high}
%%%%%%%%%%

We proceed to obtain precise behaviours for $\Lxi^j\varphi_{\sfrak}$. Applying $\pv^i$ to equation \eqref{eq:PhipsHigh1:l=1:v1} gives
\begin{align}
\label{eq:PhipsHigh1:l=1:v1:high}
\pu(\pv^i(\mu^{\half} r^{-1}V\hatPhips) ) ={}-{3M}\pv^i(\mu^{\frac{3}{2}} r^{-4}\hatPhips).
\end{align}

\begin{lemma}
Assume on $\Sigma_{\tau_0}$ that for any $i\in\mathbb{N}$, there exist  constants $\beta \in (0,\half)$ and $D_0$  such that for all $0\leq i'\leq i$ and $r\geq R$,
\begin{align}
\label{assum:initialdatawithNPconst:high}
\bigg|\prb^{i'}\bigg(\rb^2\VR\hatPhips({\tb_0},v)
-\NPCP{1}\bigg)\bigg|\lesssim v^{-{i'}-\beta}D_0.
\end{align}
Then for $v-u\geq v^{\alpha_i}$ with $\frac{i+2}{i+3}<{\alpha_i}<1$ and any  $0<\eta<-1-i-\frac{\delta}{2}+(2+i){\alpha_i}$,
\begin{align}
\label{eq:VPhinsHigh1:awayregion:high}
|v^{i+3}\pv^i(\mu^\frac{3}{2} r^{-1}\VR\hatPhips(u,v))-4(-1)^i (i+2)!\NPCP{1}|
\lesssim (v^{-\beta}+v^{-\eta})\FBT.
\end{align}
\end{lemma}

\begin{proof}
For $v-u\geq v^{\alpha_i}$, one can integrate equation \eqref{eq:PhipsHigh1:l=1:v1:high} along $v=const$ and estimate the integral on the RHS of equation \eqref{eq:PhipsHigh1:l=1:v1:high} by Theorem \ref{thm:almostPrice} and \eqref{eq:rela:a}:
\begin{align}
\label{eq:VPhinsHigh1:awayregion:n:high}\begin{split}
\bigg|v^{3+i}\int_{u_{\Sigma_{\tau_0}}(v)}^u
\partial_v^i(\mu^{\frac{3}{2}}r^{-4}\hatPhips(u',v))\di u'\bigg|&\lesssim v\int_{u_{\Sigma_{\tau_0}}(v)}^u
\mu^{\half}\Delta^{-\half}r^{-1}r^{-i}v^i\tau^{-1+\frac{\delta}{2}}(u',v)\di u'\FB_{\delta}\\
&=v^{-\eta}\int_{u_{\Sigma_{\tau_0}}(v)}^u r^{-2-i}v^{1+i+\eta}\tau^{-1+\frac{\delta}{2}}(u',v)\di u'\FB_{\delta}\\
&\lesssim v^{-\eta}\int_{u_{\Sigma_{\tau_0}}(v)}^u r^{-2-i}v^{1+i+\eta}(u')^{-1+\frac{\delta}{2}}(u',v)\di u'\FB_{\delta}\\
&\lesssim v^{-\eta}\int_{u_{\Sigma_{\tau_0}}(v)}^u
r^{-2-i+\frac{1+i+\eta}{\alpha_i}}
(u')^{-1+\frac{\delta}{2}}(u',v)du'\FB_{\delta}\\
&\lesssim v^{-\eta}\int_{u_{\Sigma_{\tau_0}}(v)}^u
(u')^{-(2+i){\alpha_i}+i+\eta+\frac{\delta}{2}}(u',v)\di u'\FB_{\delta}\\
&\lesssim v^{-\eta}\FB_{\delta},
\end{split}\end{align}
where in the last step we used $-(2+i){\alpha_i}+i+\eta+\frac{\delta}{2}<-1$ which holds true by assumption. Therefore,
\begin{align}
\label{eq:VPhinsHigh1:awayregion:1:high}
\abs{v^{3+i}\pv^i(\mu^{\frac{3}{2}} r^{-1}\VR\hatPhips(u,v))-v^{3+i}\pv^i(\mu^\frac{3}{2} r^{-1}\VR\hatPhips(u_{\Sigma_{\tau_0}}(v),v))}\lesssim{}v^{-\eta}\FB_{\delta}.
\end{align}
On the other hand,  we have
\begin{align}
\pv^i(\mu^\frac{3}{2} r^{-1}\VR\hatPhips(\tau_0,v))={}&\pv^{i}(v^{-3}\cdot \mu^{\frac{3}{2}}r^{-1}v v^2\VR\hatPhips(\tau_0,v))\notag\\
={}&\sum_{j=0}^i (-1)^j \half \frac{i!}{j!(i-j)!}(j+2)!v^{-3-j}\pv^{i-j}
(\mu^{\frac{3}{2}}r^{-1}v v^2\VR\hatPhips(\tau_0,v)).
\end{align}
Then
in view of the assumption \eqref{assum:initialdatawithNPconst:high}, the estimate \eqref{asymp-factor}, and $\partial_v=\half\mu(\partial_\rho+(2\mu^{-1}-\partial_rh(r))\Lxi)$, one obtains
\begin{align}
\label{eq:VPhinsHigh1:awayregion:2:high}
|\pv^i(\mu^\frac{3}{2} r^{-1}\VR\hatPhips(\tau_0,v))-4(-1)^i (i+2)! v^{-3-i}\NPCP{1}|
\lesssim v^{-\beta-i-3}D_0+v^{-4-i}(\abs{\NPCP{1}}+\FB_{\delta}).
\end{align}
The estimate \eqref{eq:VPhinsHigh1:awayregion:high} thus follows from the estimates \eqref{eq:VPhinsHigh1:awayregion:1:high} and \eqref{eq:VPhinsHigh1:awayregion:2:high}.
\end{proof}

Now we estimate $\Lxi^i\hatPhips$ in $v-u\geq v^{\alpha_i}$. In the estimate \eqref{eq:VPhinsHigh1:awayregion:high}, one can write $\pv=\Lxi-\pu$ and use equation \eqref{eq:PhipsHigh1:l=1:v1:high} to estimate $\pu(\pv^{i-1}(\mu^{\half} r^{-1}V\hatPhips) )$, thus
\begin{align}
\hspace{4ex}&\hspace{-4ex}
\abs{v^{i+3}\pv^{i-1}(\mu^\frac{3}{2} r^{-1}\VR\Lxi\hatPhips(u,v))-4(-1)^i (i+2)!\NPCP{1}}\notag\\
\lesssim{}& (v^{-\beta}+v^{-\eta})\FBT
+\abs{v^{i+3}\pv^{i-1}(\mu^{\frac{3}{2}} r^{-4}\hatPhips)}\notag\\
\lesssim{}& (v^{-\beta}+v^{-\eta})\FBT
+v^{-\eta}u^{-(1+i){\alpha_i}+i+\eta+\frac{\delta}{2}}\FB_{\delta}
\lesssim{}(v^{-\beta}+v^{-\eta})\FBT,
\end{align}
where we have used  a similar argument in \eqref{eq:VPhinsHigh1:awayregion:n:high} to estimate $\abs{v^{i+3}\pv^{i-1}(\mu^{\frac{3}{2}} r^{-4}\hatPhips)}$ in the second last step and $\eta<-1-i-\frac{\delta}{2}+(2+i){\alpha_i}$ in the last step. One can inductively proceed to obtain that for any $0\leq j\leq i$,
\begin{align}
\label{eq:VPhinsHigh1:awayregion:v1:high:any}
\abs{v^{i+3}\pv^{i-j}(\mu^\frac{3}{2} r^{-1}\VR\Lxi^j\hatPhips(u,v))-4(-1)^i (i+2)!\NPCP{1}}
\lesssim{}&(v^{-\beta}+v^{-\eta})\FBT.
\end{align}
In particular, for $i=j$, one has
\begin{align}
\label{eq:VPhinsHigh1:awayregion:v1:high}
\abs{v^{j+3}(\mu^\frac{3}{2} r^{-1}\VR\Lxi^j\hatPhips(u,v))-4(-1)^j (j+2)!\NPCP{1}}
\lesssim{}&(v^{-\beta}+v^{-\eta})\FBT.
\end{align}

Denote $\gamma_{\alpha_j}=\{v-u=v^{\alpha_j}\}$. On $u=const$, we have
\begin{align}\label{eq-v-hatPhips:high}
(r^{-2}\Lxi^j\hatPhips)(u,v)=(r(u,v))^ {-2}\Lxi^j\hatPhips(u,v_{\gamma_{\alpha_j}}(u))
+\half (r(u,v))^ {-2}\int_{v_{\gamma_{\alpha_j}}(u)}^{v}
 V\Lxi^j\hatPhips (u,v')\di v'.
 \end{align}
We utilize  the estimates \eqref{eq:rela:b}, \eqref{eq:rela:c} and \eqref{eq:VPhinsHigh1:awayregion:v1:high} to estimate the last term of \eqref{eq-v-hatPhips:high}:
\begin{align}\begin{split}
\hspace{4ex}&\hspace{-4ex}
\frac{1}{2}\bigg|\int_{v_{\gamma_{\alpha_j}}(u)}^{v}
 \bigg(V\Lxi^j\hatPhips (u,v')
 -4(-1)^j(j+2)!r(v')^{-3-j}\NPCP{1}\bigg)\di v'\bigg|\\
 &\lesssim \int_{v_{\gamma_{\alpha_j}}(u)}^{v}
 (r(v')^{-3-j}((v')^{-\beta}+(v')^{-\eta})\FBT
 +(v')^{-3-j}\abs{\NPCP{1}})\di v'\\
 &\lesssim \big[(v_{\gamma_{\alpha_j}}(u))^{-1-\beta-j}
 +(v_{\gamma_{\alpha_j}}(u))^{-1-\eta-j}\big] \FBT +(v_{\gamma_{\alpha_j}}(u))^{-2-j}\abs{\NPCP{1}}\\
 &\lesssim (u^{-1-\beta-j}
 +u^{-1-\eta-j} +u^{-2-j})\FBT
\end{split}\end{align}
and
\begin{align}
\label{eq:PhinsHigh1:awayregion:v1:high:any}
\half\hspace{2ex}&\hspace{-2ex}\int_{v_{\gamma_{\alpha_j}}(u)}^{v}
4(-1)^j(j+2)!r(u,v')(v')^{-3-j}\NPCP{1}\di v'\notag\\
=&(-1)^j(j+2)!\int_{v_{\gamma_{\alpha_j}}(u)}^{v}
[((v')^{-2-j}-u (v')^{-3-j})+(2r(u,v')-(v'-u))(v')^{-3-j}]\NPCP{1}\di v'\notag\\
=& (-1)^j(j+2)!\NPCP{1}\big[((j+2)^{-1}uv^{-2-j}-(j+1)^{-1}v^{-1-j})\notag\\
&\qquad \qquad \qquad \qquad
-((j+2)^{-1}uv^{-2-j}-(j+1)^{-1}v^{-1-j})(u,v_{\gamma_{\alpha_j}}(u))\big]\notag\\
& +(-1)^j(j+2)!\int_{v_{\gamma_{\alpha_j}}(u)}^{v}(2r(u,v')-(v'-u))(v')^{-3-j}\NPCP{1}\di v'\notag\\
=& (-1)^j(j+2)!\NPCP{1}[(j+2)^{-1} u (v^{-2-j} - u^{-2-j})
-(j+1)^{-1}(v^{-1-j}-u^{-1-j})]\notag\\
&
+(-1)^j(j+2)!\NPCP{1}[(j+1)^{-1}((v_{\gamma_{\alpha_j}}(u))^{-1-j}
-u^{-1-j})
-(j+2)^{-1}u((v_{\gamma_{\alpha_j}}(u))^{-2-j}-u^{-2-j})]\notag\\
& +(-1)^j(j+2)!\int_{v_{\gamma_{\alpha_j}}(u)}^{v}
(2r(u,v')-(v'-u))(v')^{-3-j}\NPCP{1}\di v'.
\end{align}
For the third last line of  equation \eqref{eq:PhinsHigh1:awayregion:v1:high:any}, it equals
\begin{align}
\hspace{4ex}&\hspace{-4ex}
(-1)^j j!\NPCP{1}(u^{-1-j}
- (j+2)v^{-1-j}
+(j+1)uv^{-2-j})\notag\\
={}&(-1)^j j!\NPCP{1}u^{-1-j}v^{-2-j}(
v(v^{1+j}-u^{1+j})
-(j+1)u^{1+j}(v-u))\notag\\
={}&(-1)^j j!\NPCP{1}u^{-1-j}v^{-2-j}(v-u)
\sum_{n=0}^{j}u^{j-n}(v^{n+1}-u^{n+1})
\notag\\
={}&(-1)^j j!\NPCP{1}u^{-1-j}v^{-2}(v-u)^2 \sum_{n=0}^{j}\sum_{i=0}^{n}\bigg(\frac{u}{v}\bigg)^{j-i},
\end{align}
and the absolute value of the last two lines of equation \eqref{eq:PhinsHigh1:awayregion:v1:high:any} is clearly bounded using \eqref{eq:rela:b} and \eqref{eq:rela:c} by $C(u^{-2-j}(\log u)^2 +u^{-2-j+\alpha_{j}})\abs{\NPCP{1}}$.
As a result,
\begin{align}
\hspace{4ex}&\hspace{-4ex}
\bigg|\half (r(u,v))^ {-2}\int_{v_{\gamma_{\alpha_j}}(u)}^{v}
 V\Lxi^j\hatPhips (u,v')\di v'
 -(-1)^j j!u^{-1-j}v^{-2}(v-u)^2 \sum_{n=0}^{j}\sum_{i=0}^{n}\bigg(\frac{u}{v}\bigg)^{j-i}\NPCP{1}\bigg|\notag\\
\lesssim{}& r^{-2}\big(u^{-1-\beta-j}
 +u^{-1-\eta-j} +u^{-2-j}
 +u^{-2-j}(\log u)^2 +u^{-2-j+\alpha_{j}}\big)\FBT\notag\\
 \lesssim{}& r^{-2}\big(u^{-1-\beta-j}
+u^{-1-\eta-j}+u^{-2-j+\alpha_{j}}\big)\FBT.
\end{align}
For the first term on the right hand of \eqref{eq-v-hatPhips:high},
\begin{align}
|(r(u,v))^ {-2}\Lxi^j\hatPhips(u,v_{\gamma_{\alpha_j}}(u))|&\lesssim \mu^{-\half}(r(u,v))^ {-2} (r(u,v_{\gamma_{\alpha_j}}(u)))^ {2}\cdot (v^{-2}\tau^{-1-j+\frac{\delta}{2}})
(u,v_{\gamma_{\alpha_j}}(u))\FB_{\delta}\notag\\
&\lesssim (r(u,v))^{-2}u^{-3+2\alpha_j-j+\frac{\delta}{2}}\FB_{\delta}.
\end{align}
Hence, we conclude
\begin{lemma}
For $v-u\geq v^{\alpha_j}$ with $\frac{j+2}{j+3}<{\alpha_j}<1$, and $0<\eta<-1-j-\frac{\delta}{2}+(2+j){\alpha_j}$, we have
\begin{align}\label{asymp-hatPhips-near-nullinfinity:high}
\hspace{4ex}&\hspace{-4ex}
\bigg|\Lxi^j\varphi_{\sfrak}
-4 (-1)^j j!u^{-1-j}v^{-2}  \sum_{n=0}^{j}\sum_{i=0}^{n}\bigg(\frac{u}{v}\bigg)^{j-i}\NPCP{1}\bigg|\notag\\
\lesssim {}&r^{-2}\big(u^{-1-\beta-j}
+u^{-1-\eta-j}
+u^{-3+2{\alpha_j}-j+\frac{\delta}{2}}
+u^{-2-j+\alpha_{j}}\big)\FBT.
\end{align}
\end{lemma}

\begin{proof}
The above discussions imply that the estimate \eqref{asymp-hatPhips-near-nullinfinity:high} holds with $\varphi_{\sfrak}$ replaced by $r^{-2}\PhipsHigh{1}$. It remains to estimate the difference between these two scalars.
By definition, $\varphi_{\sfrak}=r(r-M)^{-1}\mu^{\half}r^{-2}\PhipsHigh{1}$, and $\abs{r(r-M)^{-1}\mu^{\half}-1}\lesssim M^2r^{-2}$, thus the estimate \eqref{asymp-hatPhips-near-nullinfinity:high} follows.
\end{proof}

We then consider the asymptotics in the entire region $\Dzeroinfty$.
For $\{v-u\geq v^{\alpha_j}\}\cap\{v-u\geq\frac{v}{2}\}$, the RHS of \eqref{asymp-hatPhips-near-nullinfinity:high} is bounded by
\begin{align}
\label{eq:pricelaw:errorterm:veryfarregion:high}
Cv^{-2}\big(u^{-1-\beta-j}
+u^{-1-\eta-j}
+u^{-3+2{\alpha_j}-j+\frac{\delta}{2}}
+u^{-2-j+\alpha_{j}}\big)\FBT.
\end{align}
For $\{v-u\geq v^{\alpha'_j}\}\cap\{v-u\leq\frac{v}{2}\}$, where $\alpha'_j\in (\alpha_j,1)$, then  we can use \eqref{eq:rela:e} to bound the RHS of \eqref{asymp-hatPhips-near-nullinfinity:high} by
\begin{align}
\label{eq:pricelaw:errorterm:farregion:high}
Cv^{-2{\alpha'_j}}\big(v^{-1-\beta-j} +v^{-1-\eta-j}
+v^{-3+2{\alpha_j}-j+\frac{\delta}{2}}
+v^{-2-j+\alpha_{j}}\big)\FBT.
\end{align}
By taking $\alpha_j<1-\frac{\delta}{4}$ and ${\alpha'_j}\in (\max\{1-\frac{\eta}{2}, 1-\frac{\beta}{2},\frac{\alpha_j+1}{2}\}, 1)$, there exists a constant $\epsilon>0$ such that the expressions \eqref{eq:pricelaw:errorterm:veryfarregion:high} and \eqref{eq:pricelaw:errorterm:farregion:high} are bounded by $Cv^{-2}u^{-1-j-\epsilon}\FBT$. We have moreover that $\abs{\tb- u}\lesssim 1$ in $\{v-u\geq v^{\alpha'_j}\}$. Thus, the estimate \eqref{asymp-hatPhips-near-nullinfinity:high} yields that in $\{v-u\geq v^{\alpha'_j}\}$,
\begin{align}
\bigg|\Lxi^j\varphi_{\sfrak}
-4 (-1)^j j!\tb^{-1-j}v^{-2}  \sum_{n=0}^{j}\sum_{i=0}^{n}\bigg(\frac{\tb}{v}\bigg)^{j-i}\NPCP{1}\bigg|
\lesssim {}& v^{-2}\tb^{-1-j-\epsilon}\FBT.
\end{align}

For $\{r\geq R\}\cap\{v-u\leq v^{\alpha'_j}\}$, we integrate along $\Sigma_\tau$ from a point $(\tb,r_{\gamma_{\alpha_j'}}(\tb))\in \gamma_{\alpha_j'}= \{v-u=v^{\alpha'_j}\}$:
\begin{align}
\Lxi^j\varphi_{\sfrak}(\tau,r)={}\Lxi^j\varphi_{\sfrak}(\tau,r_{\gamma_{\alpha_j'}}(\tb))
-\int_r^{r_{\gamma_{\alpha_j'}}(\tb)}\partial_\rho\Lxi^j\varphi_{\sfrak}(\tau,\rho)\di \rho.
\end{align}
Note that $\abs{\tb-u|_{(\tau,r_{\gamma_{\alpha_j'}}(\tb))}}\lesssim 1$, $ u|_{(\tau,r_{\gamma_{\alpha_j'}}(\tb))}\sim v$, $v(r_{\gamma_{\alpha_j'}}(\tb))\sim u\sim v$ and $\abs{v(r_1)-v}\lesssim v^{\alpha'_j}$. Moreover, we have $\di \rho\sim \di v$ on $\Sigma_\tau$, thus using the estimate \eqref{dec:almostPrice:nega} but with $\delta$ replaced by a $\delta'\in (0,\delta)$ to be fixed gives
\begin{align}\begin{split}
\int_r^{r_{\gamma_{\alpha_j'}}(\tb)}|\partial_\rho\Lxi^j\varphi_{\sfrak}|(\tau,\rho)\di \rho
\lesssim {}&\tau^{-1-j+\frac{\delta'}{2}}\int_r^{r_{\gamma_{\alpha_j'}}(\tb)} v^{-3}|_{\Sigma_\tau}\di \rho\\
\lesssim {}&v^{-1-j+\frac{\delta'}{2}}(v^{-2}-(v(r_{\gamma_{\alpha_j'}}(\tb)))^{-2})\FB_{\delta'}
\\
\lesssim {}&v^{-1-j+\frac{\delta'}{2}}\frac{(v(r_{\gamma_{\alpha_j'}}(\tb))-v)
(v(r_{\gamma_{\alpha_j'}}(\tb))+v)}{v^2(v(r_{\gamma_{\alpha_j'}}(\tb)))^2}\FB_{\delta'}
\\
\lesssim {}&v^{-4-j+\frac{\delta'}{2}+\alpha'_j}\FB_{\delta'}.
\end{split}\end{align}
By taking $\delta'<2(1-\alpha'_j)$, the above is bounded by $Cv^{-3-j-\epsilon}\FB_{\delta'}$ for some $\epsilon>0$.
Hence,
\begin{align}
\label{eq:Pricelaw:posi:bigr:high}
\hspace{4ex}&\hspace{-4ex}
\bigg|\Lxi^j\varphi_{\sfrak}(\tau,r)
-4 (-1)^j j!\tb^{-1-j}v^{-2}  \sum_{n=0}^{j}\sum_{i=0}^{n}
\bigg(\frac{\tb}{v}\bigg)^{j-i}\NPCP{1}\bigg|\notag\\
\leq {}&|\Lxi^j\varphi_{\sfrak}(\tau,r)
-\Lxi^j\varphi_{\sfrak}(\tau,r_{\gamma_{\alpha_j'}}(\tb))|\notag\\
&
+\bigg|\Lxi^j\varphi_{\sfrak}(\tau,r_{\gamma_{\alpha_j'}}(\tb))
-4 (-1)^j j!\tb^{-1-j}(v(r_{\gamma_{\alpha_j'}}(\tb)))^{-2}  \sum_{n=0}^{j}\sum_{i=0}^{n}
\bigg(\frac{\tb}{v(r_{\gamma_{\alpha_j'}}(\tb))}\bigg)^{j-i}\NPCP{1}\bigg|\notag\\
&
+\bigg|4 (-1)^j j!u^{-1-j}\NPCP{1}\bigg((v(r_{\gamma_{\alpha_j'}}(\tb)))^{-2}  \sum_{n=0}^{j}\sum_{i=0}^{n}
\bigg(\frac{\tb}{v(r_{\gamma_{\alpha_j'}}(\tb))}\bigg)^{j-i}
-v^{-2}  \sum_{n=0}^{j}\sum_{i=0}^{n}
\bigg(\frac{\tb}{v}\bigg)^{j-i}\bigg)\bigg|\notag\\
\lesssim{}&v^{-3-j-\epsilon}\FBT
+v^{-3-j-\epsilon}\FB_{\delta'}
+v^{-4-j+\alpha'_j}\abs{\NPCP{1}}\notag\\
\lesssim{}&v^{-3-j-\epsilon}\FBT.
\end{align}

In the end, we consider $r\leq R$ region. Integrating along $\Sigmatb$ from the point $(\tb, R)$ and utilizing the estimate \eqref{eq:Pricelaw:posi:bigr:high} at $r=R$, the estimate \eqref{dec:almostPrice:posi}  for $\prb \varphi_{\sfrak}$ and the estimate \eqref{eq:rela:g}, we obtain
\begin{align}\begin{split}
\hspace{2ex}&\hspace{-2ex}
\abs{\Lxi^j\varphi_{\sfrak}(\tau,r)
-2 (-1)^j (j+2)!\tau^{-1}v^{-2}\NPCP{1}}\\
&\leq{}\big|\Lxi^j\varphi_{\sfrak}(\tau,R)
-2 (-1)^j (j+2)!\tau^{-1}(v|_{\Sigma_\tau}(R))^{-2}\NPCP{1}\big|\\
&\quad+\big|(2 (-1)^j (j+2)!\tau^{-1}(v|_{\Sigma_\tau}(R))^{-2}-2 (-1)^j (j+2)!\tau^{-1}v^{-2})\NPCP{1}\big|\\
&\quad
+\bigg|\int_r^{R}\partial_\rho\Lxi^j\varphi_{\sfrak}(\tau,\rho)\di \rho\bigg|\\
&\lesssim_{R}{}(v^{-2}\tau^{-1-j-\epsilon}
+v^{-3}\tau^{-1-j+\frac{\delta}{2}})\FBT\notag\\
&\lesssim_{R}{} v^{-2}\tau^{-1-j-\epsilon}\FBT,
\end{split}
\end{align}
where we have used that on $\Sigmatb\cap \{\rb=R\}$,
\begin{align}
\bigg|2 (-1)^j (j+2)!\tau^{-1-j}v^{-2}\NPCP{1}
-4 (-1)^j j!\tb^{-1-j}v^{-2}  \sum_{n=0}^{j}\sum_{i=0}^{n}
\bigg(\frac{\tb}{v}\bigg)^{j-i}\NPCP{1}\big|
\lesssim {}&v^{-4-j}\abs{\NPCP{1}}.
\end{align}

In summary, we achieve the following estimate.
\begin{thm}
\label{thm:pricelaw:ell=1:posi:high}
Let $j\in \mathbb{N}$. Assume on $\Sigma_{\tau_0}$ there are constants $\beta\in (0,\half)$ and $D_0$ such that for $r\geq R$ and all $0\leq i\leq j$,
\begin{align}
\label{eq:assum:Pricelaw:posi}
\bigg|\prb^i\bigg(\rb^2\VR\hatPhips(\tau_0,v)
-\NPCP{1}\bigg)\bigg|\lesssim v^{-i-\beta}D_0.
\end{align}
Then there exists an $\epsilon>0$ such that at any point in $\Dzeroinfty$,
\begin{align}
\label{eq:pricelaw:ell=1:posi:high}
\bigg|\Lxi^j\varphi_{\sfrak}
-4 (-1)^j j!\tb^{-1-j}v^{-2}  \sum_{n=0}^{j}\sum_{i=0}^{n}
\bigg(\frac{\tb}{v}\bigg)^{j-i}\NPCP{1}\bigg|
\lesssim{}&v^{-2}\tb^{-1-j-\epsilon}\FBT.
\end{align}
\end{thm}

%%%%%%%%%%%%%%%%%
\subsection{Spin $-\half$ component}
%%%%%%%%%%%%%%%%%%%%%%%%%%

For each $(m,\ell=1)$ mode of the spin $-\half$ component, we shall also consider its asymptotics in separate regions.

From the Dirac equations \eqref{eq:Dirac:TMEscalar}, one has
\begin{align}
\psins=-(r-M)Y\varphi_{\sfrak}+\varphi_{\sfrak}.
\end{align}
Commuting with the Killing vector $\Lxi$ gives
\begin{align}
\label{eq:Dirac:angularpsins:highorder:v1}
\Lxi^j\psins=-(r-M)Y\Lxi^j\varphi_{\sfrak}+\Lxi^j\varphi_{\sfrak}.
\end{align}

In the region $\{v-u\geq v^{\alpha'_j}\}$, one can rewrite equation \eqref{eq:Dirac:angularpsins:highorder:v1} as
\begin{align}
\label{eq:Dirac:angularpsins:highorder:v2}
\Lxi^j\psins={}&-(r-M)\mu^{-1}(2\Lxi-V)\Lxi^j\varphi_{\sfrak}
+\Lxi^j\varphi_{\sfrak}\notag\\
={}&-2(r-M)\mu^{-1}\Lxi^{j+1}\varphi_{\sfrak}
+(r-M)\VR\Lxi^j\varphi_{\sfrak}
+\Lxi^j\varphi_{\sfrak}\notag\\
={}&-2(r-M)\mu^{-1}\Lxi^{j+1}\varphi_{\sfrak}
+\mu^{\half}r^{-1}\VR \Lxi^j \PhipsHigh{1}\notag\\
&
+(r-M)\partial_r(\mu^{\half}(r-M)^{-1}r^{-1})\Lxi^j\PhipsHigh{1}
+\mu^{\half}(r-M)^{-1}r^{-1}\Lxi^j\PhipsHigh{1},
\end{align}
where in the last step we used $\varphi_{\sfrak}=(r-M)^{-1}\mu^{\half}r^{-1}\PhipsHigh{1}$. As a result,
\begin{align}
\label{eq:PhinsHigh1:awayregion:v1:high:pp}
\big|\Lxi^j\psins- \big(\mu^{\frac{3}{2}}r^{-1}\VR\Lxi^j\PhipsHigh{1}
-\Lxi^j\varphi_{\sfrak}
-(v-u)\Lxi^{j+1}\varphi_{\sfrak}\big)\big|
\lesssim{}& r^{-2}\abs{\VR \Lxi^j\PhipsHigh{1}}
+r^{-1}\abs{\Lxi^j\varphi_{\sfrak}}
+\log r\abs{\Lxi^{j+1}\varphi_{\sfrak}}.
\end{align}
We collect the estimates of the terms in the round bracket on the  LHS here: the estimate \eqref{eq:VPhinsHigh1:awayregion:v1:high} gives
\begin{subequations}
\begin{align}
\label{eq:VPhinsHigh1:awayregion:v1:high:p}
\abs{(\mu^\frac{3}{2} r^{-1}\VR\Lxi^j\hatPhips(u,v))
-4(-1)^j (j+2)!v^{-3-j}\NPCP{1}}
\lesssim{}&(v^{-3-j-\beta}+v^{-3-j-\eta})\FBT;
\end{align}
and the estimate \eqref{eq:pricelaw:ell=1:posi:high} gives  \begin{align}\label{asymp-hatPhips-near-nullinfinity:high:p1}
\hspace{4ex}&\hspace{-4ex}
\bigg|-\Lxi^j\varphi_{\sfrak}
-(-4) (-1)^j j!\tb^{-1-j}v^{-2}  \sum_{n=0}^{j}\sum_{i=0}^{n}
\bigg(\frac{\tb}{v}\bigg)^{j-i}\NPCP{1}\bigg|\notag\\
\lesssim {}&v^{-2}\tb^{-1-j-\epsilon}\FBT
\lesssim{}v^{-1}\tb^{-2-j-\epsilon}\FBT,\\
\label{asymp-hatPhips-near-nullinfinity:high:p2}
\hspace{4ex}&\hspace{-4ex}
\bigg|-(v-u)\Lxi^{j+1}\varphi_{\sfrak}
 -(-4) (-1)^{j+1} (j+1)!(v-u)\tb^{-2-j}v^{-2}  \sum_{n=0}^{j+1}\sum_{i=0}^{n}
\bigg(\frac{\tb}{v}\bigg)^{j+1-i}\NPCP{1}\bigg|\notag\\
\lesssim {}&(v-u)v^{-2}\tb^{-2-j-\epsilon}\FBT
\lesssim{}v^{-1}\tb^{-2-j-\epsilon}\FBT.
\end{align}
\end{subequations}
Summing up the above three estimates, one finds from \eqref{eq:PhinsHigh1:awayregion:v1:high:pp} that
\begin{align}
\label{eq:PhinsHigh1:awayregion:v1:high:ppp}
\hspace{4ex}&\hspace{-4ex}
\bigg|\Lxi^j\psins
-4(-1)^j j! v^{-1}\tb^{-2-j}\NPCP{1}
\bigg[
(j+1)(j+2)\bigg(\frac{\tb}{v}\bigg)^{j+2}
-\frac{\tb}{v}\sum_{n=0}^{j}\sum_{i=0}^{n}
\bigg(\frac{\tb}{v}\bigg)^{j-i}\notag\\
&\qquad \qquad \qquad \qquad \qquad \qquad\quad
+(j+1)\bigg(1-\frac{\tb}{v}\bigg)
\sum_{n=0}^{j+1}\sum_{i=0}^{n}
\bigg(\frac{\tb}{v}\bigg)^{j+1-i}
\bigg]
\bigg|\notag\\
\lesssim{}&\big(\abs{\prb \Lxi^j\varphi_{\sfrak}}
+\abs{r^{-1}\Lxi^j\varphi_{\sfrak}}
+\abs{\log v\Lxi^{j+1}\varphi_{\sfrak}}+v^{-3-j-\beta}+v^{-3-j-\eta}
+v^{-1}\tb^{-2-j-\epsilon}\big)\FBT\notag\\
\lesssim{}&\big(v^{-2-\alpha'_j}\tb^{-1-j+\frac{\delta'}{2}}
+\log v v^{-2}\tb^{-2-j+\frac{\delta'}{2}}
+v^{-3-j-\beta}+v^{-3-j-\eta}
+v^{-1}\tb^{-2-j-\epsilon}\big)\FBT\notag\\
\lesssim{}& v^{-1}\tb^{-2-j-\epsilon}\FBT
\end{align}
for some $\epsilon>0$.
Furthermore, simple but tedious calculations show that the terms in the square bracket on the LHS of \eqref{eq:PhinsHigh1:awayregion:v1:high:ppp} equal
\begin{align}
\hspace{4ex}&\hspace{-4ex}
(j+1)\bigg(\bigg(\frac{\tb}{v}\bigg)^{j}
-\bigg(\frac{\tb}{v}\bigg)^{j+2}\bigg)
+(j+1)(j+2)\bigg(\frac{\tb}{v}\bigg)^{j+1}\notag\\
\hspace{4ex}&\hspace{-4ex}
+(j+1)\sum_{n=0}^j \sum_{i=0}^n\bigg(\frac{\tb}{v}\bigg)^{j-i}
-(j+2)\frac{\tb}{v}\sum_{n=0}^j \sum_{i=0}^n \bigg(\frac{\tb}{v}\bigg)^{j-i}\notag\\
={}&(j+1)\bigg(\bigg(\frac{\tb}{v}\bigg)^{j}
-\bigg(\frac{\tb}{v}\bigg)^{j+2}\bigg)
+(j+2)\sum_{n=0}^j \bigg(\frac{\tb}{v}\bigg)^{j-n}
-\sum_{n=0}^j \sum_{i=0}^n\bigg(\frac{\tb}{v}\bigg)^{j-i}.
\end{align}
We can thus obtain that for $v-u\geq v^{\alpha'_j}$, there is an $\epsilon>0$ such that
\begin{align}
\label{eq:PhinsHigh1:awayregion:v1:high:p4}
\hspace{4ex}&\hspace{-4ex}
\bigg|\Lxi^j\psins
-4(-1)^j j! v^{-1}\tb^{-2-j}\NPCP{1}
\bigg[
(j+2)\sum_{n=0}^j \bigg(\frac{\tb}{v}\bigg)^{j-n}
-\sum_{n=0}^j \sum_{i=0}^n\bigg(\frac{\tb}{v}\bigg)^{j-i}\notag\\
&\qquad \qquad\qquad\qquad\qquad\qquad\quad
+(j+1)\bigg(\bigg(\frac{\tb}{v}\bigg)^{j}
-\bigg(\frac{\tb}{v}\bigg)^{j+2}\bigg)
\bigg]
\bigg|
\lesssim v^{-1}\tb^{-2-j-\epsilon}\FBT.
\end{align}

In the region where $\{v-u\leq v^{\alpha'_j}\}$, one can make use of equation \eqref{eq:YV:hyper} to rewrite the $Y$ derivative in \eqref{eq:Dirac:angularpsins:highorder:v1} and obtain
\begin{align}
\label{eq:Dirac:angularpsins:highorder}
\Lxi^j\psins=\Lxi^j\varphi_{\sfrak}
+(r-M)\prb\Lxi^j\varphi_{\sfrak}
-(r-M)\partial_r h\cdot \Lxi^{j+1}\varphi_{\sfrak}.
\end{align}
In this region, one has $cv\leq \tb\leq v$ and $r\leq C v^{\alpha'_j}$, and $c\leq \abs{\partial_r h}\leq C$. The absolute values of the last two terms on the RHS of \eqref{eq:Dirac:angularpsins:highorder} are bounded by $rv^{-2}\tb^{-2-j+\frac{\delta'}{2}}\FB_{\delta'}\lesssim v^{-4-j+\alpha'_j+\frac{\delta'}{2}}\FB_{\delta'}\lesssim v^{-3-j-\epsilon}\FB_{\delta'}$ for an $\epsilon>0$ since $\delta'<2(1-\alpha'_j)$. The asymptotics of the first term on the RHS of \eqref{eq:Dirac:angularpsins:highorder} are given by Theorem \ref{thm:pricelaw:ell=1:posi:high}. Therefore, in this region, we have
\begin{align}
\label{eq:PhinsHigh1:nearregion:v1:high:p4}
\bigg|\Lxi^j\psins(\tau,r)
-4 (-1)^j j!\tb^{-1-j}v^{-2}  \sum_{n=0}^{j}\sum_{i=0}^{n}
\bigg(\frac{\tb}{v}\bigg)^{j-i}\NPCP{1}\bigg|
\lesssim{}&v^{-2}\tb^{-1-j-\epsilon}\FBT.
\end{align}
The two estimates \eqref{eq:PhinsHigh1:awayregion:v1:high:p4} and \eqref{eq:PhinsHigh1:nearregion:v1:high:p4} together imply the following asymptotics for the spin $-\half$ component:
\begin{thm}
\label{thm:pricelaw:ell=1:nega:high}
Let $j\in \mathbb{N}$. Assume on $\Sigma_{\tau_0}$ that there are constants $\beta\in (0,\half)$ and $D_0$ such that for $r\geq R$ and any $0\leq i\leq j$,
\begin{align}
\label{eq:assum:Pricelaw:nega}
\bigg|\prb^i\bigg(\rb^2\VR\hatPhips(\tau_0,v)
-\NPCP{1}\bigg)\bigg|\lesssim v^{-i-\beta}D_0.
\end{align}
Then there exists an $\epsilon>0$ such that in $\Dzeroinfty$,
\begin{align}
\label{eq:Pricelaw:TNfoli:nega}
\hspace{4ex}&\hspace{-4ex}
\bigg|\Lxi^j\psins
-4(-1)^j j! v^{-1}\tb^{-2-j}\NPCP{1}
\bigg[
(j+2)\sum_{n=0}^j \bigg(\frac{\tb}{v}\bigg)^{j-n}
-\sum_{n=0}^j \sum_{i=0}^n\bigg(\frac{\tb}{v}\bigg)^{j-i}\notag\\
&\qquad \qquad\qquad\qquad\qquad\qquad\quad
+(j+1)\bigg(\bigg(\frac{\tb}{v}\bigg)^{j}
-\bigg(\frac{\tb}{v}\bigg)^{j+2}\bigg)
\bigg]
\bigg|
\lesssim v^{-1}\tb^{-2-j-\epsilon}\FBT.
\end{align}
\end{thm}

\begin{remark}
There is a different way of proving the Price's law for the spin $-\half$ component, that is, by the same argument as in Section \ref{sect:Pricelaw:ell=1:posi} applied to equation \eqref{eq:simpleformofeqofpsins}.  Both ways lead to the same asymptotics, however, the way of presenting here is simpler and much natural in the sense that we simply make use of the Dirac equations connecting these two components to obtain the asymptotics for one component from the other one.
\end{remark}

%%%%%%%%%%%%%%%%
\subsection{Proof of Theorem \ref{thm:pricelaw:nonzeroNP}}
\label{sect:Pricelaw:nonzeroNP:translation}
%%%%%%%%%%%

For the Dirac field on a Schwarzschild background, one can decompose the spin $\pm \half$ components into $\ell=1$ mode and $\ell\geq 2$ modes. For $\ell\geq 2$ part $(\Psipns)^{\ell\geq 2}$, we make use of the estimates in Proposition \ref{prop:almostpricelaw:ellgeq2} to obtain for any $\epsilon\in (0,\frac{\delta}{2})$,
\begin{subequations}
\label{eq:Pricelaw:ellbigger2:nonzero}
\begin{align}
\abs{\Lxi^j ((\varphi_{\sfrak})^{\ell\geq 2})}\lesssim {}& v^{-2}\tb^{-1-j-\epsilon}
(F^{(2)}(\regl(\ell_0,j),1+2\epsilon,\tb_0,(\Psipns)^{\ell\geq 2}))^{\half}, \\
\abs{\Lxi^j ((\psins)^{\ell\geq 2})}\lesssim {}& v^{-1}\tb^{-2-j-\epsilon}
(F^{(2)}(\regl(\ell_0,j),1+2\epsilon,\tb_0,(\Psipns)^{\ell\geq 2}))^{\half}.
\end{align}
\end{subequations}
For $\ell=1$ mode, one can further decompose it into azimuthal modes $m=-\half,  \half$ as in \eqref{eq:azimuthalmodes:decomp}, and for each spin-weighted spherical mode $(m,\ell=1)$, we can  define its corresponding N--P constants  $\NPCP{1}(m,\ell=1)$ (and $\NPCN{1}(m,\ell=1)$ which is equal to $\NPCP{1}(m,\ell=1)$ by Lemma \ref{lem:relationoftwoNPconsts}) as in Definition \ref{def:NPCs}.
Then, by the main results in Theorems \ref{thm:pricelaw:ell=1:posi:high} and \ref{thm:pricelaw:ell=1:nega:high}, and together with the estimates \eqref{eq:Pricelaw:ellbigger2:nonzero} for $\ell\geq 2$ modes, this proves Theorem \ref{thm:pricelaw:nonzeroNP}.

%%%%%%%%%%%
\section{Price's law decay for vanishing first Newman--Penrose constant}
\label{sect:pricelaw:zeroNPconst}
%%%%%%%%%%

We derive in this section the asymptotic profiles for the spin $\pm \half$ components in the case that the first N--P constant vanishes. The idea is to reduce this case to the nonvanishing first N--P constant case and then apply the estimates in the previous section. As demonstrated in \cite{angelopoulos2018late} which treats the scalar field, a fundamentally important part in realizing this idea is to construct a time integral of the $\ell=0$ mode of the scalar field, which solves the scalar wave equation as well and, in particular, the time derivative of which equals the $\ell=0$ mode of the scalar field. We follow \cite{angelopoulos2018late} in Section \ref{sect:timeintegral:def} to generalize this idea to the $\ell=1$ mode of the Dirac field and calculate the first N--P constant of the time integral in terms of the initial data of the spin $\half$ component. Section \ref{sect:controllingtimeintegralenergy} is then devoted to conducting an estimate for the initial energy of the time integral by a different weighted initial energy of the spin $\half$ component. In the end, in Section \ref{sect:proofoftheorem2:NVNP}, we apply Theorem \ref{thm:pricelaw:nonzeroNP} to this time integral and complete the proof of Theorem \ref{thm:pricelaw:zeroNP}.

%%%%%%%%%%%%%%%%%
\subsection{The time integral $g_{\sfrak}$ of $(\Phips)_{m,{\ell=1}}$}
\label{sect:timeintegral:def}
%%%%%%%%%%%%%%%%%%

We consider the $\ell=1$ mode $(\Phips)^{\ell=1}$ of the spin $\half$ component and further decompose it into $(m,\ell=1)$ modes $(\Phips)_{m,{\ell=1}}$, where $m=-\half, \half$. In what below, we consider a fixed $(m,\ell=1)$ mode of the spin $\half$ component and still use the same notations as the ones of the spin $\half$ component without ambiguity. Note in particular that these components and the scalars constructed from them are thus independent of $\theta, \pb$.

Recall the equation of this mode of the spin $\half$ component $\Phips$ from \eqref{eq:prbprbPhins:ell=1:impro}
\begin{align}
\label{eq:prbprbPhins:ell=1:impro:copy}
(r-M)^{-1}\Delta^{\half}\prb((r-M)^2\Delta^{\half}\prb((r-M)^{-1}\Phips))
=H_{\sfrak}(\Phips),
\end{align}
where
\begin{align}
H_{\sfrak}={}&
\Delta(2\mu^{-1}-H)H\Lxi^2
+2\Delta (\mu^{-1} -H) \Lxi\prb
+\Delta^{\half}\partial_r (\Delta^{\half}(2\mu^{-1}-H))
\Lxi,
\end{align}
and $H=2\mu^{-1}-\partial_rh(r)$. The time integral of this mode is constructed as follows. 

\begin{lemma}
\label{lem:timeintegral:posi}
Assume that $F^{(1)}(\regl,5-\delta,\tb_0,\Psipns)<\infty$ as defined in Definition \ref{def:Fenergies:big2} for a sufficiently large $\regl$ and a $\delta\in (0,\half)$,
 and assume further that there exists a finite constant $\tilde{D}_1$ such that
\begin{align}\label{Initial:NPCP:zero}
\lim\limits_{\rho\to\infty}\rho^{3}\VR\hatPhips\big|_{\Sigmazero}\leq \tilde{D}_1.
\end{align}
Then there exists a unique smooth solution $g_{\sfrak}$ to \eqref{eq:prbprbPhins:ell=1:impro:copy} which satisfies both
\begin{align}\label{asym:gs:infty}
\lim_{\rho\to\infty}g_{\sfrak}\big|_{{\Sigmazero}}=0
\end{align}
and
\begin{align}
\label{eq:Lxigs=Phips}
\Lxi g_{\sfrak}=\Phips.
\end{align}
Such a solution $g_{\sfrak}$ satisfies the {integrability condition}:
\begin{align}\label{integrability-condition:gs}
\lim\limits_{r\to\infty}(r-M)^2\Delta^{\half}\partial_\rho((r-M)^{-1}
g_{\sfrak})|_{{\Sigmazero}}
=\int_{2M}^{\infty}\tilde{H}_{\sfrak}(\Phips)|_{{\Sigmazero}}\di\rho,
\end{align}
where $\tilde{H}_{\sfrak}(\Phips)$ is defined as in \eqref{def:tildeHs:general} and can be rewritten as
\begin{align}
\tilde{H}_{\sfrak}(\Phips)=(r-M)[r\mu^{\half}(2\mu^{-1}-H)H\Lxi\Phips
+2r\mu^{\half}(\mu^{-1}-H)\partial_\rho\Phips
+\partial_r(\Delta^{\half}(2\mu^{-1}-H))\Phips].
\end{align}
\end{lemma}

\begin{definition}
This unique smooth scalar $g_{\sfrak}$ constructed from $\Phips$ as in Lemma \ref{lem:timeintegral:posi} is called the \emph{time integral} of the spin $\half$ component $\Phips$.
\end{definition}

\begin{proof}
We first discuss the asymptotic behaviors of $\tilde{H}_{\sfrak}(\Phips)$ as $\rho\to\infty$ and $\rho\to 2M $ on $\Sigmazero$.  By $\hat{V}=\partial_\rho+H\Lxi$ and   $\hatPhips=\mu^{-\half}r\Phips$, we have
\begin{align}\label{equ:tildeH}
\tilde{H}_{\sfrak}(\Phips)=(r-M)[(2-\mu H)\VR\hatPhips-\mu^{\half}r\partial_\rho(H\Phips)
-2M\mu^{-\half}r^{-1}H\Phips].
\end{align}
From Theorem \ref{thm:almostPrice},  the assumption implies $|\Phips|+|\rb\prb\Phips|\lesssim \rho^{-1}(F^{(1)}(\regl,5-\delta,\tb_0,\Psipns))^{\half}$ as $\rho\to\infty$. Furthermore, by \eqref{Initial:NPCP:zero}, and $|H|\lesssim \rho^{-2}$, $|\partial_\rho H|\lesssim \rho^{-3}$,  we get
\begin{align}\label{asym:tildeH:Phips:infinity}
\abs{\tilde{H}_{\sfrak}(\Phips)}\lesssim \rho^{-2}(\tilde{D}_1 + (F^{(1)}(\regl,5-\delta,\tb_0,\Psipns))^{\half}).
\end{align}
From the smoothness of $\Phips$ and $\lim\limits_{r\to 2M } \partial_rh=1$,  we have, as $\rho\to 2M $,
\begin{align}\label{asym:tildeH:Phips:horizon}
\abs{\tilde{H}_{\sfrak}(\Phips)}\lesssim \mu^{-\half}(1+\mu)(F^{(1)}(\regl,5-\delta,\tb_0,\Psipns))^{\half}.
\end{align}
Therefore, the fact that the integral $\int_{2M }^{\infty}\tilde{H}_{\sfrak}(\Phips)|_{\Sigmazero}\di\rho$ exists follows from \eqref{asym:tildeH:Phips:infinity} and \eqref{asym:tildeH:Phips:horizon}.

We then determine the initial value of $g_{\sfrak}$ on $\Sigmazero$ by requiring it to be $C^1$. By \eqref{eq:prbprbPhins:ell=1:impro:copy}, we have
\begin{align}\begin{split}
\partial_\rho((r-M)^2\Delta^{\half}\prb((r-M)^{-1}g_{\sfrak}))
=(r-M)\Delta^{-\half}H_{\sfrak}(g_{\sfrak})
=\tilde{H}_{\sfrak}(\Phips),
\end{split}\end{align}
hence, for $2M <\rho<R<\infty$, we get
\begin{align}
(r-M)^2\Delta^{\half}\partial_\rho((r-M)^{-1}g_{\sfrak})(\tb_0,\rho')\Big|_{\rho}^{R}
=\int_\rho^R\tilde{H}_{\sfrak}(\Phips)(\tb_0,\rho')\di \rho'.
\end{align}
Taking $R\to\infty$, then
\begin{align}\label{integral:equ:partialgs}\begin{split}
 &\lim\limits_{R\to\infty}(r-M)^2\Delta^{\half}
 \partial_\rho((r-M)^{-1}g_{\sfrak})(\tb_0,R)
 -(r-M)^2\Delta^{\half}\partial_\rho((r-M)^{-1}g_{\sfrak})(\tb_0,\rho)\\
 =&\int_\rho^\infty\tilde{H}_{\sfrak}(\Phips)(\tb_0,\rho')\di \rho'
 =\int_{2M }^\infty\tilde{H}_{\sfrak}(\Phips)(\tb_0,\rho')\di \rho'
 -\int_{2M }^\rho\tilde{H}_{\sfrak}(\Phips)(\tb_0,\rho')\di \rho'.
 \end{split}\end{align}
In order to make sure that $g_{\sfrak}$ is actually $C^1$ at $2M $, $g_{\sfrak}$ has to satisfy the integrability condition \eqref{integrability-condition:gs} on $\Sigmazero$. In fact, if $\lim\limits_{r\to\infty}(r-M)^2\Delta^{\half}\partial_\rho((r-M)^{-1}
g_{\sfrak})|_{\Sigmazero}-\int_{2M }^{\infty}\tilde{H}_{\sfrak}(\Phips)|_{\Sigmazero}\di\rho=c\neq 0$, then this implies
\begin{align}\label{nonregu:equ:gs:horizon}
(r-M)^2\Delta^{\half}\partial_\rho((r-M)^{-1}g_{\sfrak})(\tb_0,\rho)
=-c+\int_{2M }^{\rho}\tilde{H}_{\sfrak}(\Phips)(\tb_0,\rho')\di\rho'.
\end{align}
Integrating \eqref{nonregu:equ:gs:horizon} from a fixed point $\rho_1>2M $, and by \eqref{asym:tildeH:Phips:horizon}, we get $\abs{g_{\sfrak}}\lesssim 1$ as $\rho\to 2M $.
Using \eqref{nonregu:equ:gs:horizon} and \eqref{asym:tildeH:Phips:horizon} again, then we have $\partial_\rho g_{\sfrak}\sim \mu^{-\half}$ as $\rho\to 2M $, which contradicts with the continuity of $\partial_\rho g_{\sfrak}$ at $2M $.

On the other hand, integrability condition \eqref{integrability-condition:gs} and the assumption \eqref{asym:gs:infty} together are sufficient to uniquely determine $g_{\sfrak}$ on $\Sigmazero$. From \eqref{integral:equ:partialgs} and the integrability condition \eqref{integrability-condition:gs}, we can get
\begin{align}\label{regular:equa:gs:horizon}
(r-M)^2\Delta^{\half}\partial_\rho((r-M)^{-1}g_{\sfrak})(\tb_0,\rho)
=\int_{2M }^{\rho}\tilde{H}_{\sfrak}(\Phips)(\tb_0,\rho')\di\rho'.
\end{align}
Rewrite \eqref{regular:equa:gs:horizon} as
\begin{align}
\partial_\rho((r-M)^{-1}g_{\sfrak})(\tb_0,\rho)=
\frac{1}{(r-M)^{2}\Delta^{\half}}
\bigg(\int_{2M }^{\infty}\tilde{H}_{\sfrak}(\Phips)(\tb_0,\rho')\di\rho'
-\int_{\rho}^{\infty}\tilde{H}_{\sfrak}(\Phips)(\tb_0,\rho')\di\rho'\bigg),
\end{align}
then by integrating along $\Sigmazero$ from $\rho=\infty$ and making use of the assumption \eqref{asym:gs:infty}, we can solve $g_{\sfrak}$ uniquely everywhere on $\Sigmazero$.

We now show that this unique solution $g_{\sfrak}$ can be smoothly extended to $\rho=2M$. First, by \eqref{regular:equa:gs:horizon} and \eqref{asym:tildeH:Phips:horizon}, we get, as $\rho\to 2M $,
\begin{align}
|\partial_\rho((r-M)^{-1}g_{\sfrak})(\tb_0,\rho)|\lesssim 1,
\end{align}
hence, for fixed $\rho_1>2M $, the integral $\int_{\rho}^{\rho_1}\partial_\rho((r-M)^{-1}g_{\sfrak})(\tb_0,\rho')\di \rho'$ is continuous to $2M $, thus $g_{\sfrak}$  can be continuously extended to $2M $. Second, we prove that $g_{\sfrak}$ is smooth at $2M $. We first discuss $\partial_\rho g_{\sfrak}$.
By \eqref{regular:equa:gs:horizon} and \eqref{asym:tildeH:Phips:horizon}, we have
\begin{align}\label{regular:equa:partialgs:horizon}\begin{split}
\partial_\rho g_{\sfrak}(\tb_0,\rho)&=(r-M)^{-1}g_{\sfrak}+(r-M)^{-1}\Delta^{-\half}
\int_{2M }^{\rho}\tilde{H}_{\sfrak}(\Phips)(\tb_0,\rho')\di\rho'\\
&=(r-M)^{-1}g_{\sfrak}+(r-M)^{-1}r^{-\half}
(r-2M)^{-\half}\int_{2M }^{\rho}(\rho'-2M)^{-\half}
\check{H}_{\sfrak}(\Phips)(\tb_0,\rho')\di\rho',
\end{split}\end{align}
where $\check{H}_{\sfrak}(\Phips)$ is smooth at $2M $. Hence, the limit $\lim\limits_{\rho\to 2M }\partial_\rho g_{\sfrak}(\tb_0,\rho)$ exists following from \eqref{regular:equa:partialgs:horizon}, and, as a result, the $\rho$-derivative of $g_{\sfrak}$ at $2M $ exists. In fact, $\partial_\rho g_{\sfrak}(\tb_0,2M )=\lim\limits_{\rho\to 2M }\frac{g_{\sfrak}(\tb_0,\rho)-g_{\sfrak}(\tb_0,2M )}{\rho-2M }=
\lim\limits_{\rho\to 2M }\partial_\rho g_{\sfrak}(\tb_0,\rho)$, hence $\partial_\rho g_{\sfrak}$ is continuous at $2M $.
For the higher order derivatives, we need the following property.
\begin{remark}\label{remark:smooth:function}
Let $f(r)$ is a smooth function in $[0,1]$, and $g(r)=r^{-\half}\int_0^rs^{-\half}f(s)\di s$ for $r\in (0,1]$, $g(0)=\lim\limits_{r\to 0}r^{-\half}\int_0^rs^{-\half}f(s)\di s=2f(0)$. Then, $g(r)$ is smooth in $[0,1]$. In fact, by the smoothness of $f(r)$ at $0$, for any fixed $k$, we have
\begin{align}
f(r)=\sum_{j=0}^{k}\frac{1}{j!}f^{(j)}(0)r^j+f_k(r)r^{k+1},
\end{align}
 where $f_k(r)$ is a smooth function. By the definition of $g(r)$, we get, for $r\in(0,1]$,
 \begin{align}
 g(r)=\sum_{j=0}^k\frac{1}{(j+\half)j!}f^{(j)}(0)r^j+
r^{-\half}\int_0^rf_k(s)s^{k+\half}\di s,
\end{align}
and thus,
\begin{align}
g^{(k)}(r)=\frac{1}{k+\half}f^{(k)}(0)
+\partial^k\bigg(r^{-\half}\int_0^rf_k(s)s^{k+\half}\di s\bigg).
\end{align}
The above equality implies
\begin{align}
\bigg|g^{(k)}(r)-\frac{1}{k+\half}f^{(k)}(0)\bigg|\lesssim r,
\end{align}
which yields $\lim\limits_{r\to 0}g^{(k)}(r)=\frac{2}{1+2k}f^{(k)}(0)$. On the other hand, by $g^{(k)}(0)=\lim\limits_{r\to 0}\frac{g^{(k-1)}(r)-g^{(k-1)}(0)}{r}=\lim\limits_{r\to 0}g^{(k)}(r)$, we know the derivative of $g^{(k-1)}(r)$ at $r=0$ exists, and $g^{(k)}(r)$ is continuous at $r=0$.
Hence $g(r)$ is smooth in $[0,1]$.
\end{remark}
\noindent Following from \eqref{regular:equa:partialgs:horizon} and Remark \ref{remark:smooth:function}, we know $\partial_\rho^k g_{\sfrak}$ are continuous at $2M $, for all $k\geq 0$. Thus, $g_{\sfrak}$ is smooth in $[2M ,\infty)$. By standard theory of global well-posedness of linear wave equations, such a solution $g_{\sfrak}$ is unique.
\end{proof}

Further, we collect some expressions of the time integral $g_{\sfrak}$ and derivatives of $g_{\sfrak}$ in terms of $\Phips$.
\begin{lemma}\label{lemma:asym:formu:gs}
Let $g_{\sfrak}$ be the time integral of $\Phips$ which is constructed in Lemma \ref{lem:timeintegral:posi}. Then, for $\rb\geq R$, we have
\begin{subequations}
\label{expression:gsandr2gs:copy}
\begin{align}\label{expression:gs:copy}
g_{\sfrak}(\tb_0,\rho)
={}&-\half (\rho^{-1}+M\rho^{-2}+O(\rho^{-3}))\int_{2M}^\rho\tilde{H}_{\sfrak}(\Phips)
(\tb_0,\rho')\di\rho'\notag\\
& -\half(\rb-M)\int_{\rho}^{\infty}\big( (\rb')^{-2}+O((\rb')^{-3}))\big)\tilde{H}_{\sfrak}(\Phips)
(\tb_0,\rho')\di\rho',\\
\label{eq:expression:r2Vgs}
\rho^2\hat{V}(\mu^{-\half}\rb g_{\sfrak})(\tb_0,\rho)={}&
(M+O(\rho^{-1}))\int_{2M}^\rho
\tilde{H}_{\sfrak}(\Phips)(\tb_0,\rho')\di\rho'
+(\rho^3+M\rho^2+\rb O(1))H\Phips(\tb_0,\rho)\notag\\
&-\rb^3(1+O(\rb^{-1}))\int_{\rho}^{\infty}\big( (\rb')^{-2}+O((\rb')^{-3})\big)\tilde{H}_{\sfrak}(\Phips)
(\tb_0,\rho')\di\rho'.
\end{align}
\end{subequations}
\end{lemma}
\begin{proof}
First, we integrate
\begin{align}\label{regular:equa:gs:horizon:scaled}
\partial_\rho ((\rb-M)^{-1}g_{\sfrak})(\tb_0,\rho)
=(\rho-M)^{-2}\rho^{-1}\mu^{-\half}
\int_{2M}^\rho\tilde{H}_{\sfrak}(\Phips)(\tb_0,\rho')
\di\rho'
\end{align}
from $\rho=\infty$, and then apply integration by parts, arriving at 
\begin{align}\label{expression:gs}\begin{split}
(\rho-M)^{-1}g_{\sfrak}(\tb_0,\rho)
&=-\int_{\rho}^{\infty}
(r-M)^{-2}r^{-1}\mu^{-\half}\int_{2M }^r\tilde{H}_{\sfrak}
(\Phips)(\tb_0,\rho')
\di\rho'\di r\\
&=-\int_{\rho}^{\infty}
r^{-3}(1+3Mr^{-1}+O(r^{-2}))\int_{2M }^r\tilde{H}_{\sfrak}
(\Phips)(\tb_0,\rho')
\di\rho'\di r\\
&=\bigg(\half r^{-2}+Mr^{-3}+O(r^{-4})\bigg)\int_{2M }^r\tilde{H}_{\sfrak}
(\Phips)(\tb_0,\rho')\di\rho'
\Big|_{\rho}^{\infty}\\
&\quad -\int_{\rho}^{\infty}
\bigg(\half r^{-2}+Mr^{-3}+O(r^{-4})\bigg)\tilde{H}_{\sfrak}(\Phips)
(\tb_0,\rho')\di\rho'\\
&=-\bigg(\half \rho^{-2}+M\rho^{-3}+O(\rho^{-4})\bigg)\int_{2M }^\rho\tilde{H}_{\sfrak}(\Phips)
(\tb_0,\rho')\di\rho'\\
&\quad -\int_{\rho}^{\infty}\bigg(\half (\rb')^{-2}+M(\rb')^{-3}+O((\rb')^{-4})\bigg)\tilde{H}_{\sfrak}(\Phips)
(\tb_0,\rho')\di\rho'.
\end{split}\end{align}
Here, in the derivation from the first line to the second line, we have used $\mu^{-\half}=(1-2Mr^{-1})^{-\half}=1+Mr^{-1} +O(r^{-2})$ and $(r-M)^{-2}=r^{-2}(1-Mr^{-1})^{-2}=r^{-2}(1+2Mr^{-1}+O(r^{-2}))$.
Equation \eqref{expression:gs:copy} follows directly from \eqref{expression:gs}.

Second, one has from $\hat{V}=\partial_\rho+H\Lxi$ and \eqref{eq:Lxigs=Phips} that
\begin{align}\label{equ:gs:Vgs1}
\rho^2\hat{V}(\mu^{-\half}\rb g_{\sfrak})={}&
\big[-\mu^{-\frac{3}{2}}M\rho +\mu^{-\half}(\rb-M)^{-1}\rb^2(2\rb-M) \big] g_{\sfrak}
+\mu^{-\half}\rb^3 (\rb-M)\prb((\rb-M)^{-1}g_{\sfrak})\notag\\
&
+\mu^{-\half}\rho^3H\Phips,
\end{align}
where, the term $\mu^{-\half}\rho^2\partial_\rho(\rb g_{\sfrak})$ has been rewritten in terms of $g_\sfrak$ and $\prb((\rb-M)^{-1}g_{\sfrak})$. One can further use \eqref{regular:equa:gs:horizon} to express $\prb((\rb-M)^{-1}g_{\sfrak})$, eventually leading to
\begin{align}
\rho^2\hat{V}(\mu^{-\half}\rb g_{\sfrak})={}&
\big[-\mu^{-\frac{3}{2}}M\rho +\mu^{-\half}(\rb-M)^{-1}\rb^2(2\rb-M) \big] g_{\sfrak}
+\mu^{-\half}\rho^3H\Phips\notag\\
&+\mu^{-1}\rho^2(\rb-M)^{-1}\int_{2M }^\rho
\tilde{H}_{\sfrak}(\Phips)(\tb_0,\rho')\di\rho'\notag\\
={}&(2\rb^2+2M\rho+O(1)) g_{\sfrak}+\mu^{-\half}\rho^3H\Phips
+\mu^{-1}\rho^2(\rb-M)^{-1}\int_{2M }^\rho
\tilde{H}_{\sfrak}(\Phips)(\tb_0,\rho')\di\rho'.
\end{align}
In view of the expression \eqref{expression:gs:copy}, this further equals
\begin{align}\label{eq:r2VRrgsfrak:explicit}
\hspace{4ex}&\hspace{-4ex}
-(\rb+2M+O(\rho^{-1}))\int_{2M }^\rho\tilde{H}_{\sfrak}(\Phips)
(\tb_0,\rho')\di\rho'\notag\\
\hspace{4ex}&\hspace{-4ex}
-\rb^3(1+O(\rb^{-1}))\int_{\rho}^{\infty}( (\rb')^{-2}+O((\rb')^{-3})))\tilde{H}_{\sfrak}(\Phips)
(\tb_0,\rho')\di\rho'\notag\\
\hspace{4ex}&\hspace{-4ex}
+(\rb+3M+O(\rho^{-1}))\int_{2M }^\rho
\tilde{H}_{\sfrak}(\Phips)(\tb_0,\rho')\di\rho'\notag
+(\rho^3+M\rho^2+\rho O(1))H\Phips\\
={}&(M+O(\rho^{-1}))\int_{2M }^\rho
\tilde{H}_{\sfrak}(\Phips)(\tb_0,\rho')\di\rho'
+(\rho^3+M\rho^2+\rho O(1))H\Phips\notag\\
&-\rb^3(1+O(\rb^{-1}))\int_{\rho}^{\infty}( (\rb')^{-2}+O((\rb')^{-3})))\tilde{H}_{\sfrak}(\Phips)
(\tb_0,\rho')\di\rho',
\end{align}
which proves the equality \eqref{eq:expression:r2Vgs}.
\end{proof}

From Lemma \ref{lemma:asym:formu:gs}, we  have the following asymptotic behavior for $g_{\sfrak}$ and its derivative on $\Sigmazero$.

\begin{cor}\label{coroll:asym:gs:Vgs}
Let $g_{\sfrak}$ be the time integral of $\Phips$ which is constructed in Lemma \ref{lem:timeintegral:posi}.
We have on $\Sigmazero$ that for $\rho$ sufficiently large,
\begin{subequations}
\begin{align}\label{asym:formu:r-2:gs}
\hspace{4ex}&\hspace{-4ex}\bigg|g_{\sfrak}(\tb_0,\rho)
+\half\rho^{-1}\int_{2M}^{\infty}
\tilde{H}_{\sfrak}(\Phips)(\tb_0,\rho')\di\rho'\bigg|
\lesssim_{\delta} \rho^{-2+\frac{\delta}{2}}(F^{(1)}(\regl,5-\delta/2,\tb_0,\Psipns))^{\half},\\
\label{asym:formu:r-2:V:gs}
\hspace{4ex}&\hspace{-4ex}\bigg|\rho^{2}\VR (\mu^{-\half}\rb g_{\sfrak})
-M\int_{2M}^{\infty}\tilde{H}_{\sfrak}(\Phips)(\tb_0,\rho')\di\rho'
+\rho^3\int_\rho^{\infty}2(\rho')^{-1}\VR\hatPhips(\tb_0,\rho')\di\rho'\bigg|\notag\\
\lesssim_{\delta}{}& \rho^{-1}
(\tilde{D}_1 + (F^{(1)}(\regl,5-\delta,\tb_0,\Psipns))^{\half}).
\end{align}
\end{subequations}
Furthermore, if limit $\lim\limits_{\rho\to\infty}\rho^3\VR\hatPhips\big|_{\Sigmazero}$ exits,  then
\begin{align}\label{calcu:NPCPTI}
\lim\limits_{\rho\to\infty}\rho^2\VR(\mu^{-\half} \rb g_{\sfrak})(\tb_0,\rho)
=M\int_{2M}^\infty
\tilde{H}_{\sfrak}(\Phips)(\tb_0,\rho')\di\rho'
-\frac{2}{3}\lim\limits_{\rho\to\infty}\rho^3\VR\hatPhips\big|_{\Sigmazero}.
\end{align}
\end{cor}
\begin{proof}
First, from the fact that $H=O(r^{-2})$ for $r$ away from horizon, standard Sobolev inequalities applied to the expression \eqref{equ:tildeH} yields that for $\rb\geq R$,
\begin{align}
\label{eq:pointwisenorm:tildeHsPhips}
\absCDeri{\tilde{H}_{\sfrak}(\Phips)}{\reg}\lesssim_{\reg,\delta} \rb^{-2+\frac{\delta}{2}}(F^{(1)}(\reg+\regl,5-\delta/2,\tb_0,\Psipns))^{\half}.
\end{align}
Moreover, this trivially implies
\begin{align}
\label{eq:integral2Mtoinfty:tildeHsPhips}
\int_{2M }^{\infty}\tilde{H}_{\sfrak}(\Phips)|_{\Sigmazero}\di\rho\lesssim_{\delta}
(F^{(1)}(\regl,5-\delta/2,\tb_0,\Psipns))^{\half}.
\end{align}
As a result, the estimate
\eqref{asym:formu:r-2:gs} follows from the above two estimate and \eqref{expression:gs:copy}.

By using $|H|\lesssim \rho^{-2}$, $|\Phips|\lesssim\rho^{-1}(F^{(1)}(\regl,5-\delta/2,\tb_0,\Psipns))^{\half}$,
\eqref{asym:tildeH:Phips:infinity}
and \eqref{eq:expression:r2Vgs},
we obtain
 \begin{align}
 \rho^{2}\VR (\mu^{-\half}rg_{\sfrak})
={}&M\int_{2M}^{\infty}\tilde{H}_{\sfrak}(\Phips)(\tb_0,\rho')\di\rho'
+\rho^3H\Phips
-\rho^3\int_\rho^{\infty}[2(\rho')^{-1}\VR\hatPhips
-\partial_\rho(H\Phips)](\tb_0,\rho')\di\rho'\notag\\
&+O(\rho^{-1})(\tilde{D}_1 + (F^{(1)}(\regl,5-\delta,\tb_0,\Psipns))^{\half})\notag\\
={}&M\int_{2M}^{\infty}\tilde{H}_{\sfrak}(\Phips)(\tb_0,\rho')\di\rho'
-\rho^3\int_\rho^{\infty}2(\rho')^{-1}\VR\hatPhips
(\tb_0,\rho')\di\rho'\notag\\
&+O(\rho^{-1})(\tilde{D}_1 + (F^{(1)}(\regl,5-\delta,\tb_0,\Psipns))^{\half}),
\end{align}
 which proves the estimate \eqref{asym:formu:r-2:V:gs}.

Last, if the limit $\lim\limits_{\rho\to\infty}\rho^3\VR\hatPhips(\tb_0,\rho)$ exists,  then we have
\begin{align}
\lim\limits_{\rho\to\infty}\rho^3\int_{\rho}^{\infty}
2(\rho')^{-1}\VR\hatPhips(\tb_0,\rho')\di\rho'=\frac{2}{3}
\lim\limits_{\rho\to\infty}\rho^3\VR\hatPhips\big|_{\Sigmazero}.
\end{align}
Substituting the above equation into \eqref{asym:formu:r-2:V:gs}, we get \eqref{calcu:NPCPTI}.
\end{proof}

Similar to Definition \ref{def:NPCs}, we define the first N--P constant of the time integral:

\begin{definition}
\label{def:NPconstofgsfrak}
Define the first N--P constant of $g_{\sfrak}$ which is constructed from Lemma \ref{lem:timeintegral:posi} to be
\begin{align}
\NPCPTI{1}=\lim\limits_{\rho\to\infty}\rho^2
\VR(\mu^{-\half} rg_{\sfrak})(\tb_0,\rho).
\end{align}
\end{definition}

We can then calculate the first N--P constant of the time integral from Corollary \ref{coroll:asym:gs:Vgs}. 

\begin{lemma}
Under the assumptions \eqref{Initial:NPCP:zero} for $\Phips$ and \eqref{asym:gs:infty} for $g_{\sfrak}$, the first N--P constant $\NPCPTI{1}$ of the scalar $g_{\sfrak}$ defined as in Definition \ref{def:NPconstofgsfrak} is finite and satisfies
\begin{align}
\label{expression:NPCPTI}
\NPCPTI{1}=M\int_{2M}^\infty
\tilde{H}_{\sfrak}(\Phips)(\tb_0,\rho')\di\rho'
-\frac{2}{3}\lim\limits_{\rho\to\infty}\rho^3\VR\hatPhips\big|_{\Sigmazero}.
\end{align}
\end{lemma}

%%%%%%%%%%%%%%%%%%
\subsection{Control the initial energy of $g_{\sfrak}$ by an initial energy of $\Phips$}
\label{sect:controllingtimeintegralenergy}
%%%%%%%%%%%%%%%%

In this subsection, we shall prove that an initial energy of $g_{\sfrak}$ is bounded by an initial energy of $\Phips$. This is necessary since applying the estimates in Theorem \ref{thm:pricelaw:zeroNP} to the time integral requires a bound for the weighted initial energy of the time integral. 
To be precise, we shall prove the following result.

\begin{prop}
\label{prop:energycompare}
Let $\reg\in \mathbb{N^+}$.
There exists a universal constant $\regl\in \mathbb{N}$ such that
\begin{align}
\label{eq:initialenergycompare}
F^{(1)}(\reg,3-\delta,\tb_0,g_{\sfrak})\lesssim_{\delta,\reg}{}F^{(1)}(\reg+\regl,5-\delta/2,\tb_0,\Psipns).
\end{align}
\end{prop}

\begin{proof}
From  \eqref{equ:tildeH}, we have for $\rb$ large that
\begin{align}
\label{eq:esti:tildeHsPhis:away}
\abs{\tilde{H}_{\sfrak}(\Phips)}\lesssim{}&\mu \rb \abs{\VR \PhipsHigh{1}}
+\rb^{-2}\abs{\PhipsHigh{1}}
+\rb^{-1}\abs{\prb \PhipsHigh{1}}\notag\\
\lesssim{}&\rb^{-2+\frac{\delta}{4}}
\abs{\mu\rb^{3-\frac{\delta}{4}}\VR \PhipsHigh{1}}
+\rb^{-2}(\abs{\PhipsHigh{1}}
+\rb\abs{\prb \PhipsHigh{1}}),
\end{align}
thus
\begin{subequations}
\begin{align}
\int_{2M }^\infty \abs{\tilde{H}_{\sfrak}(\Phips)
(\tb_0,\rho')}\di\rho'\lesssim{}&
\sup_{\Sigmazero}\abs{\mu^{\frac{3}{2}}r^{3-\frac{\delta}{4}}\VR \PhipsHigh{1}}
+\norm{\Psips}_{W_{-2}^{1}(\Sigmazero)},\\
\int_{\rb}^\infty \abs{\tilde{H}_{\sfrak}(\Phips)
(\tb_0,\rho')}\di\rho'\lesssim{}&
\rb^{-1+\frac{\delta}{4}}\Big(\sup_{\Sigmazero^{\geq \rb}}\abs{\mu^{\frac{3}{2}}r^{3-\frac{\delta}{4}}\VR \PhipsHigh{1}}
+\norm{\Psips}_{W_{-1-\delta/2}^{1}(\Sigmazero^{\geq \rb})}\Big),\\
\rb^3\int_{\rho}^{\infty} (\rb')^{-2}\abs{\tilde{H}_{\sfrak}(\Phips)
(\tb_0,\rho')}\di\rho'
\lesssim{}&\rb^3\int_{\rho}^{\infty}
(\rb')^{-4}\big((\rb')^3 \abs{\VR \PhipsHigh{1}}
+\abs{\Psips}
+\rb'\abs{\prb \Psips}\big)\di\rho'\notag\\
\lesssim{}&
\rb^3\int_{\rho}^{\infty}
(\rb')^{-4}(\rb')^3 \abs{\VR \PhipsHigh{1}}\di \rb'
+\rb^3\int_{\rho}^{\infty}
(\rb')^{-4}\big(
\abs{\Psips}
+r\abs{\prb \Psips}\big)\di\rho'.
\end{align}
\end{subequations}
Therefore, applying these estimates to the expressions \eqref{expression:gsandr2gs:copy} and using the Minkowski integral inequality, one can obtain
\begin{subequations}
\begin{align}
&\norm{\rb g_{\sfrak}}^2_{W_{-2}^{0}(\Sigmazero)}
\lesssim{}\norm{\rho^2\hat{V}\PhipsHigh{1}}^2_{W_{1-\frac{\delta}{2}}^{0}(\Sigmazero^{\geq 4M})}
+\norm{\Psips}^2_{W_{-2}^{1}(\Sigmazero)}
+\sup_{\Sigmazero}\abs{\mu^{\frac{3}{2}}r^{3-\frac{\delta}{4}}\VR \PhipsHigh{1}}^2,\\
&\norm{\rho^2\hat{V}(\mu^{-\half}\rb g_{\sfrak})}^2_{W_{-1-\delta}^{0}(\Sigmazero^{\geq 4M})}\lesssim_{\delta}{}\norm{\rho^2\hat{V}\PhipsHigh{1}}^2_{W_{1-\frac{\delta}{2}}^{0}(\Sigmazero^{\geq 4M})}
+\norm{\Psips}^2_{W_{-1-\delta/2}^{1}(\Sigmazero)}
+\sup_{\Sigmazero}\abs{\mu^{\frac{3}{2}}r^{3-\frac{\delta}{4}}\VR \PhipsHigh{1}}^2.
\end{align}
\end{subequations}
A simple application of Hardy's inequality allows us to bound
\begin{align}
\norm{\Psips}^2_{W_{-1-\delta/2}^{1}(\Sigmazero)}
+\sup_{\Sigmazero}\abs{\mu^{\frac{3}{2}}r^{3-\frac{\delta}{4}}\VR \PhipsHigh{1}}^2
\lesssim_{\delta}{}&
\norm{\rho^2\hat{V}\PhipsHigh{1}}^2_{W_{1-\frac{\delta}{2}}^{\regl}(\Sigmazero^{\geq 4M})}+\norm{\Psips}^2_{W_{-2}^{\regl}(\Sigmazero)}
\end{align}
for some $\regl>0$, therefore,
\begin{align}
\norm{\rb g_{\sfrak}}^2_{W_{-2}^{0}(\Sigmazero)}
+\norm{\rho^2\hat{V}(\mu^{-\half}\rb g_{\sfrak})}^2_{W_{-1-\delta}^{0}(\Sigmazero^{\geq 4M})}
\lesssim{}&
\norm{\rho^2\hat{V}\PhipsHigh{1}}^2_{W_{1-\frac{\delta}{2}}^{\regl}(\Sigmazero^{\geq 4M})}+\norm{\Psips}^2_{W_{-2}^{\regl}(\Sigmazero)}.
\end{align}

By applying further the differential operator $\rb \prb$, one can argue in the same way as proving \eqref{eq:r2VRrgsfrak:explicit} that for any $i\in \mathbb{N}$,
\begin{align}
\label{expression:highorder:r2VRrgs}
(\rb\prb)^i(\rho^2\hat{V}(\mu^{-\half}\rb g_{\sfrak}))
={}&c_1(\rb+O(1))\int_{\rb}^\infty
\tilde{H}_{\sfrak}(\Phips)(\tb_0,\rho')\di\rho'
+O(1)\int_{2M }^\infty
\tilde{H}_{\sfrak}(\Phips)(\tb_0,\rho')\di\rho'\notag\\
&-c_3\rb^3(1+O(\rb^{-1}))\int_{\rho}^{\infty}( (\rb')^{-2}+O((\rb')^{-3})))\tilde{H}_{\sfrak}(\Phips)
(\tb_0,\rho')\di\rho'\notag\\
&
+\sum_{j=0}^{i-1}(c_{2,j}+O(\rb^{-1})) (\rb \prb)^j(\rb^2\tilde{H}_{\sfrak}(\Phips)),
\end{align}
where $c_1$, $c_2$, and $\{c_{3,j}\}_{j=0,1,\ldots, i-1}$ are finite constants depending only on $i$. Using the estimate \eqref{eq:esti:tildeHsPhis:away}, we have for the last line that
\begin{align}
\bigg\|\sum_{j=0}^{i-1}(c_{2,j}+O(\rb^{-1})) (\rb \prb)^j(\rb^2\tilde{H}_{\sfrak}(\Phips))
\bigg\|^2_{W_{-1-\delta}^{0}(\Sigmazero^{\geq 4M})}
\lesssim_{\delta,i}{}&
\norm{\rho^2\hat{V}\PhipsHigh{1}}^2_{W_{1-\frac{\delta}{2}}^{i}(\Sigmazero^{\geq 4M})}
+\norm{\Psips}^2_{W_{-2}^{i}(\Sigmazero)}.
\end{align}
Meanwhile, the first two lines of the LHS of \eqref{expression:highorder:r2VRrgs} are estimated in the same way, thus, for any $i\geq 1$,
\begin{align}
\sum_{j=0}^i\norm{(\rb\prb)^j(\rho^2\hat{V}(\mu^{-\half}\rb g_{\sfrak}))}^2_{W_{-1-\delta}^{0}(\Sigmazero^{\geq 4M})}\lesssim_{\delta,i}{}&
\norm{\rho^2\hat{V}\PhipsHigh{1}}^2_{W_{1-\frac{\delta}{2}}^{i+\regl}(\Sigmazero^{\geq 4M})}
+\norm{\Psips}_{W_{-1-\delta/2}^{i+\regl}(\Sigmazero)}.
\end{align}

By taking more $\prb$ derivatives on equation \eqref{regular:equa:partialgs:horizon}, we can bound $\prb^{i}g_{\sfrak}$ near horizon by $\Psips$ for any $i\in \mathbb{N}$, that is, for any finite $R>2M$,
\begin{align}
\sum_{j=0}^i \norm{(\rb\prb)^j g_{\sfrak}}^2_{W_{0}^{0}(\Sigmazero^{\leq R})}\lesssim_{R,i}{}&
\norm{\Psips}_{W_{-2}^{1+i}(\Sigmazero^{\leq R})}.
\end{align}
In total, we have thus for any $i\in \mathbb{N}^+$ that there exists a constant $\regl>0$ such that
\begin{align}
\label{eq:initialenergycompare:v1}
\hspace{6ex}&\hspace{-6ex}
\sum_{j=0}^i \norm{(\rb\prb)^j (\rb g_{\sfrak})}^2_{W_{-2}^{0}(\Sigmazero)}
+\sum_{j=0}^i\norm{(\rb\prb)^j(\rho^2\hat{V}(\mu^{-\half}\rb g_{\sfrak}))}^2_{W_{-1-\delta}^{0}(\Sigmazero^{\geq 4M})}\notag\\
\lesssim_{i,\delta}{}&
\norm{\rho^2\hat{V}\PhipsHigh{1}}^2_{W_{1-\frac{\delta}{2}}^{i+\regl}(\Sigmazero^{\geq 4M})}
+\norm{\Psips}_{W_{-1-\delta/2}^{i+\regl}(\Sigmazero)}.
\end{align}
In the end, by making use of $\Lxi g_{\sfrak}=\Phips$ and this estimate \eqref{eq:initialenergycompare:v1}, we achieve for any $\reg\in\mathbb{N}$,
\begin{align}
\norm{\rb g_{\sfrak}}^2_{W_{-2}^{\reg}(\Sigmazero)}
+\norm{\rho^2\hat{V}(\mu^{-\half}\rb g_{\sfrak})}^2_{W_{-1-\delta}^{\reg}(\Sigmazero^{\geq 4M})}
\lesssim_{\reg,\delta}{}&
\norm{\rho^2\hat{V}\PhipsHigh{1}}^2_{W_{1-\frac{\delta}{2}}^{\reg+\regl}(\Sigmazero^{\geq 4M})}
+\norm{\Psips}_{W_{-1-\delta/2}^{\reg+\regl}(\Sigmazero)}.
\end{align}
This is precisely the estimate \eqref{eq:initialenergycompare}.
\end{proof}

%%%%%%%%%%%%%%%%%%5
\subsection{Proof of Theorem \ref{thm:pricelaw:zeroNP}}
\label{sect:proofoftheorem2:NVNP}
%%%%%%%%%%%%%%%%%%%5

We prove Theorem \ref{thm:pricelaw:zeroNP} in this subsection. In Section \ref{sect:pfofThm2:ell1m} we derive the asymptotics for $\ell=1$ mode, and in Section \ref{sect:pfofThm2:ell2m}, we combine the asymptotics for $\ell=1$ mode with the almost sharp decay estimates for $\ell\geq 2$ modes in Theorem \ref{thm:almostPrice} (or Proposition \ref{prop:almostpricelaw:ellgeq2}) to conclude Theorem \ref{thm:pricelaw:zeroNP}.

%%%%%%%%%%%%%%%
\subsubsection{Estimates for $\ell=1$ mode}
\label{sect:pfofThm2:ell1m}
%%%%%%%%%%%%%%%%
We consider only a fixed $(m,\ell=1)$ mode of the spin $\pm \half$ components first.

To begin with, we need the following estimate for the time integral on the initial hypersurface $\Sigmazero$.

\begin{prop}
Let $j\in \mathbb{N}$. Assume on $\Sigmazero$ that there are constants  $\beta\in(0,\half)$, $\tilde{D}_0\geq 0$ and $\tilde{D}_1$ such that for all $0\leq i\leq j$ and $r\geq R$,
\begin{align}\label{assume:VhatPhips:zeroNPC}
\big|\rb^{i}\prb^i\big(\VR\hatPhips({\tb_0},\rho)
-{\tilde{D}_1}{\rho^{-3}}\big)\big|\lesssim \rho^{-3-\beta}\tilde{D}_0,
\end{align}
and assume further  for a suitably small $\delta\in (0,\half)$ and a suitably large $\regl=\regl(j)$ that
\begin{align}
F^{(1)}(\regl,5-\delta/2,\tb_0,\Psipns)<\infty.
\end{align}
Then it holds on $\Sigmazero$ that for all $0\leq i\leq j+1$ and $r\geq R$,
\begin{align}
\label{eq:assump:Vhatgsfrak:zeroNPC:fromVhatPhips}
|\rb^i\prb^i(\hat{V}(\mu^{-\half}rg_{\sfrak})-\NPCPTI{1}\rho^{-2})|\lesssim
\rho^{-2-\beta}\big(\tilde{D}_0+\abs{\tilde{D}_1}+(F^{(1)}(\regl,5-\delta/2,\tb_0,\Psipns))^{\half}\big),
\end{align}
where $\NPCPTI{1}$ is given by \eqref{expression:NPCPTI}.
\end{prop}

\begin{proof}
For simplicity, denote $D=\tilde{D}_0+\abs{\tilde{D}_1}+(F^{(1)}(\regl,5-\delta/2,\tb_0,\Psipns))^{\half}$.
We first prove the $j=0$ case. The $i=0$ case is manifest from \eqref{asym:formu:r-2:V:gs}, the assumption \eqref{assume:VhatPhips:zeroNPC} and the definition \eqref{expression:NPCPTI} of $\NPCPTI{1}$, and it remains to show the $i=1$ case.
By \eqref{assume:VhatPhips:zeroNPC} and Lemma \ref{lem:timeintegral:posi}, the time integral $g_{\sfrak}$ of $\Phips$ satisfies
\begin{align}\label{Lxigs}
\Lxi g_{\sfrak}=\Phips
\end{align}
and
\begin{align}
\label{eq:gsHigh1:l=1:v1:gs}
Y(\mu^{\frac{3}{2}} r^{-1}\hat{V}g_{\sfrak}^{(1)}) ={}-{6M} r^{-3}g_{\sfrak},
\end{align}
where $g_{\sfrak}^{(1)}=\mu^{-\half}rg_{\sfrak}$. The derivation of equation \eqref{eq:gsHigh1:l=1:v1:gs} comes from the fact that $g_{\sfrak}$ and $\PhipsHigh{1}$ satisfy the same equation \eqref{eq:wave:PhipsHigh1:simple}. From Corollary \ref{coroll:asym:gs:Vgs}, we have for the N-P constant $\NPCPTI{1}$ of $g_{\sfrak}$ that
\begin{align}
\NPCPTI{1}=M\int_{2M}^{\infty}\tilde{H}_{\sfrak}(\Phips)(\tau_0,\rho)\di \rho-\frac{2}{3}\tilde{D}_1.
\end{align}
Using $Y=-\partial_\rho+\partial_rh\Lxi$ and equations  \eqref{Lxigs} and \eqref{eq:gsHigh1:l=1:v1:gs}, we obtain
\begin{align}\label{equ:partial:V:gs1}
-\partial_\rho(\hat{V}g_{\sfrak}^{(1)})+(\rho^{-1}-3M\mu^{-1}\rho^{-2})\hat{V}g_{\sfrak}^{(1)}
+\partial_rh\VR\hatPhips=-6M\mu^{-\frac{3}{2}}\rho^{-2}g_{\sfrak}.
\end{align}
Furthermore, substituting  \eqref{asym:formu:r-2:gs}, \eqref{asym:formu:r-2:V:gs} and $|H|=|2\mu^{-1}-\partial_rh|\lesssim \rho^{-2}$ into \eqref{equ:partial:V:gs1}, we get
\begin{align}
|-\partial_\rho(\hat{V}g_{\sfrak}^{(1)})+\rho^{-1}\hat{V}g_{\sfrak}^{(1)}
+2\VR\hatPhips+6M\rho^{-2}g_{\sfrak}|\lesssim \rho^{-4}D.
\end{align}
On the other hand, by \eqref{assume:VhatPhips:zeroNPC}, we have
\begin{align}
\bigg|\int_\rho^{\infty}2({\rho'})^{-1}\VR\hatPhips(\tb_0,\rho')\di\rho'
-\frac{2c}{3}\rho^{-3}\bigg|
\lesssim \rho^{-3-\beta}\tilde{D}_0.
\end{align}
Using \eqref{asym:formu:r-2:gs}, \eqref{asym:formu:r-2:V:gs} and \eqref{assume:VhatPhips:zeroNPC} again gives
\begin{align}
\hspace{2ex}&\hspace{-2ex}-\partial_\rho(\hat{V}g_{\sfrak}^{(1)})+\rho^{-1}\hat{V}g_{\sfrak}^{(1)}
+2\VR\hatPhips+6M\rho^{-2}g_{\sfrak}\notag\\
&=-\partial_\rho(\hat{V}g_{\sfrak}^{(1)}-\NPCPTI{1}\rho^{-2})+2\NPCPTI{1}\rho^{-3}
+M\rho^{-3}\int_{2M}^{\infty}\tilde{H}_{\sfrak}(\Phips)(\tb_0,\rho')\di\rho'-\frac{2c}{3}\rho^{-3}
+O(\rho^{-3-\beta})D\notag\\
&\quad\ +2c\rho^{-3}+O(\rho^{-3-\beta})
-3M\int_{2M}^{\infty}\tilde{H}_{\sfrak}(\Phips)(\tb_0,\rho')\di\rho'+O(\rho^{-4})D\notag\\
&=-\partial_\rho(\hat{V}g_{\sfrak}^{(1)}-\NPCPTI{1}\rho^{-2})+O(\rho^{-3-\beta})D.
\end{align}
Hence, we have proved that
\begin{align}
|\partial_\rho(\hat{V}g_{\sfrak}^{(1)}-\NPCPTI{1}\rho^{-2})|\lesssim\rho^{-3-\beta}D.
\end{align}

For $j\geq 1$ cases,
we  prove by induction. Assume the statement hold for $j=j_0-1$, and to complete the induction, it suffices to prove \eqref{eq:assump:Vhatgsfrak:zeroNPC:fromVhatPhips} for $i=j_0+1$ under the assumption that \eqref{assume:VhatPhips:zeroNPC} holds for all $0\leq i\leq j_0$.
To be more precise, under the assumption that
\begin{align}\label{highderivative:V:hatPhips}
|\partial_\rho^i(\VR\hatPhips(\tau_0,\rho)-\rb^{-3}{\tilde{D}_1})|\lesssim \rho^{-3-i-\beta}\tilde{D}_0,\quad \text{for all}\ 0\leq i\leq j
\end{align}
together with the estimates followed from inductive hypothesis
\begin{align}\label{highderivative:V:gs1}
|\partial_\rho^{i}(\hat{V}g_{\sfrak}^{(1)}(\tau_0,\rho)
-\NPCPTI{1}\rho^{-2})|\lesssim\rho^{-2-i-\beta}D,\quad \text{for all }\ 0\leq i\leq j,
\end{align}
it suffices to prove \eqref{highderivative:V:gs1} for $i=j+1$ to close the induction.
From \eqref{equ:tildeH}, we have
\begin{align}
\tilde{H}_{\sfrak}(\Phips)=(r-M)[(2-\mu H)(\VR\hatPhips-\rho^{-3}{\tilde{D}_1})+\tilde{D}_1(2-\mu H)\rho^{-3}-\mu^{\half}r\partial_\rho(H\Phips)
-2M\mu^{-\half}r^{-1}H\Phips].
\end{align}
Applying $\partial_\rho^j$ to this equation, and using $|\partial_\rho^i H|\lesssim \rho^{-2-i}$ for $0\leq i\leq j$,  $|\partial_\rho^i \Phips|\lesssim \rho^{-1-i}D$ for $0\leq i\leq j+1$, and
\eqref{highderivative:V:hatPhips}, we have
\begin{align}
|\partial_\rho^i \tilde{H}_{\sfrak}(\Phips)|\lesssim \rho^{-2-i}D, \quad \text{for all }\ 0\leq i\leq j.
\end{align}
Furthermore, from \eqref{equ:gs:Vgs1}, one has
\begin{align}\begin{split}
\hspace{4ex}&\hspace{-4ex}\big[-\mu^{-\frac{3}{2}}M\rho +\mu^{-\half}(\rb-M)^{-1}\rb^2(2\rb-M) \big] \bigg(g_{\sfrak}
+\half\rho^{-1}\int_{2M}^{\infty}
\tilde{H}_{\sfrak}(\Phips)(\tb_0,\rho')\di\rho'\bigg)\\
={}&\rho^2(\hat{V}g_{\sfrak}^{(1)}-\NPCPTI{1}\rho^{-2})
+\NPCPTI{1}
-\mu^{-\half}\rho^3H\Phips\\
&+\half\rho^{-1}\big[-\mu^{-\frac{3}{2}}M\rho +\mu^{-\half}(\rb-M)^{-1}\rb^2(2\rb-M) \big]
\int_{2M}^{\infty}\tilde{H}_{\sfrak}(\Phips)(\tb_0,\rho')\di\rho',
\end{split}\end{align}
thus we achieve
\begin{align}\label{highderivative:gs}
\bigg|\partial_\rho^i\bigg(g_{\sfrak}
+\half\rho^{-1}\int_{2M}^{\infty}\tilde{H}_{\sfrak}(\Phips)(\tb_0,\rho')\di\rho'\bigg)\bigg|
\lesssim \rho^{-2-i}D, \quad \text{for all }\ 0\leq i\leq j.
\end{align}

Last, we rewrite \eqref{equ:partial:V:gs1} as
\begin{align}\begin{split}
\hspace{2ex}&\hspace{-2ex}-\partial_\rho(\hat{V}g_{\sfrak}^{(1)}-\NPCPTI{1}\rho^{-2})
+(\rho^{-1}-3M\mu^{-1}\rho^{-2})(\hat{V}g_{\sfrak}^{(1)}-\NPCPTI{1}\rho^{-2})\\
\hspace{2ex}&\hspace{-2ex}
+\partial_rh(\VR\hatPhips-\tilde{D}_1\rho^{-3})
+6M\mu^{-\frac{3}{2}}\rho^{-2}\bigg(g_{\sfrak}
+\half\rho^{-1}\int_{2M}^{\infty}\tilde{H}_{\sfrak}(\Phips)(\tb_0,\rho')\di\rho'\bigg)\\
=&-2\NPCPTI{1}\rho^{-3}-\NPCPTI{1}(\rho^{-3}-3M\mu^{-1}\rho^{-4})
-\tilde{D}_1\partial_rh\rho^{-3}
+3M\mu^{-\frac{3}{2}}\rho^{-3}\int_{2M}^{\infty}\tilde{H}_{\sfrak}(\Phips)(\tb_0,\rho')\di\rho'.
\end{split}\end{align}
Applying $\partial_\rho^{j+1}$ to the above equation, and by \eqref{highderivative:V:hatPhips}, \eqref{highderivative:V:gs1} and \eqref{highderivative:gs}, this justifies the estimate \eqref{highderivative:V:gs1} for $i=j+1$ and finishes the proof.
\end{proof}

We can now turn to the full $\ell=1$ mode. One can uniquely define a scalar function $g_{-\sfrak}$ by a Dirac system from $g_{\sfrak}$
\begin{subequations}\label{eq:Dirac:TMEscalarTI}
\begin{align}
 g_{\sfrak}={}&
(\Delta^{\Half}\VR)(\Delta^{\Half}g_{-\sfrak}),\\
- g_{-\sfrak}={}&
Yg_{\sfrak}.
\end{align}
\end{subequations}
Then the scalar $\psins$ defined by $\psins=\Lxi g_{-\sfrak}$ and the scalar $\psips=\Phips=\Lxi g_{\sfrak}$ solve the Dirac equations \eqref{eq:Dirac:TMEscalar}.
As a result, by defining $\varphi_{\sfrak,TI}=(r-M)^{-1}g_{\sfrak}$ and $\psi_{-\sfrak, TI}=g_{-\sfrak}$, where the subscript $TI$ means they are defined by the time integral of the spin $\pm \half$ components, Theorem \ref{thm:pricelaw:nonzeroNP} applies to $(g_{\sfrak},g_{-\sfrak})$ and yields that for a suitably small $\delta$,
there exists an $\epsilon>0$ and a $\regl=\regl(j)>0$ such that
\begin{subequations}
\label{eq:Pricelaw:zero:gs}
\begin{align}
\hspace{6ex}&\hspace{-6ex}\bigg|\Lxi^j \varphi_{\sfrak,TI}-c_{\sfrak,j}v^{-2}\tb^{-1-j}
\sum_{m=\pm\half}\NPCPTI{1}(m,\ell=1)
Y_{m,\ell=1}^{\sfrak}(\cos\theta)e^{im\pb}\bigg|
\notag\\
\lesssim_{j,\delta} {}& v^{-2}\tb^{-1-j-\epsilon}
\Big[
(F^{(1)}(\regl,5-\delta/2,\tb_0,\Psipns))^{\half} +\sum_{m=\pm\half}\abs{\NPCPTI{1}(m,\ell=1)}+ \tilde{D}_0+\abs{\tilde{D}_1}\Big], \\
\hspace{6ex}&\hspace{-6ex}\bigg|\Lxi^j \psi_{-\sfrak,TI}
-c_{-\sfrak,j}v^{-1}\tb^{-2-j}
\sum_{m=\pm\half}\NPCPTI{1}(m,\ell=1)
Y_{m,\ell=1}^{-\sfrak}(\cos\theta)e^{im\pb}\bigg|\notag\\
\lesssim_{j,\delta} {}& v^{-1}\tb^{-2-j-\epsilon}\Big[
(F^{(1)}(\regl,5-\delta/2,\tb_0,\Psipns))^{\half} +\sum_{m=\pm\half}\abs{\NPCPTI{1}(m,\ell=1)}+ \tilde{D}_0+\abs{\tilde{D}_1}\Big],
\end{align}
\end{subequations}
where $c_{\sfrak,j}$ and $c_{-\sfrak,j}$ are defined in \eqref{def:cpnsj}.
Note that we have used here the following estimate to achieve the above inequalities:
\begin{align}
\label{eq:energyofgpmsfrak:3-delta}
F^{(1)}(\regl,3-\delta,\tb_0,\Psi_{ \pm\sfrak, TI})\lesssim_{\regl, \delta}{}F^{(1)}(\regl,5-\delta/2,\tb_0,\Psipns).
\end{align}
This estimate can be proved in the following way.  We have shown in Proposition \ref{prop:energycompare} that
\begin{align}
\label{eq:energyofgsfrak:3-delta}
F^{(1)}(\regl,3-\delta,\tb_0,\Psi_{ \sfrak, TI})=
F^{(1)}(\regl,3-\delta,\tb_0,g_{\sfrak})\lesssim_{\regl, \delta}{}F^{(1)}(\regl,5-\delta/2,\tb_0,\Psipns).
\end{align}
For the other part $F^{(1)}(\regl,3-\delta,\tb_0,\Psi_{-\sfrak, TI})$, it is clear that the integrals over finite radius region is bounded by $C F^{(1)}(\regl,3-\delta,\tb_0,\Psi_{ \sfrak, TI})$ in view of the equations \eqref{eq:Dirac:TMEscalarTI}, thus we simply need to estimate the integrals for $r\geq 4M$. By definition \ref{def:Fenergies:big2}  and the equations \eqref{eq:Dirac:TMEscalarTI}, $F(\regl,0,\tb_0,\Psi_{-\sfrak, TI})\lesssim_{\regl} F(\regl,0,\tb_0,\Psi_{\sfrak, TI})$, and
\begin{align}
\norm{rV\tildePhinsHighTI{1}}^2_{W_{3-\delta-2}^{\regl}(\Sigmatb^{\geq 4M})}
\lesssim_{\regl}{}& \norm{rV\tildePhipsHighTI{1}}^2_{W_{3-\delta-2}^{\regl}(\Sigmatb^{\geq 4M})}\lesssim_{\regl,\delta} F^{(1)}(\regl,3-\delta,\tb_0,g_{\sfrak}).
\end{align}
In conclusion, $F^{(1)}(\regl,3-\delta,\tb_0,\Psi_{ -\sfrak, TI})\lesssim_{\regl,\delta} F^{(1)}(\regl,3-\delta,\tb_0,\Psi_{ \sfrak, TI})$. Combined with the estimate \eqref{eq:energyofgsfrak:3-delta}, we obtain the estimate \eqref{eq:energyofgpmsfrak:3-delta}.

%%%%%%%%%%%%%
\subsubsection{Asymptotics for the entire Dirac field}
\label{sect:pfofThm2:ell2m}
%%%%%%%%%%%%

Since $\Lxi\varphi_{\sfrak, TI}=\varphi_{\sfrak}$ and $\Lxi\varphi_{-\sfrak,TI}=\psins$, we obtain from \eqref{eq:Pricelaw:zero:gs} that for a suitably small $\delta$,
there exists an $\epsilon>0$ and a $\regl=\regl(j)>0$ such that
\begin{subequations}
\label{eq:Pricelaw:zero:ell=1}
\begin{align}
\hspace{6ex}&\hspace{-6ex}\bigg|\Lxi^j (\varphi_{\sfrak})^{\ell=1}-c_{\sfrak,j+1}v^{-2}\tb^{-2-j}
\sum_{m=\pm\half}\NPCPTI{1}(m,\ell=1)
Y_{m,\ell=1}^{\sfrak}(\cos\theta)e^{im\pb}\bigg|
\notag\\
\lesssim_{j,\delta} {}& v^{-2}\tb^{-2-j-\epsilon}
\Big[
(F^{(1)}(\regl,5-\delta/2,\tb_0,(\Psipns)^{\ell=1}))^{\half} \notag\\
&\qquad \qquad \qquad +\sum_{m=\pm\half}(\abs{\NPCPTI{1}(m,\ell=1)}+\abs{\tilde{D}_1(m,\ell=1)})+ \tilde{D}_0\Big], \\
\hspace{6ex}&\hspace{-6ex}\bigg|\Lxi^j (\psins)^{\ell=1}
-c_{-\sfrak,j+1}v^{-1}\tb^{-3-j}
\sum_{m=\pm\half}\NPCPTI{1}(m,\ell=1)
Y_{m,\ell=1}^{-\sfrak}(\cos\theta)e^{im\pb}\bigg|\notag\\
\lesssim_{j,\delta} {}& v^{-1}\tb^{-3-j-\epsilon}\Big[
(F^{(1)}(\regl,5-\delta/2,\tb_0,(\Psipns)^{\ell=1}))^{\half} \notag\\
&\qquad \qquad \qquad +\sum_{m=\pm\half}(\abs{\NPCPTI{1}(m,\ell=1)}+\abs{\tilde{D}_1(m,\ell=1)})+ \tilde{D}_0\Big],
\end{align}
\end{subequations}
where $c_{\sfrak,j+1}$ and $c_{-\sfrak,j+1}$ are as defined in \eqref{def:cpnsj}. Recall that this estimate holds under the assumption that
\begin{align}
(F^{(1)}(\regl,3-\delta,\tb_0,(\Psi_{\pm\sfrak,TI})^{\ell=1}))^{\half}
+\sum_{m=\pm\half}(\abs{\NPCPTI{1}(m,\ell=1)}+\abs{\tilde{D}_1(m,\ell=1)})+\tilde{D}_0
 <\infty.
\end{align}
One can utilize the estimate \eqref{eq:initialenergycompare} to bound the first term by $(F^{(1)}(\regl,5-\delta/2,\tb_0,(\Psipns)^{\ell=1}))^{\half} $, and from the expression \eqref{expression:NPCPTI} for a fixed mode $(m,\ell=1)$ and
the expression \eqref{equ:tildeH} of $\tilde{H}_{\sfrak}((\Phips)^{\ell=1})$, the second term is bounded by $C\Big(F^{(1)}(\regl,5-\delta/2,\tb_0,(\Psipns)^{\ell=1})^{\half} +\sum\limits_{m=\pm\half}\abs{\tilde{D}_1(m,\ell=1)} +\tilde{D}_0\Big)$.
Thus the estimates \eqref{eq:Pricelaw:zero:ell=1} are valid under the assumption \eqref{thm:Pricelaw:zero:assump2}.

Consider next the $\ell=2$ mode and $\ell\geq 3$ modes. It is clear from Proposition \ref{prop:almostpricelaw:ellgeq2} that for any $j\in \mathbb{N}$ and any $\delta\in (0,\half)$,
\begin{subequations}
\label{eq:pointwise:zero:ell=2}
\begin{align}
\abs{ \Lxi^j(\varphi_{\sfrak})^{\ell=2}}
\lesssim_{ j,\delta}{}&v^{-2}\tb^{-2-\frac{\delta}{2}-j}
(F^{(2)}(\regl(j),3+\delta,\tb_0,(\Psipns)^{\ell=2}))^{\half},\\
\abs{\Lxi^j(\psins)^{\ell=2}}
\lesssim_{j,\delta}{}&v^{-1}\tb^{-3-\frac{\delta}{2}-j}
(F^{(2)}(\regl(j),3+\delta,\tb_0,(\Psipns)^{\ell=2}))^{\half};
\end{align}
\end{subequations}
and
\begin{subequations}
\label{eq:pointwise:zero:ellbig3}
\begin{align}
\abs{ \Lxi^j(\varphi_{\sfrak})^{\ell\geq 3}}
\lesssim_{j,\delta}{}&v^{-2}\tb^{-2-\frac{\delta}{2}-j}
(F^{(3)}(\regl(j),1+\delta,\tb_0,(\Psipns)^{\ell\geq 3}))^{\half},\\
\abs{\Lxi^j(\psins)^{\ell\geq 3}}
\lesssim_{ j,\delta}{}&v^{-1}\tb^{-3-\frac{\delta}{2}-j}
(F^{(3)}(\regl(j),1+\delta,\tb_0,(\Psipns)^{\ell\geq 3}))^{\half}.
\end{align}
\end{subequations}
The estimates \eqref{eq:Pricelaw:zero:ell=1}--\eqref{eq:pointwise:zero:ellbig3} together prove Theorem \ref{thm:pricelaw:zeroNP}.

%%%%%%%%%%
\section*{Acknowledgment}
%%%%%%%%%

The first author S. M. acknowledges the support by the ERC grant ERC-2016 CoG 725589 EPGR. The authors are grateful to the anonymous referees for many valuable comments and suggestions.

%%%%%%%%%%%%%%%%%%%%%%%%%%%%%%%%%%%%%%%%%%%%%%
%%%%%%%%%%%%%%%%%%%%%%%%%%%%%%%%%%%%%%%%%%%%%%

\appendix

\section{Derivation of Dirac equations and Teukolsky master equation on a Kerr background}
\label{app:TMEandDiracEq}

Consider $\Phi_A$ as a test field on Kerr spacetimes. Let $\Sigma=r^2 +a^2\cos^2\theta$ and $\Delta=r^2-2Mr +a^2$, where $M$ and $a$ are the mass and angular momentum per mass of the Kerr black-hole spacetime. We follow \cite{Teukolsky1973I}
and choose a  Kinnersley null tetrad $(\tilde{l},\tilde{n},m,\bar{m})$  \cite{Kinnersley1969tetradForTypeD} which reads in Boyer-Lindquist coordinates:
\begin{align}
\begin{split}\label{eq:Kinnersley tetrad}
\tilde{l}^\mu &= {\Delta}^{-1}(r^2+a^2 , \Delta , 0 , a), \\
\tilde{n}^\mu &= \frac{1}{2\Sigma} (r^2+a^2 , - \Delta , 0 , a), \\
m^\mu &= -\frac{1}{\sqrt{2} }\rho^{\star}\left(i a \sin{\theta},0 , 1, \frac{i}{\sin{\theta}}\right),
\end{split}
\end{align}
and $(\bar{m})^{\mu}$ and ${\rho}^{\star}$ being the complex conjugate of $m^{\mu}$ and $\rho = -1/(r- i a \cos{\theta})$, respectively. Similar to the definitions \eqref{def:oandiota} and \eqref{def:compsofDiracspinor}, let  $\tilde{o}^A$ and $\tilde{\iota}^A$ be the associated dyad legs of the Kinnersley null tetrad, and let $\tilde{\chi}_0$ and $\tilde{\chi}_1$ be the components of $\tilde{\Phi}_A$ along the dyad legs $\tilde{o}^A$ and $\tilde{\iota}^A$. Then the Dirac equations \eqref{eq:Dirac:spinorform}, as shown in \cite{Teukolsky1973I}, take the form of
\begin{subequations}\label{eq:Dirac:operatorform}
\begin{align}
\label{eq:Dirac:operatorform:1}
(\delta^{\star}-\alpha+\pi)\tilde{\chi}_0={}&(D-\rho+\epsilon)\tilde{\chi}_1,\\
\label{eq:Dirac:operatorform:2}
(\Delta+\mu-\gamma)\tilde{\chi}_0={}&(\delta+\beta-\tau)\tilde{\chi}_1.
\end{align}
\end{subequations}
Here, $\delta, D, \Delta, \delta^{\star}$ are differential operators, and $\alpha, \pi, \rho,\epsilon,\mu,\gamma,\beta,\tau$ are spin coefficients. Their explicit forms and values in Kerr spacetimes are given in \eqref{eq:appx:diffopers} and \eqref{eq:appx:spincoeffs}.

However, it is well-known that this Kinnersley tetrad has singularity at $\Horizon$, thus we shall choose a regular null tetrad instead.
A Hartle--Hawking null tetrad \cite{HHtetrad72}, which is regular at $\Horizon$ in a regular coordinate system, say, the ingoing Eddington-Finkelstein coordinate system, reads in Boyer--Lindquist coordinates:
\begin{align}
\begin{split}\label{eq:HartleHawkingtetrad:Kerr}
l^\mu &= {(2\Sigma)}^{-1}(r^2+a^2 , \Delta , 0 , a), \\
n^\mu &= \Delta^{-1} (r^2+a^2 , - \Delta , 0 , a), \\
m^\mu &= -2^{-\half}\rho^{\star}\left(i a \sin{\theta},0 , 1, i\csc\theta\right).
\end{split}
\end{align}
and $\bar{m}^{\mu}$ being the complex conjugate of $m^{\mu}$.
Let  ${o}^A$ and ${\iota}^A$ be the associated dyad legs of the Hartle--Hawking null tetrad as in \eqref{def:oandiota}, and let $\chi_0$ and $\chi_1$ be the components of $\Phi_A$ along dyad legs $o^A$ and $\iota^A$ as defined in \eqref{def:compsofDiracspinor}. The components $\chi_0$ and $\chi_1$ are thus regular up to and on $\Horizon$.

Denote the future-directed ingoing and outgoing principal null vector in B-L coordinates
\begin{subequations}
\begin{align}\label{def:VectorFieldYandV:Kerr}
Y&= \frac{(r^2+a^2)\partial_t +a\partial_{\phi}}{\Delta}-\partial_r, \ &\ \VR&= \frac{(r^2+a^2)\partial_t+
a\partial_{\phi}}{\Delta}+\partial_r,
\end{align}
and define in B-L coordinates
\begin{align}
\curlL{n}={}&\partial_{\theta}-\frac{i}{\sin\theta}\partial_{\phi}
-ia\sin\theta\partial_t + n\cot\theta,\\
\curlLd{n}={}&\partial_{\theta}
+\frac{i}{\sin\theta}\partial_{\phi}
+ia\sin\theta\partial_t + n\cot\theta.
\end{align}
\end{subequations}
Denote $s$  the spin-weight $\pm \frac{1}{2}$ and $\sfrak$ its absolute value $\frac{1}{2}$, and define our Teukolsky scalars of Dirac field as
\begin{align}
\psis=\left\{
             \begin{array}{ll}
               \Sigma^{\half}\chi_0, & \hbox{$s=\half$;} \\
               (2\Sigma)^{-\half}(r-ia\cos\theta)\chi_1, & \hbox{$s=-\half$.}
             \end{array}
           \right.
\end{align}

Applying $(D+\epsilon^{\star}-\rho-\rho^{\star})$ to
\eqref{eq:Dirac:operatorform:2}
and $(\delta-\alpha^{\star}-\tau+\pi^{\star})$ to
\eqref{eq:Dirac:operatorform:2} and taking the difference, one obtains a decouple equation of $\tilde{\chi}_0$:
\begin{align}
[(D+\epsilon^{\star}-\rho-\rho^{\star})(\Delta-\gamma+\mu)
-(\delta-\alpha^{\star}-\tau+\pi^{\star})
(\delta^{\star}-\alpha+\tau)]\tilde{\chi}_0={}0.
\end{align}
Interchanging $\tilde{l}\leftrightarrow \tilde{n}$ and $m\leftrightarrow \bar{m}$ gives
\begin{align}
[(\Delta-\gamma^{\star}+\mu+\mu^{\star})(D+\epsilon-\rho)
-(\delta^{\star}+\beta^{\star}+\pi-\tau^{\star})
(\delta+\beta-\tau)]\tilde{\chi}_1={}&0.
\end{align}
In this Kinnersley tetrad, the nonvanishing spin coefficients are
\begin{align}
\label{eq:appx:spincoeffs}
\beta={}&\frac{\cot\theta}{2\sqrt{2}(r+ia\cos \theta)},&
\pi={}&\frac{ia\sin\theta}{\sqrt{2}(r-ia\cos\theta)^2},
&\rho={}&\frac{-1}{r-ia\cos\theta}, &
\tau={}&\frac{-ia\sin\theta}{\sqrt{2}\Sigma}, \notag\\
\mu={}&\frac{-\Delta}{2(r-ia\cos\theta)\Sigma}, & \alpha={}&\pi-\frac{\cot\theta}{2\sqrt{2}(r-ia\cos\theta)}, &\gamma={}&\mu+\frac{r-M}{2\Sigma},
\end{align}
and the differential operators in \eqref{eq:Dirac:operatorform} are
\begin{align}
\label{eq:appx:diffopers}
D={}&\VR, &\Delta={}&\frac{\Delta}{2\Sigma}Y,&
\delta={}&\frac{1}{\sqrt{2}(r+ia\cos\theta)}\curlLd{0}, &\delta^{\star}={}&\frac{1}{\sqrt{2}(r-ia\cos\theta)}\curlL{0}.
\end{align}
In view of the relations
\begin{align}
\psis=\left\{
             \begin{array}{ll}
               2^{-\Half}\Delta^{\Half}\tilde{\chi}_0, & \hbox{$s=\Half$;} \\
               \Delta^{-\Half}(r-ia\cos\theta)\tilde{\chi}_1, & \hbox{$s=-\Half$,}
             \end{array}
           \right.
\end{align}
we obtain from equations \eqref{eq:Dirac:operatorform} the following Dirac equations
\begin{subequations}\label{eq:Dirac:TMEscalar::Kerr}
\begin{align}
\eth'   \psips={}&
(\Delta^{\Half}\VR)(\Delta^{\Half}\psins),\\
\eth \psins={}&
Y\psips,
\end{align}
\end{subequations}
where $\eth'=\edthR'-ia\sin\theta\Lxi$ and $\eth=\edthR +ia\sin\theta\Lxi$.
As is shown by Teukolsky in \cite{Teukolsky1973I},
the scalars $\psips^{\text{Teu}}=\tilde{\chi}_0$ and $\psins^{\text{Teu}}=\rho^{-1}\tilde{\chi}_1$ satisfy the celebrated
 Teukolsky master equation (TME). Since $\psips=\frac{1}{\sqrt{2}}\Delta^{\sfrak}\psi_{\sfrak}^{\text{Teu}}$ and $\psins=-\Delta^{-\sfrak}\psi_{{-\sfrak}}^{\text{Teu}}$, by taking into account of this rescaling, one obtains the following form of TME in Boyer--Lindquist coordinates:
\begin{align}
\begin{split}\label{eq:TME:Kerr}
& -\left[\tfrac{(r^2+a^2)^2}{\Delta} -a^2 \sin^2{\theta} \right] \tfrac{\partial^2 \psis}{\partial t^2} - \tfrac{4Mar}{\Delta} \tfrac{\partial^2 \psis}{\partial t \partial \phi}-\left[\tfrac{a^2}{\Delta} -\tfrac{1}{\sin^2{\theta}} \right] \tfrac{\partial^2 \psis}{\partial \phi^2}   \\
&  +\Delta^{s} \tfrac{\partial}{\partial r} \left( \Delta^{-s+1} \tfrac{\partial \psis}{\partial r} \right) + \tfrac{1}{\sin{\theta}} \tfrac{\partial}{\partial \theta} \left( \sin{\theta} \tfrac{\partial \psis}{\partial \theta}\right) +2s \left[ \tfrac{a(r-M)}{\Delta} + \tfrac{i \cos{\theta}}{\sin^2{\theta} } \right] \tfrac{\partial \psis}{\partial \phi} \\
&  +2s\left[ \tfrac{M(r^2-a^2)}{\Delta} -r -ia \cos{\theta} \right] \tfrac{\partial \psis}{\partial t}- (s^2 \cot^2{\theta} +s) \psis = 0 .
\end{split}
\end{align}

%%%%%%%%%%%%%%%%%%%
\section{A list of scalars constructed from the spin $\pm \half$ components}
%%%%%%%%%%%%%%%

For convenience, we collect the scalars constructed from the spin $\pm \half$ components in the table below so that one can easily relate them in terms of the scalars $\psips$ and $\psins$.

\begin{table}[htbp]
\begin{center}
\begin{tabular}{l|ll}
&$s=\sfrak$   & $s=-\sfrak$\\
\hline
$\psis$ & $r\chi_0$  as in \eqref{def:psipns:Schw}& $2^{-\half}\chi_1$ as in \eqref{def:psipns:Schw}\\
$\phis$ &$r^{-1}\psips$ as in \eqref{def:phis} &$\mu^{\half}\psins$ as in \eqref{def:phis}\\
$\Phis$ & $\psips$  as in \eqref{def:varphis} & $r\mu^{\half}\psins$ as in \eqref{def:varphis} \\
$\Psi_s$ & $r\psips$ as in \eqref{def:PsipsPsins} & $r\psins$ as in \eqref{def:PsipsPsins}\\
$\varphi_s$ &$(r-M)^{-1}\psips$ as in \eqref{def:varphisfrak} &\textbackslash\\
$\Phi_{s}^{(1)}$ &$\mu^{-\half}r\psips$ as in Definition \ref{def:hatPhips}& $\curlVR \Phi_{-\sfrak}$ as in Definition \ref{def:tildePhiplusandminusHigh}\\
$\Phi_{s}^{(i)}$ &$\curlVR^{i-1}\Phi_{\sfrak}^{(1)}$
 as in Definition \ref{def:tildePhiplusandminusHigh}& $\curlVR^{i-1} \Phi_{-\sfrak}^{(1)}$ as in Definition \ref{def:tildePhiplusandminusHigh}\\
$\tilde{\Phi}_{s}^{(i)}$ &as in Definition \ref{def:tildePhiplusandminusHigh} & as in Definition \ref{def:tildePhiplusandminusHigh}\\
$g_s$ &$\partial_\tau g_{\sfrak}=\psips$ as in Lemma \ref{lem:timeintegral:posi} &$\partial_\tau g_{-\sfrak}=\psins$ as in \eqref{eq:Dirac:TMEscalarTI}
\end{tabular}
\end{center}
\caption{Expressions of the spin $\pm \half$ components.}
\end{table}

\providecommand{\bysame}{\leavevmode\hbox to3em{\hrulefill}\thinspace}
\providecommand{\MR}{\relax\ifhmode\unskip\space\fi MR }
\providecommand{\MRhref}[2]{%
  \href{http://www.ams.org/mathscinet-getitem?mr=#1}{#2}
}
\providecommand{\href}[2]{#2}


\begin{thebibliography}{10}

\bibitem{andersson16decayMaxSchw}
Lars Andersson, Thomas B{\"a}ckdahl, and Pieter Blue, \emph{Decay of solutions
  to the {M}axwell equation on the {S}chwarzschild background}, Classical and
  Quantum Gravity \textbf{33} (2016), no.~8, 085010.

\bibitem{andersson2019stability}
Lars Andersson, Thomas B{\"a}ckdahl, Pieter Blue, and Siyuan Ma,
  \emph{Stability for linearized gravity on the {K}err spacetime}, arXiv
  preprint arXiv:1903.03859 (2019).

\bibitem{larsblue15hidden}
Lars Andersson and Pieter Blue, \emph{Hidden symmetries and decay for the wave
  equation on the {K}err spacetime}, Annals of Mathematics \textbf{182} (2015),
  no.~3, 787--853.

\bibitem{larsblue15Maxwellkerr}
Lars Andersson and Pieter Blue, \emph{Uniform energy bound and asymptotics for the {M}axwell field on
  a slowly rotating {K}err black hole exterior}, Journal of Hyperbolic
  Differential Equations \textbf{12} (2015), no.~04, 689--743.

\bibitem{Jinhua17LinGraSchw}
Lars Andersson, Pieter Blue, and Jinhua Wang, \emph{Morawetz estimate for
  linearized gravity in {S}chwarzschild}, Ann. Henri Poincar\'{e} \textbf{21}
  (2020), 761--813.

\bibitem{angelopoulos2018late}
Yannis Angelopoulos, Stefanos Aretakis, and Dejan Gajic, \emph{Late-time
  asymptotics for the wave equation on spherically symmetric, stationary
  spacetimes}, Advances in Mathematics \textbf{323} (2018), 529--621.

\bibitem{angelopoulos2018vector}
Yannis Angelopoulos, Stefanos Aretakis, and Dejan Gajic, \emph{A vector field approach to almost-sharp decay for the wave
  equation on spherically symmetric, stationary spacetimes}, Annals of PDE
  \textbf{4} (2018), no.~2, 15.

\bibitem{angelopoulos2019logarithmic}
Yannis Angelopoulos, Stefanos Aretakis, and Dejan Gajic, \emph{Logarithmic corrections in the asymptotic expansion for the
  radiation field along null infinity}, Journal of Hyperbolic Differential
  Equations \textbf{16} (2019), no.~01, 1--34.
  
\bibitem{bo99}
Leor Barack and Amos Ori. \emph{Late-time decay of gravitational and electromagnetic perturbations along the event horizon}, Physical Review D 60.12 (1999): 124005.


\bibitem{batic2007scattering}
D.~Batic, \emph{Scattering for massive {D}irac fields on the {K}err metric},
  Journal of Mathematical Physics \textbf{48} (2007), no.~2, 022502.

\bibitem{blue08decayMaxSchw}
Pieter Blue, \emph{Decay of the {M}axwell field on the {S}chwarzschild
  manifold}, Journal of Hyperbolic Differential Equations \textbf{5} (2008),
  no.~04, 807--856.

\bibitem{bluesoffer03mora}
Pieter Blue and Avy Soffer, \emph{Semilinear wave equations on the
  {S}chwarzschild manifold {I}: {L}ocal decay estimates}, Advances in
  Differential Equations \textbf{8} (2003), no.~5, 595--614.

\bibitem{blue:soffer:integral}
Pieter Blue and Avy Soffer, \emph{A space--time integral estimate for a large data semi-linear
  wave equation on the {S}chwarzschild manifold}, Letters in Mathematical
  Physics \textbf{81} (2007), no.~3, 227--238.

\bibitem{boyer:lindquist:1967}
Robert~H. Boyer and Richard~W. Lindquist, \emph{Maximal analytic extension of
  the {K}err metric}, J. Mathematical Phys. \textbf{8} (1967), 265--281.

\bibitem{chandrasekhar1975linearstabSchw}
Subrahmanyan Chandrasekhar, \emph{On the equations governing the perturbations
  of the {S}chwarzschild black hole}, Proceedings of the Royal Society of
  London A: Mathematical, Physical and Engineering Sciences, vol. 343, The
  Royal Society, 1975, pp.~289--298.

\bibitem{chandrasekhar1976solution}
Subrahmanyan Chandrasekhar, \emph{The solution of {D}irac's equation in {K}err geometry},
  Proceedings of the Royal Society of London. A. Mathematical and Physical
  Sciences \textbf{349} (1976), no.~1659, 571--575.

\bibitem{chandbook1998}
Subrahmanyan Chandrasekhar, \emph{The mathematical theory of black holes}, Vol. \textbf{69}. Oxford University Press, 1998.

\bibitem{christodoulou1986global}
Demetrios Christodoulou, \emph{Global solutions of nonlinear hyperbolic
  equations for small initial data}, Communications on Pure and Applied
  Mathematics \textbf{39} (1986), no.~2, 267--282.

\bibitem{CK93global}
Demetrios Christodoulou and Sergiu Klainerman, \emph{The global gonlinear
  gtability of the {M}inkowski gpace}, Princeton Mathematical Series, vol.~41,
  Princeton University Press, Princeton, NJ, 1993.

\bibitem{dafermos2019boundedness}
Mihalis Dafermos, Gustav Holzegel, and Igor Rodnianski, \emph{Boundedness and
  decay for the {T}eukolsky equation on {K}err spacetimes {I}: the case $|a|\ll
  m $}, Annals of PDE \textbf{5} (2019), no.~1, 2.

\bibitem{dafermos2019linear}
Mihalis Dafermos, Gustav Holzegel, and Igor Rodnianski, \emph{The linear stability of the {S}chwarzschild solution to
  gravitational perturbations}, Acta Mathematica \textbf{222} (2019), no.~1,
  1--214.

\bibitem{dafermos2017interior}
Mihalis Dafermos and Jonathan Luk, \emph{The interior of dynamical vacuum black
  holes {I}: {T}he $c^0$-stability of the {K}err {C}auchy horizon}, arXiv
  preprint arXiv:1710.01722 (2017).

\bibitem{dafrod09red}
Mihalis Dafermos and Igor Rodnianski, \emph{The red-shift effect and radiation
  decay on black hole spacetimes}, Communications on Pure and Applied
  Mathematics \textbf{62} (2009), no.~7, 859--919.

\bibitem{dafermos2009new}
Mihalis Dafermos and Igor Rodnianski, \emph{A new physical-space approach to decay for the wave equation
  with applications to black hole spacetimes}, XVIth International Congress On
  Mathematical Physics, World Scientific, 2010, pp.~421--432.

\bibitem{dafermos2011bdedness}
Mihalis Dafermos and Igor Rodnianski, \emph{A proof of the uniform boundedness of solutions to the wave
  equation on slowly rotating {K}err backgrounds}, Inventiones mathematicae
  \textbf{185} (2011), no.~3, 467--559.

\bibitem{dafermos2016decay}
Mihalis Dafermos, Igor Rodnianski, and Yakov Shlapentokh-Rothman, \emph{Decay
  for solutions of the wave equation on {K}err exterior spacetimes {III}: The
  full subextremalcase $|a|< m$}, Annals of Mathematics \textbf{183} (2016),
  no.~3, 787--913.

\bibitem{Dong2019HiggsBoson}
Shijie Dong, Philippe~G. LeFloch, and Zoe Wyatt, \emph{Global evolution of the
  ${U}(1)$ {H}iggs {B}oson: nonlinear stability and uniform energy bounds},
  arXiv preprint arXiv:1902.02685 (2019).

\bibitem{donninger2011proof}
Roland Donninger, Wilhelm Schlag, and Avy Soffer, \emph{A proof of {P}rice's
  law on {S}chwarzschild black hole manifolds for all angular momenta},
  Advances in Mathematics \textbf{226} (2011), no.~1, 484--540.

\bibitem{donninger2012pointwise}
Roland Donninger, Wilhelm Schlag, and Avy Soffer, \emph{On pointwise decay of linear waves on a {S}chwarzschild black
  hole background}, Communications in Mathematical Physics \textbf{309} (2012),
  no.~1, 51--86.

\bibitem{E82E}
Michael Eastwood and Paul Tod. \emph{Edth-a differential operator on the sphere},  Mathematical Proceedings of the Cambridge Philosophical Society. Vol. 92. No. 2. Cambridge University Press, 1982.

\bibitem{fackerell:ipser:EM}
E.~D. {Fackerell} and J.~R. {Ipser}, \emph{{Weak Electromagnetic Fields Around
  a Rotating Black Hole}}, Phys. Rev. D. \textbf{5} (1972), 2455--2458.

\bibitem{finster2002decay}
Felix Finster, Niky Kamran, Joel Smoller, and Shing-Tung Yau, \emph{Decay rates
  and probability estimates for massive {D}irac particles in the
  {K}err--{N}ewman black hole geometry}, Communications in Mathematical
  Physics \textbf{230} (2002), no.~2, 201--244.

\bibitem{finster2003long}
Felix Finster, Niky Kamran, Joel Smoller, and Shing-Tung Yau, \emph{The long-time dynamics of {D}irac particles in the
  {K}err--{N}ewman black hole geometry}, Advances in Theoretical and
  Mathematical Physics \textbf{7} (2003), no.~1, 25--52.

\bibitem{finster2016linear}
Felix Finster and Joel Smoller, \emph{Linear stability of the non-extreme
  {K}err black hole}, Advances in Theoretical and Mathematical Physics
  \textbf{21} (2017), no.~8, 1991--2085.

\bibitem{Giorgi2019linearRNfullcharge}
Elena Giorgi, \emph{The linear stability of {R}eissner-{N}ordstr\" om
  spacetime: the full subextremal range}, arXiv preprint arXiv:1910.05630
  (2019).

\bibitem{gudapati2019conserved}
Nishanth Gudapati, \emph{A conserved energy for axially symmetric
  {N}ewman--{P}enrose--{M}axwell scalars on {K}err black holes}, Proceedings of
  the Royal Society A \textbf{475} (2019), no.~2221, 20180686.

\bibitem{hafner2019linear}
Dietrich H{\"a}fner, Peter Hintz, and Andr{\'a}s Vasy, \emph{Linear stability
  of slowly rotating {K}err black holes}, arXiv preprint arXiv:1906.00860
  (2019).

\bibitem{hafner20scattering}
Dietrich H\"{a}fner, Mokdad Mokdad, and Jean-Philippe Nicolas, \emph{Scattering
  theory for {D}irac fields inside a {R}eissner--{N}ordstr\"{o}m--type black hole},
  arXiv preprint arXiv:2007.16139 (2020).

\bibitem{hafner2004scattering}
Dietrich H{\"a}fner and Jean-Philippe Nicolas, \emph{Scattering of massless
  {D}irac fields by a {K}err black hole}, Reviews in Mathematical Physics
  \textbf{16} (2004), no.~01, 29--123.

\bibitem{HHtetrad72}
S.~W. {Hawking} and J.~B. {Hartle}, \emph{{Energy and angular momentum flow
  into a black hole}}, Communications in Mathematical Physics \textbf{27}
  (1972), 283--290.

\bibitem{hintz2020sharp}
Peter Hintz, \emph{A sharp version of {P}rice's law for wave decay on
  asymptotically flat spacetimes}, arXiv preprint arXiv:2004.01664 (2020).

\bibitem{HintzKds2018}
Peter Hintz and Andr\'{a}s Vasy, \emph{The global non-linear stability of the
  {K}err-de {S}itter family of black holes}, Acta Mathematica \textbf{220}
  (2018), 1--206.

\bibitem{hung2017linearstabSchw}
Pei-Ken Hung, Jordan Keller, and Mu-Tao Wang, \emph{Linear stability of
  {S}chwarzschild spacetime: The {C}auchy problem of metric coefficients},
  arXiv preprint arXiv:1702.02843 (2017).

\bibitem{kla2015globalstabwavemapKerr}
Alexandru~D Ionescu and Sergiu Klainerman, \emph{On the global stability of the
  wave-map equation in {K}err spaces with small angular momentum}, Annals of
  PDE \textbf{1} (2015), no.~1, 1.

\bibitem{johnson2019linear}
Thomas~William Johnson, \emph{The linear stability of the {S}chwarzschild
  solution to gravitational perturbations in the generalised wave gauge},
  Annals of PDE \textbf{5} (2019), no.~2, 13.

\bibitem{kay1987linear}
Bernard~S Kay and Robert~M Wald, \emph{Linear stability of {S}chwarzschild
  under perturbations which are non-vanishing on the bifurcation $2$-sphere},
  Classical and Quantum Gravity \textbf{4} (1987), no.~4, 893.

\bibitem{kerr63}
R.~Kerr, \emph{Gravitational field of a spinning mass as an example of
  algebraically special metrics}, Physical Review Letters \textbf{11} (1963),
  no.~5, 237.

\bibitem{Kinnersley1969tetradForTypeD}
William Kinnersley, \emph{Type {D} vacuum metrics}, Journal of Mathematical
  Physics \textbf{10} (1969), no.~7, 1195--1203.

\bibitem{klainerman1986null}
Sergiu Klainerman, \emph{The null condition and global existence to nonlinear
  wave equations}, Nonlinear Systems of Partial Differential Equations in
  Applied Mathematics, Part \textbf{1} (1986), 293--326.

\bibitem{lindblad2010global}
Hans Lindblad and Igor Rodnianski, \emph{The global stability of {M}inkowski
  space-time in harmonic gauge}, Annals of Mathematics (2010), 1401--1477.

\bibitem{LukOh2019SCC}
Jonathan Luk and Sung-Jin Oh, \emph{Strong cosmic censorship in spherical
  symmetry for two-ended asymptotically flat initial data {I}. {T}he interior
  of the black hole region}, Annals of Mathematics \textbf{190} (2019), no.~1,
  1--111.

\bibitem{luk2019strong}
Jonathan Luk and Sung-Jin Oh, \emph{Strong cosmic censorship in spherical symmetry for two-ended
  asymptotically flat initial data {II}: The exterior of the black hole
  region}, Annals of PDE \textbf{5} (2019), no.~1, 6.

\bibitem{Ma20almost}
Siyuan Ma, \emph{Almost {P}rice's law in {S}chwarzschild and decay estimates in
  {K}err for {M}axwell field}, arXiv preprint arXiv:2005.12492 (2020).

\bibitem{Ma2017Maxwell}
Siyuan Ma, \emph{Uniform energy bound and {M}orawetz estimate for extreme
  components of spin fields in the exterior of a slowly rotating {K}err black
  hole {I}: {M}axwell field}, Annales Henri Poincar\'{e} -- A Journal of
  Theoretical and Mathematical Physics \textbf{21} (2020), no.~3.

\bibitem{Ma17spin2Kerr}
Siyuan Ma, \emph{Uniform energy bound and {M}orawetz estimate for extreme
  components of spin fields in the exterior of a slowly rotating {K}err black
  hole {II}: linearized gravity}, Communications in Mathematical Physics
  \textbf{377} (2020), 2489--2551.

\bibitem{MMTT}
Jeremy Marzuola, Jason Metcalfe, Daniel Tataru, and Mihai Tohaneanu,
  \emph{Strichartz estimates on {S}chwarzschild black hole backgrounds},
  Communications in Mathematical Physics \textbf{293} (2010), no.~1, 37.

\bibitem{mason2012peeling}
Lionel~J Mason and Jean-Philippe Nicolas, \emph{Peeling of {D}irac and
  {M}axwell fields on a {S}chwarzschild background}, Journal of Geometry and
  Physics \textbf{62} (2012), no.~4, 867--889.

\bibitem{melnyk2003scattering}
Fabrice Melnyk, \emph{Scattering on {R}eissner--{N}ordstr{\o}m metric for
  massive charged spin $1/2$ fields}, Annales Henri Poincar{\'e}, vol.~4,
  Springer, 2003, pp.~813--846.

\bibitem{melnyk2004hawking}
Fabrice Melnyk, \emph{The {H}awking {E}ffect for {S}pin $1/2$ {F}ields},
  Communications in Mathematical Physics \textbf{244} (2004), no.~3.

\bibitem{metcalfe2012price}
Jason Metcalfe, Daniel Tataru, and Mihai Tohaneanu, \emph{Price's law on
  nonstationary space--times}, Advances in Mathematics \textbf{230} (2012),
  no.~3, 995--1028.

\bibitem{metcalfe2017pointwise}
Jason Metcalfe, Daniel Tataru, and Mihai Tohaneanu, \emph{Pointwise decay for the {M}axwell field on black hole
  space--times}, Advances in Mathematics \textbf{316} (2017), 53--93.

\bibitem{morawetz1968time}
Cathleen~S Morawetz, \emph{Time decay for the nonlinear {K}lein-{G}ordon
  equation}, Proceedings of the Royal Society of London A: Mathematical,
  Physical and Engineering Sciences, vol. 306, The Royal Society, 1968,
  pp.~291--296.

\bibitem{moschidis2016r}
Georgios Moschidis, \emph{The $r^p$-weighted energy method of {D}afermos and
  {R}odnianski in general asymptotically flat spacetimes and applications},
  Annals of PDE \textbf{2} (2016), no.~1, 6.

\bibitem{mourre1981absence}
E.~Mourre, \emph{Absence of singular continuous spectrum for certain
  self-adjoint operators}, Communications in Mathematical Physics \textbf{78}
  (1981), no.~3, 391--408.

\bibitem{nicolas1995scattering}
Jean-Philippe Nicolas, \emph{Scattering of linear {D}irac fields by a
  spherically symmetric black-hole}, Annales de l'IHP Physique th{\'e}orique,
  vol.~62, 1995, pp.~145--179.

\bibitem{nicolas2019peeling}
Jean-Philippe Nicolas and Truong~Xuan Pham, \emph{Peeling on {K}err spacetime:
  linear and semi-linear scalar fields}, Annales Henri Poincar{\'e}, vol.~20,
  Springer, 2019, pp.~3419--3470.

\bibitem{page1976dirac}
Don~N Page, \emph{Dirac equation around a charged, rotating black hole},
  Physical Review D \textbf{14} (1976), no.~6, 1509.

\bibitem{pasqualotto2019spin}
Federico Pasqualotto, \emph{The spin $\pm 1$ {T}eukolsky equations and the
  {M}axwell system on {S}chwarzschild}, Annales Henri Poincar{\'e}, vol.~20,
  Springer, 2019, pp.~1263--1323.

\bibitem{penroserindlerI}
Roger Penrose and Wolfgang Rindler, \emph{Spinors and space-time: Volume 1,
  {T}wo-spinor calculus and relativistic fields}, vol.~1, Cambridge University
  Press, 1984.

\bibitem{Price1972SchwScalar}
Richard~H Price, \emph{Nonspherical perturbations of relativistic gravitational
  collapse. {I}. {S}calar and gravitational perturbations}, Physical Review D
  \textbf{5} (1972), no.~10, 2419.

\bibitem{Price1972SchwIntegerSpin}
Richard~H Price, \emph{Nonspherical perturbations of relativistic gravitational
  collapse. {II}. {I}nteger-spin, zero-rest-mass fields}, Physical Review D
  \textbf{5} (1972), no.~10, 2439.

\bibitem{price2004late}
Richard~H Price and Lior~M Burko, \emph{Late time tails from momentarily
  stationary, compact initial data in {S}chwarzschild spacetimes}, Physical
  Review D \textbf{70} (2004), no.~8, 084039.

\bibitem{ReggeWheeler1957}
Tullio Regge and John~A Wheeler, \emph{Stability of a {Schwarzschild}
  singularity}, Physical Review \textbf{108} (1957), no.~4, 1063.

\bibitem{roken2017massive}
Christian R{\"o}ken, \emph{The massive {D}irac equation in {K}err geometry:
  separability in {E}ddington--{F}inkelstein-type coordinates and asymptotics},
  General Relativity and Gravitation \textbf{49} (2017), no.~3, 39.

\bibitem{schlue2013decay}
Volker Schlue, \emph{Decay of linear waves on higher-dimensional
  {S}chwarzschild black holes}, Analysis \& PDE \textbf{6} (2013), no.~3,
  515--600.

\bibitem{schw1916}
K.~Schwarzschild, \emph{{\"U}ber das {G}ravitationsfeld einer {K}ugel aus
  inkompressibler {F}l{\"u}ssigkeit nach der {E}insteinschen {T}heorie},
  Sitzungsberichte der K{\"o}niglich Preussischen Akademie der Wissenschaften
  zu Berlin, Phys.-Math. Klasse, 424-434 (1916), 1916.

\bibitem{smoller2012asymptotic}
Joel Smoller and Chunjing Xie, \emph{Asymptotic behavior of massless {D}irac
  waves in {S}chwarzschild geometry}, Annales Henri Poincare, vol.~13,
  Springer, 2012, pp.~943--989.

\bibitem{sterbenz2015decayMaxSphSym}
Jacob Sterbenz and Daniel Tataru, \emph{Local energy decay for {M}axwell fields
  part {I}: Spherically symmetric black-hole backgrounds}, International
  Mathematics Research Notices \textbf{2015} (2015), no.~11.

\bibitem{tataru2013local}
Daniel Tataru, \emph{Local decay of waves on asymptotically flat stationary
  space-times}, American Journal of Mathematics \textbf{135} (2013), no.~2,
  361--401.

\bibitem{tataru2011localkerr}
Daniel Tataru and Mihai Tohaneanu, \emph{A local energy estimate on {K}err
  black hole backgrounds}, International Mathematics Research Notices
  \textbf{2011} (2011), no.~2, 248--292.

\bibitem{Teukolsky1973I}
S.~A. {Teukolsky}, \emph{{Perturbations of a Rotating Black Hole. I.
  Fundamental Equations for Gravitational, Electromagnetic, and Neutrino-Field
  Perturbations}}, Astrophysical J. \textbf{185} (1973), 635--648.

\bibitem{tohaneanu2012strichartz}
Mihai Tohaneanu, \emph{Strichartz estimates on {K}err black hole backgrounds},
  Transactions of the American Mathematical Society \textbf{364} (2012), no.~2,
  689--702.

\bibitem{truong2020peeling}
Pham Truong~Xuan, \emph{Peeling of {D}irac field on {K}err spacetime}, Journal
  of Mathematical Physics \textbf{61} (2020), no.~3, 032501.

\bibitem{wald1979note}
Robert~M Wald, \emph{Note on the stability of the {S}chwarzschild metric},
  Journal of Mathematical Physics \textbf{20} (1979), no.~6, 1056--1058.

\end{thebibliography}
\end{document}